\documentclass{amsart}

\usepackage{amssymb}
\usepackage{hyperref}

\newtheorem{thm}{Theorem}[section]
\newtheorem{cor}[thm]{Corollary}
\newtheorem{lem}[thm]{Lemma}
\newtheorem{prop}[thm]{Proposition}
\theoremstyle{definition}
\newtheorem{defn}[thm]{Definition}
\theoremstyle{remark}
\newtheorem{rem}[thm]{Remark}

\numberwithin{equation}{section}
\numberwithin{thm}{section}
\newcommand{\Z}{{\mathbb{Z}}}
\newcommand{\C}{{\mathbb{C}}}
\newcommand{\R}{{\mathbb{R}}}
\newcommand{\N}{{\mathbb{N}}}
\newcommand{\HH}{{\mathbb{H}}}
\newcommand{\OO}{{\mathcal O}}

\DeclareMathOperator{\diam}{diam}
\DeclareMathOperator{\dist}{dist}
\DeclareMathOperator{\supp}{supp}
\DeclareMathOperator{\Tr}{Tr}
\let\Re=\undefined\DeclareMathOperator{\Re}{Re}
\let\Im=\undefined\DeclareMathOperator{\Im}{Im}
\DeclareMathOperator*{\wlim}{w-lim}
\DeclareMathOperator*{\tlim}{\textnormal{l}\widetilde{\textnormal{\i m}}}

\newcommand{\lng}{\langle}
\newcommand{\rng}{\rangle}
\newcommand{\eps}{{\varepsilon}}

\newcommand{\Id}{\mathrm{Id}}

\newcommand{\propagateomega}{{e^{it\Delta_{\Omega}}}}

\newcommand{\lsm}{\lesssim}
\newcommand{\ld}{\Delta_{\Omega}}
\newcommand{\po}{P^{\Omega}}

\newcommand{\lpo}{L^p(\Omega)}

\newcommand{\hd}{\dot H^1_D(\Omega)}

\newcommand{\hr}{\dot H^1(\R^3)}
\newcommand{\lt}{L_{t,x}^{10}(\R\times\Omega)}
\newcommand{\ltr}{L_{t,x}^{10}(\R\times\R^3)}
\newcommand{\tz}{2^{\Z}}
\newcommand{\prd}{e^{it\ld}}
\newcommand{\prr}{e^{it\Delta}}
\newcommand{\prdn}{e^{it_n\ld}}

\newcommand{\lon}{\Delta_{\Omega_n}}
\newcommand{\on}{\Omega_n}
\newcommand{\pnno}{P_{N_n}^{\Omega}}
\newcommand{\est}{\eps^4+t^2}
\newcommand{\xp}{x^{\perp}}
\newcommand{\logeps}{\log(\tfrac 1\eps)}
\newcommand{\llogeps}{\log\log(\tfrac 1\eps)}

\newcommand{\tildphi}{{\tilde\phi}}

\newcommand{\qtq}[1]{\quad\text{#1}\quad}

\newcounter{smalllist}
\newenvironment{SL}{\begin{list}{{\rm(\roman{smalllist})\hss}}{%
\setlength{\topsep}{0mm}\setlength{\parsep}{0mm}\setlength{\itemsep}{0mm}%
\setlength{\labelwidth}{2.0em}\setlength{\itemindent}{2.5em}\setlength{\leftmargin}{0em}\usecounter{smalllist}%
}}{\end{list}}

\newenvironment{CI}{\begin{list}{{\ $\bullet$\ }}{%
\setlength{\topsep}{0mm}\setlength{\parsep}{0mm}\setlength{\itemsep}{0mm}%
\setlength{\labelwidth}{1.5em}\setlength{\itemindent}{0em}\setlength{\leftmargin}{1.5em}%
\setlength{\labelsep}{0mm} }}{\end{list}}

\title[Energy-critical NLS outside a convex obstacle]
{Quintic NLS in the exterior of a strictly convex obstacle}

\author[R.Killip]{Rowan Killip}
\address{Department of Mathematics, UCLA}
\email{killip@math.ucla.edu}

\author[M.Visan]{Monica Visan}
\address{Department of Mathematics, UCLA}
\email{visan@math.ucla.edu}

\author[X. Zhang]{Xiaoyi Zhang}
\address{Department of Mathematics, University of Iowa, and Chinese Academy of Science, Beijing}
\email{zh.xiaoyi@gmail.com}
{\normalsize }

\allowdisplaybreaks

\begin{document}

\begin{abstract}
We consider the defocusing energy-critical nonlinear Schr\"odinger equation in the exterior of a smooth compact strictly convex obstacle in three dimensions.  For the initial-value problem with Dirichlet boundary condition we prove global well-posedness and scattering for all initial data in the energy space.
\end{abstract}

\maketitle

\tableofcontents

%%%%%%%%%%%%%%%%%%%%%%%%%%%%%%%%%%%%%%%%%%%%%%%%%%%%%%%%%%%%%%%%%%%%%%%%%%%%%%%%%%%%%%%%%%%%%%%%%%%%%%%%%%%%%%%%%%%%%%%%%%%%%%%%%%%%%%
\section{Introduction}\label{S:Introduction}
%%%%%%%%%%%%%%%%%%%%%%%%%%%%%%%%%%%%%%%%%%%%%%%%%%%%%%%%%%%%%%%%%%%%%%%%%%%%%%%%%%%%%%%%%%%%%%%%%%%%%%%%%%%%%%%%%%%%%%%%%%%%%%%%%%%%%%

We consider the defocusing energy-critical NLS in the exterior domain $\Omega$ of a smooth compact strictly convex obstacle in $\R^3$
with Dirichlet boundary conditions:
\begin{align}\label{nls}
\begin{cases}
i u_t+\Delta u=|u|^4 u, \\
u(0,x)=u_0(x),\\
u(t,x)|_{x\in \partial\Omega}=0.
\end{cases}
\end{align}
Here $u:\R\times\Omega\to\C$ and the initial data $u_0(x)$ will only be required to belong to the energy space,
which we will describe shortly.

The proper interpretation of the \emph{linear} Schr\"odinger equation with such boundary conditions was an early difficulty in mathematical quantum mechanics, but is now well understood.  Let us first whisk through these matters very quickly; see \cite{Kato:pert,RS1,RS2} for further information.

We write $-\Delta_\Omega$ for the Dirichlet Laplacian on $\Omega$.  This is the unique self-adjoint operator acting on $L^2(\Omega)$ associated with the closed quadratic form
$$
Q: H^1_0(\Omega) \to [0,\infty) \qtq{via} Q(f):=\int_\Omega \overline{\nabla f(x)} \cdot \nabla f(x) \,dx.
$$
The operator $-\Delta_\Omega$ is unbounded and positive semi-definite.  All functions of this operator will be interpreted via the Hilbert-space functional calculus.  In particular, $e^{it\Delta_\Omega}$ is unitary and provides the fundamental solution to the linear Schr\"odinger equation $i u_t+\Delta_\Omega u=0$, even when the naive notion of the boundary condition $u(t,x)|_{x\in\partial\Omega}=0$ no longer makes sense.

We now define the natural family of homogeneous Sobolev spaces associated to the operator $-\Delta_\Omega$ via the functional calculus:

\begin{defn}[Sobolev spaces]
For $s\geq 0$ and $1<p<\infty$, let $\dot H_D^{s,p}(\Omega)$ denote the completion of $C^\infty_c(\Omega)$ with respect to the norm
$$
\| f \|_{\dot H_D^{s,p}(\Omega)} := \| (-\Delta_\Omega)^{s/2} f \|_{L^p(\Omega)}.
$$
Omission of the index $p$ indicates $p=2$.
\end{defn}

For $p=2$ and $s=1$, this coincides exactly with the definition of $\dot H^1_0(\Omega)$.  For other values of parameters, the definition of $\dot H^{s,p}_D(\Omega)$ deviates quite sharply from the classical definitions of Sobolev spaces on domains, such as $\dot H^{s,p}(\Omega)$, $\dot H^{s,p}_0(\Omega)$, and the Lions--Magenes spaces $\dot H^{s,p}_{00}(\Omega)$.  Recall that all of these spaces are defined via the Laplacian in the whole space and its fractional powers.

For bounded domains $\OO\subseteq\R^d$, the relation of $\dot H^{s,p}_D(\OO)$ to the classical Sobolev spaces has been thoroughly investigated.  See, for instance, the review \cite{Seeley:ICM} and the references therein.   The case of exterior domains is much less understood; moreover, new subtleties appear.  For example, for bounded domains $\dot H^{1,p}_D(\OO)$ is equivalent to the completion of $C^\infty_c(\OO)$ in the space $\dot H^{1,p}(\R^d)$.  However, this is no longer true in the case of exterior domains; indeed, it was observed in \cite{LSZ} that this equivalence fails for $p>3$ in the exterior of the unit ball in $\R^3$, even in the case of spherically symmetric functions.

As the reader will quickly appreciate, little can be said about the problem \eqref{nls} without some fairly thorough understanding of the mapping properties of functions of $-\Delta_\Omega$ and of the Sobolev spaces $\dot H_D^{s,p}(\Omega)$, in particular.  The analogue of the Mikhlin multiplier theorem is known for this operator and it is possible to develop a Littlewood--Paley theory on this basis; see \cite{IvanPlanch:square,KVZ:HA} for further discussion.  To obtain nonlinear estimates, such as product and chain rules in $\dot H_D^{s,p}(\Omega)$, we use the main result of \cite{KVZ:HA}, which we record as Theorem~\ref{T:Sob equiv} below.  By proving an equivalence between $\dot H_D^{s,p}(\Omega)$ and the classical Sobolev spaces (for a restricted range of exponents), Theorem~\ref{T:Sob equiv} allows us to import such nonlinear estimates directly from the Euclidean setting.

After this slight detour, let us return to the question of the proper interpretation of a solution to \eqref{nls} and the energy space.  For the linear Schr\"odinger equation with Dirichlet boundary conditions, the energy space is the domain of the quadratic form associated to the Dirichlet Laplacian, namely, $\dot H^1_D(\Omega)$.  For the nonlinear problem \eqref{nls}, the energy space is again $\dot H^1_D(\Omega)$ and the energy functional is given by
\begin{align}\label{energy}
E(u(t)):=\int_{\Omega} \tfrac12 |\nabla u(t,x)|^2 + \tfrac16 |u(t,x)|^6\, dx.
\end{align}
Note that the second summand here, which is known as the potential energy, does not alter the energy space by virtue of Sobolev embedding, more precisely, the embedding $\dot H^1_D(\Omega)\hookrightarrow L^6(\Omega)$.

The PDE \eqref{nls} is the natural Hamiltonian flow associated with the energy functional \eqref{energy}.  Correspondingly, one would expect this energy to be conserved by the flow.  This is indeed the case, provided we restrict ourselves to a proper notion of solution.

\begin{defn}[Solution]\label{D:solution}
Let $I$ be a time interval containing the origin.  A function $u: I \times \Omega \to \C$ is called a (strong) \emph{solution}
to \eqref{nls} if it lies in the class $C_t(I'; \dot H^1_D(\Omega)) \cap L_t^{5}L_x^{30}(I'\times\Omega)$ for every compact subinterval $I'\subseteq I$
and it satisfies the Duhamel formula
\begin{equation}\label{E:duhamel}
u(t) = e^{it\Delta_\Omega} u_0  - i \int_0^t e^{i(t-s)\Delta_\Omega} |u(s)|^4 u(s)\, ds,
\end{equation}
for all $t \in I$.
\end{defn}

For brevity we will sometimes refer to such functions as solutions to $\text{NLS}_\Omega$.  It is not difficult to verify that strong solutions conserve energy.

We now have sufficient preliminaries to state the main result of this paper.

\begin{thm}\label{T:main}
Let $u_0\in \dot H^1_D(\Omega)$. Then there exists a unique strong solution $u$ to \eqref{nls} which is global in time and satisfies
\begin{align}\label{E:T:main}
\iint_{\R\times\Omega} |u(t,x)|^{10} \,dx\, dt\le C(E(u)).
\end{align}
Moreover, $u$ scatters in both time directions, that is, there exist asymptotic states $u_\pm\in\dot H^1_D(\Omega)$ such that
\begin{align*}
\|u(t) - e^{it\Delta_\Omega}u_\pm\|_{\dot H^1_D(\Omega)}\to 0 \qtq{as} t\to\pm\infty.
\end{align*}
\end{thm}

There is much to be said in order to give a proper context for this result.  In particular, we would like to discuss the defocusing NLS in $\R^3$ with general power nonlinearity:
\begin{equation}\label{GNLS}
i u_t+\Delta u=|u|^p u.
\end{equation}
A key indicator for the local behaviour of solutions to this equation is the scaling symmetry
\begin{align}\label{GNLSrescale}
u(t,x)\mapsto u^\lambda(t,x):=\lambda^{\frac2p} u(\lambda^2 t, \lambda x) \qtq{for any} \lambda>0,
\end{align}
which leaves the class of solutions to \eqref{GNLS} invariant.   Notice that when $p=4$ this rescaling
also preserves the energy associated with \eqref{GNLS}, namely,
$$
E(u(t)) = \int_{\R^3} \tfrac12|\nabla u(t,x)|^2+\tfrac1{p+2}|u(t,x)|^{p+2}\,dx.
$$
For this reason, the quintic NLS in three spatial dimensions is termed energy-critical.  The energy is the \emph{highest regularity} conservation law that is known for NLS; this has major consequences for the local and global theories for this equation when $p\geq 4$.  When $p>4$, the equation is ill-posed in the energy space; see \cite{CCT}.  For $p=4$, which is the focus of this paper, well-posedness in the energy space is delicate, as will be discussed below.

For $0\leq p<4$, the equation is called energy-subcritical.  Indeed, the energy strongly suppresses the short-scale behaviour of solutions, as can be read-off from its transformation under the rescaling \eqref{GNLSrescale}:
$$
E(u^\lambda) = \lambda^{\frac4p - 1} E(u).
$$
Accordingly, it is not very difficult to prove local well-posedness for initial data in $H^1(\R^3)$.   This follows by contraction mapping in Strichartz spaces and yields a local existence time that depends on the $H^1_x$ norm of the initial data.  Using the conservation of mass (= $L^2_x$-norm) and energy, global well-posedness follows immediately by iteration.  Notice that this procedure gives almost no information about the long-time behaviour of the solution.

The argument just described does not extend to $p=4$.  In this case, the local existence time cannot depend solely on the energy, which is a scale-invariant quantity. Nevertheless, a different form of local well-posedness was proved by Cazenave and Weissler \cite{cw0,cw1}, in which the local existence time depends upon the \emph{profile} of the initial data, rather than solely on its norm.  Therefore, the iteration procedure described above cannot be used to deduce global existence.  In fact, as the energy is the highest regularity conservation law that is known, global existence is non-trivial \emph{even} for Schwartz initial data.  In \cite{cw0, cw1}, the time of existence is shown to be positive via the monotone convergence theorem; on the basis of subsequent developments, we now understand that this time is determined by the spread of energy on the Fourier side.  In the case of the \emph{focusing} equation, the existence time obtained in these arguments is not fictitious; there are solutions with a fixed energy that blow up arbitrarily quickly.

The Cazenave--Weissler arguments also yield global well-posedness and scattering for initial data with \emph{small} energy, for both the focusing and defocusing equations.  Indeed, in this regime the nonlinearity can be treated perturbatively.    

The first key breakthrough for the treatment of the large-data energy-critical NLS was the paper \cite{borg:scatter}, which proved global well-posedness and scattering for spherically symmetric solutions in $\R^3$ and $\R^4$.  This paper introduced the induction on energy argument, which has subsequently become extremely influential in the treatment of dispersive equations at the critical regularity.  We will also be using this argument, so we postpone a further description until later.  The induction on energy method was further advanced by Colliander, Keel, Staffilani, Takaoka, and Tao in their proof \cite{CKSTT:gwp} of global well-posedness and scattering for the quintic NLS in $\R^3$, for all initial data in the energy space.  This result, which is the direct analogue of Theorem~\ref{T:main} for NLS in the whole space, will play a key role in the analysis of this paper.  Let us state it explicitly:

\begin{thm}[\cite{CKSTT:gwp}]\label{T:gopher}
Let $u_0\in \dot H^1(\R^3)$. Then there exists a unique strong solution $u$ to the quintic NLS in $\R^3$ which is global in time and satisfies
\begin{align*}
\iint_{\R\times\R^3} |u(t,x)|^{10} \,dx\, dt\le C(E(u)).
\end{align*}
Moreover, $u$ scatters in both time directions, that is, there exist asymptotic states $u_\pm\in\dot H^1(\R^3)$ such that
\begin{align*}
\|u(t) - e^{it\Delta_{\R^3}}u_\pm\|_{\dot H^1(\R^3)} \to 0 \qtq{as} t\to\pm\infty.
\end{align*}
\end{thm}

We will also be employing the induction on energy argument, but in the style pioneered by Kenig and Merle \cite{KenigMerle}.   The main result of this paper of Kenig and Merle was the proof of global well-posedness and scattering for the focusing energy-critical equation and data smaller than the soliton threshold.  This result was for spherically symmetric data and dimensions $3\leq d \leq 5$; currently, the analogous result for general data is only known in dimensions five and higher \cite{Berbec}.  The proof of Theorem~\ref{T:gopher} was revisited within this framework in \cite{KV:gopher}, which also incorporates innovations of Dodson \cite{Dodson:3+}.

Let us now turn our attention to the problem of NLS on exterior domains.  This is a very popular and challenging family of problems.  While we will discuss many contributions below, to get a proper sense of the effort expended in this direction one should also consult the many references therein.  In the Euclidean setting, the problem is invariant under space translations; this means that one may employ the full power of harmonic analytic tools.  Indeed, much of the recent surge of progress in the analysis of dispersive equations is based on the incorporation of this powerful technology.  Working on exterior domains breaks space translation invariance and so many of the tools that one could rely on in the Euclidean setting.  The companion paper \cite{KVZ:HA} allows us to transfer many basic harmonic analytic results from the Euclidean setting to that of exterior domains.  Many more subtle results, particularly related to the long-time behaviour of the propagator, require a completely new analysis; we will discuss examples of this below.

Working on exterior domains also destroys the scaling symmetry.  Due to the presence of a boundary, suitable scaling and space translations lead to the study of NLS in \emph{different} geometries.  While equations with broken symmetries have been analyzed before, the boundary causes the geometric changes in this paper to be of a more severe nature than those treated previously.  An additional new difficulty is that we must proceed without a dispersive estimate, which is currently unknown in this setting.  

Before we delve into the difficulties of the energy-critical problem in exterior domains, let us first discuss the energy-subcritical case.  The principal difficulty in this case has been to obtain Strichartz estimates (cf. Theorem~\ref{T:Strichartz}).  The first results in this direction hold equally well in interior and exterior domains.  There is a strong parallel between compact manifolds and interior domains, so we will also include some works focused on that case.

For both compact manifolds and bounded domains, one cannot expect estimates of the same form as for the Euclidean space.  Finiteness of the volume means that there can be no long-time dispersion of wave packets; there is simply nowhere for them to disperse to.  Indeed, in the case of the torus $\R^d/\Z^d$, solutions to the linear Schr\"odinger equation are periodic in time.  Because of this, all Strichartz estimates must be local in time.  Further, due to the existence of conjugate points for the geodesic flow, high frequency waves can reconcentrate; moreover, they can do so arbitrarily quickly.  Correspondingly, Strichartz estimates in the finite domain/compact manifold setting lose derivatives relative to the Euclidean case.  Nevertheless, the resulting Strichartz estimates are still strong enough to prove local (and so global) well-posedness, at least for a range of energy-subcritical nonlinearity exponents $p$.   See the papers \cite{BFHM:Polygon,BSS:PAMS,BSS:schrodinger,borg:torus,BGT:compact} and references therein for further information.

For exterior domains, the obstructions just identified no longer apply, at least in the case of non-trapping obstacles (we do not wish to discuss resonator cavities, or similar geometries).  Thus one may reasonably expect all Strichartz estimates to hold, just as in the Euclidean case.  There are many positive results in this direction, as will be discussed below; however, the full answer remains unknown, even for the exterior of a convex obstacle (for which there are no conjugate points).

In the Euclidean case, the explicit form of the propagator guarantees the following dispersive estimate:
\begin{equation}\label{E:EuclidDisp}
\| e^{it\Delta_{\R^d}} f \|_{L^\infty(\R^d)} \lesssim |t|^{-\frac d2} \| f\|_{L^1(\R^d)}, \qtq{for all} t\neq0.
\end{equation}
This and the unitary of the propagator on $L^2(\R^d)$ are all that is required to obtain all known Strichartz estimates.  For the basic estimates the argument is elementary; see, for example, \cite{gv:strichartz}.  The endpoint cases and exotic retarded estimates are more delicate; see \cite{Foschi,KeelTao, Vilela}.

It is currently unknown whether or not the dispersive estimate holds outside a convex obstacle, indeed, even for the exterior of a sphere.  The only positive result in this direction belongs to Li, Smith, and Zhang, \cite{LSZ}, who prove the dispersive estimate for spherically symmetric functions in the exterior of a sphere in $\R^3$.  Relying on this dispersive estimate and employing an argument of Bourgain \cite{borg:scatter} and Tao \cite{tao:radial}, these authors proved Theorem~\ref{T:main} for spherically symmetric initial data when $\Omega$ is the exterior of a sphere in $\R^3$.  

In due course, we will explain how the lack of a dispersive estimate outside convex obstacles is one of the major hurdles we needed to overcome in order to prove Theorem~\ref{T:main}.  Note that the dispersive estimate will not hold outside a generic non-trapping obstacle, since concave portions of the boundary can act as mirrors and refocus wave packets.

Even though the question of dispersive estimates outside convex obstacles is open, global in time Strichartz estimates are known to hold.  Indeed, in \cite{Ivanovici:Strichartz}, Ivanovici proves all classical Strichartz estimates except the endpoint cases.  Her result will be crucial in what follows and is reproduced below as Theorem~\ref{T:Strichartz}.  We also draw the reader's attention to the related papers \cite{Anton08,BSS:schrodinger,BGT04,HassellTaoWunsch,PlanchVega,RobZuily,StaffTataru,Tataru:Strichartz}, as well as the references therein.

The key input for the proof of Strichartz estimates in exterior domains is the local smoothing estimate; one variant is given as Lemma~\ref{L:local smoothing} below.  In the Euclidean setting, this result can be proved via harmonic analysis methods (cf. \cite{ConsSaut,Sjolin87,Vega88}).  For the exterior of a convex obstacle, the usual approach is the method of positive commutators, which connects it to both Kato smoothing (cf. \cite[\S XIII.7]{RS4}) and the Morawetz identity; this is the argument used to prove Lemma~\ref{L:local smoothing} here.  Local smoothing is also known to hold in the exterior of a non-trapping obstacle; see \cite{BGT04}.

The local smoothing estimate guarantees that wave packets only spend a bounded amount of time next to the obstacle.  This fact together with the fact that Strichartz estimates hold in the whole space can be used to reduce the problem of proving Strichartz inequalities to the local behaviour near the obstacle, locally in time.  Using this argument, Strichartz estimates have been proved for merely non-trapping obstacles; for further discussion see \cite{BSS:schrodinger, BGT04, IvanPlanch:IHP, PlanchVega, StaffTataru}.

While both local smoothing and Strichartz estimates guarantee that wave packets can only concentrate for a bounded amount of time, they do not guarantee that this period of time is one contiguous interval.  In the context of a large-data nonlinear problem, this is a severe handicap when compared to the dispersive estimate:  Once a wave packet begins to disperse, the nonlinear effects are reduced and the evolution is dominated by the linear part of the equation.  If this evolution causes the wave packet to refocus, then nonlinear effects will become strong again.  These nonlinear effects are very hard to control and one must fear the possibility that when the wave packet final breaks up again we find ourselves back at the beginning of the scenario we have just been describing.  Such an infinite loop is inconsistent with scattering and global spacetime bounds.  In Section~\ref{S:Linear flow convergence} we will prove a new kind of convergence result that plays the role of a dispersive estimate in precluding such periodic behaviour.

The next order of business is to describe what direct information the existing Strichartz estimates give us toward the proof of Theorem~\ref{T:main}.  This is how we shall begin the

\subsection{Outline of the proof}

For small initial data, the nonlinearity can be treated perturbatively, provided one has the right linear estimates, of course!  In this way, both \cite{Ivanovici:Strichartz} and \cite{BSS:schrodinger} use the Strichartz inequalities they prove to obtain small energy global well-posedness and scattering for $\text{NLS}_\Omega$.  Actually, there is one additional difficulty that we have glossed over here, namely, estimating the derivative of the nonlinearity.  Notice that in order to commute with the free propagator, the derivative in question must be the square root of the Dirichlet Laplacian (rather than simply the gradient).  In \cite{BSS:schrodinger} an $L^4_t L^\infty_x$ Strichartz inequality is proved, which allows the authors to use the equivalence of $\dot H^1_0$ and $\dot H^1_D$.  In \cite{IvanPlanch:square} a Littlewood--Paley theory is developed, which allows the use of Besov space arguments (cf. \cite{IvanPlanch:IHP}).  Indeed, the paper \cite{IvanPlanch:IHP} of Ivanovici and Planchon goes further, proving small data global well-posedness in the exterior of non-trapping obstacles.

The main result of \cite{KVZ:HA}, which is repeated as Theorem~\ref{T:Sob equiv} below, allows us to transfer the existing local well-posedness arguments directly from the Euclidean case.  Actually, a little care is required to ensure all exponents used lie within the regime where norms are equivalent; nevertheless, this can be done as documented in \cite{KVZ:HA}.  Indeed, this paper shows that our problem enjoys a strong form of continuous dependence, known under the rubric `stability theory';  see Theorem~\ref{T:stability}.  Colloquially, this says that every function that almost solves \eqref{nls} and has bounded spacetime norm lies very close to an actual solution to \eqref{nls}.  This is an essential ingredient in any induction on energy argument.

All the results just discussed are perturbative, in particular, they are blind to the sign of the nonlinearity.  As blowup can occur for the focusing problem, any large-data global theory must incorporate some deeply nonlinear ingredient which captures the dynamical effects of the sign of the nonlinearity.  At present, the only candidates for this role are the identities of Morawetz/virial type and their multi-particle (or interaction) counterparts.

Historically, the Morawetz identity was first introduced for the linear wave equation and soon found application in proving energy decay in exterior domain problems and in the study of the nonlinear wave equation; see \cite{Morawetz75}.  As noticed first by Struwe, this type of tool also provides the key non-concentration result to prove global well-posedness for the energy-critical wave equation in Euclidean spaces.  See the book \cite{ShatahStruwe} for further discussion and complete references.  More recently, this result (plus scattering) has been shown to hold outside convex obstacles \cite{SS10} and (without scattering) in interior domains \cite{BLP08}.  In both instances, the Morawetz identity provides the crucial non-concentration result.

There is a significant difference between the Morawetz identities for the nonlinear wave equation and the nonlinear Schr\"odinger equation, which explains why the solution of the well-posedness problem for the energy-critical NLS did not follow closely on the heels of that for the wave equation: \emph{scaling}.  In the wave equation case, the Morawetz identity has energy-critical scaling.  This ensures that the right-hand side of the inequality can be controlled in terms of the energy alone; it also underscores why it can be used to guarantee non-concentration of solutions.

The basic Morawetz inequality for solutions $u$ to the defocusing quintic NLS in $\R^3$ (see \cite{LinStrauss}) reads as follows:
$$
\frac{d\ }{dt} \int_{\R^3} \frac{x}{|x|} \cdot  2 \Im\bigl\{ \bar u(t,x) \nabla u(t,x)\bigr\} \,dx \geq \int _{\R^3} \frac{8|u|^6}{3|x|}\,dx.
$$
The utility of this inequality is best seen by integrating both sides over some time interval~$I$; together with Cauchy--Schwarz, this leads directly to
\begin{equation}\label{GNLSmor}
\int_I \int _{\R^3} \frac{|u|^6}{|x|}\,dx \,dt \lesssim \| u \|_{L^\infty_t L^2_x (I\times\R^3)} \| \nabla u \|_{L^\infty_t L^2_x (I\times\R^3)}.
\end{equation}
Obviously the right-hand side cannot be controlled solely by the energy; indeed, the inequality has the scaling of $\dot H^{1/2}$.  Nevertheless, the right-hand side can be controlled by the conservation of both mass and energy; this was one of the key ingredients in the proof of scattering for the inter-critical problem (i.e. $\frac43<p<4$) in \cite{GinibreVelo}.  However, at both the mass-critical endpoint $p=\frac43$ and energy-critical endpoint $p=4$, solutions can undergo dramatic changes of scale without causing the mass or energy to diverge.  In particular, by simply rescaling an energy-critical solution as in \eqref{GNLSrescale} one may make the mass as small as one wishes.

Our comments so far have concentrated on RHS\eqref{GNLSmor}, but these concerns apply equally well to LHS\eqref{GNLSmor}.  Ultimately, the Morawetz identity together with mass and energy conservation are each consistent with a solution that blows up by focusing \emph{part} of its energy at a point, even at the origin.  A scenario where \emph{all} of the energy focuses at a single point would not be consistent with the conservation of mass.

The key innovation of Bourgain \cite{borg:scatter} was the induction on energy procedure, which allowed him to reduce the analysis of general solutions to $\text{NLS}_{\R^3}$ to those which have a clear intrinsic characteristic length scale (at least for the middle third of their evolution).  This length scale is time dependent.  In this paper we write $N(t)$ for the reciprocal of this length, which represents the characteristic frequency scale of the solution.  The fact that the solution lives at a single scale precludes the scenario described in the previous paragraph.  By using suitably truncated versions of the Morawetz identity (cf. Lemma~\ref{L:morawetz} below) and the mass conservation law, Bourgain succeeded in proving not only global well-posedness for the defocusing energy-critical NLS in $\R^3$, but also global $L^{10}_{t,x}$ spacetime bounds for the solution.

As noted earlier, the paper \cite{borg:scatter} treated the case of spherically symmetric solutions only.  The general case was treated in \cite{CKSTT:gwp}, which also dramatically advanced the induction on energy method, including reducing treatment of the problem to the study of solutions that not only live at a single scale $1/N(t)$, but are even well localized in space around a single point $x(t)$.  The dispersive estimate is needed to prove this strong form of localization.  Another key ingredient in \cite{CKSTT:gwp} was the newly introduced interaction Morawetz identity; see \cite{CKSTT:interact}.  As documented in \cite{CKSTT:gwp}, there are major hurdles to be overcome in frequency localizing this identity in the three dimensional setting.  In particular, the double Duhamel trick is needed to handle one of the error terms.  This relies \emph{crucially} on the dispersive estimate; thus, we are unable to employ the interaction Morawetz identity as a tool with which to tackle our Theorem~\ref{T:main}.

In four or more spatial dimensions, strong spatial localization is not needed to employ the interaction Morawetz identity.  This was first observed in \cite{RV, thesis:art}.  Building upon this, Dodson \cite{Dodson:obstacle} has shown how the interaction Morawetz identity can be applied to the energy-critical problem in the exterior of a convex obstacle in four dimensions.  He relies solely on frequency localization; one of the key tools that makes this possible is the long-time Strichartz estimates developed by him in the mass-critical Euclidean setting \cite{Dodson:3+} and adapted to the energy-critical setting in \cite{Visan:IMRN}.  For the three dimensional problem, these innovations do not suffice to obviate the need for a dispersive estimate, even in the Euclidean setting; see \cite{KV:gopher}.  

The variant of the induction on energy technique that we will use in this paper was introduced by Kenig and Merle in \cite{KenigMerle}.  This new approach has significantly streamlined the induction on energy paradigm; in particular, it has made it modular by completely separating the induction on energy portion from the rest of the argument.  It has also sparked a rapid and fruitful development of the method, which has now been applied successfully to numerous diverse PDE problems, including wave maps and the Navier--Stokes system.

Before we can discuss the new difficulties associated with implementing the induction on energy method to prove Theorem~\ref{T:main}, we must first explain what it is.  We will do so rather quickly; readers not already familiar with this technique, may benefit from the introduction to the subject given in the lecture notes \cite{ClayNotes}.  The argument is by contradiction.

Suppose Theorem~\ref{T:main} were to fail, which is to say that there is no function $C:[0,\infty)\to[0,\infty)$ so that \eqref{E:T:main} holds.  Then there must be some sequence of solutions $u_n:I_n\times\Omega\to\C$ so that $E(u_n)$ is bounded, but $S_{I_n}(u_n)$ diverges.  Here we introduce the notation
\begin{align*}
S_I(u):=\iint_{I\times\Omega}|u(t,x)|^{10} dx\, dt,
\end{align*}
which is known as the \emph{scattering size} of $u$ on the time interval $I$.

By passing to a subsequence, we may assume that $E(u_n)$ converges.  Moreover, without loss of generality, we may assume that the limit $E_c$ is the smallest number that can arise as a limit of $E(u_n)$ for solutions with $S_{I_n}(u_n)$ diverging.  This number is known as the \emph{critical energy}. It has the following equivalent interpretation:  If
\begin{align*}
L(E):=\sup\{S_I(u) : \, u:I\times\Omega\to \C\mbox{ such that } E(u)\le E\},
\end{align*}
where the supremum is taken over all solutions $u$ to \eqref{nls} defined on some spacetime slab $I\times\Omega$ and having energy $E(u)\le E$, then
\begin{align}\label{E:induct hyp}
L(E)<\infty \qtq{for} E<E_c \quad \qtq{and} \quad L(E)=\infty \qtq{for} E\ge E_c.
\end{align}
(The fact that we can write $E\geq E_c$ here rather than merely $E>E_c$ relies on the stability result Theorem~\ref{T:stability}.) This plays the role of the inductive hypothesis; it says that Theorem~\ref{T:main} is true for energies less than $E_c$.  The argument is called induction on energy precisely because this is then used (via an extensive argument) to show that $L(E_c)$ is finite and so obtain the sought-after contradiction.

Note that by the small-data theory mentioned earlier, we know that $E_c>0$.  Indeed, in the small-data regime, one obtains very good quantitative bounds on $S_\R(u)$.  As one might expect given the perturbative nature of the argument, the bounds are comparable to those for the linear flow; see \eqref{SbyE}. 

One would like to pass to the limit of the sequence of solutions $u_n$ to exhibit a solution $u_\infty$ that has energy $E_c$ and infinite scattering size.  Notice that by virtue of \eqref{E:induct hyp}, such a function would be a \emph{minimal energy blowup solution}.   This is a point of departure of the Kenig--Merle approach from \cite{borg:scatter,CKSTT:gwp}, which worked with merely almost minimal almost blowup solutions, in essence, the sequence $u_n$.

Proving the existence of such a minimal energy blowup solution will be the key difficulty in this paper; even in the Euclidean setting it is highly non-trivial.  In the Euclidean setting, existence was first proved by Keraani \cite{keraani-l2} for the (particularly difficult) mass-critical NLS; see also \cite{BegoutVargas,CarlesKeraani}.  Existence of a minimal blowup solution for the Euclidean energy-critical problem was proved by Kenig--Merle \cite{KenigMerle} (see also \cite{BahouriGerard,keraani-h1} for some ingredients), who were also the first to realize the value of this result for well-posedness arguments.

Let us first describe how the construction of minimal blowup solutions proceeds in the Euclidean setting.  We will then discuss the difficulties encountered on exterior domains and how we overcome these.  As $\text{NLS}_{\R^3}$ has the \emph{non-compact} symmetries of rescaling and spacetime translations, we cannot expect any subsequence of the sequence $u_n$ of almost minimal almost blowup solutions to converge.  This is a well-known dilemma in the calculus of variations and lead to the development of \emph{concentration compactness}.  In its original form, concentration compactness presents us with three possibilities: a subsequence converges after applying symmetry operations (the desired \emph{compactness} outcome); a subsequence splits into one or more bubbles (this is called \emph{dichotomy}); or the sequence is completely devoid of concentration (this is called \emph{vanishing}).  

The vanishing scenario is easily precluded.  If the solutions $u_n$  concentrate at no point in spacetime (at any scale), then we expect the nonlinear effects to be weak and so expect spacetime bounds to follow from perturbation theory and the Strichartz inequality (which provides spacetime bounds for linear solutions).  As uniform spacetime bounds for the solutions $u_n$ would contradict how these were chosen in the first place, this rules out the vanishing scenario.  Actually, this discussion is slightly too naive; one needs to show that failure to concentrate actually guarantees that the linear solution has small spacetime bounds, which then allows us to treat the nonlinearity perturbatively.    

The tool that allows us to complete the argument just described is an inverse Strichartz inequality (cf. Proposition~\ref{P:inverse Strichartz}), which says that linear flows can only have non-trivial spacetime norm if they contain at least one bubble of concentration.  Applying this result inductively to the functions $e^{it\Delta_{\R^3}}u_n(0)$, one finds all the bubbles of concentration in a subsequence of these linear solutions together with a remainder term.  This is expressed in the form of a \emph{linear profile decomposition} (cf. Theorem~\ref{T:LPD}).  Two regions of concentration are determined to be separate bubbles if their relative characteristic length scales diverge as $n\to\infty$, or if their spatial/temporal separation diverges relative to their characteristic scale; see~\eqref{E:LP5}.

If there is only one bubble and no remainder term, then (after a little untangling) we find ourselves in the desired compactness regime, namely, that after applying symmetry operations to $u_n(0)$ we obtain a subsequence that converges strongly in $\dot H^1(\R^3)$.  Moreover this limit gives initial data for the needed minimal blowup solution (cf. Theorem~\ref{T:mmbs}).  But what if we find ourselves in the unwanted dichotomy scenario where there is more than one bubble?  This is where the inductive hypothesis comes to the rescue, as we will now explain.
  
To each profile in the linear profile decomposition, we associate a nonlinear profile, which is a solution to $\text{NLS}_{\R^3}$.  For bubbles of concentration that overlap time $t=0$, these are simply the nonlinear solutions with initial data given by the bubble.  For bubbles of concentration that are temporally well separated from $t=0$, they are nonlinear solutions that have matching long-time behaviour (i.e. matching scattering state).  If there is more than one bubble (or a single bubble but non-zero remainder), all bubbles have energy strictly less than $E_c$.  (Note that energies are additive due to the strong separation of distinct profiles.)  But then by the inductive hypothesis \eqref{E:induct hyp}, each one of the nonlinear profiles will be global in time and obey spacetime bounds.  Adding the nonlinear profiles together (and incorporating the linear flow of the remainder term) we obtain an approximate solution to $\text{NLS}_{\R^3}$ with finite global spacetime bounds.  The fact that the sum of the nonlinear profiles is an approximate solution relies on the separation property of the profiles (this is, after all, a \emph{nonlinear} problem).  Thus by perturbation theory, for $n$ sufficiently large there is a true solution to $\text{NLS}_{\R^3}$ with initial data $u_n(0)$ and bounded global spacetime norms.  This contradicts the criterion by which $u_n$ were chosen in the first place and so precludes the dichotomy scenario.  

This completes the discussion of how one proves the existence of minimal energy blowup solutions for the energy-critical problem in the Euclidean setting.  The argument gives slightly more, something we call (by analogy with the calculus of variations) a \emph{Palais--Smale condition} (cf. Proposition~\ref{P:PS}). This says the following: Given an optimizing sequence of solutions for the scattering size with the energy converging to $E_c$, this sequence has a convergent subsequence (modulo the symmetries of the problem).  Note that by the definition of $E_c$, such optimizing sequences have diverging scattering size.

Recall that one of the key discoveries of \cite{borg:scatter,CKSTT:gwp} was that it was only necessary to consider solutions that have a well-defined (time-dependent) location and characteristic length scale.  Mere existence of minimal blowup solutions is not sufficient; they need to have this additional property in order to overcome the intrinsic limitations of non-scale-invariant conservation/monotonicity laws.

Fortunately, this additional property follows neatly from the Palais--Smale condition.  If $u(t)$ is a minimal energy blowup solution and $t_n$ is a sequence of times, then $u_n(t)=u(t+t_n)$ is a sequence to which we may apply the Palais--Smale result.  Thus, applying symmetry operations to $u(t_n)$ one may find a subsequence that is convergent in $\dot H^1(\R^3)$.  This is precisely the statement that the solution $u$ is \emph{almost periodic}, which is to say, the orbit is cocompact modulo spatial translations and rescaling.  This compactness guarantees that the orbit is tight in both the physical and Fourier variables (uniformly in time).

Let us now turn to the problem on exterior domains.  Adapting the concentration compactness argument to this setting will cause us a great deal of trouble.  Naturally, NLS in the exterior domain $\Omega$ does not enjoy scaling or translation invariance.  Nevertheless, both the linear and nonlinear profile decompositions must acknowledge the possibility of solutions living at any scale and in any possible location.  It is important to realize that in certain limiting cases, these profiles obey \emph{different} equations.  Here are the three main examples:
\begin{CI}
\item Solutions with a characteristic scale much larger than that of the obstacle evolve as if in $\R^3$.
\item Solutions very far from the obstacle (relative to there own characteristic scale) also evolve as if in $\R^3$.
\item Very narrowly concentrated solutions lying very close to the obstacle evolve as if in a halfspace.
\end{CI}
This is both an essential idea that we will develop in what follows and extremely naive.  In each of the three scenarios just described, there are serious omissions from this superficial picture, as we will discuss below.   

Nevertheless, the Palais--Smale condition we obtain in this paper (see Proposition~\ref{P:PS}) is so strong, that it proves the existence of minimal counterexamples in the following form:

\begin{thm}[Minimal counterexamples]\label{T:mincrim}
\hskip 0em plus 1em Suppose Theorem \ref{T:main} failed. Then there exist a critical energy\/ $0<E_c<\infty$ and a global solution $u$ to \eqref{nls} with
$E(u)=E_c$, infinite scattering size both in the future and in the past
$$
S_{\ge 0}(u)=S_{\le 0}(u)=\infty,
$$
and whose orbit $\{ u(t):\, t\in \R\}$ is precompact in $\dot H_D^1(\Omega)$.
\end{thm}

As evidence of the strength of this theorem, we note that it allows us to complete the proof of Theorem~\ref{T:main} very quickly indeed (see the last half-page of this paper).

Induction on energy has been adapted to scenarios with broken symmetries before and we would like to give a brief discussion of some of these works.  Our efforts here diverge from these works in the difficulty of connecting the limiting cases to the original model.  The lack of a dispersive estimate is a particular facet of this.

In \cite{KVZ:quadpot}, the authors proved global well-posedness and scattering for the energy-critical NLS with confining or repelling quadratic potentials.  The argument was modelled on that of Bourgain \cite{borg:scatter} and Tao \cite{tao:radial}, and correspondingly considered only spherically symmetric data.  Radiality helps by taming the lack of translation invariance; the key issue was to handle the broken scaling symmetry.  This problem has dispersive estimates, albeit only for short times in the confining (i.e. harmonic oscillator) case.

In \cite{LSZ}, the Bourgain--Tao style of argument is adapted to spherically symmetric data in the exterior of a sphere in $\R^3$.  A key part of their argument is to prove that a dispersive estimate holds in this setting.

The paper \cite{KKSV:gKdV} considers the mass-critical generalized Korteweg--de Vries equation, using the concentration compactness variant of induction on energy.  This paper proves a minimal counterexample theorem in the style of Theorem~\ref{T:mincrim}.  Dispersive estimates hold; the main obstruction was to overcome the broken Galilei invariance.  In the limit of highly oscillatory solutions (at a fixed scale) the gKdV equation is shown to resemble a \emph{different} equation, namely, the mass-critical NLS.  This means that both the linear and nonlinear profile decompositions contain profiles that are embeddings of solutions to the linear/nonlinear Schr\"odinger equations, carefully embedded to mimic solutions to Airy/gKdV.

An analogous scenario arrises in the treatment of the cubic Klein--Gordon equation in two spatial dimensions, \cite{KSV:2DKG}.  Dispersive estimates hold for this problem.  Here the scaling symmetry is broken and strongly non-relativistic profiles evolve according to the mass-critical Schr\"odinger equation, which also breaks the Lorentz symmetry.  Linear and nonlinear profile decompositions that incorporate Lorentz boosts were one of the novelties of this work.

In the last two examples, the broken symmetries have led to dramatic changes in the equation, though the geometry has remained the same (all of Euclidean space).  Next, we describe some instances where the geometry changes, but the equation is essentially the same.

The paper \cite{IPS:H3} treats the energy-critical NLS on three-dimensional hyperbolic space.  Theorem~\ref{T:gopher} is used to treat highly concentrated profiles, which are embedded in hyperbolic space using the strongly Euclidean structure at small scales.  Some helpful ingredients in hyperbolic space are the mass gap for the Laplacian and its very strong dispersive and Morawetz estimates.

More recently, dramatic progress has been made on the energy-critical problem on the three dimensional flat torus.  Global well-posedness for small data was proved in \cite{HerrTataruTz:torus} and the large-data problem was treated in \cite{IonPaus}.  (See also \cite{HaniPaus,Herr:Zoll,HerrTataruTz:mixed,IonPaus1} for results in related geometries.)  While the manifold in question may be perfectly flat, the presence of closed geodesics and corresponding paucity of Strichartz estimates made this a very challenging problem.  The large data problem was treated via induction on energy, using the result for Euclidean space (i.e. Theorem~\ref{T:gopher}) as a black box to control highly concentrated profiles.  The local-in-time frequency localized dispersive estimate proved by Bourgain \cite{borg:torus} plays a key role in ensuring the decoupling of profiles.   

While the methods employed in the many papers we have discussed so far inform our work here, they do not suffice for the treatment of Theorem~\ref{T:main}.  Indeed, even the form of perturbation theory needed here spawned the separate paper \cite{KVZ:HA}.  Moreover, in this paper we encounter not only changes in geometry, but also changes in the equation; after all, the Dirichlet Laplacian on exterior domains is very different from the Laplacian on $\R^3$.

We have emphasized the dispersive estimate because it has been an essential ingredient in the concentration compactness variant of induction on energy; it is the tool that guarantees that profiles contain a single bubble of concentration and so underwrites the decoupling of different profiles.  Up to now, no one has succeeded in doing this without the aid of a dispersive-type estimate.  Moreover, as emphasized earlier, the dispersive estimate plays a seemly irreplaceable role in the treatment of the energy-critical problem in $\R^3$.  Thus, we are confronted with the problem of finding and then proving a suitable substitute for the dispersive estimate.  One of the key messages of this paper is the manner in which this issue is handled, in particular, that the weakened form of dispersive estimate we prove, namely Theorem~\ref{T:LF}, is strong enough to complete the construction of minimal blowup solutions.  The result we prove is too strong to hold outside merely non-trapping obstacles; convexity plays an essential role here.

Section~\ref{S:Linear flow convergence} is devoted entirely to the proof of Theorem~\ref{T:LF}.
Three different methods are used depending on the exact geometric setting, but in all cases, the key result is an \emph{infinite-time} parametrix that captures the action of $e^{it\Delta_\Omega}$ up to a \emph{vanishing fraction} of the mass/energy.  Both this level of accuracy and the fact that it holds for all time are essential features for the rest of the argument.

The most difficult regime in the proof of Theorem~\ref{T:LF} is when the initial data is highly concentrated, say at scale $\eps$, at a distance $\delta$
from the obstacle with $\eps\lesssim \delta\lesssim 1$.  To treat this regime, we subdivide into two cases: $\eps\lesssim \delta\lesssim\eps^{\frac67}$ and $\eps^{\frac67}\lesssim\delta\lesssim 1$, which are called Cases~(iv) and~(v), respectively.

In Case~(iv), the initial data sees the obstacle as a (possibly retreating) halfspace.  To handle this case, we first approximate the initial data by a linear combination of Gaussian wave packets (with characteristic scale $\eps$).  Next we use the halfspace evolution of these wave packets (for which there is an exact formula) to approximate their linear evolution in $\Omega$.  As the halfspace evolution does not match the Dirichlet boundary condition, we have to introduce a correction term $w$.  Moreover, we have to choose the parameters in the definition of $w$ carefully, so that the resulting error terms can be controlled for the full range of $\delta$.

In Case~(v), the obstacle is far from the initial data relative to the data's own scale, but close relative to the scale of the obstacle.  We decompose the initial data into a linear combination of Gaussian wave packets, whose characteristic scale $\sigma$ is chosen carefully to allow refection off the obstacle to be treated by means of geometric optics.  In particular, $\sigma$ is chosen so that the wave packets do not disperse prior to their collision with the obstacle, but do disperse shortly thereafter.  We divide these wave packets into three categories: those that miss the obstacle, those that are near-grazing, and those that collide non-tangentially with the obstacle.  Wave packets in the last category are the most difficult to treat.  For these, we build a Gaussian parametrix for the reflected wave.  To achieve the needed degree of accuracy, this parametrix must be very precisely constructed; in particular, it must be matched to the principal curvatures of the obstacle at the collision point.  This parametrix does not match the Dirichlet boundary condition perfectly, and it is essential to wring the last drops of cancellation from this construction in order to ensure that it is not overwhelmed by the resulting errors.  Further, the term $w$ that we introduce to match the boundary condition is carefully chosen so that it is non-resonant; note the additional phase factor in the definition of $w^{(3)}$.  This is needed so that the error terms are manageable.  

An example of how the results of Section~\ref{S:Linear flow convergence} play a role can be seen in the case of profiles that are highly concentrated at a bounded distance from the obstacle.  These live far from the obstacle relative to their own scale, and so we may attempt to approximate them by solutions to  $\text{NLS}_{\R^3}$ whose existence is guaranteed by Theorem~\ref{T:gopher}.  Such solutions scatter and so eventually dissolve into outward propagating radiation.  However, the obstacle blocks a positive fraction of directions and so a non-trivial fraction of the energy of the wave packet will reflect off the obstacle.  Theorem~\ref{T:LF3} guarantees that this reflected energy will not refocus.  Only with this additional input can we truly say that such profiles behave as if in Euclidean space.

Now consider the case when the profile is much larger than the obstacle.  In this case the equivalence of the linear flows follows from Theorem~\ref{T:LF1}.  However, the argument does not carry over to the nonlinear case.  Embedding the nonlinear profiles requires a special argument; one of the error terms is simply not small.  Nevertheless, we are able to control it by proving that it is non-resonant; see Step~2 in the proof of Theorem~\ref{T:embed2}.

The third limiting scenario identified above was when the profile concentrates very close to the obstacle.  In this regime the limiting geometry is the halfspace $\HH$.  Note that spacetime bounds for $\text{NLS}_\HH$ follow from Theorem~\ref{T:gopher} by considering solutions that are odd under reflection in $\partial\HH$.  The linear flow is treated in Theorem~\ref{T:LF2} and the embedding of nonlinear profiles is the subject of Theorem~\ref{T:embed4}.  Note that in this regime, the spacetime region where the evolution is highly nonlinear coincides with the region of collision with the boundary.  In the far-field regime, the finite size of the obstacle affects the radiation pattern; thus it is essential to patch the halfspace linear evolution together with that in $\Omega$.

Our discussion so far has emphasized how to connect the free propagator in the limiting geometries with that in $\Omega$.  The complexity of energy-critical arguments is such that we also need to understand the relations between other spectral multipliers, such as Littlewood--Paley projectors and fractional powers.  This is the subject of Section~\ref{S:Domain Convergence}.

After much toil, we show that nonlinear profiles arising from all limiting geometries obey spacetime bounds, which plays an analogous role to the induction on energy hypothesis.  Thus, when the nonlinear profile decomposition is applied to a Palais--Smale sequence, we can show that there can be only one profile and it cannot belong to either of the limiting geometries $\R^3$ or $\HH$; it must live at approximately unit scale and at approximately unit distance from the obstacle.  This is how we obtain Theorem~\ref{T:mincrim}.  The proof of this theorem occupies most of Section~\ref{S:Proof}.  The last part of that section deduces Theorem~\ref{T:main} from this result.

To close this introduction, let us quickly recount the contents of this paper by order of presentation.

Section~\ref{S:Preliminaries} mostly reviews existing material that is needed for the analysis:  equivalence of Sobolev spaces and the product rule for the Dirichlet Laplacian; Littlewood--Paley theory and Bernstein inequalities; Strichartz estimates; local and stability theories for $\text{NLS}_\Omega$; persistence of regularity for solutions of NLS that obey spacetime bounds (this is important for the embedding of profiles); the Bourgain-style Morawetz identity; and local smoothing.

Section~\ref{S:Domain Convergence} proves results related to the convergence of functions of the Dirichlet Laplacian as the underlying domains converge.  Convergence of Green's functions at negative energies is proved via direct analysis making use of the maximum principle.  This is extended to complex energies via analytic continuation and the Phragmen--Lindel\"of principle.  General functions of the operator are represented in terms of the resolvent via the Helffer--Sj\"ostrand formula.

Section~\ref{S:Linear flow convergence}  analyses the behaviour of the linear propagator under domain convergence.  In all cases, high-accuracy infinite-time parametrices are constructed.  When the geometry guarantees that a vanishing fraction of the wave actually hits the obstacle, a simple truncation argument is used (Theorem~\ref{T:LF1}).  For disturbances close to the obstacle, we base our approximation off the exact solution of the halfspace linear problem with Gaussian initial data; see Theorem~\ref{T:LF2}.  For highly concentrated wave packets a bounded distance from the obstacle, we build a parametrix based on a Gaussian beam technique; see Theorem~\ref{T:LF3}.   The fact that Gaussian beams are exact linear solutions in Euclidean space prevents the accumulation of errors at large times.

Section~\ref{S:LPD} first proves refined and inverse Strichartz inequalities (Lemma~\ref{lm:refs} and Proposition~\ref{P:inverse Strichartz}).  These show that linear evolutions with non-trivial spacetime norms must contain a bubble of concentration.  This is then used to obtain the linear profile decomposition, Theorem~\ref{T:LPD}.  The middle part of this section contains additional results related to the convergence of domains, which combine the tools from Sections~\ref{S:Domain Convergence} and~\ref{S:Linear flow convergence}.

Section~\ref{S:Nonlinear Embedding} shows how nonlinear solutions in the limiting geometries can be embedded in $\Omega$.  As nonlinear solutions in the limiting geometries admit global spacetime bounds (this is how Theorem~\ref{T:gopher} enters our analysis), we deduce that solutions to $\text{NLS}_\Omega$ whose characteristic length scale and location conform closely to one of these limiting cases inherit these spacetime bounds.  These solutions to $\text{NLS}_\Omega$ appear again as nonlinear profiles in Section~\ref{S:Proof}.

Section~\ref{S:Proof} contains the proofs of the Palais--Smale condition (Proposition~\ref{P:PS}), as well as the existence and almost periodicity of minimal blowup solutions (Theorem~\ref{T:mmbs}).  Because of all the ground work laid in the previous sections, the nonlinear profile decomposition, decoupling, and induction on energy arguments all run very smoothly.  This section closes with the proof of Theorem~\ref{T:main}; the needed contradiction is obtained by combining the space-localized Morawetz identity introduced in Lemma~\ref{L:morawetz} with the almost periodicity of minimal blowup solutions.    

\subsection*{Acknowledgements}
R. K. was supported by NSF grant DMS-1001531.  M. V. was supported by the Sloan Foundation and NSF grants DMS-0901166 and DMS-1161396.
X. Z. was supported by the Sloan Foundation.

%%%%%%%%%%%%%%%%%%%%%%%%%%%%%%%%%%%%%%%%%%%%%%%%%%%%%%%%%%%%%%%%%%%%%%%%%%%%%%%%%%%%%%%%%%%%%%%%%%%%%%%%%%%%%%%%%%%%%%%%%%%%%%%%%%%%%%
\section{Preliminaries}\label{S:Preliminaries}
%%%%%%%%%%%%%%%%%%%%%%%%%%%%%%%%%%%%%%%%%%%%%%%%%%%%%%%%%%%%%%%%%%%%%%%%%%%%%%%%%%%%%%%%%%%%%%%%%%%%%%%%%%%%%%%%%%%%%%%%%%%%%%%%%%%%%%

\subsection{Some notation}
We write $X \lesssim Y$ or $Y \gtrsim X$ to indicate $X \leq CY$ for some absolute constant $C>0$, which may change from line
to line.  When the implicit constant depends on additional quantities, this will be indicated with subscripts.  We use $O(Y)$ to
denote any quantity $X$ such that $|X| \lesssim Y$.  We use the notation $X \sim Y$ whenever $X \lesssim Y \lesssim X$.  We write $o(1)$
to indicate a quantity that converges to zero.

Throughout this paper, $\Omega$ will denote the exterior domain of a smooth compact strictly convex obstacle in $\R^3$.
Without loss of generality, we assume that $0\in \Omega^c$.  We use $\diam:=\diam(\Omega^c)$ to denote the diameter of the obstacle
and $d(x):=\dist(x,\Omega^c)$ to denote the distance of a point $x\in\R^3$ to the obstacle.

In order to prove decoupling of profiles in $L^p$ spaces (when $p\neq 2$) in Section~\ref{S:LPD}, we will make use of the
following refinement of Fatou's Lemma, due to Br\'ezis and Lieb:

\begin{lem}[Refined Fatou, \cite{BrezisLieb}]\label{lm:rf}
Let $0<p<\infty$. Suppose $\{f_n\}\subseteq L^p(\R^d)$ with $\limsup\|f_n\|_{L^p}<\infty$. If $f_n\to f$ almost everywhere, then
\begin{align*}
\int_{\R^d}\Bigl||f_n|^p-|f_n-f|^p-|f|^p \Bigr| \,dx\to 0.
\end{align*}
In particular, $\|f_n\|_{L^p}^p-\|f_n-f\|_{L^p}^p \to \|f\|_{L^p}^p$.
\end{lem}

As described in the introduction, we need adaptations of a wide variety of harmonic analysis tools to the setting
of exterior domains.   Most of these were discussed in our paper \cite{KVZ:HA}.  One of the key inputs for that
paper is the following (essentially sharp) estimate for the heat kernel:

\begin{thm}[Heat kernel bounds, \cite{qizhang}]\label{T:heat}
Let $\Omega$ denote the exterior of a smooth compact convex obstacle in $\R^d$ for $d\geq 3$. Then there exists $c>0$ such that
\begin{align*}
|e^{t\ld}(x,y)|\lsm \Bigr(\frac{d(x)}{\sqrt t\wedge \diam}\wedge 1\Bigr)\Bigl(\frac{d(y)}{\sqrt t\wedge \diam}\wedge 1\Bigr) e^{-\frac{c|x-y|^2}t} t^{-\frac d 2},
 \end{align*}
uniformly in $x, y\in \Omega$ and $t\geq 0$; recall that $A\wedge B = \min\{A,B\}$.  Moreover, the reverse inequality holds after suitable modification
of $c$ and the implicit constant.
\end{thm}

The most important result from \cite{KVZ:HA} for our applications here is the following, which identifies Sobolev spaces defined with respect
to the Dirichlet Laplacian with those defined via the usual Fourier multipliers.  Note that the restrictions on the regularity $s$ are necessary,
as demonstrated by the counterexamples discussed in~\cite{KVZ:HA}.

\begin{thm}[Equivalence of Sobolev spaces, \cite{KVZ:HA}]\label{T:Sob equiv}
Let $d\geq 3$ and let $\Omega$ denote the complement of a compact convex body $\Omega^c\subset\R^d$ with smooth boundary.  Let $1<p<\infty$.  If $0\leq s<\min\{1+\frac1p,\frac dp\}$ then
\begin{equation}\label{E:equiv norms}
\bigl\| (-\Delta_{\R^d})^{s/2} f \bigl\|_{L^p}  \sim_{d,p,s} \bigl\| (-\Delta_\Omega)^{s/2} f \bigr\|_{L^p}  \qtq{for all} f\in C^\infty_c(\Omega).
\end{equation}
\end{thm}

This result allows us to transfer several key results directly from the Euclidean setting, provided we respect the restrictions on $s$ and $p$.  This includes such basic facts as the $L^p$-Leibnitz (or product) rule for first derivatives.  Indeed, the product rule for the operator
$(-\Delta_\Omega)^{1/2}$ is non-trivial; there is certainly no pointwise product rule for this operator.

We also need to consider derivatives of non-integer order.  The $L^p$-product rule for fractional derivatives in Euclidean spaces was
first proved by Christ and Weinstein \cite{ChW:fractional chain rule}.  Combining their result with Theorem~\ref{T:Sob equiv}
yields the following:

\begin{lem}[Fractional product rule]\label{lm:product}
For all $f, g\in C_c^{\infty}(\Omega)$, we have
\begin{align}\label{fp}
\| (-\Delta_\Omega)^{\frac s2}(fg)\|_{L^p} \lsm \| (-\Delta_\Omega)^{\frac s2} f\|_{L^{p_1}}\|g\|_{L^{p_2}}+
\|f\|_{L^{q_1}}\| (-\Delta_\Omega)^{\frac s2} g\|_{L^{q_2}}
\end{align}
with the exponents satisfying $1<p, p_1, q_2<\infty$, $1<p_2,q_1\le \infty$,
\begin{align*}
\tfrac1p=\tfrac1{p_1}+\tfrac1{p_2}=\tfrac1{q_1}+\tfrac1{q_2}, \qtq{and} 0<s<\min\bigl\{ 1+\tfrac1{p_1}, 1+\tfrac1{q_2},\tfrac3{p_1},\tfrac3{q_2} \bigr\}.
\end{align*}
\end{lem}

\subsection{Littlewood--Paley theory on exterior domains}

Fix $\phi:[0,\infty)\to[0,1]$  a smooth non-negative function obeying
\begin{align*}
\phi(\lambda)=1 \qtq{for} 0\le\lambda\le 1 \qtq{and} \phi(\lambda)=0\qtq{for} \lambda\ge 2.
\end{align*}
For each dyadic number $N\in 2^\Z$, we then define
\begin{align*}
\phi_N(\lambda):=\phi(\lambda/N) \qtq{and} \psi_N(\lambda):=\phi_N(\lambda)-\phi_{N/2}(\lambda);
\end{align*}
notice that $\{\psi_N(\lambda)\}_{N\in \tz} $ forms a partition of unity for $(0,\infty)$.

With these functions in place, we can now introduce the Littlewood--Paley projections adapted to the Dirichlet Laplacian on $\Omega$
and defined via the functional calculus for self-adjoint operators:
\begin{align*}
\po_{\le N} :=\phi_N\bigl(\sqrt{-\Delta_\Omega}\,\bigr), \quad \po_N :=\psi_N(\sqrt{-\Delta_\Omega}\,\bigr), \qtq{and} \po_{>N} :=I-\po_{\le N}.
\end{align*}
For brevity we will often write $f_N := \po_N f$ and similarly for the other projections.

We will write $P_N^{\R^3}$, and so forth, to represent the analogous operators associated to the usual Laplacian in the full Euclidean space.
We will also need the analogous operators on the halfspace $\HH=\{x\in\R^3 : x \cdot e_3 >0\}$ where $e_3=(0,0,1)$, which we denote by $P^{\HH}_N$, and so forth.

Just like their Euclidean counterparts, these Littlewood--Paley projections obey Bernstein estimates.  Indeed, these follow quickly
from heat kernel bounds and the analogue of the Mikhlin multiplier theorem for the Dirichlet Laplacian.  See \cite{KVZ:HA} for further details.

\begin{lem}[Bernstein estimates]
Let $1<p<q\le \infty$ and $-\infty<s<\infty$. Then for any $f\in C_c^{\infty}(\Omega)$, we have
\begin{align*}
\|\po_{\le N} f \|_{\lpo}+\|\po_N f\|_{\lpo}&\lsm \|f\|_{\lpo},\\
\|\po_{\le N} f\|_{L^q(\Omega)}+\|\po_N f\|_{L^q(\Omega)}&\lsm N^{d(\frac 1p-\frac1q)}\|f\|_{L^p(\Omega)},\\
N^s\|\po_N f\|_{\lpo}&\sim \|(-\ld)^{\frac s2}\po_N f\|_{\lpo}.
\end{align*}
\end{lem}

A deeper application of the multiplier theorem for the Dirichlet Laplacian is the proof of the square function inequalities.
Both are discussed in \cite{IvanPlanch:square}, as well as \cite{KVZ:HA}, and further references can be found therein.

\begin{lem}[Square function estimate]\label{sq}
Fix $1<p<\infty$.  For all $f\in C_c^{\infty}(\Omega)$,
\begin{align*}
\|f\|_{L^p(\Omega)}\sim \Bigl\|\Bigl(\sum_{N\in \tz}|\po_{N} f|^2\Bigr)^{\frac12}\Bigr\|_{L^p(\Omega)}.
\end{align*}
\end{lem}

Implicit in this lemma is the fact that each $f$ coincides with $\sum f_N$ in $L^p(\Omega)$ sense for $1<p<\infty$.  This relies on the
fact that $0$ is not an eigenvalue of $-\Delta_\Omega$, as follows from Lemma~\ref{L:local smoothing}.

\subsection{Strichartz estimates and the local theory}
As the endpoint Strichartz inequality is not known for exterior domains, some care needs to be taken when defining the natural Strichartz spaces.
For any time interval $I$, we define
\begin{align*}
S^0(I)&:=L_t^{\infty} L_x^2(I\times\Omega)\cap L_t^{2+\eps}L_x^{\frac{6(2+\eps)}{2+3\eps}}(I\times\Omega)\\
\dot S^1(I) &:= \{u:I\times\Omega\to \C :\, (-\Delta_\Omega)^{1/2}u\in S^0(I)\}.
\end{align*}
By interpolation,
\begin{align}\label{Sspaces}
\|u\|_{L_t^q L_x^r(I\times\Omega)}\leq \|u\|_{S^0(I)} \qtq{for all} \tfrac2q+\tfrac3r=\tfrac32 \qtq{with} 2+\eps\leq q\leq \infty.
\end{align}
Here $\eps>0$ is chosen sufficiently small so that all Strichartz pairs of exponents used in this paper are covered.  For example, combining
\eqref{Sspaces} with Sobolev embedding and the equivalence of Sobolev spaces Theorem~\ref{T:Sob equiv}, we obtain the following lemma.

\begin{lem}[Sample spaces]
We have
\begin{align*}
\|u\|_{L_t^\infty \dot H^1_D} &+ \|(-\Delta_{\Omega})^{\frac12}u\|_{L_t^{10} L_x^{\frac{30}{13}}} + \|(-\Delta_{\Omega})^{\frac12}u\|_{L_t^5L_x^{\frac{30}{11}}} + \|(-\Delta_{\Omega})^{\frac 12}u\|_{L_{t,x}^{\frac{10}3}}\\
& + \|(-\Delta_{\Omega})^{\frac 12}u\|_{L_t^{\frac83} L_x^4} +\|u\|_{L_t^\infty L_x^6}+\|u\|_{L^{10}_{t,x}}+\|u\|_{L_t^5 L_x^{30}}\lsm \|u\|_{\dot S^1(I)},
\end{align*}
where all spacetime norms are over $I\times \Omega$.
\end{lem}

We define $N^0(I)$ to be the dual Strichartz space and
$$
\dot N^1(I):=\{F:I\times\Omega\to \C:\, (-\Delta_\Omega)^{1/2} F\in N^0(I)\}.
$$

For the case of exterior domains, Strichartz estimates were proved by Ivanovici \cite{Ivanovici:Strichartz}; see also \cite{BSS:schrodinger}.  These estimates form an essential foundation for all the analysis carried out in this papaer. 

\begin{thm}[Strichartz estimates]\label{T:Strichartz}
Let $I$ be a time interval and let $\Omega$ be the exterior of a smooth compact strictly convex obstacle in $\R^3$. Then
the solution $u$ to the forced Schr\"odinger equation $i u_t + \Delta_\Omega u = F$ satisfies the estimate
\begin{align*}
\|u\|_{S^0(I)}\lsm \|u(t_0)\|_{L^2(\Omega)}+\|F\|_{N^0(I)}
\end{align*}
for any $t_0\in I$.  In particular, as $(-\Delta_\Omega)^{1/2}$ commutes with the free propagator $e^{it\Delta_\Omega}$,
\begin{align*}
\|u\|_{\dot S^1(I)}\lsm \|u(t_0)\|_{\dot H^1_D(\Omega)}+\|F\|_{\dot N^1(I)}
\end{align*}
for any $t_0\in I$.
\end{thm}

When $\Omega$ is the whole Euclidean space $\R^3$, we may take $\eps=0$ in the definition of Strichartz spaces; indeed, for the linear propagator
$e^{it\Delta_{\R^3}}$, Strichartz estimates for the endpoint pair of exponents $(q,r)=(2,6)$ were proved by Keel and Tao \cite{KeelTao}.   Embedding functions on the halfspace $\HH$ as functions on $\R^3$ that are odd under reflection in $\partial\HH$, we immediately see that the whole range of Strichartz estimates, including the endpoint, also hold for the free propagator $e^{it\Delta_{\HH}}$.

The local theory for \eqref{nls} is built on contraction mapping arguments combined with Theorem~\ref{T:Strichartz} and the equivalence of Sobolev spaces Theorem~\ref{T:Sob equiv}.  We record below a stability result for \eqref{nls}, which is essential in extracting a minimal counterexample to Theorem~\ref{T:main}.  Its predecessor in the Euclidean case can be found in \cite{CKSTT:gwp}; for versions in higher dimensional Euclidean spaces see \cite{ClayNotes, RV, TaoVisan}.

\begin{thm}[Stability for $\text{NLS}_{\Omega}$, \cite{KVZ:HA}]\label{T:stability}
Let $\Omega$ be the exterior of a smooth compact strictly convex obstacle in $\R^3$.  Let $I$ a compact time interval and let $\tilde u$ be an approximate solution to \eqref{nls} on $I\times \Omega$ in the sense that
$$
i\tilde u_t + \Delta_\Omega \tilde u = |\tilde u|^4\tilde u + e
$$
for some function $e$.  Assume that
\begin{align*}
\|\tilde u\|_{L_t^\infty \dot H_D^1(I\times \Omega)}\le E \qtq{and} \|\tilde u\|_{L_{t,x}^{10}(I\times \Omega)} \le L
\end{align*}
for some positive constants $E$ and $L$.  Let $t_0 \in I$ and let $u_0\in \dot H_D^1(\Omega)$ satisfy
\begin{align*}
\|u_0-\tilde u(t_0)\|_{\dot H_D^1}\le E'
\end{align*}
for some positive constant $E'$.  Assume also the smallness condition
\begin{align}\label{E:stab small}
\bigl\|\sqrt{-\Delta_\Omega}\; e^{i(t-t_0)\Delta_\Omega}\bigl[u(t_0)-\tilde u_0\bigr]\bigr\|_{L_t^{10}L_x^{\frac{30}{13}}(I\times \Omega)}
+\bigl\|\sqrt{-\Delta_\Omega}\; e\bigr\|_{N^0(I)}&\le\eps
\end{align}
for some $0<\eps<\eps_1=\eps_1(E,E',L)$.  Then, there exists a unique strong solution $u:I\times\Omega\mapsto \C$ to \eqref{nls} with initial data $u_0$ at time $t=t_0$ satisfying
\begin{align*}
\|u-\tilde u\|_{L_{t,x}^{10}(I\times \Omega)} &\leq C(E,E',L)\eps\\
\bigl\|\sqrt{-\Delta_\Omega}\;  (u-\tilde u)\bigr\|_{S^0(I\times\Omega)} &\leq C(E,E',L)\, E'\\
\bigl\|\sqrt{-\Delta_\Omega}\;  u\bigr\|_{S^0(I\times\Omega)} &\leq C(E,E',L).
\end{align*}
\end{thm}

There is an analogue of this theorem for $\Omega$ an exterior domain in $\R^d$ with $d=4,5,6$; see \cite{KVZ:HA}.  For dimensions $d\geq 7$, this is an open question.  The proof of the stability result in $\R^d$ with $d\geq 7$ relies on fractional chain rules for H\"older continuous functions and `exotic' Strichartz estimates; see \cite{ClayNotes,TaoVisan}.  The equivalence of Sobolev spaces Theorem~\ref{T:main} guarantees that the fractional chain rule can be imported directly from the Euclidean setting.  However, the `exotic' Strichartz estimates are derived from the dispersive estimate \eqref{E:EuclidDisp} and it is not known whether they hold in exterior domains.

Applying Theorem~\ref{T:stability} with $\tilde u\equiv0$, we recover the standard local well-posedness theory for \eqref{nls}.  Indeed, for an arbitrary (large) initial data $u_0\in \dot H^1_D(\Omega)$, the existence of some small time interval $I$ on which the smallness hypothesis
\eqref{E:stab small} holds is guaranteed by the monotone convergence theorem combined with Theorem~\ref{T:Strichartz}.  Moreover, if the initial data $u_0$ has small norm in $\dot H^1_D(\Omega)$ (that is, $E'$ is small), then Theorem~\ref{T:Strichartz} yields \eqref{E:stab small} with $I=\R$.
Therefore, both local well-posedness for large data and global well-posedness for small data follow from Theorem~\ref{T:stability}.  These special
cases of Theorem~\ref{T:stability} have appeared before, \cite{BSS:schrodinger, IvanPlanch:IHP}; induction on energy, however, requires the full strength of
Theorem~\ref{T:stability}.

In Section~\ref{S:Nonlinear Embedding}, we will embed solutions to NLS in various limiting geometries back inside $\Omega$.  To embed solutions to $\text{NLS}_{\R^3}$ in $\Omega$, we will make use of the following persistence of regularity result for this equation:

\begin{lem}[Persistence of regularity for $\text{NLS}_{\R^3}$, \cite{CKSTT:gwp}]\label{lm:persistencer3} Fix $s\ge 0$ and let $I$ be a compact time interval and $u:I\times\R^3\to \C$ be a solution to $\text{NLS}_{\R^3}$ satisfying
\begin{align*}
E(u)\leq E<\infty \qtq{and} \|u\|_{L_{t,x}^{10}(I\times\R^3)}\leq L<\infty.
\end{align*}
If $u(t_0)\in \dot H^s(\R^3)$ for some $t_0\in I$, then
\begin{align*}
\|(-\Delta_{\R^3})^{\frac s2}u\|_{S^0(I)}\leq C(E,L) \|u(t_0)\|_{\dot H^s(\R^3)}.
\end{align*}
\end{lem}

We will also need a persistence of regularity result for $\text{NLS}_{\HH}$.  This follows by embedding solutions on the halfspace as solutions on $\R^3$ that are odd under reflection in $\partial\HH$.  In particular, one may regard $-\Delta_\HH$ as the restriction of $-\Delta_{\R^3}$ to odd functions.  For example, one can see this equivalence in the exact formula for the heat kernel in $\HH$.

\begin{lem} [Persistence of regularity for $\text{NLS}_{\HH}$]\label{lm:persistenceh} Fix $s\ge 0$ and let $I$ be a compact time interval and
$u:I\times\HH\to \C$ be a solution to $\text{NLS}_{\HH}$ satisfying
\begin{align*}
E(u)\leq E<\infty \qtq{and} \|u\|_{L_{t,x}^{10}(I\times\HH)}\leq L<\infty.
\end{align*}
If $u(t_0)\in \dot H^s_D(\HH)$ for some $t_0\in I$, then
\begin{align*}
\|(-\Delta_{\HH})^{\frac s2}u\|_{S^0(I)}\leq C(E,L) \|u(t_0)\|_{\dot H^s_D(\HH)}.
\end{align*}
\end{lem}

\subsection{Morawetz and local smoothing}
We preclude the minimal counterexample to Theorem~\ref{T:main} in Section~\ref{S:Proof} with the use of the following one-particle Morawetz inequality; cf. \cite{borg:scatter, LinStrauss}.

\begin{lem}[Morawetz inequality]\label{L:morawetz}
Let $I$ be a time interval and let $u$ be a solution to \eqref{nls} on $I$.  Then for any $A\ge 1$ with $A|I|^{1/2}\geq \diam(\Omega^c)$ we have
\begin{align}\label{mora}
\int_I\int_{|x|\le A|I|^{\frac 12}, x\in \Omega}\frac{|u(t,x)|^6}{|x|} \,dx\,dt\lsm A|I|^{\frac 12},
\end{align}
where the implicit constant depends only on the energy of $u$.
\end{lem}

\begin{proof}
Let $\phi(x)$ be a smooth radial bump function such that $\phi(x)=1$ for $|x|\le 1$ and $\phi(x)=0$ for $|x|>2$.  Let $R\geq \diam(\Omega^c)$ and define
$a(x):=|x|\phi\bigl(\frac x R\bigr)$.  Then for $|x|\le R$ we have
\begin{align}\label{cd1}
\partial_j\partial_k a(x) \text{ is positive definite,} \quad \nabla a(x)=\frac x{|x|}, \qtq{and} \Delta \Delta a(x)<0,
\end{align}
while for $|x|>R$ we have the following rough estimates:
\begin{align}\label{cd2}
|\partial_k a(x)|\lsm 1, \quad |\partial_j\partial_k a(x)|\lsm \frac 1R,\qtq{and} |\Delta\Delta a(x)|\lsm \frac 1{R^3}.
\end{align}

To continue, we use the local momentum conservation law
\begin{align}\label{lmc}
\partial_t \Im(\bar u \partial_k u)=-2\partial_j \Re(\partial_k u\partial_j\bar u)+\frac 12\partial_k\Delta(|u|^2)-\frac 23\partial_k(|u|^6).
\end{align}
Multiplying both sides by $\partial_k a$ and integrating over $\Omega$ we obtain
\begin{align}
\partial_t \Im\int_\Omega \bar u \partial_k u \partial_k a \,dx
&=-2 \Re\int_\Omega\partial_j(\partial_k u\partial_j \bar u)\partial_k a\,dx\notag\\
&\quad+\frac 12\int_\Omega\partial_k\Delta(|u|^2)\partial_k a \,dx-\frac 23\int_\Omega\partial_k(|u|^6)\partial_k a \,dx.\label{17}
\end{align}
The desired estimate \eqref{mora} will follow from an application of the fundamental theorem of calculus combined with an upper bound on $\text{LHS}\eqref{17}$ and a lower bound on $\text{RHS}\eqref{17}$.

The desired upper bound follows immediately from H\"older followed by Sobolev embedding:
\begin{align}\label{upper bound}
\Im\int_\Omega \bar u \partial_k u \partial_k a dx
\lsm \|u\|_{L^6(\Omega)} \|\nabla u\|_{L^2(\Omega)} \|\nabla a\|_{L^3(\Omega)}\lsm R\|\nabla u\|_{L^2(\Omega)}^2.
\end{align}

Next we seek a lower bound on $\text{RHS}\eqref{17}$.  From the divergence theorem, we obtain
\begin{align*}
-2\Re\int_\Omega\partial_j(\partial_k u\partial_j \bar u)\partial_k a\,dx
&=-2\Re\int_\Omega\partial_j(\partial_k u\partial_j \bar u\partial_k a) \,dx+ 2\Re\int_\Omega\partial_k u\partial_j\bar u\partial_j\partial_k a \,dx\\
&=2\Re\int_{\partial\Omega}\partial_ku\partial_k  a\partial_j \bar u\vec n_j d\sigma(x)+2\Re\int_{|x|\le R}\partial_k u\partial_j\bar u\partial_j\partial_k a\,dx\\
&\qquad + 2\int_{|x|\ge R}\partial_k u\partial_j\bar u\partial_j\partial_k a \,dx,
\end{align*}
where $\vec n$ denotes the outer normal to $\Omega^c$.  We write
\begin{align*}
\partial_j \bar u\vec n_j=\nabla \bar u\cdot\vec n=\bar u_n.
\end{align*}
Moreover, from the Dirichlet boundary condition, the tangential
derivative of $u$ vanishes on the boundary; thus,
\begin{align*}
\nabla u=(\nabla u\cdot {\vec n})\vec n=u_n\vec n \qtq{and}  \partial_k u\partial_k a=u_n a_n.
\end{align*}
Using this, \eqref{cd1}, and \eqref{cd2} we obtain
\begin{align*}
-2\Re\int_\Omega\partial_j(\partial_k u\partial_j\bar u)\partial_k a\,dx
&\ge 2\int_{\partial\Omega} a_n|u_n|^2 d\sigma(x)+2\int_{|x|\ge R}\partial_k u\partial_j \bar u\partial_j\partial_k a \,dx\\
&\ge 2\int_{\partial\Omega} a_n|u_n|^2 d\sigma(x)-\frac CR\|\nabla u\|_{L^2(\Omega)}^2.
\end{align*}

Similarly, we can estimate the second term on $\text{LHS}\eqref{17}$ as follows:
\begin{align*}
\frac 12\int_\Omega \partial_k\Delta(|u|^2)\partial_k a \,dx
&=\frac 12\int_\Omega\partial_k\bigl[\Delta(|u|^2)\partial_k a\bigr] \,dx-\frac12\int_{\Omega}\Delta(|u|^2)\Delta a \,dx\\
&=-\frac 12\int_{\partial\Omega}\Delta(|u|^2)\partial_k a\vec n_k d\sigma(x)-\frac 12\int_{\Omega}|u|^2\Delta\Delta a \,dx\\
&=-\int_{\partial\Omega}|\nabla u|^2 a_n d\sigma(x)-\frac12\int_{|x|\le R}|u|^2 \Delta\Delta a \,dx\\
&\quad -\frac 12\int_{|x|\geq R}|u|^2 \Delta\Delta a \,dx\\
&\ge-\int_{\partial\Omega}|u_n|^2 a_n d\sigma(x)-\frac CR \|\nabla u\|_{L^2(\Omega)}^2;
\end{align*}
to obtain the last inequality we have used \eqref{cd1}, \eqref{cd2}, H\"older, and Sobolev embedding.

Finally, to estimate the third term on $\text{LHS}\eqref{17}$
we use \eqref{cd1} and \eqref{cd2}:
\begin{align*}
-\frac 23\int_{\Omega}\partial_k(|u|^6)\partial_k a \,dx
&=\frac23\int_\Omega|u|^6 \Delta a \,dx=\frac 43\int_{|x|\le R}\frac{|u|^6}{|x|} \,dx-\frac CR\|u\|_{L^6(\Omega)}^6.
\end{align*}
Collecting all these bounds and using the fact that $a_n\geq 0$ on $\partial \Omega$, we obtain
\begin{align}
\text{LHS}\eqref{17}\gtrsim \int_{|x|\le R}\frac{|u|^6}{|x|} \,dx - R^{-1} \bigl[ \|\nabla u\|_{L^2(\Omega)}^2 +\|u\|_{L^6(\Omega)}^6 \bigr].\label{lower bound}
\end{align}

Integrating \eqref{17} over $I$ and using \eqref{upper bound} and \eqref{lower bound} we derive
\begin{align*}
\int_I\int_{|x|\le R, x\in \Omega}\frac{|u|^6}{|x|} \,dx\,dt\lsm R+\frac{|I|}{R}.
\end{align*}
Taking $R=A|I|^{\frac 12}$ yields \eqref{mora}. This completes the proof of the lemma.
\end{proof}

We record next a local smoothing result.  While the local smoothing estimate does guarantee local energy decay, it falls short of fulfilling the role of a dispersive estimate.  In particular, local smoothing does not preclude the possibility that energy refocuses finitely many times.  Indeed, it is known to hold in merely non-trapping geometries.  Nevertheless, it does play a key role in the proof of the Strichartz estimates.  The version we need requires uniformity under translations and dilations; this necessitates some mild modifications of the usual argument.

\begin{lem}[Local smoothing]\label{L:local smoothing}
Let $u=e^{it\Delta_\Omega} u_0$.  Then
\begin{align*}
\int_\R \int_\Omega  |\nabla u(t,x)|^2  \bigl\langle R^{-1} (x-z)\bigr\rangle^{-3} \,dx\,dt \lesssim R \| u_0 \|_{L^2(\Omega)} \|\nabla u_0 \|_{L^2(\Omega)},
\end{align*}
uniformly for $z\in \R^3$ and $R>0$.
\end{lem}

\begin{proof}
We adapt the proof of local smoothing using the Morawetz identity
from the Euclidean setting.  For the level of generality needed
here, we need to combine two Morawetz identities: one adapted to the
obstacle and a second adapted to the $R$ ball around $z$.

Recall that the origin is an interior point of $\Omega^c$.  Given $x\in\partial\Omega$, let $\vec n(x)$ denote the outward normal to the obstacle at this point.  As $\Omega^c$ is convex, there is a constant $C>0$ independent of $z$ so that
\begin{align}\label{E:ls geom}
 \bigl| \tfrac{R^{-1}(x-z)}{\langle R^{-1}(x-z)\rangle} \cdot \vec n(x) \bigr| \leq C \tfrac{x}{|x|}\cdot \vec n(x) \qtq{for all} x\in\partial\Omega.
\end{align}
Indeed, the right-hand side is bounded away from zero uniformly for $x\in\partial\Omega$, while the set of vectors $\frac{R^{-1}(x-z)}{\langle R^{-1}(x-z)\rangle}$ is compact.

For $C>0$ as above, let
$$
F(t) := \int_\Omega \Im( \bar u \nabla u) \cdot \nabla a\, dx \qtq{with} a(x) := C |x| + R \langle R^{-1} (x-z) \bigr\rangle.
$$
After integrating by parts several times (cf. Lemma~\ref{L:morawetz}) and using that
$$
-\Delta\Delta a(x) \geq 0 \qtq{and} \partial_j\partial_k a(x) \geq \tfrac{1}{R} \langle R^{-1} (x-z) \bigr\rangle^{-3} \delta_{jk}
\quad \text{(as symmetric matrices)}
$$
one obtains
$$
\partial_t F(t) \geq 2 \int_\Omega \frac{|\nabla u(t,x)|^2  \,dx}{R \langle R^{-1} (x-z) \rangle^{3}}
 + \int_{\partial\Omega} |\nabla u(t,x)|^2 \bigl[ \nabla a(x)\cdot \vec n(x) \bigr] \,d\sigma(x).
$$
Moreover, by \eqref{E:ls geom} the integral over $\partial\Omega$ is positive since
$$
\nabla a(x)\cdot \vec n(x) = \bigl[C \tfrac{x}{|x|} + \tfrac{R^{-1}(x-z)}{\langle R^{-1}(x-z)\rangle} \bigr]\cdot \vec n(x) \geq 0
	\qtq{for} x\in\partial\Omega.
$$

Noting that $|F(t)|\leq (C+1) \| u(t) \|_{L^2(\Omega)} \| \nabla u(t) \|_{L^2(\Omega)}$, the lemma now follows by applying the
fundamental theorem of calculus.
\end{proof}

The remainder term in the linear profile decomposition Theorem~\ref{T:LPD} goes to zero in $L^{10}_{t,x}$; however, in order
to prove the approximate solution property (cf. Claim~3 in the proof of Proposition~\ref{P:PS}), we need to show
smallness in Strichartz spaces with one derivative. This is achieved via local smoothing (cf. Lemma~3.7 from
\cite{keraani-h1}); the uniformity in Lemma~\ref{L:local smoothing} is essential for this application.

\begin{cor}\label{C:Keraani3.7}
Given $w_0\in \dot H^1_D(\Omega)$,
$$
\| \nabla e^{it\Delta_\Omega} w_0 \|_{L^{\frac52}_{t,x}([\tau-T,\tau+T]\times\{|x-z|\leq R\})} \lesssim
     T^{\frac{31}{180}} R^{\frac7{45}} \| e^{it\Delta_\Omega} w_0 \|_{L^{10}_{t,x}(\R\times\Omega)}^{\frac1{18}}
     	\| w_0 \|_{\dot H^1_D(\Omega)}^{\frac{17}{18}},
$$
uniformly in $w_0$ and the parameters $R,T > 0$, $\tau\in\R$, and  $z\in\R^3$.
\end{cor}

\begin{proof}
Replacing $w_0$ by $e^{i\tau\Delta_\Omega} w_0$, we see that it suffices to treat the case $\tau=0$.

By H\"older's inequality,
\begin{align*}
\| \nabla e^{it\Delta_\Omega} & w_0
\|_{L^{\frac52}_{t,x}([-T,T]\times\{|x-z|\leq R\})} \\
&\lesssim \| \nabla e^{it\Delta_\Omega} w_0
\|_{L^2_{t,x}([-T,T]\times\{|x-z|\leq R\})}^{\frac13}
\|\nabla e^{it\Delta_\Omega} w_0
\|_{L^{\frac{20}7}_{t,x}([-T,T]\times\Omega)}^{\frac23}.
\end{align*}
We will estimate the two factors on the right-hand side separately. By the H\"older and Strichartz inequalities, as well as the equivalence
of Sobolev spaces, we estimate
\begin{align*}
\|\nabla e^{it\Delta_\Omega} w_0 \|_{L^{\frac{20}7}_{t,x}([-T,T]\times\Omega)}
&\lesssim T^{\frac 18} \| (-\Delta_\Omega)^{\frac12}
e^{it\Delta_\Omega} w_0 \|_{L^{\frac{40}9}_t L^{\frac{20}7}_x}
\lesssim T^{\frac 18} \|w_0\|_{\dot H^1_D(\Omega)} .
\end{align*}
In this way, the proof of the corollary reduces to showing
\begin{align}\label{E:LS1022}
\| \nabla e^{it\Delta_\Omega} w_0
\|_{L^2_{t,x}([-T,T]\times\{|x-z|\leq R\})}
\lsm T^{\frac4{15}} R^{\frac7{15}} \| e^{it\Delta_\Omega} w_0 \|_{L^{10}_{t,x}}^{\frac16} \| w_0 \|_{\dot H^1_D(\Omega)}^{\frac56}.
\end{align}

Given $N>0$, using the H\"older, Bernstein, and Strichartz inequalities, as well as the equivalence of Sobolev spaces, we have
\begin{align*}
\bigl\| \nabla e^{it\Delta_\Omega} P_{< N}^\Omega & w_0 \bigr\|_{L^2_{t,x}([-T,T]\times\{|x-z|\leq R\})} \\
&\lesssim T^{\frac25} R^{\frac9{20}} \bigl\| \nabla e^{it\Delta_\Omega} P_{< N}^\Omega w_0 \bigr\|_{L^{10}_tL^{\frac{20}7}_x} \\
&\lsm T^{\frac25} R^{\frac9{20}} N^{\frac14} \| (-\Delta_\Omega)^{\frac38} e^{it\Delta_\Omega} P_{< N}^\Omega w_0 \|_{L^{10}_t L^{\frac{20}7}_x} \\
&\lesssim T^{\frac25} R^{\frac9{20}} N^{\frac14} \| e^{it\Delta_\Omega} w_0 \|_{L^{10}_{t,x}}^{\frac14}
	\| (-\Delta_\Omega)^{\frac12} e^{it\Delta_\Omega} w_0 \|_{L^{10}_t L^{\frac{30}{13}}_x}^{\frac34} \\
&\lesssim T^{\frac25} R^{\frac9{20}} N^{\frac14} \| e^{it\Delta_\Omega} w_0 \|_{L^{10}_{t,x}}^{\frac14} \| w_0 \|_{\dot H^1_D(\Omega)}^{\frac34}.
\end{align*}
We estimate the high frequencies using Lemma~\ref{L:local smoothing} and the Bernstein inequality:
\begin{align*}
\bigl\| \nabla e^{it\Delta_\Omega} P_{\geq N}^\Omega w_0 \bigr\|_{L^2_{t,x}([-T,T]\times\{|x-z|\leq R\})} ^2
&\lesssim R \| P_{\geq N}^\Omega w_0 \|_{L^2_x}  \| \nabla P_{\geq N}^\Omega w_0 \|_{L^2_x} \\
&\lesssim  R N^{-1} \| w_0 \|_{\dot H^1_D(\Omega)}^2.
\end{align*}
The estimate \eqref{E:LS1022} now follows by optimizing in the choice of $N$.
\end{proof}

%%%%%%%%%%%%%%%%%%%%%%%%%%%%%%%%%%%%%%%%%%%%%%%%%%%%%%%%%%%%%%%%%%%%%%%%%%%%%%%%%%%%%%%%%%%%%%%%%%%%%%%%%%%%%%%%%%%%%%%%%%%%%%%%%%%%%%
\section{Convergence of domains}\label{S:Domain Convergence}
%%%%%%%%%%%%%%%%%%%%%%%%%%%%%%%%%%%%%%%%%%%%%%%%%%%%%%%%%%%%%%%%%%%%%%%%%%%%%%%%%%%%%%%%%%%%%%%%%%%%%%%%%%%%%%%%%%%%%%%%%%%%%%%%%%%%%%

The region $\Omega$ is not invariant under scaling or translation; indeed, under suitable choices of such operations,
the obstacle may shrink to a point, march off to infinity, or even expand to fill a halfspace.  The objective of this
section is to prove some rudimentary statements about the behaviour of functions of the Dirichlet Laplacian under
such circumstances.  In the next section, we address the much more subtle question of the convergence of propagators
in Strichartz spaces.

We begin by defining a notion of convergence of domains that is general enough to cover the scenarios discussed in this
paper, without being so general as to make the arguments unnecessarily complicated.  Throughout, we write
$$
G_\OO(x,y;z) := (-\Delta_\OO - z)^{-1}(x,y)
$$
for the Green's function of the Dirichlet Laplacian in a general open set $\OO$.  This function is symmetric under
the interchange of $x$ and $y$.

\begin{defn}\label{D:converg}
Given a sequence $\OO_n$ of open subsets of $\R^3$ we define
$$
\tlim \OO_n := \{ x\in \R^3 :\, \liminf_{n\to\infty} \dist(x,\OO_n^c) > 0\}.
$$
Writing $\tilde\OO=\tlim \OO_n$, we say $\OO_n\to\OO$ if the following two conditions hold: $\OO\triangle\tilde\OO$ is a finite set and
\begin{align}\label{cr2}
G_{\OO_n}(x,y;z)\to G_{\OO}(x,y;z)
\end{align}
for all $z\in (-2,-1)$, all $x\in \tilde\OO$, and uniformly for $y$ in compact subsets of $\tilde\OO\setminus \{x\}$.
\end{defn}

The arguments that follow adapt immediately to allow the symmetric difference $\OO\triangle\tilde\OO$ to be a set
of vanishing Minkowski $1$-content, rather than being merely finite.  The role of this hypothesis is to guarantee that this
set is removable for $\dot H^1_D(\OO)$; see Lemma~\ref{L:dense} below.  We restrict $z$ to the interval $(-2,-1)$
in \eqref{cr2} for simplicity and because it allows us to invoke the maximum principle when checking this
hypothesis.  Nevertheless, this implies convergence for all $z\in\C\setminus[0,\infty)$, as we will show in Lemma~\ref{lm:allz}.

\begin{lem}\label{L:dense} If $\OO_n\to \OO$, then $C^\infty_c(\tilde\OO)$ is dense in $\dot H^1_D(\OO)$.
\end{lem}

\begin{proof}
By definition, $C^\infty_c(\OO)$ is dense in $\dot H^1_D(\OO)$.  Given $f\in C^\infty_c(\OO)$ and $\eps>0$ define
$
f_\eps(x) := f(x) \prod_{k=1}^m \theta\bigl(\tfrac{x-x_k}{\eps}\bigr)
$
where $\{x_k\}_{k=1}^m$ enumerates $\OO\triangle\tilde\OO$ and $\theta:\R^3\to[0,1]$ is a smooth function that vanishes when $|x|\leq1$
and equals one when $|x|\geq2$.

Then $f_\eps \in C^\infty_c(\tilde\OO)\cap C^\infty_c(\OO)$ and
$$
\| f - f_\eps \|_{\dot H^1(\R^3)} \lesssim \sqrt{m\eps^3} \; \|\nabla f\|_{L^\infty} + \sqrt{m\eps}\; \|f\|_{L^\infty}.
$$
As $\eps$ can be chosen arbitrarily small, the proof is complete.
\end{proof}

In what follows, we will need some crude bounds on the Green's function that hold uniformly for the rescaled domains we consider.
While several existing methods could be used to obtain more precise results (cf. \cite{Hislop}), we prefer to give a simple
argument that gives satisfactory bounds and for which the needed uniformity is manifest.

\begin{lem}\label{L:G bnds}
For all open sets $\OO\subseteq\R^3$ and $z\in\C\setminus[0,\infty)$,
\begin{align}\label{moron}
\bigl|G_\OO(x,y;z)\bigr| \lesssim  \frac{|z|^2}{(\Im z)^2} e^{-\frac12\Re \sqrt{-z}|x-y|}  \Bigl(\frac1{|x-y|} + \sqrt{\Im z}\Bigr).
\end{align}
Moreover, if\/ $\Re z\leq 0$, then
\begin{align}\label{moron'}
\bigl|G_\OO(x,y;z)\bigr| \lesssim e^{-\frac12\Re \sqrt{-z}|x-y|}  \Bigl(\frac1{|x-y|} + \sqrt{\Im z}\Bigr).
\end{align}
\end{lem}

\begin{proof}
By the parabolic maximum principle, $0\leq e^{t\Delta_{\OO}}(x,y) \leq e^{t\Delta_{\R^3}}(x,y)$.  Thus,
$$
|G_{\OO} (x,y;z)| = \bigg| \int_0^\infty e^{tz + t\Delta_{\OO}}(x,y) \,dt \biggr| \leq \int_0^\infty e^{t\Re(z) + t\Delta_{\R^3}}(x,y) \,dt
	= G_{\R^3} (x,y;\Re z)
$$
for all $\Re z \leq 0$.  Using the explicit formula for the Green's function in $\R^3$, we deduce
\begin{equation}\label{Go bound}
| G_{\OO} (x,y;z) | \leq \frac{e^{-\sqrt{-\Re z}|x-y|}}{4\pi|x-y|} \qtq{whenever} \Re z\leq0.
\end{equation}
(When $z\in(-\infty,0]$ this follows more simply from the elliptic maximum principle.)

Note that the inequality \eqref{Go bound} implies \eqref{moron'} in the sector $\Re z < -|\Im z|$. Indeed, in this region we have
$\Re \sqrt{-z} \leq \sqrt{|z|} \leq 2^{\frac14} \sqrt{-\Re z}$.  In the remaining cases of \eqref{moron'}, namely, $-|\Im z|\leq \Re z\leq0$, we have
$1\leq\frac{|z|^2}{(\Im z)^2}\leq 2$ and so in this case \eqref{moron'} follows from \eqref{moron}.  Thus, it remains to establish \eqref{moron}.

To obtain the result for general $z\in\C\setminus[0,\infty)$, we combine \eqref{Go bound} with a crude bound elsewhere and the Phragm\'en--Lindel\"of principle.

From \eqref{Go bound} and duality, we have
$$
\bigl\| (-\Delta_{\OO}+|z|)^{-1} \bigr\|_{L^1\to L^2} = \bigl\| (-\Delta_{\OO}+|z|)^{-1} \bigr\|_{L^2\to L^\infty} \lesssim |z|^{-1/4}.
$$
Combining this information with the identity
$$
(-\Delta_{\OO}-z)^{-1} = (-\Delta_{\OO}+|z|)^{-1} +
	(-\Delta_{\OO}+|z|)^{-1}\biggl[\frac{(z+|z|)(-\Delta_{\OO}+|z|)}{-\Delta_{\OO}-z}\biggr](-\Delta_{\OO}+|z|)^{-1}
$$
and elementary estimations of the $L^2$-norm of the operator in square brackets, we deduce that
\begin{equation}\label{Go bound'}
|G_{\OO} (x,y;z)| \lesssim \frac{1}{|x-y|} + \frac{|z|^{3/2}}{|\Im z|} \qtq{for all} z\in\C\setminus[0,\infty).
\end{equation}

Using \eqref{Go bound} when $\Re z \leq 0$ and \eqref{Go bound'} when $\Re z >0$, we see that for given $\eps>0$ we have
\begin{equation}\label{log Go}
\log \biggl|\frac{G_{\OO} (x,y;z)}{(z+i\eps)^2}\biggr|
	\leq -|x-y|\Re\bigl(\sqrt{i\eps-z}\bigr) + \log\bigl(\tfrac1{|x-y|}+\sqrt\eps\bigr) + 2\log\bigl(\tfrac1\eps\bigr) + C
\end{equation}
for a universal constant $C$ and all $z$ with $\Im z =\eps$.  By the Phragm\'en--Lindel\"of principle, this inequality extends to the
entire halfspace $\Im z \geq \eps$.  Indeed, LHS\eqref{log Go} is subharmonic and converges to $-\infty$ at infinity, while RHS\eqref{log Go}
is harmonic and grows sublinearly.  To obtain the lemma at a fixed $z$ in the upper halfplane we apply \eqref{log Go} with $\eps=\frac12\Im z$
and use the elementary inequality
$$
\Re \sqrt{-u-\smash[b]{\tfrac i2 v}}  \geq \tfrac12 \Re \sqrt{-u-iv} \qtq{for all} u\in\R \qtq{and} v>0.
$$
The result for the lower halfplane follows by complex conjugation symmetry.
\end{proof}

\begin{lem}\label{lm:allz} If $\OO_n\to \OO$, then \eqref{cr2} holds uniformly for $z$ in compact subsets of $\C\setminus[0,\infty)$,
$x\in \tilde\OO$, and $y$ in compact subsets of $\tilde\OO\setminus \{x\}$.
\end{lem}

\begin{proof} We argue by contradiction. Suppose not.  Then there exist an $x\in \tilde\OO$ and a sequence $y_n\to y_\infty\in \tilde \OO\setminus \{x\}$ so that
\begin{align*}
f_n(z):=G_{\OO_n}(x, y_n, z)
\end{align*}
does not converge uniformly to $G_\OO(x,y_\infty;z)$ on some compact subset of $\C\setminus[0,\infty)$.

By Lemma~\ref{L:G bnds}, we see that $\{f_n\}$ are a normal family and so, after passing to a subsequence, converge
uniformly on compact sets to some $f(z)$. As $G_{\OO_n}(x, y_n;z)\to G_\OO(x,y_\infty;z)$ whenever $z\in (-2,-1)$,
the limit must be $f(z)=G_\OO(x,y_\infty;z)$.  This shows that it was unnecessary to pass to a subsequence, thus providing the sought-after contradiction.
\end{proof}

Given sequences of scaling and translation parameters $N_n\in 2^\Z$ and $x_n\in\Omega$, we wish to consider the domains $N_n(\Omega-\{x_n\})$.
Writing $d(x_n):=\dist(x_n,\Omega^c)$ and passing to a subsequence, we identify four specific scenarios:

\begin{CI}
\item Case 1: $N_n\equiv N_\infty$ and $x_n\to x_\infty\in \Omega$.  Here we set $\Omega_n:=\Omega$.

\item Case 2: $N_n\to 0$ and $-N_n x_n \to x_\infty\in\R^3$. Here $\Omega_n:=N_n(\Omega-\{x_n\})$.

\item Case 3: $N_nd(x_n)\to \infty$.  Here $\Omega_n:=N_n(\Omega-\{x_n\})$.

\item\mbox{Case 4: }$N_n\to\infty$ and $N_n d(x_n)\to d_\infty>0$.  Here $\Omega_n:=N_nR_n^{-1}(\Omega-\{x_n^*\})$, where $x_n^*\in\partial\Omega$ and $R_n\in SO(3)$ are chosen so that $d(x_n)=|x_n-x_n^*|$ and $R_n e_3 = \frac{x_n-x_n^*}{|x_n-x_n^*|}$.
\end{CI}
The seemingly missing possibility, namely, $N_n \gtrsim 1$ and $N_n d(x_n)\to 0$ will be precluded in the proof of Proposition~\ref{P:inverse Strichartz}.

In Case~1, the domain modifications are so tame as to not require further analysis, as is reflected by the choice of $\Omega_n$.  The definition of
$\Omega_n$ in Case~4 incorporates additional translations and rotations to normalize the limiting halfspace to be
$$
\HH := \{ x\in\R^3 : e_3 \cdot x >0 \} \qtq{where} e_3:=(0,0,1).
$$
In Cases~2 and~3, the limiting domain is $\R^3$, as we now show.

\begin{prop}\label{P:convdomain}
In Cases~2 and~3, $\Omega_n\to \R^3;$ in Case 4, $\Omega_n\to \HH$.
\end{prop}

\begin{proof} In Case 2, we have $\tlim \Omega_n = \R^3\setminus\{x_\infty\}$. It remains to show convergence of the Green's functions.

Let $C_0>0$ be a constant to be chosen later. We will show that for $z\in(-2,-1)$ and $n$ sufficiently large,
\begin{align}\label{G to R lb}
G_{\Omega_n}(x,y;z)\ge G_{\R^3}(x,y;z)-C_0N_nG_{\R^3}(x,-x_nN_n;z)
\end{align}
for $x\in \R^3\setminus\{ x_\infty\}$ fixed and $y$ in any compact subset $K$ of $\R^3\setminus\{x,x_\infty\}$. Indeed, for $n$ large enough we
have $x\in \Omega_n$ and $K\subseteq\Omega_n$. Also, for $x_0\in \partial\Omega_n$ we have $|x_0+x_n N_n|\le \diam(\Omega^c)N_n$.  Thus,
for $z\in(-2,-1)$ we estimate
\begin{align*}
G_{\R^3}(x_0,y;z)-C_0N_nG_{\R^3}(x_0,-x_nN_n;z)&=\frac{e^{-\sqrt{-z}|x_0-y|}}{4\pi|x_0-y|}
-C_0N_n\frac{e^{-\sqrt{-z}|x_0+x_nN_n|}}{4\pi|x_0+x_nN_n|}\\
&\le \frac{e^{-\sqrt{-z}|x_0-y|}}{4\pi|x_0-y|}-C_0\frac{e^{-\diam\sqrt{-z}N_n}}{4\pi \diam}<0,
\end{align*}
provided $C_0>\sup_{y\in K}\frac {\diam}{|x_0-y|}$ and $n$ is sufficiently large. Thus \eqref{G to R lb} follows from the maximum principle.

The maximum principle also implies $G_{\R^3}(x,y;z)\ge G_{\Omega_n}(x,y;z) \geq 0$. Combining this with \eqref{G to R lb},
we obtain
\begin{align*}
G_{\R^3}(x,y;z)-C_0N_nG_{\R^3}(x,-x_nN_n;z)\le G_{\Omega_n}(x,y;z)\le G_{\R^3}(x,y;z)
\end{align*}
for $n$ sufficiently large. As $N_n\to0$ and $-x_nN_n\to x_\infty$, this proves the claim in Case~2.

Next we consider Case 3. From the condition $N_nd(x_n)\to\infty$ it follows easily that $\tlim\Omega_n=\R^3$.
It remains to show the convergence of the Green's functions. By the maximum principle, $G_{\R^3}(x,y;z)\ge G_{\Omega_n}(x,y;z)$; thus, it suffices to prove a suitable lower bound.  To this end, let $\HH_n$ denote the halfspace containing $0$ for which $\partial\HH_n$ is the hyperplane
perpendicularly bisecting the line segment from $0$ to the nearest point on $\partial\Omega_n$. Note that $\dist(0,\partial\HH_n)\to\infty$
as $n\to\infty$. Given $x\in\R^3$ and a compact set $K\subset\R^3\setminus\{x\}$, the maximum principle guarantees that
\begin{align*}
G_{\Omega_n}(x,y;z)\ge G_{\HH_n}(x,y;z) \qtq{for all} y\in K \qtq{and} z\in(-2,-1),
\end{align*}
as long as $n$ is large enough that $x\in \HH_n$ and $K\subset \HH_n$.  Now
$$
G_{\HH_n}(x,y;z)=G_{\R^3}(x,y;z)-G_{\R^3}(x,y_n;z),
$$
where $y_n$ is the reflection of $y$ across $\partial\HH_n$.  Thus,
\begin{align*}
G_{\HH_n}(x,y;z)=\frac{e^{-\sqrt{-z}|x-y|}}{4\pi|x-y|} - \frac{e^{-\sqrt{-z}|x-y_n|}}{4\pi|x-y_n|}\to G_{\R^3}(x,y;z) \qtq{as} n\to \infty.
\end{align*}
This completes the treatment of Case 3.

It remains to prove the convergence in Case 4, where $\Omega_n=N_nR_n^{-1}(\Omega-\{x_n^*\})$,
$N_n\to\infty$, and $N_n d(x_n)\to d_\infty>0$.  It is elementary to see that $\tlim \Omega_n = \HH$; in particular, $\HH\subset \Omega_n$ for all $n$.
We need to verify that
\begin{align*}
G_{\Omega_n}(x,y;z)\to G_{\HH}(x,y;z)\qtq{for} z\in (-2,-1), \quad x\in \HH,
\end{align*}
and uniformly for $y$ in a compact set $K\subset\HH\setminus\{x\}$. By the maximum principle, $G_{\Omega_n}(x,y;z)\ge G_\HH(x,y;z)$.
On the other hand, we will show that
\begin{align}\label{s21}
G_{\Omega_n}(x,y;z)\le G_\HH(x,y;z)+ C{N_n^{-\eps}}e^{-\sqrt{-z}x_3},
\end{align}
for any $0<\eps<\frac 13$ and a large constant $C$ depending on $K$. As $N_n\to\infty$, these two bounds together
immediately imply the convergence of the Green's functions.

We now prove the upper bound \eqref{s21}. From the maximum principle it suffices to show that this holds just for $x\in \partial\Omega_n$,
which amounts to
\begin{align}\label{s22}
 |G_{\HH}(x,y;z)| \le C {N_n^{-\eps}}e^{-\sqrt{-z}x_3} \quad \text{for all } z\in (-2,-1), \ x\in \partial\Omega_n,\text{ and } y\in K. \!\!
\end{align}
Note that $G_{\HH}$ is negative for such $x$.  Recall also that
\begin{align*}
G_{\HH}(x,y;z)=\frac 1{4\pi}\biggl(\frac1{|x-y|}e^{-\sqrt{-z}|x-y|}-\frac 1{|x-\bar y|}e^{-\sqrt{-z}|x-\bar y|}\biggr),
\end{align*}
where $\bar y=(y^{\perp},-y_3)$ denotes the reflection of $y$ across $\partial\HH$.

If $x\in\partial\Omega_n$ with $|x|\ge N_n^\eps$ then we have $|x-y|\sim|x-\bar y|\gtrsim N_n^\eps$ for $n$ large and so
\begin{align*}
|G_{\HH}(x,y;z)|\le CN_n^{-\eps}e^{-\sqrt{-z}x_3},
\end{align*}
provided we choose $C \gtrsim \sup_{y\in K} \exp\{\sqrt{2}\,|y_3|\}$.

Now suppose $x\in\partial\Omega_n$ with $|x|\le N_n^{\eps}$.  As the curvature of $\partial\Omega_n$ is $O(N_n^{-1})$, for such
points we have $|x_3| \lesssim N_n^{2\eps-1}$.  Correspondingly,
$$
0 \leq |x-y| - |x-\bar y| = \frac{|x-y|^2 - |x-\bar y|^2}{|x-y|+|x-\bar y|} = \frac{4|x_3||y_3|}{|x-y|+|x-\bar y|} \lesssim_K N_n^{2\eps-1}.
$$
Thus, by the Lipschitz character of $r\mapsto e^{-\sqrt{-z}r}/r$ on compact subsets of $(0,\infty)$,
$$
|G_{\HH}(x,y;z)| \lesssim_K N_n^{2\eps-1}.
$$
On the other hand, since $|x_3| \lesssim N_n^{2\eps-1}\to 0$ as $n\to\infty$,
$$
{N_n^{-\eps}}e^{-\sqrt{-z}x_3} \gtrsim N_n^{-\eps}.
$$
As $0<\eps<\frac13$, this completes the justification of \eqref{s22} and so the proof of the lemma in Case~4.
\end{proof}

We conclude this section with two results we will need in Sections~\ref{S:LPD} and~\ref{S:Nonlinear Embedding}.

\begin{prop}\label{P:converg}
Assume $\Omega_n\to \Omega_\infty$ in the sense of Definition~\ref{D:converg} and let $\Theta\in C_c^{\infty}((0,\infty))$. Then
\begin{align}\label{E:P converg1}
\|[\Theta(-\lon)-\Theta(-\Delta_{\Omega_\infty})]\delta_y\|_{\dot H^{-1}(\R^3)} \to 0
\end{align}
uniformly for $y$ in compact subsets of $\,\tlim \Omega_n$.  Moreover, for any fixed $t\in \R$ and $h\in C_c^{\infty}(\tlim \Omega_n)$, we have
\begin{align}\label{E:P converg2}
\lim_{n\to\infty}\|e^{it\lon}h-e^{it\Delta_{\Omega_\infty}} h\|_{\dot H^{-1}(\R^3)}=0.
\end{align}
\end{prop}

\begin{proof}
By the Helffer-Sj\"ostrand formula (cf. \cite[p. 172]{HelfferSjostrand}), we may write
\begin{align*}
\Theta(-\Delta_{\OO})(x,y)=\int_{\C} G_{\OO}(x,y;z)\rho_\Theta(z) \, d{\hbox{Area}},
\end{align*}
where $\rho_\theta\in C_c^\infty(\C)$ with $|\rho_{\Theta}(z)|\lsm |\Im z|^{20}$.  Note that by Lemma~\ref{L:G bnds}
this integral is absolutely convergent; moreover, we obtain the following bounds:
\begin{align}\label{s18}
|\Theta(-\Delta_{\OO})(x,y)|\lsm |x-y|^{-1}\langle x-y\rangle^{-10},
\end{align}
uniformly for any domain $\OO$.

As $\Omega_n\to \Omega_\infty$, applying dominated convergence in the Helffer-Sj\"ostrand formula also guarantees that
$\Theta(-\Delta_{\Omega_n})(x,y)\to \Theta(-\Delta_{\Omega_\infty})(x,y)$ for each
$x\in \tilde\Omega_\infty:=\tlim \Omega_n$ fixed and uniformly for $y$ in
compact subsets of $\tilde\Omega_\infty\setminus\{x\}$.  Combining this with \eqref{s18} and applying
the dominated convergence theorem again yields
\begin{align*}
\|\Theta(-\Delta_{\Omega_n})\delta_y-\Theta(-\Delta_{\Omega_\infty})\delta_y\|_{L_x^{\frac
65}}\to 0,
\end{align*}
which proves \eqref{E:P converg1} since by Sobolev embedding $L_x^{6/5}(\R^3)\subseteq \dot H^{-1}_x(\R^3)$.

We turn now to \eqref{E:P converg2}. From the $L^{6/5}_x$-convergence of Littlewood--Paley expansions (cf. \cite[\S4]{KVZ:HA}),
we see that given $\eps>0$ and $h\in C_c^{\infty}(\tlim \Omega_n)$, there is a smooth function $\Theta:(0,\infty)\to[0,1]$
of compact support so that
$$
\| [1 - \Theta(-\Delta_{\Omega_\infty})] h \|_{\dot H^{-1}(\R^3)} \leq \eps.
$$
Combining this with \eqref{E:P converg1} we deduce that
$$
\limsup_{n\to\infty} \| [1 - \Theta(-\Delta_{\Omega_n})] h \|_{\dot H^{-1}(\R^3)} \leq \eps.
$$
In this way, the proof of \eqref{E:P converg2} reduces to showing
$$
\lim_{n\to\infty}\bigl\|e^{it\lon}\Theta(-\Delta_{\Omega_n}) h-e^{it\Delta_\infty} \Theta(-\Delta_{\Omega_\infty}) h\bigr\|_{\dot H^{-1}(\R^3)}=0,
$$
which follows immediately from \eqref{E:P converg1}.
\end{proof}

\begin{lem}[Convergence of $\dot H^1_D$ spaces]\label{L:n3}
Let $\Omega_n\to \Omega_\infty$ in the sense of Definition~\ref{D:converg}.
Then we have
\begin{align}\label{n4}
\|(-\Delta_{\Omega_n})^{\frac 12} f-(-\Delta_{\Omega_\infty})^{\frac 12} f\|_{L^2(\R^3)}\to 0 \qtq{for all} f\in C^\infty_c(\tlim\Omega_n).
\end{align}
\end{lem}

\begin{proof}
By the definition of $\tlim\Omega_n$, any $f\in C_c^{\infty}(\tlim \Omega_n)$ obeys $\supp(f)\subseteq\Omega_n$
for $n$ sufficiently large $n$ and for such $n$ we have
\begin{align}\label{normequal}
\| (-\Delta_{\Omega_n})^{\frac 12} f \|_{L^2(\R^3)} =\|\nabla f\|_{L^2(\R^3)} = \| (-\Delta_{\Omega_\infty})^{\frac 12} f \|_{L^2(\R^3)}.
\end{align}

Given $\eps>0$, there exists $\Theta_\eps\in C_c^\infty((0,\infty))$ such that
\begin{align*}
\sup_{\lambda\in[0,\infty)}\ \biggl|\frac{\sqrt\lambda}{1+\lambda}-\Theta_\eps(\lambda) \biggr| <\eps.
\end{align*}
Thus for any $g\in C_c^{\infty}(\R^3)$, we have
\begin{align*}
\langle g, (-\Delta_{\Omega_n})^{\frac 12} f\rangle &=\bigl\langle g, \frac{(-\Delta_{\Omega_n})^{\frac12}}{1-\Delta_{\Omega_n}}(1-\Delta) f\bigr\rangle
	=\langle g, \Theta_\eps(-\Delta_{\Omega_n})(1-\Delta)f\rangle +O(\eps).
\end{align*}
Using Proposition \ref{P:converg} and the same reasoning, we obtain
\begin{align*}
\lim_{n\to\infty} \langle g, \Theta_\eps(-\Delta_{\Omega_n})&(1-\Delta) f\rangle
	 = \langle g, \Theta_\eps(-\Delta_{\Omega_\infty})(1-\Delta)f\rangle
	 = \langle g, (-\Delta_{\Omega_\infty})^{\frac 12} f\rangle + O(\eps).
\end{align*}
Putting these two equalities together and using the fact that $\eps>0$ was arbitrary, we deduce that
$(-\Delta_{\Omega_n})^{\frac 12} f \rightharpoonup (-\Delta_{\Omega_\infty})^{\frac 12} f$ weakly in $L^2(\R^3)$.
Combining this with \eqref{normequal} gives strong convergence in $L^2(\R^3)$ and so proves the lemma.
\end{proof}

%%%%%%%%%%%%%%%%%%%%%%%%%%%%%%%%%%%%%%%%%%%%%%%%%%%%%%%%%%%%%%%%%%%%%%%%%%%%%%%%%%%%%%%%%%%%%%%%%%%%%%%%%%%%%%%%%%%%%%%%%%%%%%%%%%%%%%
\section{Convergence of linear flows}\label{S:Linear flow convergence}
%%%%%%%%%%%%%%%%%%%%%%%%%%%%%%%%%%%%%%%%%%%%%%%%%%%%%%%%%%%%%%%%%%%%%%%%%%%%%%%%%%%%%%%%%%%%%%%%%%%%%%%%%%%%%%%%%%%%%%%%%%%%%%%%%%%%%%

In this section we prove convergence of free propagators in Strichartz spaces, as we rescale and translate the domain $\Omega$ by parameters $N_n\in 2^\Z$ and $x_n\in \Omega$ conforming to one of the following three scenarios:
\begin{equation}\label{scenarios}
\left\{ \ \begin{aligned}
&\text{(i) $N_n\to 0$ and $-N_n x_n \to x_\infty\in \R^3$}\\
&\text{(ii) $N_n d(x_n)\to \infty$}\\
&\text{(iii) $N_n\to\infty$ and $N_n d(x_n) \to d_\infty>0$.}
\end{aligned}\right.
\end{equation}
Here we use the shorthand $d(x_n)=\dist(x_n, \Omega^c)$.  Notice that these are Cases~2--4 discussed in Section~\ref{S:Domain Convergence}.  We will not discuss Case~1 of Section~\ref{S:Domain Convergence} here; there is no change in geometry in Case~1, which renders the results of this section self-evident.

As seen in Section~\ref{S:Domain Convergence}, the  limiting geometry in the first and second scenarios is the whole space $\R^3$, while in the third scenario the limiting geometry is the halfspace $\HH$ (after a suitable normalization).  More precisely, in the first and second scenarios writing
$\Omega_n=N_n(\Omega-\{x_n\})$, Proposition~\ref{P:convdomain} gives $\Omega_n\to \R^3$.  In the third scenario, we define
$\Omega_n=N_nR_n^{-1}(\Omega-\{x_n^*\})$, where $x_n^*\in\partial\Omega$ and $R_n\in SO(3)$ are chosen so that $d(x_n)=|x_n-x_n^*|$ and $R_n e_3 = \frac{x_n-x_n^*}{|x_n-x_n^*|}$; in this scenario, Proposition~\ref{P:convdomain} gives $\Omega_n\to \HH=\{x\in\R^3:x\cdot e_3>0\}$.

The main result in this section is the following:

\begin{thm}[Convergence of linear flows in Strichartz spaces]\label{T:LF}\hskip 0em plus 1em
Let $\Omega_n$ be as above and let $\Omega_\infty$ be such that
$\Omega_n\to \Omega_\infty$. Then
\begin{align*}
\lim_{n\to \infty}\| e^{it\lon} \psi-e^{it\Delta_{\Omega_\infty}}\psi\|_{L_t^q L_x^r(\R\times\R^3)}=0
\end{align*}
for all $\psi\in C_c^{\infty}(\tlim \Omega_n)$ and all pairs $(q,r)$ satisfying $\frac2q+\frac3r=\frac32$ with $2<q<\infty$ and $2<r<6$.
\end{thm}

In this paper we are considering an energy-critical problem and so need an analogue of this theorem with the corresponding scaling.  To this end, we prove the following corollary, which will be used to obtain a linear profile decomposition for the propagator $e^{it\Delta_\Omega}$ in the following section.

\begin{cor}[Convergence of linear flows in $L_{t,x}^{10}$]\label{C:LF} Let $\Omega_n$ be as above and let $\Omega_\infty$ be such that
$\Omega_n\to \Omega_\infty$. Then
\begin{align*}
\lim_{n\to \infty}\| e^{it\lon} \psi-e^{it\Delta_{\Omega_\infty}}\psi\|_{L_{t,x}^{10}(\R\times\R^3)}=0
\end{align*}
for all $\psi\in C_c^{\infty}(\tlim \Omega_n)$.
\end{cor}

\begin{proof}
By H\"older's inequality,
\begin{align*}
\| e^{it\lon} \psi-e^{it\Delta_{\Omega_\infty}}\psi\|_{L_{t,x}^{10}}
&\lesssim \| e^{it\lon} \psi-e^{it\Delta_{\Omega_\infty}}\psi\|_{L_{t,x}^{10/3}}^{1/3} \| e^{it\lon} \psi-e^{it\Delta_{\Omega_\infty}}\psi\|_{L_{t,x}^{\infty}}^{2/3}.
\end{align*}
The corollary then follows from Theorem~\ref{T:LF} and the following consequence of Sobolev embedding:
$$
\| e^{it\lon} \psi\|_{L_{t,x}^{\infty}} + \| e^{it\Delta_{\Omega_\infty}}\psi\|_{L_{t,x}^{\infty}} \lesssim \|  (1-\Delta_{\Omega_n}) \psi \|_{L_t^{\infty} L_x^2} + \|  (1-\Delta_{\Omega_\infty}) \psi \|_{L_t^{\infty} L_x^2} \lesssim_\psi 1.
$$
Note that the implicit constant here does not depend on $n$, because the domains obey the interior cone condition with uniform constants.
\end{proof}

The proof of Theorem~\ref{T:LF} will occupy the remainder of this lengthy section.  We will consider three different regimes of behaviour for $N_n$ and $x_n$.  These do not exactly correspond to the three scenarios above, but rather are dictated by the method of proof.  The first such case is when $N_n\to 0$ or $d(x_n)\to \infty$.  The limiting geometry in this case is the whole of $\R^3$.

\begin{thm} \label{T:LF1} Let $\Omega_n=N_n(\Omega-\{x_n\})$ and assume that $N_n\to 0$ or $d(x_n)\to\infty$.  Then
\begin{align*}
\lim_{n\to\infty}\|e^{it\Delta_{\Omega_n}}\psi-e^{it\Delta_{\R^3}}\psi\|_{L_t^q L_x^r(\R\times\R^3)}=0
\end{align*}
for all $\psi\in C_c^{\infty}(\tlim \Omega_n)$  and all pairs $(q,r)$ satisfying $\frac2q+\frac3r=\frac32$ with $2<q<\infty$ and $2<r<6$.
\end{thm}

\begin{proof}
By interpolation and the Strichartz inequality, it suffices to prove convergence in the symmetric Strichartz space $q=r=\frac{10}3$.  To ease notation, we will simply write $-\Delta$ for $-\Delta_{\R^3}$.

Let $\Theta$ be a smooth radial cutoff such that
\begin{align*}
\Theta(x)=\begin{cases}0, &|x|\le \frac 14\\1, &|x|\ge \frac12\end{cases}
\end{align*}
and let $\chi_n(x):=\Theta\bigl(\frac{\dist(x,\Omega_n^c)}{\diam(\Omega_n^c)}\bigr)$.

Note that if $N_n\to 0$ then $\diam(\Omega_n^c)\to 0$ and so $\supp(1-\chi_n)$ is a collapsing neighbourhood of the point $-N_n x_n$.

On the other hand, if $d(x_n)\to\infty$ then we have $\frac{\dist(0,\Omega_n^c)}{\diam (\Omega_n^c)}\to\infty$.  As for $x\in \supp(1-\chi_n)$ we have
$\dist(x,\Omega_n^c)\le \frac 12\diam(\Omega_n^c)$, this gives
\begin{align*}
|x|\ge\dist(0,\Omega_n^c)-\dist(x,\Omega_n^c)\ge\dist(0,\Omega_n^c)-\tfrac 12\diam(\Omega_n^c)\to\infty \qtq{as} n\to \infty.
\end{align*}

Now fix $\psi\in C_c^{\infty}(\tlim \Omega_n)$.  From the considerations above, for $n$ sufficiently large we have $\supp \psi\subseteq\{x\in \R^3:\, \chi_n(x)=1\}$.
Moreover, if $N_n\to 0$ or $d(x_n)\to\infty$, the monotone convergence theorem together with the Strichartz inequality give
\begin{align*}
\lim_{n\to\infty}\bigl\|(1-\chi_n)e^{it\Delta}\psi\bigr\|_{L_{t,x}^{\frac{10}3} (\R\times\R^3)}=0.
\end{align*}
We are thus left to estimate $e^{it\Delta_{\Omega_n}}\psi-\chi_n e^{it\Delta}\psi$.  From the Duhamel formula,
\begin{align*}
e^{it\Delta_{\Omega_n}}\psi=\chi_n e^{it\Delta}\psi+i\int_0^t e^{i(t-s)\Delta_{\Omega_n}}[\Delta, \chi_n] e^{is\Delta}\psi \,ds.
\end{align*}
Using the Strichartz inequality, we thus obtain
\begin{align}\label{1:43}
\|e^{it\Delta_{\Omega_n}}\psi-\chi_n e^{it\Delta}\psi\|_{L_{t,x}^{\frac{10}3}(\R\times\R^3)}
\lsm\bigl\|[\Delta,\chi_n] e^{it\Delta}\psi\bigr\|_{(L_{t,x}^{\frac{10}7}+L_t^1L_x^2)(\R\times\R^3)}.
\end{align}

To estimate the right-hand side of \eqref{1:43}, we discuss separately the cases: $(1)$ $N_n\to 0$ and $(2)$ $d(x_n)\to \infty$ with $N_n\gtrsim 1$.  In the first case, we estimate
\begin{align*}
\|&[\Delta, \chi_n]\prr\psi\bigr\|_{L_{t,x}^{\frac{10}7}}
\lsm \Bigl[\|\Delta\chi_n\|_{L_x^{\frac{10}7}}+\|\nabla \chi\|_{L_x^{\frac{10}7}}\Bigr]
	\Bigl[\|\prr\psi\|_{L_t^{\frac{10}7}L_x^\infty} + \|\prr\nabla\psi\|_{L_t^{\frac{10}7}L_x^\infty} \Bigr]\\
&\lsm\Bigl[\diam(\Omega_n^c)^{-2}+\diam(\Omega_n^c)^{-1}\Bigr]\diam(\Omega_n^c)^{\frac{21}{10}}
	\Bigl[\|\prr\psi\|_{L_t^{\frac{10}7}L_x^\infty} + \|\prr\nabla\psi\|_{L_t^{\frac{10}7}L_x^\infty} \Bigr] \\
&\lsm\Bigl[N_n^{\frac 1{10}}+N_n^{\frac{11}{10}}\Bigr] 
	\Bigl[\|\prr\psi\|_{L_t^{\frac{10}7}L_x^\infty} + \|\prr\nabla\psi\|_{L_t^{\frac{10}7}L_x^\infty} \Bigr].
\end{align*}
From the dispersive estimate and Sobolev embedding,
\begin{align}\label{E:4disp}
\|\prr\psi\|_{L_x^{\infty}}\lsm \lng t\rng^{-\frac32}\bigl[\|\psi\|_{L_x^1}+\|\psi\|_{H^2_x}\bigr]\lsm_\psi\lng t\rng^{-\frac 32},
\end{align}
and similarly with $\psi$ replaced by $\nabla\psi$.  Thus we obtain
\begin{align*}
\lim_{n\to\infty}\|[\Delta,\chi_n] \prr \psi\|_{L_{t,x}^{\frac{10}7}(\R\times\R^3)}=0.
\end{align*}

Consider now the case $N_n\gtrsim 1$ and $d(x_n)\to\infty$.  Then
\begin{align*}
\|[\Delta,\chi_n]\prr\psi\|_{L_t^1L_x^2}
&\lsm \bigl[\|\Delta\chi_n\|_{L_x^{\infty}}+\|\nabla\chi_n\|_{L_x^{\infty}}\bigr]\|\prr\lng\nabla\rng \psi\|_{L_t^1L_x^2(\dist(x,\Omega_n^c)\sim N_n)}\\
&\lsm \bigl[N_n^{-2}+N_n^{-1}\bigr]\|\prr\lng\nabla\rng\psi\|_{L_t^1L_x^2(\dist(x,\Omega_n^c)\sim N_n)}.
\end{align*}
Using H\"older's inequality and \eqref{E:4disp}, we obtain
\begin{align*}
\|\prr\lng\nabla\rng\psi\|_{L_x^2(\dist(x,\Omega_n^c)\sim N_n)}
&\lsm N_n^{\frac 32}\|\prr\lng\nabla\rng\psi\|_{L_x^\infty}
\lsm_{\psi} N_n^{\frac 32}\lng t\rng^{-\frac 32}.
\end{align*}
On the other hand, from the virial identity,
\begin{align*}
\|x\prr\lng\nabla\rng\psi\|_{L_x^2}\lsm_\psi\lng t\rng
\end{align*}
and so,
\begin{align*}
\|\prr\lng\nabla\rng\psi\|_{L_x^2(\dist(x,\Omega_n^c)\sim N_n)}
&\lsm\frac 1{\dist(0,\Omega_n^c)}\|x\prr\lng \nabla\rng\psi\|_{L_x^2}\lsm_\psi\frac{\lng t\rng}{N_nd(x_n)}.
\end{align*}
Collecting these estimates we obtain
\begin{align*}
\|\prr\lng\nabla\rng\psi\|_{L_t^1L_x^2(\dist(x,\Omega_n^c)\sim N_n)}
&\lsm_\psi\int_0^{\infty}\min\biggl\{\frac{N_n^{\frac 32}}{\lng t\rng^{\frac 32}}, \ \frac{\lng t\rng}{N_nd(x_n)}\biggr\}\\
&\lsm _\psi N_nd(x_n)^{-\frac 15}+\min\bigl\{N_n^{\frac 32}, N_n^{-1}d(x_n)^{-1}\bigr\}
\end{align*}
and so,
\begin{align*}
\|[\Delta,\chi_n]\prr\psi\|_{L_t^1L_x^2}\lsm_\psi d(x_n)^{-\frac 15}+N_n^{-2}d(x_n)^{-1}\to 0 \qtq{as} n\to \infty.
\end{align*}

This completes the proof of the theorem.
\end{proof}

Theorem~\ref{T:LF1} settles Theorem~\ref{T:LF} for $N_n$ and $x_n$ conforming to scenario (i) in \eqref{scenarios}, as well as part of scenario (ii).  The missing part of the second scenario is $N_n d(x_n)\to \infty$ with $N_n\to \infty$  and $d(x_n)\lesssim 1$.  Of course, we also have to prove Theorem~\ref{T:LF} for $N_n$ and $x_n$ conforming to scenario (iii), namely, $N_n d(x_n) \to d_\infty>0$ and  $N_n\to\infty$.  We will cover these remaining cases in two parts:
\begin{SL}\addtocounter{smalllist}{3}
\item $N_n\to \infty$  and $1\lesssim N_nd(x_n) \leq N_n^{1/7}$
\item $N_n\to \infty$  and $N_n^{1/7}\leq N_nd(x_n) \lesssim N_n$.
\end{SL}
Note that in case (iv) the obstacle $\Omega_n^c$ grows in diameter much faster than its distance to the origin.  As seen from the origin, the obstacle is turning into a (possibly retreating) halfspace.  By comparison, case (v) includes the possibility that the obstacle grows at a rate comparable to its  distance to the origin.

The two cases will receive different treatments.  In Case~(iv), we use a parametrix construction adapted to the halfspace evolution.  We also prove that when the halfspace is retreating, the halfspace propagators converge to $e^{it\Delta_{\R^3}}$; see Proposition~\ref{P:HtoR}.  In Case~(v), the parametrix construction will be inspired by geometric optics considerations and will require a very fine analysis.

We now turn to the details of the proof of Theorem~\ref{T:LF} in Case~(iv).

\subsection{Case (iv)}  After rescaling we find ourselves in the setting shown schematically in Figure~\ref{F:case4} below, with
$\eps = N_n^{-1}$.  This restores the obstacle to its original size.  We further rotate and translate the problem
so that the origin lies on the boundary of the obstacle, the outward normal is $e_3$ at this point, and the wave packet $\psi$ is centered
around the point $\delta e_3$.  Abusing notation, we will write $\Omega$ for this new rotated/translated domain.  As before, we write
$\HH=\{(x_1, x_2,x_3)\in\R^3:\, x_3>0\}$; by construction, $\partial\HH$ is the tangent plane to $\Omega^c$ at the origin.  Throughout this subsection, we write $x^{\perp}:=(x_1,x_2)$; also $\bar x:=(x_1,x_2,-x_3)$ denotes the reflection of $x$ in $\partial\HH$.

\begin{figure}[ht]
\begin{center}
%\fbox{
\setlength{\unitlength}{1mm}
% 1mm  = 2.84pt
\begin{picture}(95,40)(-45,-20)
\put(0,-20){\line(0,1){40}}
\put(30, 0){\circle*{4}}                                            % Blob for \supp(\psi)
\put(36.5,0){\vector(-1,0){4}}\put(37,-1){$\supp(\psi_{\eps,\delta})$}            % Label for \supp(\psi_\eps)
\put(30,-4){\vector(-1,0){ 2}}\put(30,-4){\vector(1,0){ 2}}\put(30,-8){\hbox to 0mm{\hss$\sim\eps$\hss}}       % dimension of data
\put(15,+4){\vector(-1,0){14}}\put(15,+4){\vector(1,0){15}}\put(15,+6){\hbox to 0mm{\hss$\delta$\hss}}     % distance of data to 0
\qbezier(0,0)(0,10)(-5,20)\qbezier(0,0)(0,-10)(-5,-20)                                                     % The obstacle
\put(-25,10){\vector(2,1){20}}\put(-25,10){\vector(-2,-1){20}}\put(-27,12){\hbox to 0mm{\hss$\sim 1$\hss}} % dimension of the obstacle
\put(+4.5,-18){\vector(-1,0){4}} \put(+5,-19){$\partial \HH$}       % Label for bndry of H
\put(-8.7,-18){\vector(1,0){4}} \put(-13.8,-19){$\partial \Omega$}  % Label for bndry of \Omega
\end{picture}
%}
\vspace*{-2ex}
\end{center}
\caption{Depiction of Case~(iv); here $\eps\leq\delta\leq\eps^{6/7}$ and $\eps\to 0$.}\label{F:case4}
\end{figure}
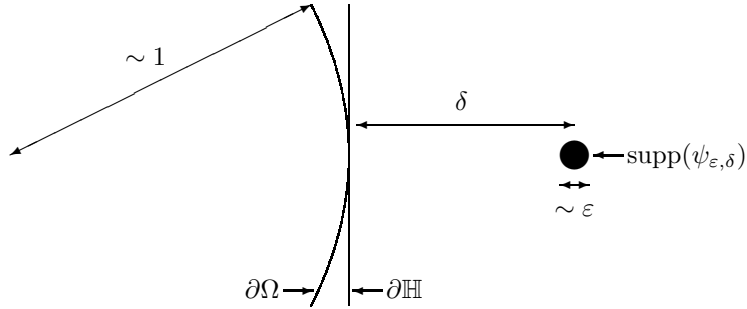

This subsection will primarily be devoted to the proof of the following

\begin{thm}\label{T:LF2}  Fix $\psi\in C_c^{\infty}(\HH - \{e_3\})$ and let
\begin{align*}
\psi_{\eps,\delta}(x):=\eps^{-\frac 32}\psi\biggl(\frac{ x-\delta e_3}\eps\biggr).
\end{align*}
Then for any pair $(q,r)$ satisfying $\frac2q+\frac3r=\frac32$ with $2<q<\infty$ and $2<r<6$ we have
\begin{align}\label{cas4}
\|e^{it\Delta_{\Omega(\eps)}}\psi_{\eps,\delta}-e^{it\Delta_{\HH}}\psi_{\eps,\delta}\|_{L_t^q L_x^r(\R\times\R^3)} \to 0
\end{align}
as $\eps\to0$ with any $\delta=\delta(\eps)$ obeying $\eps\leq\delta\leq\eps^{6/7}$.  Here $\Omega(\eps)$ is any family of affine images (i.e. rotations and translations) of $\Omega$ for which $\HH\subseteq\Omega(\eps)$ and $\partial\HH$ is the tangent plane to $\Omega(\eps)$ at the origin. 
\end{thm}

Theorem~\ref{T:LF2} gives Theorem~\ref{T:LF} for $N_n$ and $x_n$ conforming to scenario (iii) in \eqref{scenarios}.  Indeed, one applies Theorem~\ref{T:LF2} to the function $\tilde\psi(x)=\psi(x+e_3)$ with $\delta=\eps=N_n^{-1}$ and $\Omega(\eps)=R_n^{-1}(\Omega-\{x_n^*\})$.

With the aid of Proposition~\ref{P:HtoR} below, Theorem~\ref{T:LF2} also implies Theorem~\ref{T:LF} for $N_n$ and $x_n$ conforming to scenario (ii) with the additional restriction that $N_n^{1/7}\geq N_nd(x_n) \to \infty$.  In this case, we apply Theorem~\ref{T:LF2} to the function $\tilde\psi(x)=\psi_\infty(x-\rho e_3)$ with $\rho=\sup\{|x|:x\in\supp(\psi)\}$, $\eps=N_n^{-1}$, $\delta=d(x_n)-\eps\rho$, $\Omega(\eps)=R_n^{-1}(\Omega-\{x_n^*\})$, and $\psi_\infty$ being any subsequential limit of $\psi\circ R_n$.  As $\psi\circ R_n\to\psi_\infty$ in $L^2$ sense, the Strichartz inequality controls the resulting errors. 

\begin{prop}\label{P:HtoR}  Fix $\psi\in C_c^{\infty}(\HH - \{e_3\})$ and let
\begin{align*}
\psi_{\eps,\delta}(x):=\eps^{-\frac 32}\psi\biggl(\frac{ x-\delta e_3}\eps\biggr).
\end{align*}
Then for any pair $(q,r)$ satisfying $\frac2q+\frac3r=\frac32$ with $2<q<\infty$ and $2<r<6$ we have
\begin{align*}
\|e^{it\Delta_{\HH}}\psi_{\eps,\delta}-e^{it\Delta_{\R^3}}\psi_{\eps,\delta}\|_{L_t^q L_x^r(\R\times\R^3)} \to 0
\end{align*}
as $\eps\to0$ with any $\delta=\delta(\eps)$ obeying $\frac{\delta}{\eps}\to \infty$.
\end{prop}

\begin{proof}
We will prove the proposition in the special case $q=r=\frac{10}3$.  The result for general exponents follows from
the Strichartz inequality and interpolation, or by a simple modification of the arguments that follow.

Using the exact formulas for the propagator in $\R^3$ and $\HH$ and rescaling reduces the question to
\begin{align}\label{E:H2R1}
\|e^{it\Delta_{\R^3}} \tilde\psi_{\eps,\delta}\|_{L^{\frac{10}3}_{t,x}(\R\times\HH)} \to 0 \qtq{where}
	\tilde\psi_{\eps,\delta}(y) = \psi( \bar y - \tfrac\delta\eps e_3 ).
\end{align}
Notice that $\tilde\psi_{\eps,\delta}$ is supported deeply inside the complementary halfspace $\R^3\setminus\HH$.

For large values of $t$ we estimate as follows:  Combining the $L^1_x\to L^\infty_x$ dispersive estimate with mass conservation gives
$$
\|e^{it\Delta_{\R^3}} \tilde\psi_{\eps,\delta}\|_{L^{\frac{10}3}_{x}(\R^3)}
	\lesssim |t|^{-3/5} \|\tilde\psi_{\eps,\delta}\|_{L^1_x}^{\frac25} \|\tilde\psi_{\eps,\delta}\|_{L^2_x}^{\frac35} \lesssim_\psi |t|^{-3/5}.
$$
We use this bound when $|t| \geq T:=\sqrt{\delta/\eps}$ to obtain
\begin{align*}
\|e^{it\Delta_{\R^3}} \tilde\psi_{\eps,\delta}\|_{L^{\frac{10}3}_{t,x}(\{|t|\geq T\}\times\HH)}
&\lesssim_\psi \Bigl( \int_{T}^\infty t^{-2}\,dt\Bigr)^{\frac{3}{10}} \to 0 \qtq{as} \eps\to 0.
\end{align*}

For $|t|\leq T$, we use the virial estimate
$$
\bigl\| \bigl( y + \tfrac{\delta}{\eps} e_3) e^{it\Delta_{\R^3}} \tilde\psi_{\eps,\delta} \bigr\|_{L^2_{x}(\R^3)}^2
	\lesssim \bigl\| \bigl( y + \tfrac{\delta}{\eps} e_3) \tilde\psi_{\eps,\delta} \bigr\|_{L^2_{x}(\R^3)}^2
 + t^2 \bigl\| \nabla \tilde\psi_{\eps,\delta} \bigr\|_{L^2_{x}(\R^3)}^2  \lesssim_\psi \tfrac{\delta}{\eps}.
$$
This together with the H\"older and Strichartz inequalities gives
\begin{align*}
\|e^{it\Delta_{\R^3}} \tilde\psi_{\eps,\delta}\|_{L^{\frac{10}3}_{t,x}(\{|t|\leq T\}\times\HH)}
&\lesssim \|e^{it\Delta_{\R^3}} \tilde\psi_{\eps,\delta}\|_{L^\infty_t L^2_x(\{|t|\leq T\}\times\HH)}^{\frac25}
		\|e^{it\Delta_{\R^3}} \tilde\psi_{\eps,\delta}\|_{L^2_t L^6_x(\R\times\HH)}^{\frac35} \\
&\lesssim \bigl(\tfrac{\eps}{\delta}\bigr)^{\frac25}\bigl\| \bigl( y + \tfrac{\delta}{\eps} e_3) e^{it\Delta_{\R^3}} \tilde\psi_{\eps,\delta} \bigr\|_{L^\infty_t L^2_x(\{|t|\leq T\}\times\HH)}^{\frac25} \|\psi\|_{L_x^2}^{\frac35}\\
&\lsm_\psi \bigl(\tfrac\eps\delta\bigr)^{\frac15} \to 0 \qtq{as} \eps\to 0.
\end{align*}
This completes the proof of the proposition.
\end{proof}

We begin the proof of Theorem~\ref{T:LF2} by showing that we can approximate $\psi_{\eps,\delta}$ by Gaussians.

\begin{lem}[Approximation by Gaussians] \label{lm:exp4}
Let $\psi\in C_c^{\infty}(\HH-\{e_3\})$. Then for any $\eta>0$, $0<\eps\leq 1$, and $\delta\geq\eps$ there exist $N>0$, points
$\{y^{(n)}\}_{n=1}^N \subset \HH$, and constants $\{c_n\}_{n=1}^N\subset \C$ such that
\begin{align*}
\biggl\|\psi_{\eps,\delta}(x)-\sum_{n=1}^N c_n (2\pi\eps^2)^{-\frac34}\Bigl[\exp\bigl\{-\tfrac {|x-\delta y^{(n)}|^2}{4\eps^2}\bigr\}
-\exp\bigl\{-\tfrac {|x-\delta\bar{ y}^{(n)}|^2}{4\eps^2}\bigr\}\Bigr]\biggr\|_{L_x^2(\HH)}<\eta.
\end{align*}
Here, $\bar y^{(n)}$ denotes the reflection of $y^{(n)}$ in $\partial\HH$.  Moreover, we may ensure that
$$
\sum_n |c_n| \lesssim_\eta 1 \qtq{and} \sup_n |y^{(n)}-e_3| \lesssim_\eta \eps\delta^{-1},
$$
uniformly in $\eps$ and $\delta$.
\end{lem}

\begin{proof}
Wiener showed that linear combinations of translates of a fixed function in $L^2(\R^d)$ are dense in this space if and only if
the Fourier transform of this function is a.e. non-vanishing.  (Note that his Tauberian theorem is the analogous statement for $L^1$.)
In this way, we see that we can choose vectors $z^{(n)}\in \R^3$ and numbers $\tilde c_n$ so that
\begin{align*}
\biggl\| \psi(x) - \sum_{n=1}^N \tilde c_n (2\pi)^{-\frac 34} e^{-\frac{|x-z^{(n)}|^2}4}\biggr\|_{L_x^2(\R^3)}<\tfrac12 \eta.
\end{align*}

Rescaling, translating, and combining with the reflected formula, we deduce immediately that
\begin{align*}
\biggl\| \psi_{\eps,\delta}(x) - \psi_{\eps,\delta}(\bar x) - \sum_{n=1}^N c_n (2\pi\eps^2)^{-\frac34} \Bigl[
		e^{-\frac {|x-\delta y^{(n)}|^2}{4\eps^2}} - e^{-\frac {|x-\delta\bar{ y}^{(n)}|^2}{4\eps^2}}\Bigr]
		\biggr\|_{L_x^2(\R^3)}<\eta,
\end{align*}
where $y^{(n)} = \eps\delta^{-1} z^{(n)} + e_3$ and $c_n=\tilde c_n$ when $y^{(n)}\in \HH$; otherwise
we set $\bar y^{(n)} = \eps\delta^{-1} z^{(n)} + e_3$ and $c_n= - \tilde c_n$, which ensures $y^{(n)} \in \HH$.

As $\psi_{\delta,\eps}(x)$ is supported wholely inside $\HH$, so $\psi_{\delta,\eps}(\bar x)$ vanishes there.  Thus the lemma now follows.
\end{proof}

By interpolation and the Strichartz inequality, it suffices to prove Theorem~\ref{T:LF2} for the symmetric Strichartz pair $q=r=\frac{10}3$.  Also, to ease notation, we simply write $\Omega$ for $\Omega(\eps)$ in what follows.
 
Combining Lemma~\ref{lm:exp4} with the Strichartz inequality for both propagators $e^{it\Delta_{\Omega}}$ and $e^{it\Delta_{\HH}}$, we obtain
\begin{align*}
&\|e^{it\Delta_{\Omega}}\psi_{\eps, \delta}-e^{it\Delta_{\HH}}\psi_{\eps,\delta}\|_{L_{t,x}^{\frac{10}3}(\R\times\Omega)}\\
&\leq \sum_{n=1}^N |c_n|\Bigl\|\bigl[e^{it\Delta_{\Omega}}\chi_{\HH}-e^{it\Delta_{\HH}}\chi_{\HH}\bigr]
		(2\pi\eps^2)^{-\frac34}\bigl[e^{-\frac{|x-\delta y^{(n)}|^2}{4\eps^2}}-e^{-\frac{|x-\delta\bar y^{(n)}|^2}{4\eps^2}}\bigr]\Bigl\|_{L_{t,x}^{\frac{10}3}		 (\R\times\Omega)}\\
&\qquad+C\biggl\|\psi_{\eps,\delta} -\sum_{n=1}^N c_n(2\pi\eps^2)^{-\frac34}\bigl[e^{-\frac{|x-\delta y^{(n)}|^2}{4\eps^2}}-e^{-\frac{|x-\delta \bar y^  		 {(n)}|^2}{4\eps^2}}\bigr]\biggr\|_{L^2(\HH)}.
\end{align*}
Therefore, Theorem~\ref{T:LF2} is reduced to showing
\begin{align}\label{fr}
\Bigl\|\bigl[e^{it\Delta_{\Omega}}\chi_{\HH}-e^{it\Delta_{\HH}}\chi_{\HH}\bigr](2\pi\eps^2)^{-\frac34}\bigl[e^{-\frac {|x-\delta y|^2}{4\eps^2}}
-e^{-\frac {|x-\delta\bar{y}|^2}{4\eps^2}}\bigr]\Bigr\|_{L_{t,x}^{\frac{10}3}(\R\times\Omega)}=o(1)
\end{align}
as $\eps\to0$ with any $\delta=\delta(\eps)$ obeying $\eps\leq\delta\leq\eps^{6/7}$, and $y$ as in Lemma~\ref{lm:exp4}.

Next we show that we can further simplify our task to considering only $y\in\HH$ of the form $y=(0,0,y_3)$ in the estimate \eqref{fr}.
Given $y\in\HH$ with $|y-e_3|\lesssim\eps\delta^{-1}$ that is not of this form, let $\HH_y$ denote the halfspace containing $\delta y$ with $\partial\HH_y$ being the tangent plane to $\partial\Omega$ at the point nearest $\delta y$.  Moreover, let $\delta\tilde y$ be the reflection of $\delta y$ in $\partial\HH_y$.  Elementary geometric considerations show that the angle between $\partial\HH$ and $\partial\HH_y$ is $O(\eps)$.  Correspondingly, $|\delta\tilde y - \delta\bar y|\lesssim\delta\eps$ and so
\begin{align}\label{619}
\eps^{-\frac 32}\bigl\|e^{-\frac {|x-\delta \bar y|^2}{4\eps^2}}-e^{-\frac{|x-\delta\tilde y|^2}{4\eps^2}}\bigr\|_{L^2(\R^3)}\to 0 \qtq{as} \eps \to 0.
\end{align}
As we will explain, \eqref{fr} (and so Theorem~\ref{T:LF2}) follows by combining \eqref{619} with the Strichartz inequality and Proposition~\ref{P:LF2} below.  Indeed, the only missing ingredient is the observation that
$$
\eps^{-\frac32} \Bigl\| e^{it\Delta_{\HH}}\chi_{\HH} \bigl[e^{-\frac {|x-\delta y|^2}{4\eps^2}} -e^{-\frac {|x-\delta\bar{y}|^2}{4\eps^2}}\bigr]
- e^{it\Delta_{\HH_y}}\chi_{\HH_y}\bigl[e^{-\frac {|x-\delta y|^2}{4\eps^2}} -e^{-\frac {|x-\delta\tilde{y}|^2}{4\eps^2}}\bigr] \Bigr\|_{L_{t,x}^{\frac{10}3}(\R\times\R^3)}
$$
is $o(1)$ as $\eps\to0$, which follows from \eqref{619} and the exact formula for the propagator in halfspaces. 

Therefore, it remains to justify the following proposition, whose proof will occupy the remainder of this subsection.

\begin{prop} \label{P:LF2} We have
\begin{align}\label{fn4}
\Bigl\|\bigl[e^{it\Delta_{\Omega}}\chi_{\HH}-e^{it\Delta_{\HH}}\chi_{\HH}\bigr](2\pi\eps^2)^{-\frac34}\bigl[e^{-\frac {|x-\delta y|^2}{4\eps^2}}
-e^{-\frac {|x-\delta\bar{y}|^2}{4\eps^2}}\bigr]\Bigr\|_{L_{t,x}^{\frac{10}3}(\R\times\Omega)}=o(1)
\end{align}
as $\eps\to0$ with any $\delta=\delta(\eps)$ obeying $\eps\leq\delta\leq\eps^{6/7}$, uniformly for $y=(0,0,y_3)$ and $y_3$ in a compact subset of $(0,\infty)$.
\end{prop}

\begin{proof}
To prove \eqref{fn4}, we will build a parametrix for the evolution in $\Omega$ and show that this differs little from the evolution in $\HH$, for which we have an exact formula:
\begin{align*}
&(2\pi\eps^2)^{-\frac 34} e^{it\Delta_{\HH}}\chi_{\HH}\bigl[e^{-\frac{|x-\delta y|^2}{4\eps^2}}-e^{-\frac {|x-\delta\bar y|^2}{4\eps^2}}\bigr]
=(2\pi)^{-\frac 34}\bigl(\tfrac{\eps}{\eps^2+it}\bigr)^{\frac 32} \bigl[e^{-\frac {|x-\delta  y|^2}{4(\eps^2+it)}}
-e^{-\frac {|x-\delta\bar y|^2}{4(\eps^2+it)}}\bigr],
\end{align*}
for all $t\in \R$ and $x\in \HH$.  We write
\begin{align*}
u(t,x):=(2\pi)^{-\frac 34}\biggl(\frac \eps{\eps^2+it}\biggr)^{\frac 32}e^{-\frac{|x-\delta y|^2}{4(\eps^2+it)}},
\end{align*}
and so for all $t\in\R$ and $x\in\HH$,
\begin{align*}
(2\pi\eps^2)^{-\frac 34} e^{it\Delta_{\HH}}\chi_{\HH}\bigl[e^{-\frac{|x-\delta y|^2}{4\eps^2}}-e^{-\frac {|x-\delta\bar y|^2}{4\eps^2}}\bigr]
=u(t,x)-u(t,\bar x).
\end{align*}

We start by showing that a part of the halfspace evolution does not contribute to the $L_{t,x}^{10/3}$ norm.  Let $\phi:[0, \infty)\to \R$ and
$\theta:\R\to \R$ be smooth functions such that
\begin{align*}
\phi(r)=\begin{cases} 0, & 0\le r\le \frac 12\\ 1, & r\geq1\end{cases} \quad \qtq{and}\quad
\theta(r)=\begin{cases} 1, & r\le 0 \\ 0, &r\geq 1. \end{cases}
\end{align*}
We define
\begin{align*}
v(t,x):=\bigl[u(t,x)-u(t, \bar x)\bigr]\Bigl[1-\phi\Bigl(\frac{x_1^2+x_2^2}\eps\Bigr)\theta\Bigl(\frac{x_3}\eps\Bigr)\Bigr]\chi_{\{x_3\ge-\frac 12\}}.
\end{align*}
We will prove that $v$ is a good approximation for the halfspace evolution.

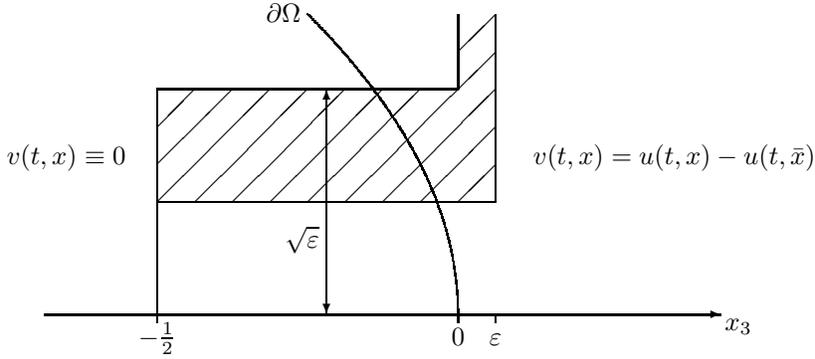
\begin{figure}[ht]
\begin{center}
%\fbox{
\setlength{\unitlength}{1mm}
% 1mm  = 2.84pt
\begin{picture}(115,45)(-65,-4)
%\put(0,-2){\line(0,1){42}}
\put(-55,0){\vector(1,0){90}}\put(35.5,-2){$x_3$}  % x_3 axis
\qbezier(0,0)(0,20)(-20,40)\put(-25.5,39){$\partial\Omega$} % The obstacle
\put(-40,-1){\line(0,1){31}}\put(-40,-4){\hbox to 0mm{\hss$-\tfrac12$\hss}}  % Vertical line at x_3 = -1/2
\put(  0,-1){\line(0,1){1}}\put(0,-4){\hbox to 0mm{\hss$0$\hss}}
\put(  5,-1){\line(0,1){1}}\put(5,-4){\hbox to 0mm{\hss$\eps$\hss}}
\put(5,15){\line(0,1){25}}\put(0,30){\line(0,1){10}}        % More vertical lines
\put(-40,30){\line(1,0){40}}\put(-40,15){\line(1,0){45}}    % Horizontal lines at x^\perp = L and L/2
\put(-60,20){$v(t,x)\equiv 0$} \put(10,20){$v(t,x)=u(t,x)-u(t,\bar x)$} %%
\put(-17.5,15){\vector(0,1){15}}\put(-17.5,15){\vector(0,-1){15}}\put(-23,9){\hbox{$\sqrt{\eps}$}}
%% Cross hatching
%%
\put(-40,25){\line(1,1){ 5}}\put(-40,20){\line(1,1){10}} % Top left corner
\put(-40,15){\line(1,1){15}}\put(-35,15){\line(1,1){15}} % Along the bottom
\put(-30,15){\line(1,1){15}}\put(-25,15){\line(1,1){15}}
\put(-20,15){\line(1,1){15}}\put(-15,15){\line(1,1){20}}
\put(-10,15){\line(1,1){15}}\put(- 5,15){\line(1,1){10}}\put(
0,15){\line(1,1){05}} % finishing at the bottom right corner
\put(  0,35){\line(1,1){ 5}} % ''pan handle''
\end{picture}
%}%end of fbox
\vspace*{-2ex}
\end{center}
\caption{The role of the cutoffs defining $v(t,x)$.  The cutoff
function takes values between $0$ and $1$ in the shaded region. We
depict only one half of a cross-section; one obtains the full 3D
figure by rotating about the~$x_3$-axis.}\label{F:v}
\end{figure}
\begin{lem} \label{L:v matters}
We have
\begin{align*}
\|u(t,x)-u(t,\bar x)-v(t,x)\|_{L_{t,x}^{\frac{10}3}(\R\times\HH)}=o(1)
\end{align*}
as $\eps\to0$ with any $\delta=\delta(\eps)$ obeying $\eps\leq\delta\leq\eps^{6/7}$, uniformly for $y=(0,0,y_3)$ and $y_3$ in a compact subset of $(0,\infty)$.
\end{lem}

\begin{proof}
By the definition of $v$, we have to prove
\begin{align*}
\biggl\|\bigl[u(t,x)-u(t,\bar x)\bigr]\phi\biggl(\frac{x_1^2+x_2^2}\eps\biggr)\theta\biggl(\frac{x_3}\eps\biggr)\biggr\|_ {L_{t,x}^{\frac{10}3}(\R\times\HH)}=o(1)  \qtq{as} \eps\to 0,
\end{align*}
which, considering the supports of $\phi$ and $\theta$, amounts to showing
\begin{align}\label{pts}
&\|u(t,x)\|_{L_{t,x}^{\frac{10}3}(|x^{\perp}|\ge \sqrt{\eps/2},\ 0\le x_3\le \eps)}+\| u(t,\bar x)\|_{L_{t,x}^{\frac{10}3}(|x^{\perp}|\ge \sqrt{\eps/2},\ 0\le
x_3\le \eps)}=o(1).
\end{align}
We only prove \eqref{pts} for $u(t,x)$ with $t\in[0,\infty)$; the proof for negative times and for $u(t,\bar x)$ is similar.

Let $T:=\eps^2\log(\frac 1\eps)$.  We will consider separately the short time contribution $[0,T]$ and the long time contribution $[T, \infty)$.
The intuition is that for short times the wave packet does not reach the cutoff, while for large times the wave packet has already disintegrated.
Thus, we do not need to take advantage of the cancelation between $u(t,x)$ and $u(t,\bar x)$.

We start with the long time contribution.  A simple change of variables yields
\begin{align*}
\int_T^\infty\int_{\R^3} |u(t,x)|^{\frac{10}3} \,dx\,dt
&\lsm\int_T^\infty\int_{\R^3}\biggl(\frac{\eps^2}{\eps^4+t^2}\biggr)^{\frac52} e^{-\frac{5\eps^2|x-\delta y|^2}{6(\eps^4+t^2)}} \,dx \,dt\\
&\lsm \int_T^\infty \biggl(\frac{\eps^2}{\eps^4+t^2}\biggr)^{\frac52-\frac 32} \,dt\lsm\eps^2\int_T^\infty t^{-2} \,dt\lsm \log^{-1}(\tfrac 1\eps).
\end{align*}
For short times, we estimate
\begin{align*}
\int_0^T\int_{|x^{\perp}|\ge\sqrt{\eps/2},0\le x_3\le\eps}|u(t,x)|^{\frac{10}3} \,dx\,dt
&\lsm \eps\int_0^T \biggl(\frac{\eps^2}{\eps^4+t^2}\biggr)^{\frac52}\int_{\sqrt{\eps/2}}^\infty e^{-\frac{5\eps^2r^2}{6(\eps^4+t^2)}}r \,dr\,dt\\
&\lsm \eps\int_0^T \biggl(\frac{\eps^2}{\eps^4+t^2}\biggr)^{\frac52-1}e^{-\frac{5\eps^3}{12(\eps^4+t^2)}} \,dt\\
&\lsm \eps e^{-\frac{5 \eps^3}{24 \eps^4\log^2(\frac1\eps)}}\int_0^T \biggl(\frac{\eps^2}{\eps^4+t^2}\biggr)^{\frac 32} \,dt\\
&\le \eps^{100}.
\end{align*}
This completes to the proof of the lemma.
\end{proof}

In view of Lemma~\ref{L:v matters}, Proposition~\ref{P:LF2} reduces to showing
\begin{align}\label{compl}
\Bigl\| (2\pi\eps^2)^{-\frac 34} e^{it\Delta_\Omega}\chi_{\HH}\bigl[e^{-\frac{|x-\delta y|^2}{4\eps^2}}-e^{-\frac{|x-\delta\bar y|^2}{4\eps^2}}\bigr]-v(t,x)\Bigl\|_{L_{t,x}^{\frac{10}3}(\R\times\Omega)}=o(1).
\end{align}
To achieve this, we write
\begin{align*}
(2\pi\eps^2)^{-\frac 34}e^{it\Delta_\Omega}\chi_{\HH}\bigl[e^{-\frac{|x-\delta y|^2}{4\eps^2}}-e^{-\frac{|x-\delta\bar y|^2}{4\eps^2}}\bigr]
=v(t,x)-w(t,x)-r(t,x),
\end{align*}
where $w$ is essentially $v$ evaluated on the boundary of $\Omega$ and $r(t,x)$ is the remainder term.  More precisely,
\begin{align*}
w(t,x):=\bigl[u(t,x_*)-u(t,\bar x_*)\bigr]\biggl[1-\phi\biggl(\frac{x_1^2+x_2^2}\eps\biggr)\theta\biggl(\frac{x_3}\eps\biggr)\biggr]
		\theta\biggl(\frac{\dist(x,\Omega^c)}\eps\biggr)\chi_{\{x_3\ge -\frac 12\}},
\end{align*}
where $x_*$ denotes the point on $\partial\Omega$ such that $x_*^{\perp}=x^\perp$ and $\bar x_*$ denotes the reflection of $x_*$ in $\partial\HH$.  Note that for $x\in\partial\Omega$, we have $w(t,x)=v(t,x)$. Thus, on $\R\times\Omega$ the remainder $r(t,x)$ satisfies
\begin{align*}
(i\partial_t+\Delta_\Omega )r=(i\partial_t+\Delta )(v-w).
\end{align*}
Therefore, by the Strichartz inequality, \eqref{compl} will follow from
\begin{align}\label{E:case4 estimates}
\|w\|_{L_{t,x}^{\frac{10}3}(\R\times\Omega)} + \|(i\partial_t+\Delta)v\|_{L_t^1L_x^2(\R\times\Omega)}
		+ \|(i\partial_t+\Delta)w\|_{L_t^1L_x^2(\R\times\Omega)}=o(1).
\end{align}

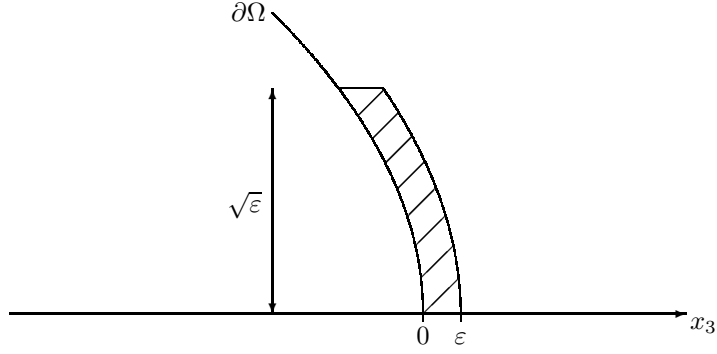
\begin{figure}[ht]
\begin{center}
%\fbox{
\setlength{\unitlength}{1mm}
% 1mm  = 2.84pt
\begin{picture}(115,45)(-65,-4)
\put(-55,0){\vector(1,0){90}}\put(35.5,-2){$x_3$}  % x_3 axis
\put(  0,-1){\line(0,1){1}}\put(0,-4){\hbox to 0mm{\hss$0$\hss}}
\put(  5,-1){\line(0,1){1}}\put(5,-4){\hbox to 0mm{\hss$\eps$\hss}}
\qbezier(0,0)(0,20)(-20,40)\put(-25.5,39){$\partial\Omega$} % The obstacle
\qbezier(5,0)(5,14.680)(-5.314,30)
\put(-20,15){\vector(0,1){15}}\put(-20,15){\vector(0,-1){15}}\put(-26,13.5){\hbox{$\sqrt{\eps}$}}
\put(-11.25,30){\line(1,0){5.936}} % top of shading
%Cross hatching
\put(0.000000,0.000000){\line(1,1){4.733750}}
\put(-0.278640,4.721360){\line(1,1){4.314602}}
\put(-1.010205,8.989795){\line(1,1){4.022505}}
\put(-2.084974,12.915026){\line(1,1){3.816245}}
\put(-3.431457,16.568542){\line(1,1){3.671747}}
\put(-5.000000,20.000000){\line(1,1){3.573782}}
\put(-6.754447,23.245553){\line(1,1){3.52}}
\put(-8.667504,26.332496){\line(1,1){3.52}}
\end{picture}
%}
\vspace*{-2ex}
\end{center}
\caption{The shaded area indicates the support of $w(t,x)$.  As in Figure~\ref{F:v} we depict only one half of a cross-section.}\label{F:w}
\end{figure}

To prove \eqref{E:case4 estimates}, we will make repeated use of the following

\begin{lem}
For $\alpha\geq 0$,
\begin{align}
\int_{|x^\perp|\le \sqrt \eps}|x^\perp|^{\alpha}e^{-\frac{\eps^2|x^\perp|^2}{2(\eps^4+t^2)}} \,dx^\perp
&\lsm \min\biggl\{\biggl(\frac{\eps^4+t^2}{\eps^2}\biggr)^{\frac{\alpha+2}2},\eps^{\frac{\alpha+2}2}\biggr\}\label{estsmall}\\
\int_{|\xp|\ge \sqrt{\eps/2}}e^{-\frac{\eps^2|\xp|^2}{2(\est)}}|\xp|^\alpha \,d\xp
&\lsm\biggl(\frac{\est}{\eps^2}\biggr)^{\frac{\alpha+2}2}\min\biggl\{1,\biggl(\frac{\est}{\eps^3}\biggr)^{20}\biggr\}.\label{estbig}
\end{align}
In particular, for $\alpha\geq0$, $\beta>\frac12$, and $\gamma=\min\{3-4\beta+\frac\alpha2,2-3\beta+\frac\alpha4\}$,
\begin{align}\label{512}
\int_0^{\infty}(\eps^4+t^2)^{-\beta}\biggl(\int_{|x^\perp|\le\sqrt{\eps}}|x^\perp|^{\alpha}e^{-\frac{\eps^2|x^\perp|^2}{2(\eps^2+t^2)}}\,d x^\perp\biggr)^{\frac 12} \,dt\lsm \eps^\gamma.
\end{align}
Moreover, for $\frac12<\beta<10$,
\begin{align} \label{514}
\int_0^\infty(\eps^4+t^2)^{-\beta}\min\biggl\{1, \biggl(\frac {\eps^4+t^2}{\eps^3}\biggr)^{10}\biggr\}\,dt\lsm \eps^{\frac32-3\beta}.
\end{align}
\end{lem}

\begin{proof}
Passing to polar coordinates, we estimate
\begin{align*}
\text{LHS}\eqref{estsmall}=\int_0^{\sqrt\eps}e^{-\frac{\eps^2r^2}{2(\eps^4+t^2)}} r^{\alpha+1} \,dr
&\lsm \biggl(\frac{\eps^4+t^2}{\eps^2}\biggr)^{\frac{\alpha+2}2}\int_0^{\frac{\eps^{\frac32}}{\sqrt{\eps^4+t^2}}} e^{-\frac {\rho^2}2} \rho^{\alpha+1} \,d\rho\\
&\lsm \biggl(\frac{\eps^4+t^2}{\eps^2}\biggr)^{\frac{\alpha+2}2}\min\biggl\{1,\biggl(\frac{\eps^{\frac32}}{\sqrt{\eps^4+t^2}}\biggr)^{\alpha+2}\biggr\},
\end{align*}
which settles \eqref{estsmall}.  The proof of \eqref{estbig} follows along similar lines.

Using \eqref{estsmall}, we estimate
\begin{align*}
\text{LHS}\eqref{512}
&\lsm \int_0^\infty(\eps^4+t^2)^{-\beta}\min\biggl\{\biggl(\frac{\eps^4+t^2}{\eps^2}\biggr)^{\frac{\alpha+2}4},\,\eps^{\frac{\alpha+2}4}\biggr\} \,dt\\
&\lsm \int_0^{\eps^{\frac 32}}(\eps^4+t^2)^{-\beta+\frac{\alpha+2}4}\eps^{-\frac{\alpha+2}2} \,dt
	+\int_{\eps^{\frac 32}}^\infty(\eps^4+t^2)^{-\beta}\eps^{\frac{\alpha+2}4} \,dt\\
&\lsm \eps^{-\frac{\alpha+2}2}\int_0^{\eps^2} \eps^{-4\beta+\alpha+2}\,dt + \eps^{-\frac{\alpha+2}2}\int_{\eps^2}^{\eps^{\frac32}}t^{-2\beta+\frac{\alpha+2}2}\, dt +\eps^{\frac{\alpha+2}4}\eps^{\frac 32(1-2\beta)}\\
&\lsm \eps^{\frac{\alpha+2}2} \eps^{2-4\beta}+ \eps^{-\frac{\alpha+2}2}\eps^{\frac 32(-2\beta+\frac{\alpha+2}2+1)}+\eps^{\frac{\alpha+2}4}\eps^{\frac 32(1-2\beta)}\\
&\lsm \eps^{3-4\beta+\frac\alpha2} +\eps^{2-3\beta+\frac\alpha4}.
\end{align*}
To establish \eqref{514} one argues as for \eqref{512}; we omit the details.
\end{proof}

We are now ready to prove \eqref{E:case4 estimates}, which will complete the proof of Proposition~\ref{P:LF2}.  We will estimate each of the three summands appearing on the left-hand side of \eqref{E:case4 estimates}.  We start with the first one.

\begin{lem}[Estimate for $w$]\label{L:we}
We have
\begin{align}\label{we}
\|w\|_{L_{t,x}^{\frac{10}3}(\R\times\Omega)}\lsm \delta\eps^{-\frac15}.
\end{align}
\end{lem}

\begin{proof}
We first obtain a pointwise bound for $w$.  Note that on the support of $w$,
\begin{align*}
|x^\perp|\le \eps^{\frac 12}, \quad |x_3|\lsm \eps, \qtq{and} |x_{*3}|\lsm |x^\perp|^2\lsm \eps,
\end{align*}
where the last two estimates follow from the finite curvature assumption.  Here we use the notation $x_{*3}:=x_*\cdot e_3$.  Thus, using the fact that
$|\bar x_* -\delta y|=|x_*+\delta y|$ and the mean value theorem, on the support of $w$ we have
\begin{align}\label{dif}
\biggl| e^{-\frac{|x_*-\delta y|^2}{4(\eps^2+it)}}-e^{-\frac{|\bar x_*-\delta y|^2}{4(\eps^2+it)}}\biggr|
&=\biggl| e^{-\frac{|x^\perp|^2}{4(\eps^2+it)}}\biggl(e^{-\frac{|x_{*3}-\delta y_3|^2}{4(\eps^2+it)}}-e^{-\frac{|x_{*3}+\delta y_3|^2}{4(\eps^2+it)}}\biggr)\biggr|\notag\\
&\lsm e^{-\frac{\eps^2|x^{\perp}|^2}{4(\eps^4+t^2)}}\frac\delta{\sqrt{\eps^4+t^2}}|x_{*3}|\notag\\
&\lsm \delta {(\eps^4+t^2)}^{-\frac12} |x^\perp|^2 e^{-\frac{\eps^2|x^\perp|^2}{4(\eps^4+t^2)}}.
\end{align}
Therefore,
\begin{align}\label{ptw}
|w(t,x)|\le |u(t,x_*)-u(t,\bar x_*)|\lsm \delta\eps^{\frac32}(\eps^4+t^2)^{-\frac 54}  |x^\perp|^2e^{-\frac{\eps^2|x^\perp|^2}{4(\eps^4+t^2)}}.
\end{align}
To control the $L_{t,x}^{\frac {10}3}$ norm of $w$ we use \eqref{ptw} together with \eqref{estsmall}, as follows:
\begin{align*}
\int_\R\int_\Omega|w(t,x)|^{\frac{10}3} \,dx\,dt
&\lsm \delta^{\frac{10}3}\eps^{-\frac73}\int _0^{\infty}\biggl(\frac{\eps^2}{\eps^4+t^2}\biggr)^{\frac{25}6} \!\! \int_{|x^\perp|\le\eps^{\frac 12}} e^{-\frac{5\eps^2|x^\perp|^2}{6(\eps^4+t^2)}}|x^\perp|^{\frac {20}3} \,dx^\perp dt\\
&\lesssim \delta^{\frac{10}3}\eps^{-\frac73}\int_0^{\infty}\biggl(\frac{\eps^2}{\eps^4+t^2}\biggr)^{\frac{25}6}\min\biggl\{\biggl(\frac{\eps^4+t^2}{\eps^2}\biggr)^{\frac{13}3},\eps^{\frac{13}3}\biggr\}\,dt\\
&\lsm \delta^{\frac{10}3}\eps^{-\frac 83}\int_0^{\eps^{\frac 32}}(\eps^4+t^2)^{\frac 16}\,dt+\delta^{\frac{10}3}\eps^{\frac{31}3}\int_{\eps^{\frac 32}}^\infty(\eps^4+t^2)^{-\frac{25}6}\,dt\\
&\lesssim \delta^{\frac{10}3}\eps^{-\frac 23}.
\end{align*}
This completes the proof of the lemma.
\end{proof}

\begin{lem}[Estimate for $(i\partial_t+\Delta)v$]\label{L:ve}
We have
\begin{align}\label{ve}
\|(i\partial_t+\Delta)v\|_{L_t^1L_x^2(\R\times\Omega)}\lesssim \delta \eps^{-\frac34}.
\end{align}
\end{lem}

\begin{proof}
Using the definition of $v(t,x)$, we compute
\begin{align}
(i\partial_t+\Delta)v(t,x)
&=(i\partial_t+\Delta)\Bigl\{\bigl[u(t,x)-u(t,\bar x)\bigr]\bigl[1-\phi\bigl(\tfrac{x_1^2+x_2^2}\eps\bigr)\theta\bigl(\tfrac{x_3}\eps\bigr)\bigr]\chi_{\{x_3\ge -\frac 12\}}\Bigr\}\notag\\
&=\bigl[u(t,x)-u(t,\bar x)\bigr]\Delta\Bigl\{\bigl[1-\phi\bigl(\tfrac{x_1^2+x_2^2}\eps\bigr)\theta\bigl(\tfrac{x_3}\eps\bigr)\bigr]\chi_{\{x_3\ge-\frac 12\}}\Bigr\}\label{1v}\\
&\quad+2\nabla\bigl[u(t,x)-u(t,\bar x)\bigr]\cdot \nabla\Bigl\{\bigl[1-\phi\bigl(\tfrac{x_1^2+x_2^2}\eps\bigr)\theta\bigl(\tfrac{x_3}\eps\bigr)\bigr]\chi_{\{x_3\ge -\frac 12\}}\Bigr\}.\label{2v}
\end{align}

We first consider the contribution of \eqref{1v}.  A direct analysis yields that for $x\in \Omega$ in the support of \eqref{1v},
\begin{align*}
|x_3|\lsm \eps, \quad |x^{\perp}|\ge \sqrt{\eps/2}, \qtq{and} \Bigl|\Delta\Bigl\{\bigl[1-\phi\bigl(\tfrac{x_1^2+x_2^2}\eps\bigr)\theta\bigl(\tfrac{x_3}\eps\bigr)\bigr]\chi_{\{x_3\ge -\frac 12\}}\Bigr\}\Bigr|\lsm \eps^{-2}.
\end{align*}
Thus, by the mean value theorem,
\begin{align}
\biggl| e^{-\frac{|x-\delta y|^2}{4(\eps^2+it)}}-e^{-\frac{|\bar x-\delta y|^2}{4(\eps^2+it)}}\biggr|
&\lsm  e^{-\frac{\eps^2|x^\perp|^2}{4(\eps^4+t^2)} }\frac{\delta|x_3|}{\sqrt{\eps^4+t^2}}
\lsm \delta\eps(\eps^4+t^2)^{-\frac12}e^{-\frac{\eps^2|x^\perp|^2}{4(\eps^4+t^2)}}.\label{7c}
\end{align}
This yields the pointwise bound
\begin{align}
\eqref{1v}\lsm \eps^{-2}|u(t,x)-u(t,\bar x)|\lsm \delta\eps^{\frac 12}(\eps^4+t^2)^{-\frac54}e^{-\frac{\eps^2|x^\perp|^2}{4(\eps^4+t^2)}}.\label{p1v}
\end{align}
Using \eqref{p1v} followed by \eqref{estbig} (with $\alpha=0$) and \eqref{514} (with $\beta=\frac 34$), we obtain
\begin{align*}
\|\eqref{1v}\|_{L_t^1 L_x^2(\R\times\Omega)}
& \lsm \eps^{\frac12}\delta\eps^{\frac 12} \int_0^{\infty}(\eps^4+t^2)^{-\frac 54}\biggl(\int_{|x^\perp|\ge \sqrt{\eps/2}}e^{-\frac{\eps^2|x^\perp|^2}{2(\eps^4+t^2)}}\,dx^\perp\biggr)^{\frac 12} \,dt\\
&\lsm \delta\int_0^{\infty}(\eps^4+t^2)^{-\frac 54+\frac12}\min\biggl\{1,\biggl(\frac{\eps^4+t^2}{\eps^3}\biggr)^{10}\biggr\}\,dt\\
&\lsm \delta\eps^{-\frac 34}.
\end{align*}

We now consider the contribution of \eqref{2v}.  For $x\in \Omega$ in the support of \eqref{2v}, we have
\begin{align*}
|x_3|\lesssim \eps, \quad |\xp|\ge \sqrt{\eps/2}, \qtq{and}
	\Bigl|\nabla\Bigl\{\bigl[1-\phi\bigl(\tfrac{x_1^2+x_2^2}\eps\bigr)\theta\bigl(\tfrac{x_3}\eps\bigr)\bigr]\chi_{\{x_3\ge -\frac 12\}}\Bigr\}\Bigr|\lsm \eps^{-1}.
\end{align*}
Using that $|x-\delta \bar y|=|\bar x-\delta y|$, we compute
\begin{align*}
\nabla \biggl(e^{-\frac{|x-\delta y|^2}{4(\eps^2+it)}}-e^{-\frac{|\bar x- \delta y|^2}{4(\eps^2+it)}}\biggr)
&=-\frac{x-\delta y}{2(\eps^2+it)}e^{-\frac{|x-\delta y|^2}{4(\eps^2+it)}}+ \frac{x-\delta\bar y}{2(\eps^2+it)}e^{-\frac{|x-\delta\bar y|^2}{4(\eps^2+it)}}\\
&=-\frac x{2(\eps^2+it)}\biggl(e^{-\frac{|x-\delta y|^2}{4(\eps^2+it)}}-e^{-\frac{|x-\delta \bar y|^2}{4(\eps^2+it)}}\biggr)\\
&\quad+ \frac{\delta y_3 e_3}{2(\eps^2+it)}\biggl(e^{-\frac{|x-\delta y|^2}{4(\eps^2+it)}}+e^{-\frac{|x-\delta \bar y|^2}{4(\eps^2+it)}}\biggr).
\end{align*}
Thus, for $x\in \Omega$ in the support of \eqref{2v} we have
\begin{align*}
\bigl|\nabla\bigl[ & u(t,x)-u(t,\bar x)\bigr]\bigr|\\
&\lsm\biggl(\frac{\eps^2}{\est}\biggr)^{\frac34}\biggl\{\frac{|x|}{\sqrt{\est}}\frac{\eps\delta}{\sqrt{\est}}e^{-\frac{\eps^2|\xp|^2}{4(\est)}}+\frac{\delta}{\sqrt{\est}}e^{-\frac{\eps^2|\xp|^2}{4(\est)}}\biggr\} \\
&\lsm\Bigl\{\eps^{\frac 72}\delta(\est)^{-\frac 74}+\eps^{\frac 52}\delta(\est)^{-\frac 74}|\xp|+\eps^{\frac 32}\delta(\est)^{-\frac 54}\Bigr\}e^{-\frac{\eps^2|\xp|^2}{4(\est)}}\\
&\lsm \Bigl\{\eps^{\frac 32}\delta(\est)^{-\frac 54}+\eps^{\frac52}\delta(\est)^{-\frac74}|\xp|\Bigr\}e^{-\frac{\eps^2|\xp|^2}{4(\est)}},
\end{align*}
which yields the pointwise bound
\begin{align*}
|\eqref{2v}|\lsm \Bigl\{\eps^{\frac 12}\delta(\est)^{-\frac54}+\eps^{\frac 32}\delta(\est)^{-\frac74}|\xp|\Bigr\}e^{-\frac{\eps^2|\xp|^2}{4(\est)}}.
\end{align*}
Using \eqref{estbig} followed by \eqref{514}, we estimate the contribution of \eqref{2v} as follows:
\begin{align*}
\|\eqref{2v}\|_{L_t^1L_x^2(\R\times\Omega)}
&\lsm \eps^{\frac12}\eps^{\frac 12}\delta\int_0^\infty(\est)^{-\frac 54}\biggl(\int_{|\xp|\ge\sqrt{\eps/2}}e^{-\frac{\eps^2|\xp|^2}{2(\est)}} \,d\xp\biggr)^{\frac12}\,dt\\
&\quad+\eps^{\frac 12}\eps^{\frac32}\delta\int_0^\infty(\est)^{-\frac74}\biggl(\int_{|\xp|\ge\sqrt{\eps/2}}|\xp|^2e^{-\frac{\eps^2|\xp|^2}{2(\est)}} \,d\xp\biggr)^{\frac 12}\,dt\\
&\lsm \delta\int_0^{\infty}(\est)^{-\frac 34}\min\biggl\{1,\biggl(\frac{\est}{\eps^3}\biggr)^{10}\biggr\} \,dt\\
&\lsm \delta\eps^{-\frac 34}.
\end{align*}
This completes the proof of the lemma.
\end{proof}

\begin{lem}[Estimate for $(i\partial_t+\Delta)w$]\label{L:we1}
We have
\begin{align}\label{we1}
\|(i\partial_t+\Delta)w\|_{L_t^1L_x^2(\R\times\Omega)}\lesssim \delta \eps^{-\frac34}  + \delta^3\eps^{-2}.
\end{align}
\end{lem}

\begin{proof}
We compute
\begin{align}
(i\partial_t + \Delta)w\!\!&\notag\\
&=\Bigl\{(i\partial_t+\Delta)\bigl[u(t,x_*)-u(t,\bar x_*)\bigr]\label{w1}\\
&\quad+2\nabla\bigl[u(t,x_*)-u(t,\bar x_*)\bigr]\cdot \nabla\label{w2}\\
&\quad+\bigl[u(t,x_*)-u(t,\bar x_*)\bigr]\ \Delta \Bigr\}\bigl[1-\phi\bigl(\tfrac{x_1^2+x_2^2}\eps\bigr)\theta\bigl(\tfrac{x_3}\eps\bigr)\bigr]\theta\bigl(\tfrac{\dist(x,\Omega^c)}{\eps}\bigr)\chi_{\{x_3\ge-\frac 12\}}\label{w3}.
\end{align}

We first consider the contribution of \eqref{w3}.  Using \eqref{ptw}, we obtain the pointwise bound
\begin{align*}
|\eqref{w3}|\lsm \delta\eps^{-\frac 12}(\est)^{-\frac 54}|\xp|^2 e^{-\frac{\eps^2|\xp|^2}{4(\est)}}.
\end{align*}
Thus using \eqref{512} and the fact that $|x_3|\lesssim \eps$ for $x\in\supp w$, we obtain
\begin{align*}
\|\eqref{w3}\|_{L_t^1L_x^2(\R\times\Omega)}
&\lsm \delta\int_0^\infty(\est)^{-\frac 54} \biggl(\int_{|\xp|\le\sqrt \eps}|\xp|^4e^{-\frac{\eps^2|\xp|^2}{2(\est)}} \,d\xp\biggr)^{\frac 12} \,dt
\lesssim\delta\eps^{-\frac 34}.
\end{align*}

Next we consider the contribution of \eqref{w2}.  As $\frac{\partial x_*}{\partial x_3}=0$, $\nabla[u(t,x_*)-u(t,\bar x_*)]$ has no component in the $e_3$ direction.  For the remaining directions we have
\begin{equation*}
\begin{aligned}
\nabla_{\perp}\bigl[u(t,x_*)-u(t,\bar x_*)\bigr]
&= \tfrac{-\xp}{2(\eps^2+it)}\bigl[u(t,x_*)-u(t,\bar x_*)\bigr] \\
&\quad 	- (\nabla_\perp x_{*3}) \bigl[\tfrac{x_{*3}-\delta y_3}{2(\eps^2+it)} u(t,x_*)- \tfrac{x_{*3}+\delta y_3}{2(\eps^2+it)} u(t,\bar x_*)\bigr].
\end{aligned}
\end{equation*}
Using \eqref{ptw}, $|\nabla_\perp x_{*3}|\lsm |\xp|$, and $|x_{*3}|\lsm \eps$, we deduce
\begin{align*}
\bigl|\nabla_{\perp}\bigl[u(t,x_*)-u(t,\bar x_*)\bigr]\bigr|
&\lsm \bigl[ \delta\eps^{\frac32}(\est)^{-\frac74}|\xp|^3 + \delta\eps^{\frac32}(\est)^{-\frac54}|\xp| \bigr] e^{-\frac{\eps^2|\xp|^2}{4(\est)}}.
\end{align*}
This gives the pointwise bound
\begin{align*}
|\eqref{w2}|&\lsm \bigl[\delta\eps^{\frac 12} (\est)^{-\frac74}|\xp|^3 + \delta\eps^{\frac 12} (\est)^{-\frac54}|\xp| \bigr]e^{-\frac{\eps^2|\xp|^2}{4(\est)}}.
\end{align*}
Using \eqref{512}, we thus obtain
\begin{align*}
\|\eqref{w2}\|_{L_t^1L_x^2(\R\times\Omega)}
&\lsm \eps\delta\int_0^\infty (\est)^{-\frac74}\biggl(\int_{|\xp|\le\sqrt\eps}|\xp|^6e^{-\frac{\eps^2|\xp|^2}{2(\est)}} \,d\xp\biggl)^{\frac 12} \,dt\\
&\quad + \eps\delta\int_0^\infty (\est)^{-\frac54}\biggl(\int_{|\xp|\le\sqrt\eps}|\xp|^2e^{-\frac{\eps^2|\xp|^2}{2(\est)}} \,d\xp\biggl)^{\frac 12} \,dt\\
&\lesssim \delta\eps^{-\frac34} + \delta\eps^{-\frac14}\lsm \delta\eps^{-\frac34}.
\end{align*}

Lastly, we consider \eqref{w1}.  We begin with the contribution coming from the term $\partial_t\bigl[u(t,x_*)-u(t,\bar x_*)\bigr]$, which we denote by $\eqref{w1}_{\partial_t}$.  We start by deriving a pointwise bound on this term.  A straightforward computation using \eqref{dif} yields
\begin{align*}
&\bigl|\partial_t\bigl[u(t,x_*)-u(t,\bar x_*)\bigr]\bigr|\\
&\lsm \frac{\eps^{\frac32}}{(\est)^{\frac 54}}\Bigl| e^{-\frac{|x_*-\delta y|^2}{4(\eps^2+it)}}-e^{-\frac{|\bar x_*-\delta y|^2}{4(\eps^2+it)}}\Bigr|\\
&\quad +\biggl(\frac {\eps^2}{\est}\biggr)^{\frac 34}\frac 1{\est}\biggl||x_*-\delta y|^2 e^{-\frac{|x_*-\delta y|^2}{4(\eps^2+it)}}-|\bar x_*-\delta y|^2 e^{-\frac{|\bar x_*-\delta y|^2}{4(\eps^2+it)}}\biggr|\\
&\lsm\bigl[\eps^{\frac 32}(\est)^{-\frac54}+\eps^{\frac 32}(\est)^{-\frac74}|\xp|^2\bigl]\Bigl|e^{-\frac{|x_*-\delta
y|^2}{4(\eps^2+it)}}-e^{-\frac{|\bar x_*-\delta y|^2}{4(\eps^2+it)}}\Bigr|\\
&\quad +\eps^{\frac 32}(\est)^{-\frac 74}e^{-\frac{\eps^2|\xp|^2}{4(\est)}}\Bigl||x_{*3}-\delta y_3|^2e^{-\frac{|x_{*3}-\delta y_3|^2}{4(\eps^2+it)}}-|x_{*3}+\delta y_3|^2 e^{-\frac{|x_{*3}+\delta y_3|^2}{4(\eps^2+it)}}\Bigr|\\
&\lsm \bigl[\eps^{\frac 32}(\est)^{-\frac 54}+\eps^{\frac32}(\est)^{-\frac 74}|\xp|^2\bigr]e^{-\frac{\eps^2|\xp|^2}{4(\est)}}\delta(\est)^{-\frac12}|\xp|^2\\
&\quad +\eps^{\frac 32}(\est)^{-\frac 74}e^{-\frac{\eps^2|\xp|^2}{4(\est)}}\bigl[\delta|x_{*3}|+(|x_{*3}|^2+\delta^2)\delta (\est)^{-\frac12}|x^\perp|^2\bigr]\\
&\lsm e^{-\frac{\eps^2|\xp|^2}{4(\est)}}\Bigl[\delta\eps^{\frac32}(\est)^{-\frac 74}|\xp|^2+
	\delta\eps^{\frac 32}(\est)^{-\frac94}|\xp|^4+\eps^{\frac 32}\delta^3(\est)^{-\frac94}|\xp|^2\Bigr],
\end{align*}
where in order to obtain the third inequality we have used the identity $2(ab-cd)=(a-c)(b+d)+(a+c)(b-d)$.  Using \eqref{512} as before, we obtain
\begin{align*}
\|\eqref{w1}_{\partial_t}\|_{L_t^1L_x^2(\R\times\Omega)}
&\lsm \delta\eps^{-\frac 14} + \delta\eps^{-\frac 34} + \delta^3\eps^{-2}\lsm \delta\eps^{-\frac 34} + \delta^3\eps^{-2}.
\end{align*}

We now turn to the Laplacian term in \eqref{w1}, which we denote by $\eqref{w1}_\Delta$.  For a generic function $f:\R^3\to\C$,
\begin{equation}\label{bits&pieces}
\Delta f(x_*) = [\Delta_\perp f + 2 (\nabla_\perp x_{*3})\cdot (\nabla_\perp \partial_3 f) + (\Delta_\perp x_{*3})(\partial_3 f)
	+ |\nabla_\perp x_{*3}|^2(\partial_3^2 f) ](x_*). 
\end{equation}
Using this formula with $f(x) := u(t,x) - u(t,\bar x)$, we first derive a pointwise bound on $\eqref{w1}_\Delta$.

A direct computation gives
\begin{align*}
(\Delta_{\perp}f)(x_*)=\biggl[-\frac{1}{\eps^2+it} + \frac{|\xp|^2}{4(\eps^2+it)^2}\biggr] \bigl[u(t,x_*)-u(t,\bar x_*)\bigr].
\end{align*}
Therefore, using \eqref{ptw} we obtain the pointwise bound
\begin{align*}
\bigl|(\Delta_{\perp}f)(x_*)\bigr|
&\lsm \bigl[\delta\eps^{\frac 32}(\est)^{-\frac74}|\xp|^2+\delta\eps^{\frac 32}(\est)^{-\frac 94}|\xp|^4\bigr] e^{-\frac{\eps^2|\xp|^2}{4(\est)}}.
\end{align*}

Next, we combine $|\nabla_\perp x_{*3}|\lesssim |x^\perp|$ with
$$
(\nabla_\perp \partial_3 f)(x_*) = \frac{(x_{*3}-\delta y_3)x^\perp}{4(\eps^2+it)^2} u(t,x_*)
	- \frac{(x_{*3}+\delta y_3)x^\perp}{4(\eps^2+it)^2} u(t,\bar x_*),
$$
$|x_{*3}|\lesssim |x^\perp|^2$, \eqref{ptw}, and the crude bound
\begin{equation}\label{u size}
|u(t,x_*)|+|u(t,\bar x_*)| \lesssim \eps^{\frac32}(\est)^{-\frac34} e^{-\frac{\eps^2|\xp|^2}{4(\est)}}
\end{equation}
to obtain
\begin{align*}
\bigl|(\nabla_\perp x_{*3})\cdot (\nabla_\perp \partial_3 f)\bigr|
&\lesssim \bigl[\delta\eps^{\frac32}(\est)^{-\frac94}|x^\perp|^6 + \delta\eps^{\frac32}(\est)^{-\frac74}|x^\perp|^2\bigr]e^{-\frac{\eps^2|\xp|^2}{4(\est)}}. 
\end{align*}

Next we use $|\Delta_\perp x_{*3}|\lesssim 1$ and $|x_{*3}\pm\delta y|\lesssim \delta$ together with elementary computations to find
$$
\bigl|(\Delta_\perp x_{*3})(\partial_3 f)\bigr| \lesssim \delta\eps^{\frac32}(\est)^{-\frac54} e^{-\frac{\eps^2|\xp|^2}{4(\est)}}.
$$

Toward our last pointwise bound, we compute
$$
(\partial_3^2 f)(x_*) = \biggl[\frac{-1}{2(\eps^2+it)}+\frac{|x_{*3}-\delta y_3|^2}{4(\eps^2+it)^2}\biggr] u(t,x_*)
	- \biggl[\frac{-1}{2(\eps^2+it)}+\frac{|x_{*3}+\delta y_3|^2}{4(\eps^2+it)^2}\biggr] u(t,\bar x_*),
$$
Combining this with \eqref{ptw}, \eqref{u size}, $|\nabla_\perp x_{*3}|\lesssim |x^\perp|$, and $|x_{*3}|\lesssim |x^\perp|^2\lesssim\eps\lsm\delta$ yields
\begin{align*}
&\bigl||\nabla_\perp x_{*3}|^2(\partial_3^2 f) (x_*)\bigr|\\
&\lesssim \bigl[\delta\eps^{\frac32}(\est)^{-\frac74}|x^\perp|^4 + \delta^3\eps^{\frac32}(\est)^{-\frac94}|x^\perp|^4
	+ \delta\eps^{\frac32}(\est)^{-\frac74}|x^\perp|^2\bigr]e^{-\frac{\eps^2|\xp|^2}{4(\est)}}.
\end{align*}

We now put together all the pieces from \eqref{bits&pieces}.  Using $|x^\perp|^2 \lesssim \delta \lesssim 1$ so as to keep only the largest terms, we obtain
$$
\bigl| \Delta f (x_*) \bigr| \lesssim \bigl[
\delta\eps^{\frac32}(\est)^{-\frac54} + \delta\eps^{\frac32}(\est)^{-\frac74}|x^\perp|^2 + \delta\eps^{\frac32}(\est)^{-\frac94}|x^\perp|^4\bigr]
e^{-\frac{\eps^2|\xp|^2}{4(\est)}}.
$$
Using \eqref{512} as before, we thus obtain
\begin{align*}
\|\eqref{w1}_\Delta\|_{L_t^1L_x^2(\R\times\Omega)}\lsm \delta+\delta\eps^{-\frac 14}+\delta\eps^{-\frac 34}\lsm \delta\eps^{-\frac 34}.
\end{align*}

This completes the proof of the lemma.
\end{proof}

Collecting Lemmas~\ref{L:we}, \ref{L:ve}, and \ref{L:we1} and recalling that $\eps\leq \delta\leq \eps^{6/7}$, we derive \eqref{E:case4 estimates}.
This in turn yields \eqref{compl}, which combined with Lemma~\ref{L:v matters} proves Proposition~\ref{P:LF2}.
\end{proof}

This completes the proof of Theorem~\ref{T:LF2} and so the discussion of Case~(iv).

\subsection{Case (v)} In this case we have $N_n\to \infty$  and $N_n^{-6/7}\leq d(x_n) \lesssim 1$.  As in Case~(iv), we rescale so that the
obstacle is restored to its original (unit) size.  Correspondingly, the initial data has characteristic scale $\eps:= N_n^{-1}\to 0$ and is supported within a distance $O(\eps)$ of the origin, which is at a distance $\delta:=d(x_n)$ from the obstacle.  A schematic representation is given in Figure~\ref{F:case5}.  There and below,
\begin{equation}\label{E:psi eps defn}
\psi_\eps(x) :=  \eps^{-3/2} \psi\bigl(\tfrac x{\eps}\bigr).
\end{equation}

In this way, the treatment of Case~(v) reduces to the following assertion:

\begin{thm}\label{T:LF3}  Fix $\psi\in C^\infty_c(\R^3)$ and let $\psi_\eps$ be as in \eqref{E:psi eps defn}.  Then for any pair $(q,r)$ satisfying
$\frac2q+\frac3r=\frac32$ with $2<q<\infty$ and $2<r<6$, we have
\begin{align}\label{main}
\|e^{it\Delta_{\Omega(\eps)}}\psi_\eps-e^{it\Delta}\psi_\eps\|_{L_t^q L_x^r(\R\times\R^3)}\to 0 \qtq{as} \eps\to 0,
\end{align}
for any $\eps$-dependent family of domains $\Omega(\eps)$ that are affine images of $\Omega$
with the property that $\delta:=\dist(0,\Omega(\eps)^c) \geq \eps^{6/7}$ and $\delta\lesssim 1$.
\end{thm}

\begin{figure}[ht]
\begin{center}
%\fbox{
\setlength{\unitlength}{1mm}
% 1mm  = 2.84pt
\begin{picture}(95,40)(-45,-20)
\put(30, 0){\circle*{4}}                                            % Blob for \supp(\psi)
\put(36.5,0){\vector(-1,0){4}}\put(37,-1){$\supp(\psi_\eps)$}            % Label for \supp(\psi_\eps)
\put(30,-4){\vector(-1,0){ 2}}\put(30,-4){\vector(1,0){ 2}}\put(30,-8){\hbox to 0mm{\hss$\sim\eps$\hss}}       % dimension of data
\put(15,+4){\vector(-1,0){14}}\put(15,+4){\vector(1,0){15}}\put(15,+6){\hbox to 0mm{\hss$\delta$\hss}}     % distance of data to 0
\qbezier(0,0)(0,10)(-5,20)\qbezier(0,0)(0,-10)(-5,-20)                                                     % The obstacle
\put(-25,10){\vector(2,1){20}}\put(-25,10){\vector(-2,-1){20}}\put(-27,12){\hbox to 0mm{\hss$\sim 1$\hss}} % dimension of the obstacle
\end{picture}
%}
\vspace*{-2ex}
\end{center}
\caption{Depiction of Case~(v); here $\eps^{6/7}\leq \delta\lesssim1 $ and $\eps\to 0$.}\label{F:case5}
\end{figure}
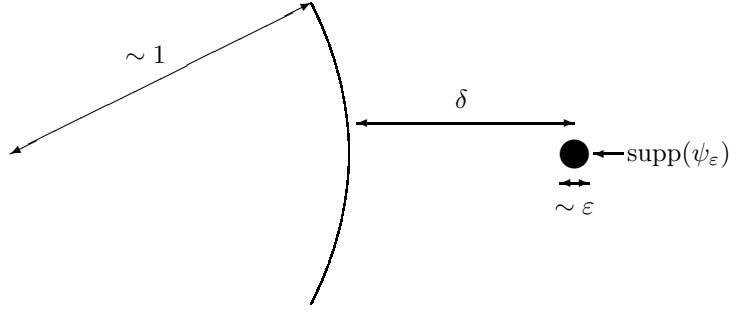

We now begin the proof of Theorem~\ref{T:LF3}.  By interpolation and the Strichartz inequality, it suffices to treat the case $q=r=\frac{10}3$.  By time-reversal symmetry, it suffices to consider positive times only, which is what we will do below.  To ease notation, we write $\Omega$ for $\Omega(\eps)$ for the remainder of this subsection.

The first step in the proof is to write $\psi_\eps$ as a superposition of Gaussian wave packets; we will then investigate the evolution of the individual wave packets.  The basic decomposition is given by the following lemma.  The parameter $\sigma$ denotes the initial width of the Gaussian wave packets. It is chosen large enough so that the wave packets hold together until they collide with the obstacle.  This ensures that they reflect in an almost particle-like manner and allows us to treat the reflected wave in the realm of geometric optics.  Indeed, the particle-like regime lasts for time $\sim\sigma^2$, while the velocity of the wave packet is $2\xi=\tfrac{2n}{L}$, which is $\sim\eps^{-1}$ for the dominant terms, up to double logarithmic factors (cf. \eqref{auto}).  As the obstacle is $\delta$ away from the origin, it takes the dominant wave packets time $\sim\delta\eps$ to reach the obstacle (up to $\log\log(\frac1\eps)$ factors), which is much smaller than $\sigma^{2}=\delta\eps\log^2(\frac1\eps)$.  Moreover, $\sigma$ is chosen small enough that the individual wave packets disperse shortly after this collision.  In addition to the main geometric parameters $\eps$ and $\delta$, we also need two degrees of small parameters; these are $[\log(\frac1\eps)]^{-1}$ and $[\log\log(\frac1\eps)]^{-1}$.

\begin{lem}[Wave packet decomposition] \label{decomposition}
Fix $\psi\in C_c^{\infty}(\R^3)$ and let $0<\eps\ll 1$,
$$
\sigma:=\sqrt{\eps\delta}\log(\tfrac 1{\eps}) \qquad\text{and}\qquad L:=\sigma \log\log(\tfrac 1{\eps}).
$$
Then there exist coefficients $\{c_n^{\eps}\}_{n\in\Z^3}$ so that
\begin{equation}\label{expand}
\biggl\|\psi_\eps(x)-\sum_{n\in\Z^3}c_n^{\eps}(2\pi\sigma^2)^{-\frac34}\exp\Bigl\{-\frac {|x|^2}{4\sigma^2}+in\cdot\frac
xL\Bigr\}\biggr\|_{L^2(\R^3)}=o(1)
\end{equation}
as $\eps\to 0$.  Moreover,
\begin{equation}\label{bdforc}
|c_n^{\eps}|\lesssim_{k,\psi} \frac {(\sigma\eps)^{\frac 32}}{L^3}\min\biggl\{1,\Bigl(\frac L{\eps|n|}\Bigr)^k\biggr\}  \qtq{for all} k\in \N
\end{equation}
and \eqref{expand} remains true if the summation is only taken over those $n$ belonging to
\begin{align}\label{auto}
\mathcal S:=\biggl\{n\in \Z^3: \,\frac 1{\eps\log\log(\frac 1{\eps})}\leq \frac {|n|}L \leq \frac {\log\log(\frac1{\eps})}{\eps}\biggr\}.
\end{align}
\end{lem}

\begin{proof}
For $n\in\Z^3$, let
$$
\gamma_n(x):=(2\pi\sigma^2)^{-\frac 34}  \exp\bigl\{-\tfrac{|x|^2}{4\sigma^2}+in\cdot\tfrac xL\bigr\}.
$$
Note that $\|\gamma_n\|_{L^2(\R^3)}=1$. We define
$$
c_n^{\eps}:=(2\pi L)^{-3}\int_{[-\pi L,\pi L]^3}\psi_\eps(x)(2\pi\sigma^2)^{\frac 34} \exp\bigl\{\tfrac{|x|^2}{4\sigma^2}-in\cdot\tfrac xL\bigr\} \, dx.
$$
Then by the convergence of Fourier series we have
\begin{equation}\label{eq1}
\psi_\eps(x)=\sum_{n\in \Z^3}c_n^{\eps} \gamma_n(x) \quad\text{for all}\quad x\in[-\pi L,\pi L]^3.
\end{equation}
Taking $\eps$ sufficiently small, we can guarantee that $\supp \psi_\eps\subseteq[-\frac{\pi L}2, \frac{\pi L}2]^3$.  Thus,
to establish \eqref{expand} we need to show that outside the cube $[-\pi L,\pi L]^3$, the series only contributes a small error.
Indeed, let $k\in \Z^3\setminus \{ 0 \}$ and $Q_k:=2\pi kL+[-\pi L,\pi L]^3$; using the periodicity of Fourier series, we obtain
$$
\Bigl\|\sum_{n\in \Z^3}c_n^{\eps}\gamma_n\Bigr\|_{L^2(Q_k)}^2=\int_{[-\pi L,\pi L]^3}|\psi_\eps(x)|^2\exp\bigl\{\tfrac{|x|^2}{2\sigma^2} - \tfrac{|x+2\pi kL|^2}{2\sigma^2}\bigr\} \, dx.
$$
As on the support of $\psi_\eps$ we have $|x|\le \tfrac12 \pi L\leq \tfrac 14 |2\pi kL|$, we get
$$
\Bigl\|\sum_{n\in \Z^3} c_n^{\eps}\gamma_n\Bigr\|_{L^2(Q_k)}^2
\lesssim \|\psi\|_{L^2(\R^3)}^2 \exp\bigl\{ -\tfrac{\pi^2 k^2L^2}{2\sigma^2} \bigr\}.
$$
Summing in $k$ and using \eqref{eq1}, we obtain
\begin{align*}
\Bigl\|\psi_\eps-\sum_{n\in \Z^3}c_n^{\eps}\gamma_n\Bigr\|_{L^2(\R^3)}
&\lesssim \sum_{k\in \Z^3\setminus\{0\}} \Bigl\|\sum_{n\in \Z^3}c_n^{\eps}\gamma_n\Bigr\|_{L^2(Q_k)}\\
&\lesssim \|\psi\|_{L^2(\R^3)}\sum_{k\in\Z^3\setminus \{0\}}\exp \bigl\{ -\tfrac {\pi^2 k^2L^2}{4\sigma^2} \bigr\}\\
&\lesssim_{\psi} e^{-\frac {\pi^2 L^2}{4\sigma^2}}=o(1) \qtq{as} \eps\to 0.
\end{align*}
This proves \eqref{expand}.

Next we prove the upper bound \eqref{bdforc}.  From the definition of $c_n^{\eps}$, we immediately obtain
\begin{align*}
|c_n^{\eps}|&\lesssim \frac {\sigma^{\frac32}}{L^3}\bigl\|\psi_\eps(x)e^{\frac{|x|^2}{4\sigma^2}}\bigr\|_{L^1(\R^3)} \lesssim\frac{(\sigma\eps)^{\frac 32}}{L^3}.
\end{align*}
To derive the other upper bound, we use integration by parts.  Let $\mathbb D:= i\frac {Ln}{|n|^2}\cdot\nabla$; note that
$\mathbb D^k e^{-in\frac xL}=e^{-in\frac xL}$.  The adjoint of $\mathbb D$ is given by $\mathbb D^t=-i\nabla\cdot\frac{Ln}{|n|^2}$.
We thus obtain
\begin{align*}
|c_n^{\eps}|
&=(2\pi L)^{-3}\biggl|\int_{\R^3} \mathbb D^k e^{-in\frac xL}\psi_\eps(x)(2\pi\sigma^2)^{\frac 34} e^{\frac{|x|^2}{4\sigma^2}}dx\biggr|\\
&=(2\pi L)^{-3}\biggl|\int_{\R^3} e^{-in\frac xL}(\mathbb D^t)^k\Bigl[ \eps^{-\frac 32}\psi\Bigl(\frac x\eps\Bigr)(2\pi\sigma^2)^{\frac 34}e^{\frac{|x|^2}{4\sigma^2}}\Bigr]\,dx\biggr|\\
&\lesssim L^{-3}\Bigl(\frac L{|n|}\Bigr)^k\Bigl(\frac {\sigma}{\eps}\Bigr)^{\frac 32}\sum_{|\alpha|\leq k}\Bigl\|\partial^\alpha\Bigl[\psi\Bigl(\frac x{\eps}\Bigr)e^{\frac {|x|^2}{4\sigma^2}}\Bigr]\Bigr\|_{L^1(\R^3)}\\
&\lesssim_{k,\psi}L^{-3}\Bigl(\frac L{|n|}\Bigr)^k\Bigl(\frac {\sigma}{\eps}\Bigr)^{\frac 32}\eps^{3-k}\\
&\lesssim_{k,\psi}\frac {(\eps\sigma)^{\frac 32}}{L^3}\Bigl(\frac L{\eps|n|}\Bigr)^k.
\end{align*}
This proves \eqref{bdforc}.

To derive the last claim, we first note that
\begin{equation}\label{E:gamma inner prod}
\int_{\R^3} \gamma_n(x)\overline{\gamma_m(x)}\,dx=e^{-\frac{\sigma^2}{2L^2}|n-m|^2}.
\end{equation}
Now fix $N\in \N$.  For $n\le N$, we use the first upper bound for $c_n^{\eps}$ to estimate
\begin{align*}
\Bigl\| \sum_{|n|\le N}c_n^{\eps}\gamma_n \Bigr\|_{L^2(\R^3)}^2
&\lesssim_{\psi}\frac {(\sigma\eps)^3}{L^6}\sum_{|n|,|m|\le N} e^{-\frac{\sigma^2}{2L^2}|n-m|^2}
\lesssim_{\psi}\frac {(\sigma\eps)^3}{L^6}N^3\Bigl(\frac L{\sigma}\Bigr)^3 \lesssim_\psi \Bigl(\frac {\eps N}L\Bigr)^3.
\end{align*}
For $n\geq N$, we use the second upper bound for $c_n^{\eps}$ (with $k=3$) to estimate
\begin{align*}
\biggl\|\sum_{|n|\ge N}c_n^{\eps}\gamma_n \biggr\|_{L^2}^2
&\lesssim_{\psi}\frac {(\sigma\eps)^3}{L^6}\Bigl(\frac L{\eps}\Bigr)^6
    \sum_{|n|\ge |m|\ge N} \frac 1{|n|^3} \frac 1{|m|^3} e^{-\frac {\sigma^2}{2L^2}|n-m|^2}\\
&\lesssim_{\psi} \Bigl(\frac\sigma{\eps}\Bigr)^3\sum_{|n|\ge|m|\ge N}\frac 1{|m|^6} e^{-\frac {\sigma^2}{2L^2}|n-m|^2}\\
&\lesssim_\psi \Bigl(\frac {\sigma}{\eps}\Bigr)^3\Bigl(\frac L{\sigma}\Bigr)^3 \sum_{|m|\ge N}\frac 1{|m|^6}\\
&\lesssim_\psi \Bigl(\frac L{\eps N}\Bigr)^3.
\end{align*}
Thus,
\begin{align}\label{error}
\biggl\|&\sum_{|n|\le \frac L{\eps\log\log(\frac 1{\eps})}}c_n^{\eps}\gamma_n\biggr\|_{L^2_x}^2+\biggl\|\sum_{|n|\ge  {\frac L\eps\log\log(\frac 1{\eps})}}c_n^{\eps}\gamma_n\biggr\|_{L^2_x}^2\lesssim_{\psi}[\log\log(\tfrac 1{\eps})]^{-3}=o(1)
\end{align}
as $\eps\to 0$.  This completes the proof of Lemma~\ref{decomposition}.
\end{proof}

Combining the Strichartz inequality with Lemma~\ref{decomposition}, proving Theorem~\ref{T:LF3} reduces to showing
\begin{align*}
\Bigl\|\sum_{n\in \mathcal S} c_n^{\eps}\bigl[\propagateomega(1_\Omega\gamma_n)-e^{it\Delta_{\R^3}}\gamma_n\bigr]\Bigr\|_{L_{t,x}^{\frac {10}3}(\R\times\R^3)}=o(1) \qtq{as}\eps\to 0.
\end{align*}
Recall that the linear Schr\"odinger evolution of a Gaussian wave packet in the whole space has a simple explicit expression:
\begin{align*}
u_n(t,x):=[e^{it\Delta_{\R^3}}\gamma_n](x)=\frac 1{(2\pi)^{\frac34}}\biggl(\frac {\sigma}{\sigma^2+it}\biggr)^{\frac32}\exp\biggl\{ix\cdot \xi_n-it|\xi_n|^2-\frac {|x-2\xi_nt|^2}{4(\sigma^2+it)}\biggr\},
\end{align*}
where $\xi_n:=\frac nL$.

\begin{defn}[Missing, near-grazing, and entering rays] \label{D:MEG}
Fix $n\in \mathcal S$. We say $u_n$ \emph{misses the obstacle} if
\begin{align*}
\dist(2t\xi_n, \Omega^c)\ge \frac{|2t\xi_n|}{[\llogeps]^4} \qtq{for all} t\geq 0.
\end{align*}
Let $$\mathcal M=\{n\in \mathcal S : u_n\mbox{ misses the obstacle}\}.$$

If the ray $2t\xi_n$ intersects the obstacle, let $t_c\geq0$ and $x_c=2t_c\xi_n\in \partial\Omega$ denote the time and location of
first incidence, respectively.  We say $u_n$ \emph{enters the obstacle} if in addition
$$
\frac{|\xi_n\cdot\nu|}{|\xi_n|} \geq [\llogeps]^{-4},
$$
where $\nu$ denotes the unit normal to the obstacle at the point $x_c$. Let
\begin{align*}
\mathcal E=\{n\in \mathcal S : u_n\mbox{ enters the obstacle}\}.
\end{align*}

Finally, we say $u_n$ is \emph{near-grazing} if it neither misses nor enters the obstacle. Let
\begin{align*}
\mathcal G=\{n\in \mathcal S : u_n \mbox{ is near-grazing}\}.
\end{align*}
\end{defn}

We first control the contribution of the near-grazing directions.

\begin{lem}[Counting $\mathcal G$] \label{L:counting G} The set of near-grazing directions constitutes a
vanishing fraction of the total directions. More precisely,
$$
 \# \mathcal G \lesssim [\llogeps]^{-4} \# \mathcal S \lesssim \Bigl(\frac L\eps\Bigr)^3 [\llogeps]^{-1}.
$$
\end{lem}

\begin{proof}
We claim that the near-grazing directions are contained in a
neighbourhood of width $O( [\log\log(\frac1\eps)]^{-4} )$ around the
set of grazing rays, that is, rays that are tangent to
$\partial\Omega$. We will first verify this claim and then explain
how the lemma follows. The objects of interest are depicted in
Figure~\ref{Fig.NG}.  Two rays are show, one which collides with the
obstacle and another that does not.  The horizontal line
represents the nearest grazing ray.  The origin, from which the
rays emanate, is marked $O$.

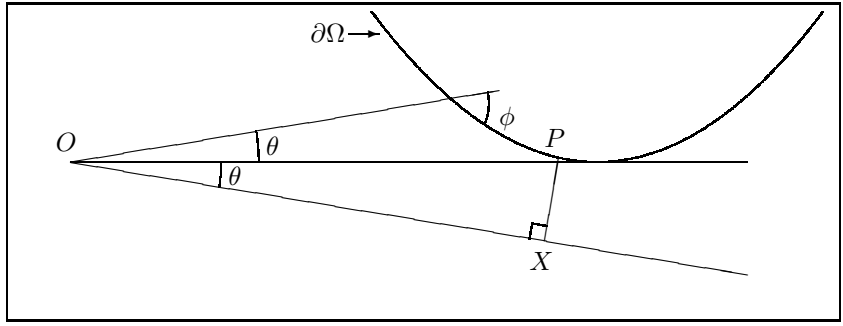
\begin{figure}[ht]
\begin{center}
\fbox{ \setlength{\unitlength}{1mm}
% 1mm  = 2.84pt
\begin{picture}(105,40)(-75,-20)
\put(-72,1.5){$O$}
\put(-70,0){\line(1,0){90}}   % horizontal line
\put(-70,0){\line(6,1){57}}   % skewed up line.  Angle to horizontal = atan(1/6) = 0.165
\qbezier(-14.568,9.322)(-14.157,6.856)(-15,5)\put(-13,5){$\phi$} % The angle \phi
\put(-70,0){\line(6,-1){90}}  % skewed down line
\qbezier(-50,0)(-50,-1.6)(-50.272,-3.287)\put(-49,-3){$\theta$}
\qbezier(-45,0)(-45,+2)(-45.340,4.110)\put(-44,1){$\theta$}
\put(-7,-10.5){\line(1,6){1.85}}\put(-9,-14.3){$X$}\put(-7,2){$P$}
\qbezier(-9,-10.167)(-8.834,-9.167)(-8.667,-8.167)  % Right-angle marker
\qbezier(-8.667,-8.167)(-7.667,-8.334)(-6.667,-8.5) % (cont.)
%\put(0,-5.833){\vector(0,1){5.833}}\put(0,-5.833){\vector(0,-1){5.833}} % missing distance
%\put(1,-7.5){$\sim\theta\ell$}                                % (cont.)
%\put(0,-16){\line(0,1){2}}\put(-70,-16){\line(0,1){2}}          % Length of adjacent side
%\put(-30,-15){\vector(1,0){30}}\put(-40,-15){\vector(-1,0){30}} % (cont.)
%\put(-35,-16){\hbox to 0mm{\hss$\ell$\hss}}       % (cont.)
%\put(-0.7,-16.667){\line(1,6){0.6}}\put(-70.7,-5){\line(1,6){0.6}}
% Length of downslope
%\put(-40.3,-8.2){\line(-6,1){30}}\put(-40.3,-8.2){\line(6,-1){39.8}} %
\qbezier(0,0)(15,0)(30,20)\qbezier(0,0)(-15,0)(-30,20)           % The obstacle
\put(-33,17){\vector(1,0){4}} \put(-38.1,16){$\partial \Omega$}  % Label for bndry of \Omega
\end{picture}
} \vspace*{-2ex}
\end{center}
\caption{Near-grazing rays.} \label{Fig.NG}
\end{figure}

For rays that collide with the obstacle, the condition to be near-grazing is that $\sin(\phi) \leq [\log\log(\frac1\eps)]^{-4}$.
Here $\phi$ is the angle between the ray and the tangent plane to the obstacle at the point of collision. Convexity of the obstacle
guarantees that $\phi \geq \theta$.  From this we deduce $\theta\lesssim [\log\log(\frac1\eps)]^{-4}$, in accordance with the claim made above.

Let us now consider rays that do not collide with the obstacle.  We recall that to be near-grazing in this case there must be some time $t>0$ so that
$X=2\xi t$ is within a distance $2|\xi|t[\log\log(\frac1\eps)]^{-4}$ of a point $P$ on the obstacle.  Then
$\theta\leq\tan\theta \leq \frac{|XP|}{|OX|} \leq [\log\log(\frac1\eps)]^{-4}$. This finishes the proof of the claim.

The set of directions corresponding to grazing rays is a smooth
curve whose length is uniformly bounded in terms of
the geometry of $\Omega$ alone. Moreover, we have shown that all
near-grazing directions lie within a neighbourhood of this curve of
thickness $O( [\log\log(\frac1\eps)]^{-4} )$.  Noting that the
directions $\{\frac n{|n|} : n\in\mathcal S\}$ are uniformly
distributed on the sphere and much more tightly packed than the width of this neighbourhood,
the lemma follows from a simple area estimate.
\end{proof}

\begin{prop}[The near-grazing contribution]\label{P:ng} We have
$$
\Bigl\|\sum_{n\in \mathcal G} c_n^{\eps}\bigl[\propagateomega(1_\Omega\gamma_n)-u_n\bigr]\Bigr\|_{L_{t,x}^{\frac {10}3}(\R\times\R^3)}=o(1) \qtq{as}\eps\to 0.
$$
\end{prop}

\begin{proof}
From the Strichartz inequality, it suffices to prove
\begin{align*}
\Bigl\| \sum_{n\in \mathcal G} c_n^{\eps}\gamma_n \Bigr\|_{L^2(\R^3)}=o(1) \qtq{as} \eps\to 0.
\end{align*}
Using \eqref{bdforc}, \eqref{E:gamma inner prod}, and Lemma~\ref{L:counting G}, we estimate
\begin{align*}
\Bigl\|\sum_{n\in \mathcal G} c_n^{\eps} \gamma_n\Bigr\|_{L^2(\R^3)}^2
&\lesssim \sum_{n,m\in \mathcal G}\frac {(\sigma \eps)^3}{L^6} e^{-\frac {\sigma^2}{2L^2}|n-m|^2}
\lesssim\sum_{n\in \mathcal G} \frac {(\sigma\eps)^3}{L^6}\Bigl(\frac L{\sigma}\Bigr)^3\lesssim [\log\log(\tfrac 1{\eps})]^{-1},
\end{align*}
which converges to $0$ as $\eps\to 0$.
\end{proof}

We now consider the contribution of rays that miss the obstacle in the sense of Definition~\ref{D:MEG}.

\begin{prop}[Contribution of rays that miss the obstacle]\label{P:missing}
Assume $n\in \mathcal M$. Then
\begin{equation}\label{432}
\|e^{it\Delta_{\Omega}}(1_{\Omega}\gamma_n)-u_n\|_{L_{t,x}^{\frac {10}3}(\R\times \R^3)}\lesssim \eps^{100}
\end{equation}
for sufficiently small $\eps$. Furthermore, we have
\begin{align}\label{249}
\Bigl\|\sum_{n\in \mathcal M} c_n^\eps \bigl[e^{it\Delta_{\Omega}}(1_\Omega\gamma_n)-u_n\bigr]\Bigr\|_{L_{t,x}^{\frac{10}3} (\R\times\R^3)}=o(1) \qtq{as} \eps\to 0.
\end{align}
\end{prop}

\begin{proof} We first notice that \eqref{249} is an immediate consequence of \eqref{432}. Indeed, using the upper bound \eqref{bdforc} for
$c_n^\eps$, we estimate
\begin{align*}
\Bigl\|\sum_{n\in \mathcal M} c_n^\eps\bigl[e^{it\Delta_\Omega}(1_\Omega\gamma_n)-u_n\bigr]\Bigr\|_{L_{t,x}^{\frac{10}3}(\R\times\R^3)}
&\lsm\sum_{|n|\le \frac L\eps\llogeps}\frac{(\sigma\eps)^{\frac32}}{L^3}\eps^{100}\\
&\lsm (\sigma\eps)^{\frac 32}\eps^{97}[\llogeps]^3=o(1).
\end{align*}
We are thus left to prove \eqref{432}.

As $u_n$ misses the obstacle, we have
\begin{align}\label{845}
\dist(2t\xi_n,\Omega^c)\ge \tfrac 12\delta[\llogeps]^{-4} \qtq{for all} t\geq 0.
\end{align}
Indeed, when $|2t\xi_n|<\frac \delta 2$, the triangle inequality gives $\dist(2t\xi_n,\Omega^c)\ge \frac \delta 2$; when
$|2t\xi_n|\geq\frac \delta 2$, this bound follows immediately from Definition~\ref{D:MEG}.

Now let $\chi$ be a smooth cutoff that vanishes on the obstacle and equals $1$ when
\begin{align*}
\dist(x,\Omega^c)\ge \delta \log^{-1}(\tfrac 1\eps).
\end{align*}
This cutoff can be chosen to also obey the following:
\begin{equation}\label{533}
%\begin{aligned}
|\nabla \chi|\lesssim \delta^{-1}\logeps, \quad |\Delta \chi|\lesssim \delta^{-2}\log^2(\tfrac 1\eps),\quad |\supp(\Delta\chi)|\lesssim \delta\log^{-1}(\tfrac 1\eps).
%\end{aligned}
\end{equation}
From \eqref{845} and the triangle inequality, we obtain
\begin{align}\label{1001}
\dist(2t\xi_n, \supp(1-\chi))&\ge\dist(2t\xi_n,\Omega^c)-\delta\log^{-1}(\tfrac 1\eps)\notag\\
&\ge \tfrac12 \delta[\llogeps]^{-4}-\delta\log^{-1}(\tfrac 1\eps)\ge \tfrac14\delta[\llogeps]^{-4}.
\end{align}
Moreover, when $|t|\ge \sigma^2$, we observe that
\begin{align}\label{1002}
\dist(2t\xi_n, \supp(1-\chi))&\ge \dist(2t\xi_n,\Omega^c)-\delta\log^{-1}(\tfrac 1\eps)\notag\\
&\ge \frac{|2t\xi_n|}{[\llogeps]^4}-\frac\delta{\logeps}\ge\frac{|t\xi_n|}{[\llogeps]^4}.
\end{align}
Here we have used the fact that $\delta\ll |2t\xi_n|$ for $t\ge\sigma^2$.

With these preliminaries out of the way, we are ready to begin proving \eqref{432}.  By the triangle inequality,
\begin{align}
\text{LHS}\eqref{432}\le\|e^{it\Delta_{\Omega}}(1_{\Omega}\gamma_n)-\chi u_n\|_{L_{t,x}^{\frac{10}3}(\R\times\R^3)}+\|\chi
u_n-u_n\|_{L_{t,x}^{\frac{10}3}(\R\times\R^3)}.\label{E:M}
\end{align}
We begin with the first term on the right-hand side of \eqref{E:M}.  Using the Duhamel formula, we write
\begin{align*}
e^{it\Delta_{\Omega}}(1_{\Omega}\gamma_n)-\chi u_n
&=e^{it\Delta_{\Omega}}(1_{\Omega}\gamma_n)-e^{it\Delta_{\Omega}}(\chi\gamma_n)+e^{it\Delta_{\Omega}}(\chi \gamma_n)-\chi u_n\\
&=e^{it\Delta_{\Omega}}[(1_{\Omega}-\chi)\gamma_n]+i\int_0^te^{i(t-s)\Delta_{\Omega}}\bigl[\Delta\chi u_n+2\nabla \chi \cdot \nabla u_n\bigr]\,ds.
\end{align*}
Similarly, for the second term on the right-hand side of \eqref{E:M} we have
\begin{align*}
(1-\chi)u_n&=e^{it\Delta}(1-\chi)\gamma_n+i\int_0^te^{i(t-s)\Delta}\bigl[\Delta \chi u_n+2\nabla \chi \cdot \nabla u_n\bigr](s)\,ds.
\end{align*}
Thus, using the Strichartz inequality we obtain
\begin{align}
\text{LHS}\eqref{432}
&\lesssim\|(1-\chi)\gamma_n\|_{L^2(\R^3)}+\|\Delta \chi u_n\|_{L^1_tL_x^2(\R\times \R^3)} +\|\nabla \chi \cdot \nabla u_n\|_{L_t^1L_x^2(\R\times \R^3)}.\label{530}
\end{align}
The first term on the right-hand side of \eqref{530} can be easily controlled:
\begin{align*}
\|(1-\chi)\gamma_n\|_{L^2(\R^3)}^2
&\lesssim \sigma^{-3}\int_{\supp(1-\chi)}e^{-\frac {|x|^2}{2\sigma^2}}dx \\
&\lesssim \sigma^{-3}\sigma^3 \exp\Bigl\{-\frac {\dist^2(0,\supp(1-\chi))}{4 \sigma^2}\Bigr\}\\
&\lsm \exp\Bigl\{-\frac{\delta^2}{8\eps\delta\log^2(\tfrac1\eps)}\Bigr\}\le\eps^{200}.
\end{align*}

To estimate the remaining terms on the right-hand side of \eqref{530}, we first observe that
\begin{align*}
|\nabla u_n|\lesssim |\xi_n||u_n|+\frac {|x-2\xi_n t|}{\sqrt{\sigma^4+t^2}}|u_n|
&\lesssim \bigl[|\xi_n |+\sigma^{-1}\bigr]\biggl(\frac{\sigma^2}{\sigma^4+t^2}\biggr)^{\frac 34}e^{-\frac{{\sigma^2}|x-2\xi_n t|^2}{8(\sigma^4+t^2)}}.
\end{align*}
As $\sigma^{-1}\leq |\xi_n|\le \frac {\log\log(\frac 1{\eps})}{\eps}$, we obtain
\begin{align}\label{1244}
|u_n|+ |\nabla u_n| \lesssim \frac{\llogeps}{\eps}\biggl(\frac{\sigma^2}{\sigma^4+t^2}\biggr)^{\frac 34}e^{-\frac{{\sigma^2}|x-2\xi_n t|^2}{8(\sigma^4+t^2)}}.
\end{align}
To estimate the contribution of these terms, we discuss short and long times separately.  For $0\leq t\le \sigma^2$, we use \eqref{1001} to estimate
\begin{align*}
\|u_n&\|_{L_t^1L_x^2(t\le \sigma^2, \ x\in \supp(1-\chi))} + \|\nabla u_n\|_{L_t^1L_x^2(t\le \sigma^2, \ x\in \supp(1-\chi))}\\
&\lsm \frac{\llogeps}{\eps} \sigma^2\sup_{0\leq t\le\sigma^2}\biggl(\frac{\sigma^2}{\sigma^4+t^2}\biggr)^{\frac 34}\biggl\|\exp\Bigl\{-\frac{\sigma^2|x-2t\xi_n|^2}{8(\sigma^4+t^2)}\Bigr\}\biggr\|_{L_x^2(\supp(1-\chi))}\\
&\lsm\frac{\llogeps}{\eps} \sigma^2 \sup_{0\leq t\le\sigma^2}\biggl\|\exp\Bigl\{-\frac{\sigma^2|x-2t\xi_n|^2}{16(\sigma^4+t^2)}\Bigr\}\biggr\|_{L_x^{\infty}
(\supp(1-\chi))}\\
&\lsm \frac{\llogeps}{\eps}\sigma^2 \exp\biggl\{-\frac\delta{\eps\log^3(\tfrac 1\eps)}\biggr\}\\
&\le\eps^{110}.
\end{align*}
For $|t|>\sigma^2$, we use \eqref{533} and \eqref{1002} to obtain
\begin{align*}
\|u_n&\|_{L_t^1L_x^2(t>\sigma^2, \ x\in \supp(1-\chi))} + \|\nabla u_n\|_{L_t^1L_x^2(t> \sigma^2, \ x\in \supp(1-\chi))}\\
&\lsm\bigl[\delta\log^{-1}(\tfrac 1\eps)\bigr]^{\frac 12}\bigl\||u_n|+|\nabla u_n|\bigr\|_{L_t^1L_x^{\infty}(t>\sigma^2, \ x\in\supp(1- \chi))}\\
&\lsm \frac{\delta^{\frac 12}\sigma^{\frac 32}\llogeps}{\eps\log^{\frac 12}(\tfrac 1\eps)}\biggl\|t^{-\frac 32}\exp\Bigl\{-\frac{\sigma^2\dist^2(2t\xi_n,
\supp(1-\chi))}{8(\sigma^4+t^2)}\Bigr\}\biggr\|_{L_t^1(t>\sigma^2)}\\
&\lsm \frac{\delta^{\frac 12}\sigma^{\frac 12}\llogeps}{\eps\log^{\frac 12}(\tfrac 1\eps)} \exp\Bigl\{-\frac\delta\eps\Bigr\}\\
&\le \eps^{110}.
\end{align*}
Putting these two pieces together, we find
\begin{align*}
\|\Delta \chi u_n\|_{L_t^1L_x^2(\R\times\R^3)}+\|\nabla \chi \cdot \nabla u_n\|_{L_t^1L_x^2(\R\times \R^3)}
\lsm\delta^{-2}\log^2(\tfrac 1\eps)\eps^{110}\le \eps^{100}.
\end{align*}
This completes the proof of Proposition~\ref{P:missing}.
\end{proof}

In order to complete the proof of Theorem~\ref{T:LF3}, we need to estimate the contribution from the Gaussian wave packets $\gamma_n$ that
collide non-tangentially with the obstacle, that is, for $n\in\mathcal E$.  This part of the argument is far more subtle than the treatment of
$n\in \mathcal G$ or $n\in \mathcal M$.  Naturally, the entering wave packets reflect off the obstacle and we will need to build a careful parametrix to capture this reflection.  Moreover, the convexity of the obstacle enters in a crucial way --- it ensures that the reflected waves do not refocus.

The treatment of the entering rays will occupy the remainder of this subsection.  We begin with the simplest part of the analysis, namely,
the short time contribution.  Here, short times means well before the wave packets have reached the obstacle.  The estimate applies equally
well to all wave packets, irrespective of whether $n\in\mathcal E$ or not.

\begin{prop}[The contribution of short times]\label{P:short times}
Let $T:=\frac{\eps\delta}{10\llogeps}$.  Then
\begin{align}\label{st}
\sum_{n\in \mathcal S} |c_n^\eps|\bigl\|e^{it\Delta_{\Omega}}(1_\Omega\gamma_n)-u_n\bigr\|_{L_{t,x}^{\frac{10}3} ([0,T]\times\R^3)}=o(1)
\qtq{as} \eps\to 0.
\end{align}
\end{prop}

\begin{proof}
Let $\chi $ be a smooth cutoff that vanishes on the obstacle and equals $1$ when $\dist(x,\Omega^c)>\frac \delta{10}$.  This cutoff can be chosen to also satisfy
\begin{align}\label{deta}
|\nabla \chi|\lsm \delta^{-1} \qtq{and} |\Delta \chi|\lsm \delta^{-2}.
\end{align}
Moreover, for $t\in[0,T]$ we have
\begin{align*}
|2t\xi_n |\le 2\frac{\eps\delta}{10\llogeps}\cdot\frac{\llogeps}{\eps}=\frac15 \delta
\end{align*}
and so
\begin{align}\label{615}
\dist(2t\xi_n, \supp(1-\chi))\ge \tfrac12 \delta \qtq{for all} t\in[0, T].
\end{align}

The proof of this proposition is almost identical to that of Proposition~\ref{P:missing}, with the roles of \eqref{1001} and \eqref{1002} being played by \eqref{615}.  Indeed, using the Duhamel formula and the Strichartz inequality as in the proof of Proposition~\ref{P:missing},
\begin{align*}
\|e^{it\Delta_{\Omega}}(&1_{\Omega}\gamma_n)-u_n\|_{L_{t,x}^{\frac{10}3}([0,T]\times\R^3)}\\
&\le\|e^{it\Delta_{\Omega}}(1_{\Omega}\gamma_n)-\chi u_n\|_{L_{t,x}^{\frac{10}3}([0,T]\times\R^3)}+\|\chi u_n-u_n\|_{L_{t,x}^{\frac{10}3}([0,T]\times\R^3)}\\
&\lesssim \|(1-\chi)\gamma_n\|_{L^2(\R^3)}+\|\Delta \chi u_n\|_{L^1_tL_x^2([0,T]\times\R^3)} +\|\nabla \chi \cdot \nabla u_n\|_{L_t^1L_x^2([0,T]\times\R^3)}.
\end{align*}
The first term is estimated straightforwardly
\begin{align*}
\|(1-\chi)\gamma_n\|_{L^2(\R^3)}^2
&\lsm\sigma^{-3}\int_{\supp(1-\chi)} e^{-\frac{|x|^2}{2\sigma^2}} \,dx
\lsm e^{-\frac{\dist^2(0,\supp(1-\chi))}{4\sigma^2}}
\lsm e^{-\frac{\delta^2}{16\sigma^2}}\le \eps^{200}.
\end{align*}
For the remaining two terms, we use \eqref{1244} and \eqref{615} to estimate
\begin{align*}
\|u_n\|_{L_t^1L_x^2([0,T]\times\supp(1-\chi))} & + \|\nabla u_n\|_{L_t^1L_x^2([0,T]\times\supp(1-\chi))}\\
&\lsm\delta\sup_{t\in[0,T]} \biggl(\frac{\sigma^2}{\sigma^4+t^2}\biggr)^{\frac34}\Bigl\|e^{-\frac{\sigma^2|x-2\xi_n t|^2}{8(\sigma^4+t^2)}}\Bigr\|_{L_x^2(\supp(1-\chi))}\\
&\lsm \delta \sup_{t\in[0,T]}\Bigl\|e^{-\frac{\sigma^2|x-2\xi_nt|^2}{16(\sigma^4+t^2)}}\Bigr\|_{L_x^{\infty}(\supp(1-\chi))}\\
&\lsm \delta \sup_{t\in[0,T]} \exp\Bigl\{-\frac{\sigma^2\dist^2(2t\xi_n,\supp(1-\chi))}{32\sigma^4}\Bigr\} \\
&\lsm \delta e^{-\frac{\delta^2}{128\sigma^2}}\le \eps^{110}.
\end{align*}
This implies
\begin{align*}
\|\Delta \chi u_n\|_{L^1_tL_x^2([0,T]\times \R^3)} +\|\nabla \chi \cdot \nabla u_n\|_{L_t^1L_x^2([0,T]\times \R^3)}
	\lsm \delta^{-2}\eps^{110}\le \eps^{100}.
\end{align*}

Collecting these estimates and using \eqref{bdforc} we obtain
\begin{align*}
\text{LHS}\eqref{st}\lesssim \sum_{n\in \mathcal S} \frac{(\sigma\eps)^{\frac32}}{L^3} \eps^{100}=o(1) \qtq{as} \eps\to 0.
\end{align*}
This finishes the proof of Proposition~\ref{P:short times}.
\end{proof}

Now take $n\in \mathcal E$, which means that the wave packet $u_n(t,x)$ enters the obstacle.  We write $t_c$ for the first time of intersection and $x_c=2t_c\xi_n$ for the location of this collision.  Naturally both $t_c$ and $x_c$ depend on $n$; however, as most of the analysis will focus on one wave packet at a time, we suppress this dependence in the notation.

We approximate the wave generated by $u_n$ reflecting off $\partial\Omega$ by a Gaussian wave packet $v_n$ (or more accurately
by $-v_n$ since the Dirichlet boundary condition inverts the profile), which we define as follows:
\begin{align}\label{forv}
v_n(t,x):=&\Bigl(\frac {\sigma^2}{2\pi}\Bigr)^{\frac 34}\frac {(\det\Sigma)^{\frac 12}}{(\sigma^2+it_c)^{\frac 32}} [\det(\Sigma+i(t-t_c))]^{-\frac12}
\exp\Bigl\{i(x-x_c)\eta-it|\eta|^2\notag\\
&\qquad\qquad\qquad\qquad+ix_c\cdot \xi-\tfrac14(x-x(t))^T(\Sigma+i(t-t_c))^{-1}(x-x(t))\Bigr\},
\end{align}
where for simplicity we write $\xi=\xi_n$.  The parameters $\eta$, which represents the momentum of the reflected wave packet, and $\Sigma$, which gives its covariance structure, will be specified shortly.  Correspondingly, $x(t):=x_c+2\eta(t-t_c)$ represents the center of the reflected wave packet.

We define an orthonormal frame $(\vec\tau,\vec \gamma,\vec \nu)$ at the point $x_c\in\partial\Omega$, where $\vec \tau,\vec \gamma$ are
two tangent vectors to $\partial \Omega$ in the directions of the principal curvatures $\frac 1{R_1},\ \frac 1{R_2}$ and $\vec \nu$ is the
unit outward normal to the obstacle. Note that the obstacle being strictly convex amounts to $1\lesssim R_1,R_2<\infty$. Without loss
of generality, we may  assume $R_1\le R_2$.

With this frame, we define $\eta:=\xi-2(\xi\cdot\vec\nu)\vec\nu$ as the reflection of $\xi$, in accordance with the basic law of reflection,
namely, the angle of incidence equals the angle of reflection. In this frame, $\Sigma^{-1}$ is defined as follows:
\begin{align}\label{E:Sigma defn}
\Sigma^{-1}=\frac 1{\sigma^2+it_c}\Id+iB,
\end{align}
where
\begin{align*}
B=\begin{pmatrix}
\frac {4\xi_3}{R_1} & 0 &\frac {4\xi_1}{R_1}\\
0 &\frac {4\xi_3}{R_2} &\frac {4\xi_2}{R_2}\\
\frac {4\xi_1}{R_1} &\frac {4\xi_2}{R_2}  &\frac{4\xi_1^2}{R_1\xi_3}+\frac {4\xi_2^2}{R_2\xi_3}
\end{pmatrix}
\end{align*}
and
\begin{align*}
\eta_1:=\eta\cdot\vec\tau&=\xi\cdot\vec\tau=:\xi_1 \\
\eta_2:=\eta\cdot \vec\gamma&=\xi\cdot \vec\gamma=:\xi_2\\
\eta_3:=\eta\cdot\vec\nu=-\xi\cdot\vec\nu&=:-\xi_3=\tfrac12 |\xi-\eta|.
\end{align*}
The matrix $B$ encodes the additional spreading of the reflected wave packet induced by the curvature of the obstacle; incorporating this subtle effect is essential for the analysis that follows.  The structure of the matrix $B$ captures the basic rule of mirror manufacture:  the radius of curvature equals twice the focal length.  

\begin{lem}[Bounds for collision times and locations]\label{L:xc}
For rays that enter the obstacle, we have
$$
\xi_3 < 0, \quad |\xi_3| \geq |\xi| [\llogeps]^{-4}, \qtq{and} \delta \leq |x_c| \lesssim \delta[\log\log(\tfrac1\eps)]^8.
$$
In particular, $\delta|\xi|^{-1}\leq 2t_c\lesssim \delta|\xi|^{-1}[\llogeps]^8$.
\end{lem}

\begin{proof}
The first inequality simply expresses the fact that the ray approaches the obstacle from without.  The second inequality is an exact repetition of
$n\in \mathcal E$ as given in Definition~\ref{D:MEG}.  The lower bound on $|x_c|$ follows directly from the fact that $\delta=\dist(0,\Omega^c)$.

The proof of the upper bound on $|x_c|$ divides into two cases.  When $\delta\gtrsim[\log\log(\tfrac1\eps)]^{-8}$, the result follows from
$|x_c|\leq \dist(0,\Omega^c) + \diam(\Omega^c) \lesssim1$.

It remains to consider the case when $\delta\leq \tfrac{1}{8C} [\log\log(\tfrac1\eps)]^{-8}$ for some fixed large $C=C(\Omega)$. By
approximating $\partial\Omega$ from within by a paraboloid, this case reduces to the analysis of the following system of equations:
\begin{align*}
y &= m x  \quad \text{and} \quad y = Cx^2 + \delta \quad \text{with} \quad m\geq [\log\log(\tfrac1\eps)]^{-4}.
\end{align*}
The first equation represents the ray, whose slope is restricted by that permitted for an entering ray. (Note that the convexity of the obstacle implies that the angle between the ray and $\partial\Omega$ is larger than the angle between the ray and the axis $y=0$.)  Using the quadratic formula, we see that the solution obeys
$$
|x_c|\leq \sqrt{ x^2 + y^2 } = \frac{2\delta \sqrt{1+m^2}}{m + \sqrt{m^2 - 4C\delta}} \sim \frac{\delta\sqrt{1+m^2}}{m},
$$
where we used the restriction on $\delta$ in the last step.
\end{proof}

\begin{lem}[Reflected waves diverge]\label{L:diverging rays}
For $j=1,2$, let $x^{(j)}(t)$ denote the broken ray beginning at the origin, moving with velocity $2\xi^{(j)}$ and reflecting off the convex body $\Omega^c$.  Then
$$
| x^{(1)}(t) - x^{(2)}(t) | \geq 2| \xi^{(1)} - \xi^{(2)} | \, t
$$
whenever $t \geq \max\{t_c^{(1)},t_c^{(2)}\}$, that is, greater than the larger collision time.
\end{lem}

\begin{proof}
In the two dimensional case, this result follows from elementary planar geometry.  A particularly simple argument is to reflect the
outgoing rays across the line joining the two collision points.  By convexity, the continuations of the incoming rays will both lie
between the reflected outgoing rays. Note that the geometry involved is dictated solely by the two tangent lines at the collision points, not by the
shape of the convex body in between.

We note that given two vectors $v^{(j)}\in\R^2$ and two points $y^{(j)}\in\R^2$, there is a convex curve passing through these points
and having these vectors as outward normals at these points if and only if
\begin{equation}\label{Convex position}
v^{(1)} \cdot \bigl(y^{(1)}-y^{(2)}\bigr) \geq 0\quad\text{and}\quad v^{(2)} \cdot \bigl(y^{(2)}-y^{(1)}\bigr) \geq0.
\end{equation}
Indeed, by convexity, $\Omega^c\subseteq\{x:\, (x-y^{(j)})\cdot v^{(j)}\leq 0\}$ for $j=1,2$.

We will use this two dimensional case as a stepping stone to treat three dimensions.  (The argument carries over to higher dimensions
also.)  If $\xi^{(1)}$ and $\xi^{(2)}$ are parallel, then the analysis is one-dimensional and totally elementary.  In what
follows, we assume that these vectors are not parallel.

Let $\nu^{(1)}$ and $\nu^{(2)}$ denote the unit outward normals to $\partial\Omega$ at the collision points.  These are linearly
independent.  We write $P$ for the orthogonal projection into the plane that they span and $Q=\Id-P$ for the complementary projection.
By the law of reflection, $Q [ x^{(j)}(t) ] = Q [2\xi^{(j)} t]$ and the broken rays $P [ x^{(j)}(t) ]$ make equal angles of incidence
and reflection with the projected normals $P[\nu^{(j)}]=\nu^{(j)}$ at the projected collision points.  We now apply the two-dimensional
result.  To do this, we need to see that the projected collision points and the projected normals obey the chord/normal condition \eqref{Convex position};
this follows immediately from the convexity of the original obstacle.

Using the two-dimensional result, we get
\begin{align*}
\bigl| x^{(1)}(t) - x^{(2)}(t) \bigr|^2 &= \bigl| P[x^{(1)}(t)] - P[x^{(2)}(t)] \bigr|^2 + \bigl| Q[x^{(1)}(t)] - Q[x^{(2)}(t)]\bigr|^2\\
&\geq 4 \bigl| P [ \xi^{(1)} ] - P [ \xi^{(2)} ] \bigr|^2 t^2 +4\bigl| Q [ \xi^{(1)} ] - Q [ \xi^{(2)} ] \bigr|^2 t^2 \\
&= 4 | \xi^{(1)}- \xi^{(2)} |^2 t^2,
\end{align*}
which proves the lemma.
\end{proof}

Next we investigate in more detail the properties of the matrix $\Sigma$.

\begin{lem}[Bounds for the covariance matrix]\label{L:matrix}
Let $n\in \mathcal E$.  Then
\begin{align}
\Re \vec v^{\;\!T}(\Sigma+i(t-t_c))^{-1}\vec v&\ge\frac{\sigma^2}{[\llogeps]^{25}[\sigma^4+\log^4(\tfrac 1\eps)t^2]} |\vec v|^2\label{sig41}\\
\|(\Sigma+i(t-t_c))^{-1}\|_{\max}&\leq\frac{\log^{5}(\frac1\eps)}{\sqrt{\sigma^4+t^2}}\label{sig42}\\
|\det( \Id +i(t-t_c)\Sigma^{-1})|^{-\frac 12}&\leq \log^{\frac52}(\tfrac1{\eps})\biggl(\frac{\sigma^4+t_c^2}{\sigma^4+t^2}\biggr)^{\frac34}\label{sig3}
\end{align}
for all $t\geq 0$ and $\vec v\in \R^3$.  If in addition $|t-t_c|\le 4\frac{\sigma\logeps}{|\xi|}$, then
\begin{align}
\|(\Sigma+i(t-t_c))^{-1}-\Sigma^{-1}\|_{HS}&\le\eps^{-\frac 12}\delta^{-\frac 32}\log^3(\tfrac 1\eps)\label{sig1}\\
\bigl|1-\det(\Id+i(t-t_c)\Sigma^{-1})^{-\frac12}\bigr|&\le\eps^{\frac12}\delta^{-\frac12}\log^3(\tfrac 1\eps)\label{sig2}.
\end{align}
Here $\|\cdot\|_{HS}$ denotes the Hilbert--Schmidt norm: for a matrix $A=(a_{ij})$, this is given by $\|A\|_{HS}=(\sum_{i,j}|a_{ij}|^2)^{\frac 12}$.
Also, $\|A\|_{\max}$ denotes the operator norm of $A$.
\end{lem}

\begin{proof} We first prove \eqref{sig1}. Using Lemma~\ref{L:xc}, we get
\begin{align}\label{sig}
\|\Sigma^{-1}\|_{HS}\leq \frac{\sqrt{3}}{|\sigma^2+it_c|}+\|B\|_{HS}
\le \sqrt{3}\sigma^{-2}+\frac {4\sqrt{10}|\xi|^2}{ R_1|\xi_3|}
&\lesssim\sigma^{-2}+|\xi|[\log\log(\tfrac 1{\eps})]^4 \notag\\
&\lesssim \eps^{-1}\delta^{-1}[\log\log(\tfrac 1{\eps})]^5.
\end{align}
Thus, for $|t-t_c|\le 4\frac{\sigma\logeps}{|\xi|}$ we obtain
\begin{align}\label{306}
\|(t-t_c)\Sigma^{-1}\|_{HS}&\lsm \eps^{\frac12}\delta^{-\frac12}\log^2(\tfrac 1\eps)[\llogeps]^6\ll1.
\end{align}
Combining this with the resolvent formula
\begin{align*}
(\Sigma+i(t-t_c))^{-1}-\Sigma^{-1}&=-i(t-t_c)\Sigma^{-1}(\Sigma+i(t-t_c))^{-1}\\
&=-i(t-t_c)\Sigma^{-2}(\Id+i(t-t_c)\Sigma^{-1})^{-1}
\end{align*}
and using \eqref{306}, we estimate
\begin{align*}
\|(\Sigma+i(t-t_c))^{-1}-\Sigma^{-1}\|_{HS}&\le
|t-t_c|\|\Sigma^{-1}\|_{HS}^2\|(\Id+i(t-t_c)\Sigma^{-1})^{-1}\|_{HS}\\
&\lesssim \frac{4\sigma\logeps}{|\xi|}\eps^{-2}\delta^{-2}[\llogeps]^{10}\\
&\lsm \eps^{-\frac 12}\delta^{-\frac 32}\log^2(\tfrac1\eps)[\llogeps]^{11}\\
&\le \eps^{-\frac 12}\delta^{-\frac 32}\log^3(\tfrac 1\eps).
\end{align*}
This settles \eqref{sig1}.  The estimate \eqref{sig2} follows from \eqref{306} and the fact that the determinant function is Lipschitz on a small neighborhood of the identity.

We now turn to the remaining estimates; the key is to understand the real symmetric matrix $B$.  A direct computation gives
\begin{align*}
\det(\lambda \Id-B)=\lambda\biggl[\lambda^2-4\biggl(\frac{\xi_1^2+\xi_3^2}{R_1\xi_3}+\frac{\xi_2^2+\xi_3^2}{R_2\xi_3}\biggr)\lambda
+\frac{16|\xi|^2}{R_1R_2}\biggr].
\end{align*}
Hence one eigenvalue is $0$ and it is easy to check that $\eta$ is the corresponding eigenvector.  We write $-\infty<\lambda_2\le\lambda_1<0$ for the remaining eigenvalues. Moreover, as
\begin{align*}
\lambda_1\lambda_2=\frac{16|\xi|^2}{R_1R_2}\qtq{and}
|\lambda_1|+|\lambda_2|=4\Bigl(\frac{\xi_1^2+\xi_3^2}{R_1|\xi_3|}+\frac{\xi_2^2+\xi_3^2}{R_2|\xi_3|}\Bigr),
\end{align*}
using Lemma~\ref{L:xc} we get
\begin{align*}
[\llogeps]^{-4}|\xi|\lsm |\lambda_1|\le |\lambda_2|\lsm\frac{|\xi|^2}{|\xi_3|}\lsm |\xi|[\llogeps]^4.
\end{align*}
In particular,
\begin{align}\label{B norm}
\| B \|_{\max} \lesssim |\xi|[\llogeps]^4 \lesssim \eps^{-1} [\llogeps]^5.
\end{align}

The orthonormal eigenbasis for $B$ is also an eigenbasis for $\Sigma^{-1}$ with eigenvalues
\begin{align*}
\frac 1{\sigma^2+it_c}, \quad \frac 1{\sigma^2+it_c}+i\lambda_1,\qtq{and} \frac1{\sigma^2+it_c}+i\lambda_2.
\end{align*}
In this basis, $(\Sigma+i(t-t_c))^{-1}$ is diagonal with diagonal entries
\begin{align*}
\frac 1{\sigma^2+it}, \ \Bigl[\Bigl(\frac1{\sigma^2+it_c}+i\lambda_1\Bigr)^{-1}+i(t-t_c)\Bigr]^{-1},\ %
\Bigl[\Bigl(\frac1{\sigma^2+it_c}+i\lambda_2\Bigr)^{-1}+i(t-t_c)\Bigr]^{-1}.
\end{align*}

An exact computation gives
\begin{align*}
\Re \Bigl[\Bigl(\frac1{\sigma^2+it_c}+i\lambda_j\Bigr)^{-1}+i(t-t_c)\Bigr]^{-1}=\frac{\sigma^2}{\sigma^4[1-\lambda_j(t-t_c)]^2 +[t-\lambda_jt_c(t-t_c)]^2}.
\end{align*}
Using $\delta\lesssim 1$, the upper bound for $t_c$ given by Lemma~\ref{L:xc}, and the upper bound for $\lambda_j$ obtained above, we get
\begin{align*}
|\lambda_j t_c|&\lsm |\xi|[\llogeps]^4\frac{\delta[\llogeps]^8}{|\xi|}\lsm [\llogeps]^{12}\\
\lambda_j^2 t_c^4&\lsm |\xi|^2[\llogeps]^8\frac{\delta^4[\llogeps]^{32}}{|\xi|^4}\lsm\delta^4\eps^2[\llogeps]^{42}\le \sigma^4\\
\sigma^4\lambda_j^2&\lsm \eps^2\delta^2\log^4(\tfrac 1\eps)|\xi|^2[\llogeps]^8 \lsm  \log^4(\tfrac1\eps)[\llogeps]^{10}.
\end{align*}
Therefore,
\begin{align*}
\sigma^4[1-\lambda_j(t-t_c)]^2 &+[t-\lambda_jt_c(t-t_c)]^2\\
&\lsm \sigma^4(1+\lambda_j^2t^2+\lambda_j^2t_c^2)+t^2+\lambda_j^2t_c^2t^2+\lambda_j^2t_c^4\\
&\lsm\sigma^4(1+\lambda_j^2t_c^2)+\lambda_j^2t_c^4+t^2(1+\lambda_j^2t_c^2+\sigma^4\lambda_j^2)\\
&\lsm \sigma^4[\llogeps]^{24}+t^2\log^4(\tfrac1\eps)[\llogeps]^{10}\\
&\le [\llogeps]^{25}[\sigma^4+\log^4(\tfrac 1\eps) t^2].
\end{align*}
Thus,
\begin{align*}
\Re \Bigl[\Bigl(\frac 1{\sigma^2+it_c}+i\lambda_j\Bigr)^{-1}+i(t-t_c)\Bigr]^{-1}\ge\frac{ \sigma^2}{[\llogeps]^{25}[\sigma^4+t^2\log^4(\frac 1\eps)]}.
\end{align*}
As $\Re \frac 1{\sigma^2+it}$ admits the same lower bound, we derive \eqref{sig41}.

We now turn to \eqref{sig42}.  Our analysis is based on the identity
\begin{align*}
\biggl|\Bigl[\Bigl(\frac1{\sigma^2+it_c}+i\lambda_j\Bigr)^{-1}+i(t-t_c)\Bigr]^{-1}\biggr|^2
&=\biggl|\frac{1-\lambda_jt_c+i\lambda_j\sigma^2}{\sigma^2[1-\lambda_j(t-t_c)]+i[t-\lambda_jt_c(t-t_c)]}\biggr|^2\\
&=\frac{(1-\lambda_jt_c)^2+(\lambda_j\sigma^2)^2}{\sigma^4[1-\lambda_j(t-t_c)]^2+[t-\lambda_jt_c(t-t_c)]^2}.
\end{align*}
We have
\begin{align*}
(1-\lambda_jt_c)^2+(\lambda_j\sigma^2)^2\lesssim 1+ [\llogeps]^{24} + \log^4(\tfrac1\eps)[\llogeps]^{10} \leq \log^5(\tfrac1\eps).
\end{align*}
To estimate the denominator we use Lemma~\ref{L:xc} to see that $t_c\ll \sigma^2$ and so
\begin{align}\label{823}
\sigma^4[1-\lambda_j(t-t_c)]^2+[t-\lambda_jt_c(t-t_c)]^2
&\geq t_c^2\Bigl\{[1-\lambda_j(t-t_c)]^2 + \bigl[ \tfrac{t}{t_c} -\lambda_j(t-t_c)\bigr]^2\Bigr\}\notag\\
&=2\bigl[\tfrac{t+t_c}2-\lambda_j t_c (t-t_c)\bigr]^2 + \tfrac12(t-t_c)^2\notag\\
&\gtrsim [\llogeps]^{-24} (t+t_c)^2\notag\\
&\gtrsim \frac{\sigma^4+t^2}{\log^4(\frac1{\eps}) [\llogeps]^{26}},
\end{align}
where we have used the bound $|\lambda_j t_c|\lsm [\llogeps]^{12}$ to derive the penultimate inequality.  Combining these bounds we obtain
\begin{align*}
\biggl|\Bigl[\Bigl(\frac1{\sigma^2+it_c}+i\lambda_j\Bigr)^{-1}+i(t-t_c)\Bigr]^{-1}\biggr|^2
&\lesssim \frac{\log^9(\frac1\eps)[\llogeps]^{26}}{\sigma^4+t^2} \le \frac{\log^{10}(\frac1{\eps})}{\sigma^4+t^2}.
\end{align*}
As $(\Sigma+i(t-t_c))^{-1}$ is orthogonally diagonalizable, this bound on its eigenvalues yields \eqref{sig42}.

Finally, we compute
\begin{align*}
|\det(\Id+i(t-t_c)\Sigma^{-1})|
&=\biggl|\Bigl(1+\frac{i(t-t_c)}{\sigma^2+it_c}\Bigr)\prod_{j=1,2}\Bigl[1+i(t-t_c)\Bigl(\frac 1{\sigma^2+it_c}+i\lambda_j\Bigr)\Bigr]\biggr|\\
&=\biggl|\frac{\sigma^2+it}{(\sigma^2+it_c)^3}\prod_{j=1,2}\Bigl\{\sigma^2[1-\lambda_j(t-t_c)]+i[t-\lambda_jt_c(t-t_c)]\Bigr\}\biggr|.
\end{align*}
Using \eqref{823} we obtain
\begin{align*}
&|\det(\Id+i(t-t_c)\Sigma^{-1})|^{-1}\\
&\quad\le \frac{(\sigma^4+t_c^2)^{\frac 32}}{(\sigma^4+t^2)^{\frac12}}
	\prod_{j=1,2}\Bigl\{\sigma^4[1-\lambda_j(t-t_c)]^2+[t-\lambda_jt_c(t-t_c)]^2\Bigr\}^{-\frac12}\\
&\quad\leq \log^5(\tfrac1{\eps})\biggl(\frac{\sigma^4+t_c^2}{\sigma^4+t^2}\biggr)^{\frac32}.
\end{align*}
This completes the proof of the lemma.
\end{proof}

Using this lemma, we will see that the reflected wave $v_n$ agrees with $u_n$ to high order on $\partial\Omega$, at least for $x$ near $x_c^{(n)}$ and $t$ near $t_c^{(n)}$; compare \eqref{A} with $|u_n(t_c^{(n)},x_c^{(n)})|\sim \sigma^{-3/2}$. Indeed, requiring this level of agreement can be used to derive the matrix $B$ given above.  Without this level of accuracy we would not be able to show that the contribution of entering rays is $o(1)$ as $\eps\to0$.

\begin{lem} \label{L:uv match} Fix $n\in \mathcal E$.  For each $x\in \Omega$, let $x_*=x_*(x)$ denote the nearest point to $x$ in $\partial\Omega$.
Let
$$
A_n(t,x):=\exp\{it|\xi_n|^2-i\xi_n\cdot(x_*-x_c^{(n)})\}\bigl[u_n(t,x_*)-v_n(t,x_*)\bigr].
$$
Then for each $(t,x)\in\R\times\Omega$ such that $|x_*-x_c^{(n)}|\le\sigma \log(\frac 1{\eps})$ and $|t-t_c^{(n)}|\le \frac {4\sigma\log(\frac 1{\eps})}{|\xi_n|}$ we have
\begin{align}
|A_n(t,x)|&\lesssim\eps^{-\frac 14}\delta^{-\frac 54}\log^{12}(\tfrac 1{\eps}) \label{A}\\
|\nabla A_n(t,x)|&\lesssim \eps^{-\frac 34}\delta^{-\frac 74}\log^{12}(\tfrac 1{\eps}) \label{deriv A}\\
|\partial_t A_n(t,x)| + |\Delta A_n(t,x)|&\lesssim \eps^{-\frac 74}\delta^{-\frac 74}\log^9(\tfrac 1{\eps}) \label{laplace A}.
\end{align}
\end{lem}

\begin{proof}
Throughout the proof, we will suppress the dependence on $n\in\mathcal E$; indeed, all estimates will be uniform in $n$.  Let
\begin{align*}
F(t,x) &:= \biggl( \frac{\sigma^2+it_c}{\sigma^2+it}\biggr)^{\frac32} e^{-\frac{|x-2\xi t|^2}{4(\sigma^2+it)} } - \det (1+i(t-t_c)\Sigma^{-1})^{-\frac12}  e^{\Phi(t,x)}
\end{align*}
with
\begin{align*}
\Phi(t,x) &:= i(x-x_c)(\eta-\xi) -\tfrac14(x-x(t))^T(\Sigma +i(t-t_c))^{-1}(x-x(t)),
\end{align*}
so that
\begin{equation}\label{AfromF}
A(t,x) = \Bigl(\frac {\sigma^2}{2\pi}\Bigr)^{\frac 34}(\sigma^2+it_c)^{-\frac 32} e^{ix_c \xi} F(t,x_*).
\end{equation}

We further decompose
\begin{align*}
F(t,x)=F_1(t,x)+F_2(t,x)+F_3(t,x),
\end{align*}
where
\begin{align*}
F_1(t,x)&:= \biggl[\biggl(\frac{\sigma^2+it_c}{\sigma^2+it}\biggr)^{\frac32} -1\biggr]e^{-\frac{|x-2\xi t|^2}{4(\sigma^2+it)} }\\
F_2(t,x)&:=\bigl[1-\det(1+i(t-t_c)\Sigma^{-1})^{-\frac 12}\bigr] e^{-\frac {|x-2\xi t|^2}{4(\sigma^2+it)}}\\
F_3(t,x)&:= \det (1+i(t-t_c)\Sigma^{-1})^{-\frac12}\Bigl\{e^{-\frac{|x-2\xi t|^2}{4(\sigma^2+it)}}-e^{\Phi(t,x)}\Bigr\}.
\end{align*}

We begin by estimating the time derivative of $F$ on $\partial\Omega$.  We will make repeated use of the following bounds:
\begin{align}\label{E:bounds1}
t\sim t_c \ll \sigma^2 \qtq{and} |x-2\xi t| + |x-x(t)|\lesssim\sigma\log(\tfrac1{\eps}),
\end{align}
for all $|x-x_c|\le\sigma \log(\frac 1{\eps})$ and $|t-t_c|\le \frac {4\sigma\log(\frac 1{\eps})}{|\xi|}$.  Moreover, from \eqref{sig1}, \eqref{sig2}, and \eqref{sig},
we obtain
\begin{align}\label{E:bounds2}
\bigl|\partial_t\det(1+i(t-t_c)\Sigma^{-1})^{-\frac 12}\bigr| &=\tfrac12 \bigl|\det(1+i(t-t_c)\Sigma^{-1})^{-\frac 12}\bigr| \bigl| \Tr (\Sigma+i(t-t_c))^{-1}\bigr|\notag\\
&\lesssim \|(\Sigma+i(t-t_c))^{-1}\|_{HS}\lesssim \eps^{-1}\delta^{-1}[\llogeps]^5.
\end{align}
Lastly, as $\xi-\eta$ is normal to $\partial\Omega$ at $x_c$, we see that
\begin{align}\label{E:bounds3}
|(\xi-\eta)\cdot(x-x_c)| \lesssim  |\xi| \, |x-x_c|^2  \lesssim \delta \log^5(\tfrac1\eps),
\end{align}
for  all $x\in \partial\Omega$ with $|x-x_c|\lesssim \sigma\logeps$.

A straightforward computation using \eqref{E:bounds1} gives
\begin{align*}
|\partial_tF_1(t,x)|&\lesssim \sigma^{-2} + |t-t_c|\sigma^{-2}\bigl[\sigma^{-2}|\xi||x-2\xi t| + \sigma^{-4}|x-2\xi t|^2\bigr]\lesssim \eps^{-1}\delta^{-1}.
\end{align*}
Using also \eqref{sig2} and \eqref{E:bounds2} we obtain
\begin{align*}
|\partial_tF_2(t,x)|&\lesssim \eps^{-1}\delta^{-1}[\llogeps]^5 + \eps^{\frac12}\delta^{-\frac12} \log^3(\tfrac1\eps)\bigl[\sigma^{-2}|\xi||x-2\xi t| + \sigma^{-4}|x-2\xi t|^2\bigr]\\
&\lesssim \eps^{-1}\delta^{-1} \log^{4}(\tfrac1\eps).
\end{align*}
As $|\partial_t A| \lesssim \sigma^{-3/2} |\partial_t F|$, the contributions of $\partial_t F_1$ and $\partial_t F_2$ are consistent with \eqref{laplace A}.

We now turn to $F_3$.  In view of \eqref{sig41}, \eqref{sig2}, and \eqref{E:bounds2},
\begin{align*}
&|\partial_t F_3(t,x)|
\lesssim\eps^{-1}\delta^{-1}[\llogeps]^5+\biggl|\Bigl[\frac{\xi(x-2\xi t)}{\sigma^2+it}+\frac{i|x-2\xi t|^2}{4(\sigma^2+it)^2}\Bigr]e^{-\frac{|x-2\xi t|^2}{4(\sigma^2+it)}}\\
&- \Bigl[ \eta^T(\Sigma + i(t-t_c))^{-1}(x-x(t)) +\tfrac i4(x-x(t))^T(\Sigma + i(t-t_c))^{-2}(x-x(t)) \Bigr] e^{\Phi(t,x)}\biggr|.
\end{align*}
To simplify this expression we use the following estimates
\begin{align*}
\biggl| \frac{\xi(x-2\xi t)}{\sigma^2+it}+\frac{i|x-2\xi t|^2}{4(\sigma^2+it)^2} - \frac{\xi(x-2\xi t)}{\sigma^2+it_c} - \frac{i|x-2\xi t|^2}{4(\sigma^2+it_c)^2} \biggr|
	&\lesssim \eps^{-1}\delta^{-1} \\
\bigl| \eta^T \bigl[ (\Sigma +i(t-t_c))^{-1} - \Sigma^{-1} \bigr] (x-x(t)) \bigr| &\lesssim \eps^{-1}\delta^{-1} \log^6(\tfrac1\eps) \\
\bigl|  (x-x(t))^T\bigl[ (\Sigma +i(t-t_c))^{-2} - \Sigma^{-2} \bigr] (x-x(t)) \bigr| &\lesssim \eps^{-\frac12}\delta^{-\frac32} \log^8(\tfrac1\eps)  \\
\biggl| \frac{|x-2\xi t|^2}{4(\sigma^2+it)} - \frac{|x-2\xi t|^2}{4(\sigma^2+it_c)} \biggr| &\lesssim \eps^{\frac12} \delta^{-\frac12} \log^3(\tfrac1\eps) \\
\bigl| (x-x(t))^T\bigl[ (\Sigma +i(t-t_c))^{-1} - \Sigma^{-1} \bigr] (x-x(t)) \bigr| &\lesssim \eps^{\frac12} \delta^{-\frac12} \log^7(\tfrac1\eps),
\end{align*}
which follow from \eqref{sig1}, \eqref{sig}, and \eqref{E:bounds1}.  Combining these estimates with the fact that $z\mapsto e^{z}$ is $1$-Lipschitz
on the region $\Re z <0$ yields
\begin{align*}
&|\partial_t F_3(t,x)|\lesssim\eps^{-1}\delta^{-1} \log^{10}(\tfrac1\eps) \\
&\quad+\biggl| \frac{\xi(x-2\xi t)}{\sigma^2+it_c}+\frac{i|x-2\xi t|^2}{4(\sigma^2+it_c)^2} - \eta^T \Sigma^{-1}(x-x(t)) - \tfrac i4(x-x(t))^T\Sigma^{-2}(x-x(t)) \biggr| \\
&\quad+ \eps^{-\frac32}\delta^{-\frac12} \logeps \biggl|\frac{|x-2\xi t|^2}{4(\sigma^2+it_c)} + i(x-x_c)(\eta-\xi)-\tfrac14(x-x(t))^T \Sigma^{-1}(x-x(t)) \biggr|.
\end{align*}
As $|\partial_t A| \lesssim \sigma^{-3/2} |\partial_t F|$, the first term on the right-hand side is consistent with \eqref{laplace A}.  Thus to complete our analysis of $ |\partial_t F|$, it remains only to bound the second and third lines in the display above.  Recalling \eqref{E:Sigma defn}, the fact that $\eta$ belongs to the kernel of the symmetric matrix $B$, and $x(t)=x_c+2\eta(t-t_c)$, we can simplify these expressions considerably.  First, we have
\begin{align*}
\Bigl| \tfrac{\xi(x-2\xi t)}{\sigma^2+it_c} & +\tfrac{i|x-2\xi t|^2}{4(\sigma^2+it_c)^2} - \eta^T \Sigma^{-1}(x-x(t)) - \tfrac i4(x-x(t))^T\Sigma^{-2}(x-x(t)) \Bigr| \\
={}& \Bigl|\tfrac{(\xi-\eta)(x-x_c)}{\sigma^2+it_c} - i \tfrac{(t-t_c)(\xi-\eta)(x-x_c)}{(\sigma^2+it_c)^2} +  \tfrac{(x-x_c)^T B (x-x_c)}{2(\sigma^2+it_c)}
	+ i\tfrac{(x-x_c)^T B^2(x-x_c)}{4} \Bigr| \\
\lesssim{}& \eps^{-1} \log^5(\tfrac1\eps),
\end{align*}
where we used \eqref{E:bounds1}, \eqref{E:bounds3}, and \eqref{B norm} to obtain the inequality.  This shows that the second line in the estimate on $\partial_t F_3$ is acceptable for \eqref{laplace A}.

For the last line of our estimate on $\partial_t F_3$ above, we use the same tools to obtain
\begin{align*}
\eps^{-\frac32}\delta^{-\frac12} &\logeps\Bigl|\tfrac{|x-2\xi t|^2}{4(\sigma^2+it_c)} + i(x-x_c)(\eta-\xi)-\tfrac14(x-x(t))^T \Sigma^{-1}(x-x(t)) \Bigr| \\
={}&  \eps^{-\frac32}\delta^{-\frac12} \logeps\Bigl| \tfrac{(t-t_c)(\xi-\eta)(x-x_c)}{\sigma^2+it_c} + i(x-x_c)(\xi-\eta) + \tfrac{i}4 (x-x_c)^T B (x-x_c) \Bigr| \\
\lesssim {}& \eps^{-1} \log^7(\tfrac1\eps) + \eps^{-\frac32}\delta^{-\frac12} \logeps\Bigl| (x-x_c)(\xi-\eta) + \tfrac14 (x-x_c)^T B (x-x_c) \Bigr|.
\end{align*}
The first summand is acceptable.  To bound the second summand, we need to delve deeper.

Using the orthonormal frame introduced earlier, we write
\begin{align}\label{y1}
y_1:=(x-x_c)\cdot \vec{\tau},\quad y_2:=(x-x_c) \cdot \vec{\gamma},\qtq{and} y_3:=(x-x_c) \cdot \vec{\nu}.
\end{align}
Then
\begin{align}\label{xi3y3}
(x-x_c)\cdot(\xi-\eta)=2\xi_3(x-x_c)\cdot{\vec{\nu}}=2\xi_3y_3.
\end{align}
For $x\in\partial \Omega$ near $x_c$, we have
\begin{align}\label{y3}
y_3=-\frac {y_1^2}{2R_1}-\frac {y_2^2}{2R_2}+O(|y_1|^3+|y_2|^3).
\end{align}
On the other hand, for any $z\in \R^3$ a direct computation shows that
$$
\frac 14 z^TBz=\frac 1{\xi_3R_1}(\xi_3z_1+\xi_1z_3)^2+\frac1{\xi_3R_2}(\xi_3z_2+\xi_2z_3)^2.
$$
Applying this to
\begin{align*}
z=x-x(t)&=x-x_c-2\eta(t-t_c)
=\begin{pmatrix}
y_1-2\xi_1(t-t_c)\\
y_2-2\xi_2(t-t_c)\\
2\xi_3(t-t_c)\\
\end{pmatrix} + \begin{pmatrix} 0\\ 0\\ y_3\\ \end{pmatrix}
\end{align*}
and noting that $|y_3|\lsm |y_1|^2+|y_2|^2\lesssim \sigma^2\log^2(\frac 1\eps)$, we get
\begin{align*}
\frac 14(x-x(t))^T B(x-x(t))
&=\frac{(y_1\xi_3+y_3\xi_1)^2}{\xi_3R_1}+\frac{(y_2\xi_3+y_3\xi_2)^2}{\xi_3R_2}\\
&=\xi_3\frac {y_1^2}{R_1}+\xi_3\frac {y_2^2}{R_2}+O\Bigl(\sigma^3\log^3(\tfrac 1{\eps})\cdot \frac{|\xi|^2}{|\xi_3|}\Bigr).
\end{align*}
Combining this with \eqref{xi3y3}, \eqref{y3}, and Lemma~\ref{L:xc}, we deduce
\begin{align}\label{B cancel}
\Bigl|(x-x_c)\cdot(\xi-\eta)+\tfrac14 (x-x(t))^T B(x-x(t))\Bigr|
	&\lesssim\sigma^3\log^3(\tfrac 1\eps)\frac{|\xi|^2}{|\xi_3|} \notag \\
&\lesssim \eps^{\frac 12}\delta^{\frac 32}\log^7(\tfrac 1\eps).
\end{align}

This is the missing piece in our estimate of $|\partial_t F_3|$.  Putting everything together yields
$$
|\partial_t A(t,x)| \lesssim \sigma^{-\frac32} |\partial_t F(t,x)| \lesssim \sigma^{-\frac32} \eps^{-1}\delta^{-1} \log^{10}(\tfrac1\eps)
	\lesssim \eps^{-\frac74}\delta^{-\frac74} \log^{9}(\tfrac1\eps),
$$
which proves the first half of \eqref{A}.

This bound on the time derivative of $A$ allows us to deduce \eqref{A} by just checking its validity at $t=t_c$.  Note that
both $F_1$ and $F_2$ vanish at this point and so, by the fundamental theorem of calculus, we have
\begin{align*}
|A(t,x)| &\lesssim |t-t_c| \eps^{-\frac74}\delta^{-\frac74} \log^{9}(\tfrac1\eps) + \sigma^{-\frac32} |F_3(t_c,x)| \\
&\lesssim \eps^{-\frac14}\delta^{-\frac54} \log^{12}(\tfrac1\eps)  + \sigma^{-\frac32} \bigl| (x-x_c)(\eta-\xi) + \tfrac14(x-x_c)^T B (x-x_c) \bigr|.
\end{align*}
Combining this with \eqref{B cancel} yields \eqref{A}.

It remains to estimate the spatial derivatives of $A$.  Notice that this corresponds to derivatives of $F$ in directions parallel to $\partial\Omega$.
To compute these, we need to determine the unit normal $\vec\nu_x$ to $\partial\Omega$ at a point $x\in\partial\Omega$;  indeed, the projection matrix onto the tangent space at $x$ is given by $\Id - \vec\nu_x^{\vphantom{T}} \vec\nu_x^T$.   Writing $y=x-x_c$ as in \eqref{y1} and \eqref{y3}, we have
\begin{equation}\label{nu_x}
\vec\nu_x = \begin{pmatrix} y_1/R_1 \\  y_2/R_2 \\ 1 \end{pmatrix} + |y|^2 \vec \psi(y),
\end{equation}
where $\vec \psi$ is a smooth function with all derivatives bounded.

However, the Laplacian of $A$ does not involve only the tangential derivatives of $F$; due to the curvature of the obstacle, the normal derivative of $F$ also enters:
$$
|\Delta A| \lesssim \sigma^{-\frac32} \Bigl\{ |\nabla F|_{\R^3} + |\partial^2 F|_{T_x\partial\Omega} \Bigr\}.
$$
Here $\partial^2 F$ denotes the full matrix of second derivatives of $F$, while the subscript $T_x\partial\Omega$ indicates that only the tangential components are considered; no subscript or $\R^3$ will be used to indicate that all components are considered.  In this way, verifying \eqref{deriv A} and the remaining part of \eqref{laplace A} reduces to proving
\begin{equation}\label{E:lap A needs}
\begin{gathered}
|\nabla F|_{T_x\partial\Omega} \lesssim \delta^{-1} \log^{13}(\tfrac1\eps)
	\qtq{and}  |\nabla F| + |\partial^2 F|_{T_x\partial\Omega} \lesssim \eps^{-1}\delta^{-1} \log^{10}(\tfrac1\eps).
\end{gathered}
\end{equation}

Again we decompose $F$ into the three parts $F_1$, $F_2$, and $F_3$.  The first two are easy to estimate; indeed, we do not even need to
consider normal and tangential components separately:
\begin{align*}
|\nabla F_1(t,x)| \lesssim \frac{|x-2\xi t|}{\sigma^2} \frac{|t-t_c|}{\sigma^2} \lesssim \delta^{-1} \llogeps
\end{align*}
and similarly, using \eqref{sig2},
\begin{align*}
|\nabla F_2(t,x)| \lesssim \frac{|x-2\xi t|}{\sigma^2} \eps^{\frac12}\delta^{-\frac12} \log^3(\tfrac1\eps) \lesssim \delta^{-1} \log^3(\tfrac1\eps).
\end{align*}
These are both consistent with the needs of \eqref{E:lap A needs}.

We can bound the second derivatives of $F_1$ and $F_2$ in a similar manner:
\begin{align*}
|\partial^2 F_1(t,x)| &\lesssim \Bigl[ \sigma^{-2} + \frac{|x-2\xi t|^2}{\sigma^4} \Bigr] \frac{|t-t_c|}{\sigma^2} \lesssim \eps^{-\frac12}\delta^{-\frac32} \llogeps  \\
|\partial^2 F_2(t,x)| &\lesssim \Bigl[ \sigma^{-2} + \frac{|x-2\xi t|^2}{\sigma^4} \Bigr] \eps^{\frac12}\delta^{-\frac12} \log^3(\tfrac1\eps)
	\lesssim \eps^{-\frac12}\delta^{-\frac32} \log^3(\tfrac1\eps).
\end{align*}
Both are acceptable for \eqref{E:lap A needs}.

This leaves us to estimate the derivatives of $F_3$; now it will be important to consider tangential derivatives separately.  We have
\begin{align*}
|\nabla F_3(t,x)|_{T_x\partial\Omega} &\lesssim \frac{|x-2\xi t|}{\sigma^2} |F_3(t,x)| \\
	& \quad  + \biggl| \frac{x-2\xi t}{2(\sigma^2+it)} - i(\xi-\eta) - \tfrac12 (\Sigma +i(t-t_c))^{-1}(x-x(t)) \biggr|_{T_x\partial\Omega}.
\end{align*}
From the proof of \eqref{A}, we know that $|F_3| \lesssim \eps^{1/2}\delta^{-1/2}\log^{13}(\frac1\eps)$.  To estimate the second line we begin by
simplifying it.  Using \eqref{E:bounds1} and \eqref{sig1}, we have
\begin{align}
|\nabla F_3(t,x)|_{T_x\partial\Omega} &\lesssim \frac{\sigma\log(\frac1\eps)}{\sigma^2} \eps^{\frac12}\delta^{-\frac12}\log^{13}(\tfrac1\eps) \notag\\
&\quad + \frac{|t-t_c|}{\sigma^{4}} |x-2\xi t| + \|(\Sigma +i(t-t_c))^{-1}-\Sigma^{-1}\|_{HS}|x-x(t)| \notag\\
& \quad + \biggl| \frac{x-2\xi t}{2(\sigma^2+it_c)}-i(\xi-\eta) - \tfrac12 \Sigma^{-1}(x-x(t)) \biggr|_{T_x\partial\Omega} \notag\\
& \lesssim \delta^{-1}\log^{13}(\tfrac1\eps)+ \biggl| \frac{(\xi-\eta)(t-t_c)}{\sigma^2+it_c} \biggr|_{T_x\partial\Omega} + \biggl| \xi-\eta + \tfrac12 B (x-x_c) \biggr|_{T_x\partial\Omega}. \label{nab F3}
\end{align}
Thus far, we have not used the restriction to tangential directions.  Thus, using \eqref{B norm} we may pause to deduce
$$
|\nabla F_3(t,x)|_{\R^3} \lesssim \delta^{-1}\log^{13}(\tfrac1\eps) + \sigma^{-1}\logeps + |\xi| \lesssim \eps^{-1}\llogeps,
$$
which is consistent with \eqref{E:lap A needs}.

We now return to \eqref{nab F3}.  To estimate the last two summands we write $x-x_c=y$ and use \eqref{nu_x} to obtain
\begin{equation*}
\bigl(\Id - \vec\nu_x^{\vphantom{T}} \vec\nu_x^T\bigr) (\xi-\eta) = - \begin{pmatrix} 2\xi_3y_1/R_1\\ 2\xi_3y_2/R_2 \\ 0 \end{pmatrix} + O(|\xi| \,  |y|^2).
\end{equation*}
Similarly,
\begin{equation*}
\tfrac12 \bigl(\Id - \vec\nu_x^{\vphantom{T}} \vec\nu_x^T\bigr) B (x-x_c)= \begin{pmatrix} 2\xi_3y_1/R_1\\ 2\xi_3y_2/R_2 \\ 0 \end{pmatrix} + O(\|B\|_{\max} |y|^2).
\end{equation*}
Using \eqref{B norm}, this allows us to deduce that
\begin{equation}\label{E:T cancel}
\bigl| \xi-\eta + \tfrac12 B (x-x_c) \bigr|_{T_x\partial\Omega} \lesssim \|B\|_{\max} |y|^2 + |\xi| \,  |y|^2 \lesssim \delta\log^5(\tfrac1\eps).
\end{equation}
Therefore,
$$
|\nabla F_3(t,x)|_{T_x\partial\Omega} \lesssim \delta^{-1}\log^{13}(\tfrac1\eps) + \log^{2}(\tfrac1\eps) + \delta\log^{5}(\tfrac1\eps)
\lesssim \delta^{-1}\log^{13}(\tfrac1\eps).
$$
This is consistent with \eqref{E:lap A needs}, thereby completing the proof of \eqref{deriv A}.

Estimating the second order derivatives of $F_3$ is very messy.  We get
\begin{align*}
\partial_k \partial_l e^{-\frac{|x-2\xi t|^2}{4(\sigma^2+it)}}
&=\biggl[ \frac{-\delta_{kl}}{2(\sigma^2+it)} + \frac{(x-2\xi t)_k(x-2\xi t)_l}{4(\sigma^2+it)^2}\biggr] e^{-\frac{|x-2\xi t|^2}{4(\sigma^2+it)}}  \\
&=\biggl[ \frac{-\delta_{kl}}{2(\sigma^2+it_c)} + \frac{(x-2\xi t)_k(x-2\xi t)_l}{4(\sigma^2+it_c)^2}\biggr] e^{-\frac{|x-2\xi t|^2}{4(\sigma^2+it)}}
	+ O\biggl(\frac{\log\log(\frac1\eps)}{\eps^{\frac12}\delta^{\frac32}}\biggr).
\end{align*}
Proceeding similarly and using \eqref{sig1} and \eqref{sig} yields
\begin{align*}
& \partial_k \partial_l e^{\Phi(t,x)} = \partial_k \partial_l  e^{i(x-x_c)(\eta-\xi)-\frac14(x-x(t))^T(\Sigma +i(t-t_c))^{-1}(x-x(t))} \\
&=\biggl[ -\tfrac12\Sigma^{-1}_{kl} + \Bigl\{ i(\eta-\xi) - \tfrac12\Sigma^{-1}(x-x(t))\Bigr\}_k \Bigl\{ i(\eta-\xi) - \tfrac12\Sigma^{-1}(x-x(t))\Bigr\}_l\biggr]  e^{\Phi(t,x)}\\
&\quad + O\bigl(\eps^{-1}\delta^{-1}\log^6(\tfrac1\eps)\bigr).
\end{align*}

We now combine these formulas, using \eqref{B norm}, \eqref{E:T cancel}, and the definition of $\Sigma^{-1}$ in the process.  This yields
\begin{align*}
|\partial^2 F_3|_{T_x\partial\Omega} &\lesssim \frac{\log^2(\frac1\eps)}{\sigma^2} |F_3| + |B|_{T_x\partial\Omega}
	+ \eps^{-\frac12}\delta^{\frac12}\log^5(\tfrac1\eps) +\eps^{-1}\delta^{-1}\log^6(\tfrac1\eps)\\
&\lesssim \eps^{-\frac12}\delta^{-\frac32}\log^{13}(\tfrac1\eps) + \eps^{-1} \log(\tfrac1\eps) + \eps^{-\frac12}\delta^{\frac12}\log^5(\tfrac1\eps) +\eps^{-1}\delta^{-1}\log^6(\tfrac1\eps)\\
&\lesssim \eps^{-1}\delta^{-1} \log^{6}(\tfrac1\eps).
\end{align*}
This completes the proof of \eqref{E:lap A needs} and so that of \eqref{laplace A}.
\end{proof}

With all these preparations, we are ready to begin estimating the contribution of the wave packets that enter the obstacle.  In view of
Proposition~\ref{P:short times}, it suffices to prove the following

\begin{prop}[The long time contribution of entering rays]\label{P:long times}
We have
\begin{equation}\label{enter}
\Bigl\|\sum_{n\in \mathcal E} c_n^{\eps}\bigl[\propagateomega (1_\Omega\gamma_n)-u_n\bigr]\Bigr\|_{L_{t,x}^{\frac {10}3}([T,\infty)\times\R^3)}=o(1) \qtq{as} \eps\to 0,
\end{equation}
where $T=\frac1{10}\eps\delta[\llogeps]^{-1}$.
\end{prop}

Now fix $n\in \mathcal E$. We denote by $\chi^{u_n}(t)$ a smooth time cutoff such that
$$
\chi^{u_n}(t)=1 \qtq{for} t\in \bigl[0, t_c+2\tfrac{\sigma\log(\frac1\eps)}{|\xi_n|}\bigr] \qtq{and} \chi^{u_n}(t)=0 \qtq{for} t\ge t_c+4\tfrac{\sigma\log(\frac1\eps)}{|\xi_n|}.
$$
Denote by $\chi^{v_n}(t)$ a smooth time cutoff such that
$$
\chi^{v_n}(t)=1 \qtq{for} t\ge t_c-2\tfrac{\sigma\log(\frac1\eps)}{|\xi_n|} \qtq{and} \chi^{v_n}(t)=0 \qtq{for} t\in\bigl[0,t_c- 4\tfrac{\sigma\log(\frac1\eps)}{|\xi_n|}\bigr].
$$
We then define
\begin{align*}
\tilde u_n(t,x) :=\chi^{u_n}(t) u_n(t,x) \qtq{and} \tilde v_n(t,x):=\chi^{v_n}(t)v_n(t,x).
\end{align*}
The cutoff $\chi^{u_n}$ kills $u_n$ shortly after it enters the obstacle; the additional time delay (relative to $t_c$) guarantees that the bulk of the wave packet is deep inside $\Omega^c$ when the truncation occurs.  Note that the cutoff also ensures that $u_n$ does not exit the obstacle.  Analogously,
$\chi^{v_n}$ turns on the reflected wave packet shortly before it leaves the obstacle.

By the triangle inequality,
\begin{align}
\text{LHS}\eqref{enter}
&\lesssim \Bigl\|\sum_{n\in \mathcal E} c_n^{\eps}(u_n-1_\Omega\tilde u_n)\Bigr\|_{L_{t,x}^{\frac{10}3}([T,\infty)\times\R^3)}
+ \Bigl\|\sum_{n\in \mathcal E} c_n^{\eps}\tilde v_n\Bigr\|_{L_{t,x}^{\frac {10}3}([T,\infty)\times\Omega)}\notag\\
&\quad + \Bigl\|\sum_{ n\in \mathcal E} c_n^{\eps}\bigl[\propagateomega(1_\Omega\gamma_n)-(\tilde u_n -\tilde v_n)\bigr]\Bigr\|_{L_{t,x}^{\frac
{10}3}([T,\infty)\times\Omega)}.\label{E:control 462}
\end{align}
We prove that the first two summands are $o(1)$ in Lemmas~\ref{L:small u} and \ref{L:bdfv}, respectively.  Controlling the last summand is a much lengthier enterprise and follows from Lemma~\ref{L:rem}.  This will complete the proof of Proposition~\ref{P:long times}, which together with Propositions~\ref{P:ng}, \ref{P:missing}, and \ref{P:short times} yields Theorem~\ref{T:LF3}.

\begin{lem}\label{L:small u}
We have
\begin{align*}
\Bigl\|\sum_{n\in \mathcal E} c_n^{\eps}(u_n-1_\Omega\tilde u_n)\Bigr\|_{L_{t,x}^{\frac{10}3}([T,\infty)\times\R^3)}=o(1) \qtq{as} \eps\to 0.
\end{align*}
\end{lem}

\begin{proof}
By the triangle inequality and H\"older,
\begin{align}
\Bigl\|\sum_{n\in \mathcal E} c_n^\eps(u_n-1_\Omega\tilde u_n)\Bigr\|_{L_{t,x}^{\frac{10}3}}
&\lsm\Bigl\|\sum_{n\in \mathcal S}|c_n^\eps||u_n|\Bigr\|_{L_{t,x}^{\frac{10}3}}\label{interp}\notag\\
&\lsm \Bigl\|\sum_{n\in \mathcal S}|c_n^\eps| |u_n|\Bigr\|_{L_t^\infty L_x^2}^{\frac15}
	\Bigl\|\sum_{n\in \mathcal S}|c_n^\eps||u_n|\Bigr\|_{L_t^{\frac83}L_x^4}^{\frac 45},
\end{align}
where all spacetime norms are over $[T,\infty)\times\R^3$.  To estimate the first factor on the right-hand side of \eqref{interp}, we use
\begin{align*}
|x-2t\xi_n|^2+|x-2t\xi_m|^2=2|x-(\xi_n+\xi_m)t|^2+2t^2|\xi_n-\xi_m|^2
\end{align*}
together with \eqref{bdforc} to get
\begin{align*}
\Bigl\|\sum_{n\in \mathcal S}& |c_n^\eps||u_n|\Bigr\|_{L_x^\infty L_x^2([T,\infty)\times\R^3)}^2\\
&\lsm \biggl\|\sum_{n,m\in \mathcal S}\frac{(\sigma\eps)^3}{L^6}\biggl(\frac{\sigma^2}{\sigma^4+t^2}\biggr)^{\frac32}
	\int_{\R^3} e^{-\frac{\sigma^2|x-2t\xi_n|^2}{4(\sigma^4+t^2)}-\frac{\sigma^2|x-2t\xi_m|^2}{4(\sigma^4+t^2)}}\,dx\biggr\|_{L_t^{\infty}([T,\infty))}\\
&\lsm \biggl\|\sum_{n, m\in \mathcal S}\frac{(\sigma\eps)^3}{L^6}\exp\Bigl\{-\frac{\sigma^2 t^2|\xi_n-\xi_m|^2}{2(\sigma^4+t^2)}\Bigr\}\biggr\|_{L_t^{\infty}([T,\infty))}.
\end{align*}
For $t\geq T$ we have
\begin{align*}
\frac{\sigma^2 t^2|\xi_n-\xi_m|^2}{2(\sigma^4+t^2)}
&\ge\frac{\sigma^2|n-m|^2 T^2}{2L^2(\sigma^4+T^2)}\ge\frac{|n-m|^2T^2}{4\sigma^4[\llogeps]^2}\ge\frac{|n-m|^2}{\log^5(\tfrac 1\eps)},
\end{align*}
and so derive
\begin{align*}
\Bigl\|\sum_{n\in \mathcal S}|c_n^\eps||u_n|\Bigr\|_{L_t^\infty L_x^2([T,\infty)\times\R^3)}^2
&\lsm \sum_{n, m\in \mathcal S}\frac{(\sigma\eps)^3}{L^6}\exp\Bigl\{-\frac{|n-m|^2}{\log^5(\tfrac 1\eps)}\Bigr\}\\
&\lsm \sum_{n\in \mathcal S}\frac{(\sigma\eps)^3}{L^6}\log^{\frac{15}2}(\tfrac 1\eps)\\
&\lsm\frac{(\sigma\eps)^3}{L^6}\log^{\frac {15}2}(\tfrac1\eps)\cdot\biggl(\frac L\eps\llogeps\biggr)^3\\
&\lsm \log^{\frac{15}2}(\tfrac 1\eps).
\end{align*}

We now turn to the second factor on the right-hand side of \eqref{interp}.  As
\begin{align*}
\sum_{j=1}^4\frac 14 |x-2\xi_{n_j} t|^2&= \biggl|x-\frac {\sum_{j=1}^4\xi_{n_j}}2 t\biggr|^2+\frac {t^2}4\sum_{j<k}|\xi_{n_j} -\xi_{n_k}|^2,
\end{align*}
we have
\begin{align}\label{129}
\Bigl\|\sum_{n\in \mathcal S}&|c_n^{\eps}| |u_n|\Bigr\|_{L_x^4(\R^3)}^4\notag\\
&\lesssim \sum_{n_1,\cdots,n_4\in \mathcal S}\bigl|c_{n_1}^{\eps}c_{n_2}^{\eps} c_{n_3}^{\eps}c_{n_4}^\eps\bigr|
\biggl(\frac {\sigma^2}{\sigma^4+t^2}\biggr)^3\int_{\R^3} \exp\Bigl\{-\sum_{j=1}^4 \frac{\sigma^2|x-2\xi_{n_j} t|^2}{4(\sigma^4+t^2)}\Bigr\}\,dx\notag\\
&\lsm \sum_{n_1,\cdots,n_4\in \mathcal S} \frac{(\sigma\eps)^6}{L^{12}}\biggl(\frac{\sigma^2}{\sigma^4+t^2}\biggr)^{\frac 32} \exp\Bigl\{-\frac
{\sigma^2t^2}{4(\sigma^4+t^2)}\sum_{j<k}|\xi_{n_j}-\xi_{n_k}|^2\Bigr\}.
\end{align}

To estimate the sum in \eqref{129} we divide it into two parts.  Let $N:=\log^3(\frac1{\eps}).$

\textbf{Part 1:} $|n_j-n_k|\ge N$ for some $1\le j\neq k\le 4$.  We estimate the contribution of the summands conforming to this case to LHS\eqref{129} by
\begin{align*}
\frac{(\sigma\eps)^6}{L^{12}}\biggl(\frac L{\eps}\log\log(\tfrac 1{\eps})\biggr)^{12}& \biggl(\frac {\sigma^2}{\sigma^4+t^2}\biggr)^{\frac 32}
	\exp\Bigl\{-\frac {\sigma^2N^2t^2}{4L^2(\sigma^4+t^2)} \Bigr\}\\
&\lsm t^{-3} \sigma^9\eps^{-6}[\llogeps]^{12}\exp\Bigl\{-\frac{\sigma^2N^2t^2}{4L^2(\sigma^4+t^2)}\Bigr\}.
\end{align*}
For $T\le t\le \sigma^2$, we estimate
\begin{align*}
\exp\Bigl\{-\frac{\sigma^2N^2t^2}{4L^2(\sigma^4+t^2)}\Bigr\}\le\exp\Bigl\{-\frac{T^2N^2}{8\sigma^2L^2}\Bigr\}\le\eps^{100}
\end{align*}
while for $t\ge \sigma^2$,
\begin{align*}
\exp\Bigl\{-\frac{\sigma^2N^2t^2}{4L^2(\sigma^4+t^2)}\Bigr\}\le\exp\Bigl\{-\frac{\sigma^2N^2}{8L^2}\Bigr\}\le \eps^{100}.
\end{align*}
Thus the contribution of Part 1 is $O( \eps^{80}t^{-3})$.

\textbf{Part 2:} $|n_j-n_k|\le N$ for all $1\le j\neq k\le 4$.  We estimate the contribution of the summands conforming to this case to LHS\eqref{129} by
\begin{align*}
\frac {(\sigma\eps)^6}{L^{12}}\biggl(\frac {\sigma^2}{\sigma^4+t^2}\biggr)^{\frac 32}N^9 \biggl(\frac L{\eps}\log\log(\tfrac 1{\eps})\biggr)^3
\lesssim \frac{\sigma^9N^9\eps^3}{L^9t^3}[\llogeps]^3\lsm \frac{\eps^3\log^{27}(\tfrac1\eps)}{[\llogeps]^6} t^{-3}.
\end{align*}

Collecting the two parts and integrating in time, we obtain
\begin{align*}
\Bigl\|\sum_{n\in \mathcal S}|c_n^\eps||u_n|\Bigr\|_{L_t^{\frac83}L_x^4([T,\infty)\times\R^3)}
\lesssim \frac{\eps^{\frac34}\log^{\frac{27}4}(\tfrac1\eps)}{[\llogeps]^{\frac 32}}\cdot T^{-\frac 38}
\lsm \eps^{\frac38}\delta^{-\frac38}\log^8(\tfrac 1\eps).
\end{align*}

Putting everything together and invoking \eqref{interp} we get
$$
\Bigl\|\sum_{n\in \mathcal E} c_n^\eps(u_n-1_\Omega\tilde u_n)\Bigr\|_{L_{t,x}^{\frac{10}3}([T, \infty)\times\R^3)}
\lesssim \eps^{\frac3{10}}\delta^{-\frac3{10}}\log(\tfrac1\eps)^{\frac34+\frac{32}5}=o(1) \qtq{as}\eps\to0.
$$
This completes the proof of the lemma.
\end{proof}

\begin{lem}\label{L:bdfv}
We have
$$
\Bigl\|\sum_{n\in \mathcal E} c_n^{\eps} \tilde v_n\Bigr\|_{L_{t,x}^{\frac {10}3}([T,\infty)\times\R^3)}=o(1) \qtq{as}\eps\to0.
$$
\end{lem}

\begin{proof}
By H\"older's inequality,
\begin{align}\label{E:interp}
\Bigl\|\sum_{n\in \mathcal E} c_n^{\eps} \tilde v_n\Bigr\|_{L_{t,x}^{\frac {10}3}}
\lsm \Bigl\|\sum_{n\in \mathcal E}c_n^\eps \tilde v_n\Bigr\|_{L_t^\infty L_x^2}^{\frac15}
	\Bigl\|\sum_{n\in \mathcal E}c_n^\eps \tilde v_n\Bigr\|_{L_t^{\frac83}L_x^4}^{\frac 45},
\end{align}
where all spacetime norms are over $[T,\infty)\times\R^3$.

First we note that from \eqref{sig41} and \eqref{sig3}, we can bound
\begin{align}\label{bdfv}
|v_n(t,x)|
&\lsm \log^{\frac52}(\tfrac1{\eps})\biggl(\frac{\sigma^2}{\sigma^4+t^2}\biggr)^{\frac34}\exp\Bigl\{-\frac{\sigma^2|x-x_n(t)|^2}{4[\llogeps]^{25}[\sigma^4+t^2\log^4(\frac1\eps)]}\Bigr\}.
\end{align}
Using this bound, we estimate
\begin{align*}
&\int_{\R^3} \! |\tilde v_{n_1}||\tilde v_{n_2}| \,dx
\lesssim \log^5(\tfrac1{\eps})\biggl(\frac{\sigma^2}{\sigma^4+t^2}\biggr)^{\frac 32} \!\!\int_{\R^3}\!\! \exp\Bigl\{-\frac{\sigma^2[|x-x_1(t)|^2+|x-x_2(t)|^2]}{4[\log\log(\frac 1{\eps})]^{25}[\sigma^4+t^2\log^4(\frac 1{\eps})]} \Bigr\} \,dx\\
&\lesssim\log^5(\tfrac1{\eps})\biggl[\frac {[\log\log(\frac1{\eps})]^{25}[\sigma^4+t^2\log^4(\frac1{\eps})]}{\sigma^4+t^2}\biggr]^{\frac 32} \exp\Bigl\{-\frac{\sigma^2|x_1(t)-x_2(t)|^2}{8[\log\log(\frac1{\eps})]^{25}[\sigma^4+t^2\log^4(\frac 1{\eps})]}\Bigr\}\\
&\lesssim \log^{12}(\tfrac 1{\eps})\exp\Bigl\{-\frac{\sigma^2|x_1(t)-x_2(t)|^2}{8[\log\log(\frac1{\eps})]^{25}[\sigma^4+t^2\log^4(\frac 1{\eps})]}\Bigr\},
\end{align*}
where $x_j(t)$ denotes the trajectory of $v_{n_j}$, that is,
$$
x_j(t):=2\xi^{(j)} t_c^{(j)}+2\eta^{(j)} (t-t_c^{(j)})
$$
with $\xi^{(j)}:=\xi_{n_j}$, $\eta^{(j)}:=\eta_{n_j}$, and $t_c^{(j)}$ representing the corresponding collision times.

Therefore,
\begin{align}\label{E:tilde v}
\Bigl\|\sum_{n\in\mathcal  E} c_n^{\eps}\tilde v_n\Bigr\|_{L_t^{\infty}L_x^2}^2
&\lesssim  \sup_t \sum_{ n_1, n_2\in \mathcal E}|c_{n_1}^{\eps}||c_{n_2}^{\eps}|\log^{12}(\tfrac 1{\eps}) e^{-\frac {\sigma^2|x_1(t)-x_2(t)|^2}{8[\log\log(\frac1{\eps})]^{25}[\sigma^4+t^2\log^4(\frac 1{\eps})]}},
\end{align}
where the supremum in $t$ is taken over the region
\begin{align*}
t\ge  \max\biggl\{t_c^{(1)}-4\frac {\sigma \log(\frac1{\eps})}{|\xi^{(1)}|},\ t_c^{(2)}-4\frac {\sigma\log(\frac1{\eps})}{|\xi^{(2)}|}\biggr\}.
\end{align*}

Next we show that for $|n_1-n_2|\geq \log^4(\frac1\eps)$ and all such $t$,
\begin{equation}\label{difray}
|x_1(t)-x_2(t)|\ge |\xi^{(1)}-\xi^{(2)}|t.
\end{equation}
We discuss two cases.  When $t\ge \max\{t_c^{(1)},t_c^{(2)}\}$, this follows immediately from Lemma~\ref{L:diverging rays}.  It remains to prove \eqref{difray} for
$$
\max\biggl\{t_c^{(1)}-4\frac {\sigma\log(\frac1{\eps})}{|\xi^{(1)}|},\ t_c^{(2)}-4\frac {\sigma\log(\frac1{\eps})}{|\xi^{(2)}|}\biggr\}\le t\le \max\bigl\{t_c^{(1)}, \ t_c^{(2)}\bigr\}.
$$
Without loss of generality, we may assume $t_c^{(1)}\ge t_c^{(2)}$.  Using Lemmas~\ref{L:xc} and~\ref{L:diverging rays} and the fact that
$|n_1-n_2|\ge \log^4(\frac 1\eps)$, we estimate
\begin{align*}
|x_1(t)-x_2(t)|&\ge |x_1(t_c^{(1)})-x_2(t_c^{(1)})|-|x_1(t)-x_1(t_c^{(1)})|-|x_2(t)-x_2(t_c^{(1)})|\\
&\ge 2|\xi^{(1)}-\xi^{(2)}|t_c^{(1)}-2|\xi^{(1)}||t-t_c^{(1)}|-2|\xi^{(2)}||t-t_c^{(1)}|\\
&\ge 2\frac {|n_1-n_2|}L t_c^{(1)}-8\sigma \log(\tfrac 1{\eps})-
8\frac {|\xi^{(2)}|\sigma\logeps}{|\xi^{(1)}|}\\
&\ge 2\frac {|n_1-n_2|}{L}t_c^{(1)} -16\sigma\log(\tfrac 1{\eps})[\log\log(\tfrac 1{\eps})]^2\\
&\ge\frac {|n_1-n_2|}{L} t=|\xi^{(1)}-\xi^{(2)}| t.
\end{align*}
This completes the verification of \eqref{difray}.

Using \eqref{bdforc} and \eqref{difray}, \eqref{E:tilde v} implies
\begin{align*}
\Bigl\|\sum_{n\in \mathcal E} c_n^{\eps}\tilde v_n\Bigr\|_{L_x^2}^2
&\lesssim \sum_{|n_1-n_2|\ge \log^4(\frac1{\eps}),\ n_i\in \mathcal E}\frac {(\sigma\eps)^3}{L^6}\log^{12}(\tfrac1{\eps}) e^{-\frac{\sigma^2|n_1-n_2|^2t^2}{8L^2[\log\log(\frac 1{\eps})]^{25}[\sigma^4+t^2\log^4(\frac 1{\eps})]}}\\
&\quad +\sum_{|n_1-n_2|\le\log^4(\frac 1{\eps}),\ n_i\in \mathcal E}\frac{(\sigma\eps)^3}{L^6}\log^{12}(\tfrac 1{\eps})\\
&\lesssim \frac {(\sigma\eps)^3}{L^6}\log^{12}(\tfrac 1{\eps})\biggl(\frac {L\log\log(\tfrac 1{\eps})}{\eps}\biggr)^3 \biggl(\frac{[\log\log(\tfrac 1{\eps})]^{27}[\sigma^4+t^2\log^4(\frac 1{\eps})]}{t^2}\biggr)^{\frac 32}\\
&\quad+\frac{(\sigma\eps)^3}{L^6}\log^{24}(\tfrac 1{\eps})\biggl(\frac {L\log\log(\frac 1{\eps})}{\eps}\biggr)^3\\
&\lesssim \log^{13}(\tfrac1{\eps})\biggl(\frac{\sigma^6}{t^3}+\log^6(\tfrac 1\eps)\biggr) + \log^{25}(\tfrac 1{\eps}).
\end{align*}
Thus,
\begin{align}\label{E:536}
\Bigl\|\sum_{n\in \mathcal E} c_n^{\eps}\tilde v_n\Bigr\|_{L_t^{\infty}L_x^2([T,\infty)\times\R^3)}\lesssim \log^{\frac{25}2}(\tfrac1{\eps}).
\end{align}

We now turn to estimating the second factor on the right-hand side of \eqref{E:interp}.  Combining \eqref{bdfv} with
\begin{align*}
\frac 14\sum_{j=1}^4|x-x_j(t)|^2=\biggl|x-\frac 14\sum_{j=1}^4 x_j(t)\biggr|^2+\frac 1{16}\sum_{j<l}|x_j(t)-x_l(t)|^2,
\end{align*}
we get
\begin{align*}
&\int_{\R^3}|\tilde v_{n_1}||\tilde v_{n_2}||\tilde v_{n_3}||\tilde v_{n_4}|\,dx\\
&\lesssim \log^{10}(\tfrac1{\eps})\biggl(\frac{\sigma^2}{\sigma^4+t^2}\biggr)^3\!\!\int_{\R^3}\!\!\exp\Bigl\{ -\frac {\sigma^2}{4[\log\log(\frac1{\eps})]^{25}[\sigma^4+t^2\log^4(\frac1{\eps})]}\sum_{j=1}^4|x-x_j(t)|^2\Bigr\}\,dx\\
&\lesssim \log^{10}(\tfrac1{\eps})\biggl(\frac {\sigma^2}{\sigma^4+t^2}\biggr)^3\biggl(\frac{[\log\log(\frac 1{\eps})]^{25}[\sigma^4+t^2\log^4(\frac 1{\eps})]}{\sigma^2}\biggr)^{\frac 32}
e^{-\frac{\sigma^2\sum_{j<l}|x_j(t)-x_l(t)|^2}{16[\log\log(\frac 1{\eps})]^{25}[\sigma^4+t^2\log^4(\frac 1{\eps})]}}\\
&\lesssim\sigma^3t^{-3}\log^{17}(\tfrac 1{\eps})e^{-\frac{\sigma^2\sum_{j<l}|x_j(t)-x_l(t)|^2}{16[\log\log(\frac1{\eps})]^{25}[\sigma^4+t^2\log^4(\frac 1{\eps})]}}.
\end{align*}
Combining this with \eqref{bdforc}, we obtain
\begin{align}\label{359}
\Bigl\|\sum_{n\in \mathcal E} &c_n^{\eps} \tilde v_n(t)\Bigr\|_{L_x^4}^4
\lesssim \!\sum_{n_1,\ldots,n_4\in \mathcal E}\! \frac{(\sigma\eps)^6}{L^{12}}\sigma^3t^{-3}\log^{17}(\tfrac1{\eps}) e^{-\frac{\sigma^2\sum_{j<l}|x_j(t)-x_l(t)|^2}{16[\log\log(\frac1{\eps})]^{25}[\sigma^4+t^2\log^4(\frac 1{\eps})]}}.
\end{align}

To estimate the sum above we break it into two parts.  Let $N:= \log^4(\frac 1{\eps})$.

\textbf{Part 1:} $|n_j-n_k|\ge N$ for some $1\le j\neq k\le 4$.  By \eqref{difray}, we have
$$
|x_j(t)-x_k(t)|\ge \frac {|n_j-n_k|}{L} t \qtq{for all} t\in \supp \prod_{l=1}^4 \tilde v_{n_l}.
$$
As $t\geq T$, we estimate the contribution of the summands conforming to this case to LHS\eqref{359} by
\begin{align*}
\frac {(\sigma\eps)^6}{L^{12}}\sigma^3t^{-3}& \log^{17}(\tfrac1{\eps})\biggl[\frac {L\log\log(\frac 1{\eps})}{\eps}\biggr]^{12}
\exp\Bigl\{-\frac{\sigma^2t^2N^2}{16L^2[\log\log(\frac1{\eps})]^{25}[\sigma^4+t^2\log^4(\frac 1{\eps})]}\Bigr\}\\
&\lesssim \frac {\sigma^9}{\eps^6}t^{-3}\log^{18}(\tfrac 1{\eps})\exp\Bigl\{-\frac{N^2}{\log^5(\frac 1{\eps})}\Bigr\}
\le\eps^{100}t^{-3}.
\end{align*}

\textbf{Part 2:} $|n_j-n_k|\le N$ for all $1\le j<k \le 4$.  We estimate the contribution of the summands conforming to this case to LHS\eqref{359} by
\begin{align*}
\frac {(\sigma \eps)^6}{L^{12}}&\sigma^3 t^{-3} \log^{17}(\tfrac1{\eps}) N^9\biggl(\frac {L\log\log(\frac 1{\eps})}{\eps}\biggr)^3
\lesssim \biggl(\frac {\sigma N}L\biggr)^9\eps^3 t^{-3}\log^{18}(\tfrac1{\eps})\le \eps^3 t^{-3}\log^{56}(\tfrac 1{\eps}).
\end{align*}

Combining the estimates from the two cases, we obtain
\begin{align*}
\Bigl\|\sum_{n\in \mathcal E} c_n^\eps \tilde v_n\Bigr\|_{L_t^{\frac83}L_x^4([T,\infty)\times\R^3)}
&\lesssim \bigl[\eps^{25} + \eps^{\frac 34}\log^{14}(\tfrac 1{\eps})\bigr] T^{-\frac 38}
\lesssim \eps^{\frac 38}\delta^{-\frac 38}\log^{15}(\tfrac1\eps).
\end{align*}
Combining this with \eqref{E:interp} and \eqref{E:536} completes the proof of Lemma~\ref{L:bdfv}.
\end{proof}

To complete the proof of Proposition~\ref{P:long times}, we are left to estimate the last term on RHS\eqref{E:control 462}.  For each $n\in \mathcal E$,
we write
\begin{equation}\label{E:259}
\propagateomega (1_\Omega\gamma_n)-(\tilde u_n -\tilde v_n)=-(w_n+r_n),
\end{equation}
where $w_n$ is chosen to agree with $\tilde u_n-\tilde v_n$ on $\partial \Omega$ and $r_n$ is the remainder.  More precisely, let
$\phi\in C_c^{\infty}([0,\infty))$ with $\phi\equiv 1$ on $[0,\frac 12]$ and $\phi\equiv 0$ on $[1,\infty)$.  For each $x\in \Omega$, let $x_*\in\partial\Omega$ denote the point obeying $|x-x_*|=\dist(x,\Omega^c)$.  Now we define
$$
w_n(t,x):=w_n^{(1)}(t,x)+w_n^{(2)}(t,x)+w_n^{(3)}(t,x)
$$
and
\begin{align*}
w_n^{(j)}(t,x):=(\tilde u_n-\tilde v_n)(t,x_*)\phi\bigl(\tfrac{|x-x_*|}{\sigma}\bigr)
\!\begin{cases}
(1-\phi)\bigl(\frac {|x_*-x_c^{(n)}|}{\sigma\log(\frac 1{\eps})}\bigr), & j=1,\\[2mm]
\phi\bigl(\frac {|x_*-x_c^{(n)}|}{\sigma\log(\frac 1{\eps})}\bigr)(1-\phi)\bigl(\frac {|t-t_c^{(n)}||\xi_n|}{2\sigma\log(\frac 1{\eps})}\bigr), &j=2,\\[2mm]
\phi\bigl(\frac {|x_*-x_c^{(n)}|}{\sigma\log(\frac 1{\eps})}\bigr)     \phi\bigl(\frac {|t-t_c^{(n)}||\xi_n|}{2\sigma\log(\frac 1{\eps})}\bigr)e^{i\xi\cdot (x-x_*)},\!\! &j=3.
\end{cases}
\end{align*}
We will estimate $w_n$ by estimating each $w_n^{(j)}$ separately. Note that $w_n^{(3)}$ is the most significant of the three;
spatial oscillation has been introduced into this term to ameliorate the temporal oscillation of $\tilde u_n-\tilde v_n$.  This subtle modification is essential to achieve satisfactory estimates.

To estimate $r_n$, we use \eqref{E:259} to write
$$
0=(i\partial_t +\Delta_\Omega)(\tilde u_n-\tilde v_n-w_n-r_n)=(i\partial_t+\Delta)(\tilde u_n-\tilde v_n-w_n)-(i\partial_t+\Delta_\Omega) r_n,
$$
which implies
$$
(i\partial_t+\Delta_\Omega) r_n=iu_n\partial_t \chi^{u_n} -iv_n\partial_t \chi^{v_n} -(i\partial_t+\Delta)w_n.
$$
Using the Strichartz inequality, we estimate
\begin{align*}
\Bigl\|\sum_{n\in \mathcal E} c_n^{\eps}r_n\Bigr\|_{L_{t,x}^{\frac{10}3}([T,\infty)\times\Omega)}
&\lesssim \Bigl\|\sum_{n\in \mathcal E}c_n^{\eps}\bigl[u_n\partial_t\chi^{u_n}-v_n\partial_t\chi^{v_n}\bigr]\Bigr\|_{L_t^1L_x^2([T,\infty)\times\Omega)}\\
&\quad + \Bigl\|\sum_{n\in \mathcal E}c_n^{\eps}(i\partial_t+\Delta)w_n\Bigr\|_{L_t^1L_x^2([T,\infty)\times\Omega)}.
\end{align*}

Putting everything together, we are thus left to prove the following

\begin{lem}\label{L:rem} As $\eps\to0$, we have
\begin{align}
&\Bigl\|\sum_{n\in \mathcal E} c_n^{\eps} u_n\partial_t \chi^{u_n}\Bigr\|_{L_t^1L_x^2([T,\infty)\times \Omega)}
+\Bigl\|\sum_{n\in \mathcal E} c_n^{\eps} v_n \partial_t \chi^{v_n}\Bigr\|_{L_t^1L_x^2([T,\infty)\times \Omega)}=o(1)\label{rem1}\\
&\Bigl\|\sum_{n\in \mathcal E} c_n^{\eps}w_n^{(j)}\Bigr\|_{L_{t,x}^{\frac{10}3}([T,\infty)\times\Omega)}
+\Bigl\|\sum_{n\in \mathcal E} c_n^{\eps}(i\partial_t+\Delta)w_n^{(j)}\Bigr\|_{L_t^1L_x^2([T,\infty)\times\Omega)}=o(1),\label{rem2}
\end{align}
for each $j=1,2,3$.  As previously, $T=\frac1{10}\eps\delta[\log\log(\frac1\eps)]^{-1}$.
\end{lem}

\begin{proof}
We first prove the estimate \eqref{rem1} for $u_n$. Recall the following bound for $u_n$:
\begin{align*}
|u_n(t,x)|\lesssim \biggl(\frac{\sigma^2}{\sigma^4+t^2}\biggr)^{\frac 34} \exp\Bigl\{-\frac{\sigma^2|x-2\xi_n t|^2}{4(\sigma^4+t^2)}\Bigr\}.
\end{align*}
Also, for $t\in \supp \partial_t \chi^{u_n}=[t_c^{(n)}+\frac{2\sigma\log(\frac 1{\eps})}{|\xi_n|},t_c^{(n)}+\frac{4\sigma\log(\frac 1{\eps})}{|\xi_n|}]$ we have
$t\le \sigma^2$ and, by the definition of $\mathcal E$,
\begin{align*}
\dist(2\xi_n t,\Omega)\gtrsim \frac{|\xi_n||t-t_c^{(n)}|}{[\llogeps]^4}\ge\frac {\sigma\log(\frac1{\eps})}{[\log\log(\frac 1{\eps})]^5}.
\end{align*}
Thus,
\begin{align*}
|\partial_t\chi^{u_n}|^2\int_{\Omega}|u_n(t,x)|^2\,dx
&\lesssim\biggl(\frac {\sigma^2}{\sigma^4+t^2}\biggr)^{\frac 32}|\partial_t\chi^{u_n}|^2\int_{\Omega}\exp\Bigl\{-\frac {\sigma^2|x-2\xi_{n}t|^2}{2(\sigma^4+t^2)}\Bigr\}\,dx\\
&\lesssim \sigma^{-3}\biggl(\frac{|\xi_n|}{\sigma\logeps}\biggr)^2 \int_{|y|\ge\frac{\sigma\log(\frac1{\eps})}{[\log\log(\frac 1{\eps})]^5}}e^{-\frac {|y|^2}{4\sigma^2}}\,dy\\
&\lesssim \eps^{200}.
\end{align*}
Summing in $n$ and using \eqref{bdforc}, we obtain
\begin{align*}
\Bigl\|\sum_{n\in \mathcal E} c_n^{\eps}u_n\partial_t\chi^{u_n}\Bigr\|_ {L_t^1L_x^2([T,\infty)\times\Omega)}
&\lesssim \sum_{n\in \mathcal E} \frac {(\sigma\eps)^{\frac 32}}{L^3}\eps^{100} \frac{\sigma\log(\frac1{\eps})}{|\xi_n|}\leq \eps^{90}.
\end{align*}

The estimate for $v_n$ is similar.  Note that by the definition of $\mathcal E$, for $t\in\supp \partial_t \chi^{v_n}=[t_c^{(n)}-4\frac {\sigma \log(\frac1{\eps})}{|\xi_n|},t_c^{(n)}-2\frac {\sigma\log(\frac1{\eps})}{|\xi_n|}]$ we have
\begin{align}\label{124}
\dist(x_n(t), \Omega)\gtrsim\frac{|\xi_n||t-t_c^{(n)}|}{[\llogeps]^4}\ge \frac {\sigma\log(\frac1{\eps})} {[\log\log(\frac 1{\eps})]^5}
\end{align}
and, by Lemma \ref{L:xc},
\begin{align*}
t\leq t_c^{(n)}\lsm \frac{\delta}{|\xi_n|}[\llogeps]^8 \qtq{and} t^2\log^4(\tfrac1\eps)\le\sigma^4[\llogeps]^{19}.
\end{align*}
Therefore, using \eqref{bdfv} we get
\begin{align*}
|(\partial_t\chi^{v_n})v_n(t,x)|^2
&\lsm \log^5(\tfrac1{\eps})\biggl(\frac{\sigma^2}{\sigma^4+t^2}\biggr)^{\frac32}\exp\biggl\{-\frac{\sigma^2|x-x_n(t)|^2}{2[\llogeps]^{25}[\sigma^4+t^2\log^4(\tfrac
1\eps)]}\biggr\}\cdot\frac{|\xi_n|^2}{[\sigma\logeps]^2}\\
&\lsm \sigma^{-5}|\xi_n|^2\log^{3}(\tfrac1\eps)\exp\biggl\{-\frac{|x-x_n(t)|^2}{4\sigma^2[\llogeps]^{44}}\biggr\}.
\end{align*}
Using \eqref{124} and computing as for $u_n$, we obtain
\begin{align*}
\int_{\Omega} |\partial _t\chi^{v_n} v_n(t,x)|^2 \,dx\lsm \eps^{200}
\end{align*}
and then
$$
\Bigl\|\sum_{n\in \mathcal E} c_n^{\eps}v_n\partial_t\chi^{v_n}\Bigr\|_{L_t^1L_x^2([T,\infty)\times\Omega)} \le \eps^{90}.
$$
This completes the proof of \eqref{rem1}.

We now turn to estimating \eqref{rem2}.  We begin with the contribution from $w_n^{(1)}$.  Using the definitions of $\tilde u_n(t,x)$ and $\tilde v_n(t,x)$, as well as \eqref{sig42}, \eqref{sig3}, and the fact that $\partial_t[\det M(t)]^{-1/2} = -\frac12 [\det M(t)]^{-1/2} \Tr[M(t)^{-1}\partial_t M(t)]$, we estimate
\begin{align}
|w_n^{(1)}(t,x)|&+|(i\partial_t+\Delta)w_n^{(1)}(t,x)|\notag\\
&\lsm \Bigl[\sigma^{-2}+|\xi_n|^2+\frac{|x_*-2\xi_nt|^2}{\sigma^4}\Bigr]| u_n(t,x_*)| \chi_1(t,x)\label{cun}\\
&\quad+\Bigl[\sigma^{-2}+|\xi_n|^2+\frac{\log^{10}{(\frac1\eps)}}{\sigma^4+t^2}|x_*-x_n(t)|^2\Bigr]|v_n(t,x_*)| \chi_2(t,x)\label{cvn},
\end{align}
where $\chi_1(t,x)$ is a cutoff to the spacetime region
\begin{align*}
\biggl\{(t,x)\in [0,\infty) \times\Omega: \, |x-x_*|\le \sigma, \ |x_*-x_c^{(n)}|\geq\tfrac12 \sigma\logeps,\ t\leq t_c^{(n)}+4\frac{\sigma\logeps}{|\xi_n|}\biggr\}
\end{align*}
and $\chi_2(t,x)$ is a cutoff to the spacetime region
\begin{align*}
\biggl\{(t,x)\in [0,\infty) \times\Omega: \, |x-x_*|\le \sigma, \ |x_*-x_c^{(n)}|\geq\tfrac12\sigma\logeps,\ t\ge t_c^{(n)}-4\frac{\sigma\logeps}{|\xi_n|}\biggr\}.
\end{align*}
Note that
\begin{align}\label{E:chi}
\int_{\Omega} \chi_1(t,x) + \chi_2(t,x) \, dx\lesssim \sigma \qtq{for all} t\geq 0.
\end{align}

To estimate the contribution from \eqref{cun}, we note that on the spacetime support of this term we have $t\leq \sigma^2$ and,
by the definition of $\mathcal E$,
\begin{align*}
\frac{|x_*-x_c^{(n)}|}{[\llogeps]^4}&\lsm |x_*-2\xi_n t|\leq |x_*|  + |x_c^{(n)}| + 2|\xi_n(t-t_c^{(n)})|\lesssim 1.
\end{align*}
Thus we can estimate
\begin{align}
\eqref{cun}
&\lsm\bigl[\sigma^{-2}+|\xi_n|^2+\sigma^{-4}\bigr]|u_n(t,x_*)|\chi_1(t,x)\notag\\
&\lsm \eps^{-2}\delta^{-2}[\llogeps]^2 \sigma^{-\frac32}\exp\Bigl\{-\frac{c\sigma^2|x_*-x_c^{(n)}|^2}{\sigma^4[\llogeps]^8}\Bigr\}\chi_1(t,x)\notag\\
&\lsm \eps^{-2}\delta^{-2}\sigma^{-\frac32}[\llogeps]^2\exp\Bigl\{-\frac{\log^2(\tfrac1\eps)}{[\llogeps]^9}\Bigr\}\chi_1(t,x)\notag\\
&\le \eps^{100}\chi_1(t,x).\label{E:cun}
\end{align}

To estimate the contribution from \eqref{cvn}, we discuss long and short times separately.  If $t\geq \delta[\llogeps]^{10}|\xi_n|^{-1}$, then $t\gg t_c^{(n)}$ and so $2 |t-t_c^{(n)}|\geq t$.  Using the definition of $\mathcal E$, we thus obtain
\begin{align*}
|x_*-x_n(t)|\ge \dist(x_n(t), \partial\Omega)\gtrsim \frac{|2\xi_n(t-t_c^{(n)})|}{[\llogeps]^4}\geq \frac{|\xi_nt|}{[\llogeps]^5}.
\end{align*}
Noting also that $\sigma^4\le t^2\log^4(\tfrac 1\eps)$, we estimate
\begin{align}\label{l2}
\frac{\sigma^2|x_*-x_n(t)|^2}{4[\llogeps]^{25}[\sigma^4+t^2\log^4(\tfrac1\eps)]}
&\ge \frac{\sigma^2 |\xi_n|^2t^2}{8[\llogeps]^{35}t^2\log^4(\tfrac 1\eps)}\ge \frac\delta{\eps\log^3(\frac 1\eps)}.
\end{align}
Using the crude upper bound
$$
|x_*-x_n(t)|\leq |x_*|+|x_c^{(n)}| +2|\xi_n|(t-t_c^{(n)})\lesssim 1 + |\xi_n|t,
$$
together with \eqref{bdfv} and \eqref{l2}, we obtain
\begin{align*}
\eqref{cvn}
&\lesssim \log^{10}{(\tfrac1\eps)}\eps^{-2}\delta^{-2}[\llogeps]^2 \log^{\frac52}{(\tfrac1\eps)}\sigma^{\frac 32}t^{-\frac32}\exp\Bigl\{-\frac \delta{\eps\log^3(\tfrac1\eps)}\Bigr\}\chi_2(t,x)\\
&\le t^{-\frac 32}\eps^{100}\chi_2(t,x)
\end{align*}
for $t\geq \delta[\llogeps]^{10}|\xi_n|^{-1}$.

Now consider the regime $t_c^{(n)}-4\sigma\logeps|\xi_n|^{-1}\leq t\leq \delta[\llogeps]^{10}|\xi_n|^{-1}$.  By the definition of $\mathcal E$, we have
\begin{align}\label{E:515}
|x_*-x_n(t)|\gtrsim \frac{|x_*-x_c^{(n)}|}{[\llogeps]^4}\ge\frac{\sigma\logeps}{[\llogeps]^5}.
\end{align}
For the times under consideration,
\begin{align*}
\sigma^4+t^2\log^4(\tfrac 1\eps)\le \sigma^4+\delta^2\eps^2\log^4(\tfrac 1\eps)[\llogeps]^{22}\le \sigma^4[\llogeps]^{23},
\end{align*}
and so we obtain
\begin{align}\label{l1}
\frac{\sigma^2|x_*-x_n(t)|^2}{4[\llogeps]^{25}[\sigma^4+t^2\log^4(\tfrac1\eps)]}&\geq\frac{\log^2(\tfrac 1\eps)}{[\llogeps]^{60}}.
\end{align}
Using the crude upper bound
\begin{align*}
|x_*-x_n(t)|\lesssim |x_*|+|x_c^{(n)}|+|\xi_n t|\lsm [\llogeps]^{10}
\end{align*}
together with \eqref{bdfv} and \eqref{l1}, we obtain
\begin{align*}
\eqref{cvn}
&\lesssim \eps^{-2}\delta^{-2}\log^6(\tfrac1{\eps})[\llogeps]^{20} \log^{\frac52}{(\tfrac1\eps)}\sigma^{-\frac 32}\exp\Bigl\{-\frac{\log^2(\tfrac 1\eps)}{[\llogeps]^{60}}\Bigr\}\chi_2(t,x)\\
&\le \eps^{100}\chi_2(t,x)
\end{align*}
in the short time regime.

Collecting our estimates for long and short times, we get
\begin{align*}
\eqref{cvn}\lsm \langle t\rangle^{-\frac32}\eps^{100} \chi_2(t,x).
\end{align*}
Combining this with \eqref{bdforc}, \eqref{E:chi}, and the bound \eqref{E:cun} for \eqref{cun}, we obtain
\begin{align*}
\Bigl\|\sum_{n\in\mathcal E} c_n^{\eps}w_n^{(1)}\Bigr\|_{L_{t,x}^{\frac{10}3}([T,\infty)\times\Omega)}+\Bigl\|\sum_{n\in\mathcal E} c_n^{\eps}(i\partial_t+\Delta) w_n^{(1)}\Bigr\|_{L_t^1L_x^2([T,\infty)\times\Omega)}=o(1).
\end{align*}
This proves \eqref{rem2} for $w_n^{(1)}$.

Next we consider the term $w_n^{(2)}$.  Just as for $w_n^{(1)}$, we have the following pointwise bound:
\begin{align*}
|w_n^{(2)}(t,x)| & +|(i\partial_t+\Delta )w_n^{(2)}(t,x)|\lsm \biggl\{\Bigl[\sigma^{-2}+|\xi_n|^2+\frac{|x_*-2\xi_nt|^2}{\sigma^4}\Bigr]|\tilde u_n(t,x_*)|\\
&\qquad+\Bigl[\sigma^{-2}+|\xi_n|^2+\frac{\log^{10}(\frac1\eps)}{\sigma^4+t^2}|x_*-x_n(t)|^2\Bigr]|\tilde v_n(t,x_*)|\biggr\}\cdot \chi(t,x),
\end{align*}
where $\chi(t,x)$ is a cutoff to the spacetime region
\begin{align*}
\biggl\{(t,x)\in [0,\infty)\times\Omega: \, |x-x_*|\le \sigma, \ |x_*-x_c^{(n)}|\le \sigma\logeps,\ |t-t_c^{(n)}|\ge \frac{\sigma\logeps}{|\xi_n|}\biggr\}.
\end{align*}

On the support of $\tilde u_n(t,x_*) \chi(t,x)$ we have $t\le \sigma^2$ and
\begin{align*}
|x_*-2\xi_n t|\ge |2\xi_n(t-t_c^{(n)})|-|x_*-x_c^{(n)}|\ge \sigma\logeps.
\end{align*}
Hence
\begin{align*}
|\tilde u_n(t,x_*)|\chi(t,x)&\lsm \biggl(\frac{\sigma^2}{\sigma^4+t^2}\biggr)^{\frac34}\exp\Bigl\{-\frac{\sigma^2|x_*-2\xi_nt|^2}{4(\sigma^4+t^2)}\Bigr\}
\lsm \sigma^{-\frac 32}\exp\bigl\{-\tfrac 18 \log^2(\tfrac 1\eps)\bigr\}\\
&\le\eps^{100}.
\end{align*}
As before, this estimate is good enough to deduce
\begin{align*}
\Bigl\|\Bigl[\sigma^{-2}+|\xi_n|^2+\frac{|x_*-2\xi_nt|^2}{\sigma^4}\Bigr]|\tilde u_n(t,x_*)|\chi(t,x)\Bigr\|_{L_{t,x}^{\frac{10}3}\cap L_t^1L_x^2([T,\infty)\times\Omega)}\le \eps^{90}.
\end{align*}

To estimate the contribution of the $\tilde v_n$ term, we split into short and long times as in the treatment of the corresponding term in $w_n^{(1)}$.  Indeed, the treatment of the regime $t\ge \delta[\llogeps]^{10}|\xi_n|^{-1}$ follows verbatim as there.  For the complementary set of times $t_c^{(n)}-4\sigma\logeps|\xi_n|^{-1}\leq t\leq \delta[\llogeps]^{10}|\xi_n|^{-1}$, we estimate
\begin{align*}
|x_*-x_n(t)|&\ge|x_c^{(n)}-x_n(t)|-|x_c^{(n)}-x_*|=2|\xi_n(t-t_c^{(n)})|-|x_c^{(n)}-x_*|\ge \sigma\logeps.
\end{align*}
This plays the role of \eqref{E:515}; indeed, it is a stronger bound.  With this in place, arguing as for $w_n^{(1)}$ we obtain
\begin{align*}
\Bigl\|\Bigl[\sigma^{-2}+|\xi_n|^2+\frac{\log^{10}(\frac1\eps)}{\sigma^4+t^2}|x_*-x_n(t)|^2\Bigr]|\tilde v_n(t,x_*)|\chi(t,x)\Bigr\|_{L_{t,x}^{\frac{10}3}\cap L_t^1L_x^2([T,\infty)\times\Omega)}\le \eps^{90}.
\end{align*}

Combining the two estimates and using \eqref{bdforc} yields \eqref{rem2} for $w_n^{(2)}$.

It remains to prove \eqref{rem2} for $w_n^{(3)}$, which is the most subtle of all.

\begin{lem}[Almost disjointness of the $w^{(3)}_n$]\label{L:disjoint w3}  Fix $(t,x)\in\R\times\Omega$.  Then
\begin{equation}\label{E:disjoint w3}
\# \big\{ n\in \mathcal E : w^{(3)}_n(t,x) \neq 0 \bigr\}  \lesssim
\logeps^{12}.
\end{equation}
\end{lem}

\begin{proof}
From the definition of $w^{(3)}_n$ we see that if $w^{(3)}_n(t,x) \cdot w^{(3)}_m(t,x) \neq 0$, then
\begin{align*}
|t_c^{(n)}-t_c^{(m)}| \leq |t_c^{(n)}-t| + |t-t_c^{(m)}| \le 2\bigl( |\xi_n|^{-1}+ |\xi_m|^{-1} \bigr)\sigma\logeps
\end{align*}
and
\begin{align*}
|x_c^{(n)}-x_c^{(m)}| &\leq |x_c^{(n)}-x_*| + |x_*-x_c^{(m)}| \le 2\sigma\logeps .
\end{align*}
Combining these with
\begin{align*}
|x_c^{(n)}-x_c^{(m)}| &= 2 |\xi_n t_c^{(n)}-\xi_m t_c^{(m)} | \\
&= \bigl| (\xi_n+\xi_m) (t_c^{(n)}-t_c^{(m)}) + (\xi_n-\xi_m)(t_c^{(n)}+t_c^{(m)}) \bigr|
\end{align*}
and $\eps^{-1} [\llogeps]^{-1} \leq |\xi_n|,|\xi_m|\leq \eps^{-1}\llogeps$ yields
\begin{align*}
|\xi_n-\xi_m| \, (t_c^{(n)}+t_c^{(m)}) \lesssim  \sigma\logeps +
\sigma\logeps[\llogeps]^{2} \lesssim  \sigma\logeps[\llogeps]^{2}.
\end{align*}
From Lemma~\ref{L:xc} we have $t_c^{(n)}+t_c^{(m)} \geq \delta\eps[\llogeps]^{-1}$ and so
$$
|n-m| = L |\xi_n-\xi_m| \lesssim \frac{L\sigma}{\delta\eps}\logeps[\llogeps]^3 = [\logeps]^3[\llogeps]^4 \leq [\logeps]^4.
$$
The lemma now follows; RHS\eqref{E:disjoint w3} bounds the number of lattice points in a ball of this radius.
\end{proof}

To continue, we note that on the support of $w_n^{(3)}$ we have $\tilde u_n(t,x)=u_n(t,x)$ and $\tilde v_n(t,x)=v_n(t,x)$.
We rewrite $w_n^{(3)}$ as follows:
\begin{align*}
w^{(3)}_n(t,x)&=\exp\{it|\xi_n|^2-i\xi_n\cdot(x_*-x_c^{(n)})\}\bigl[u_n(t,x_*)-v_n(t,x_*)\bigr]\\
&\qquad\quad \cdot
\phi\biggl(\frac{|x-x_*|}{\sigma}\biggr)\phi\biggl(\frac{|x_*-x_c^{(n)}|}{\sigma\logeps}\biggr) \phi\biggl(\frac{|t-t_c^{(n)}| |\xi_n|}{2\sigma\logeps}\biggr)\\
&\qquad\quad \cdot\exp\{-it|\xi_n|^2+i\xi_n\cdot(x-x_c^{(n)})\}\\
&=:A_n(t,x)\cdot B_n(t,x)\cdot C_n(t,x).
\end{align*}
We have the following pointwise bounds on $A_n,B_n,C_n$, and their derivatives that are uniform in $n$:
\begin{align*}
&\begin{cases}
|C_n(t,x)|\le1,\quad  |\nabla C_n(t,x)|\le |\xi_n|\lsm \eps^{-1}\llogeps,\\
(i\partial_t+\Delta) C_n(t,x)=0,
\end{cases}\\
&\begin{cases}
|B_n(t,x)|\le 1, \ |\nabla B_n(t,x)|\lsm \sigma^{-1}+[\sigma\log(\tfrac 1\eps)]^{-1}\lsm \eps^{-\frac12}\delta^{-\frac12}, \\
|(i\partial_t+\Delta)B_n(t,x)|\lsm \sigma^{-2} +[\sigma\logeps]^{-2}+\frac{|\xi_n|}{\sigma\log(\frac1\eps)}\lsm \eps^{-\frac32}\delta^{-\frac12},
\end{cases}\\
&\begin{cases}
|A_n(t,x)|\lesssim \eps^{-\frac 14}\delta^{-\frac54}\log^{12}(\tfrac 1\eps), \ |\nabla A_n(t,x)|\lsm \eps^{-\frac34}\delta^{-\frac 74}\log^{12}(\tfrac 1\eps),\\
|(i\partial_t+\Delta) A_n(t,x)|\lesssim \eps^{-\frac 74}\delta^{-\frac74}\log^9(\tfrac 1\eps),
\end{cases}
\end{align*}
on the support of $w_n^{(3)}$.  Indeed, the bounds on $C_n$ and $B_n$ follow from direct computations, while the bounds on $A_n$ were proved in Lemma~\ref{L:uv match}.  Using these bounds we immediately get
\begin{align}
\bigl\|w_n^{(3)}\bigr\|_{L_{t,x}^\infty([T,\infty)\times\Omega)}&\lsm\eps^{-\frac 14}\delta^{-\frac 54}\log^{12}(\tfrac 1\eps)\label{E:w3}\\
\bigl\|(i\partial_t+\Delta)w_n^{(3)}\bigr\|_{L_{t,x}^\infty([T,\infty)\times\Omega)}&\lsm\eps^{-\frac 74}\delta^{-\frac 74}\log^{13}(\tfrac 1\eps),\label{E:laplace w3}
\end{align}
uniformly for $n\in \mathcal E$.  Additionally, the spacetime support of $w_n^{(3)}$ has measure
$$
\bigl|\supp w_n^{(3)}\bigr| \lsm \bigl[\sigma \log(\tfrac1\eps) \eps\llogeps] \sigma \bigl[\sigma\log(\tfrac1\eps)\bigr]^2 \lsm \sigma^4\eps\log^3(\tfrac1\eps)\llogeps.  
$$
Using this together with \eqref{bdforc}, Lemma~\ref{L:disjoint w3}, \eqref{E:w3}, and H\"older's inequality, we estimate
\begin{align*}
&\Bigl\|\sum_{n\in \mathcal E} c_n^\eps w_n^{(3)}\Bigr\|_{L_{t,x}^{\frac{10}3}([T,\infty)\times\Omega)}^{\frac{10}3}\\
&\lsm \sum_{n_1, \ldots, n_4\in \mathcal E} |c_{n_1}^\eps|^{\frac56} \cdot \ldots \cdot |c_{n_4}^\eps|^{\frac56} \int_T^\infty\int_\Omega |w_{n_1}^{(3)}(t,x)|^{\frac56}\cdot\ldots\cdot |w_{n_4}^{(3)}(t,x)|^{\frac56}\, dx\, dt\\
&\lsm \frac{(\sigma \eps)^5}{L^{10}}\log^{36}(\tfrac1\eps)\Bigl[\frac L\eps\llogeps\Bigr]^3 \bigl[\eps^{-\frac 14}\delta^{-\frac 54}\log^{12}(\tfrac 1\eps)\bigr]^{\frac{10}3} \sigma^4\eps\log^3(\tfrac1\eps)\llogeps\\
&\lsm \eps^{\frac{19}6}\delta^{-\frac{19}6}\log^{82}(\tfrac1\eps) = o(1) \qtq{as} \eps\to 0.
\end{align*}

Arguing similarly and using \eqref{E:laplace w3} in place of \eqref{E:w3}, we obtain
\begin{align*}
&\Bigl\|\sum_{n\in \mathcal E} c_n^\eps (i\partial_t+\Delta)w_n^{(3)}\Bigr\|_{L_{t,x}^2([T,\infty)\times\Omega)}^2\\
&\lsm \sum_{n_1,n_2\in \mathcal E} |c_{n_1}^\eps||c_{n_2}^\eps|\int_T^\infty\int_\Omega \bigl| (i\partial_t+\Delta)w_{n_1}^{(3)}(t,x)\bigr|\bigl| (i\partial_t+\Delta)w_{n_2}^{(3)}(t,x)\bigr|\, dx\, dt\\
&\lsm \frac{(\sigma \eps)^3}{L^6}\log^{12}(\tfrac1\eps)\Bigl[\frac L\eps\llogeps\Bigr]^3 \bigl[\eps^{-\frac 74}\delta^{-\frac 74}\log^{13}(\tfrac 1\eps)\bigr]^2 \sigma^4\eps\log^3(\tfrac1\eps)\llogeps\\
&\lsm \eps^{-\frac12}\delta^{-\frac32}\log^{46}(\tfrac1\eps).
\end{align*}
To convert this to a bound in $L^1_tL^2_x$, we need the following consequence of Lemma~\ref{L:xc}:
\begin{align*}
\bigl| \bigl\{ t : {\textstyle\sum_{n\in \mathcal E}} c_n^\eps w_n^{(3)}(t,x)\not\equiv 0\bigr\} \bigr|
&\leq \max_{n,m\in\mathcal E} |t_c^{(n)}- t_c^{(m)}| +\tfrac{2\sigma\log(\frac1\eps)}{|\xi_n|} + \tfrac{2\sigma\log(\frac1\eps)}{|\xi_m|}\\
&\lesssim \eps\delta[\llogeps]^9.
\end{align*}
Applying H\"older's inequality in the time variable, we get
\begin{align*}
\Bigl\|\sum_{n\in\mathcal E} c_n^\eps (i\partial_t+\Delta) & w_n^{(3)}\Bigr\|_{L_t^1L_x^2([T,\infty)\times\Omega)}\\
&\lesssim \bigl[\eps\delta\log\log^9(\tfrac1\eps)\bigr]^{\frac12}\Bigl\|\sum_{n\in \mathcal E} c_n^\eps (i\partial_t+\Delta)w_n^{(3)}\Bigr\|_{L_{t,x}^2([T,\infty)\times\Omega)}\\
&\lesssim \eps^{\frac14}\delta^{-\frac 14}\log^{24}(\tfrac 1\eps) = o(1) \qtq{as} \eps\to 0.
\end{align*}

This proves \eqref{rem2} for $w_n^{(3)}$ and so completes the proof of Lemma~\ref{L:rem}.
\end{proof}

Combining Lemmas~\ref{L:small u}, \ref{L:bdfv}, and \ref{L:rem} yields Proposition~\ref{P:long times}, which controls the contribution for large times of rays that enter the obstacle.  The contribution from short times was estimated in Proposition~\ref{P:short times}, while the contributions of near-grazing rays and rays that miss the obstacle were estimated in Propositions~\ref{P:ng} and \ref{P:missing}, respectively.  Putting everything together completes the proof of Theorem~\ref{T:LF3} and so the discussion of Case~(v).

%%%%%%%%%%%%%%%%%%%%%%%%%%%%%%%%%%%%%%%%%%%%%%%%%%%%%%%%%%%%%%%%%%%%%%%%%%%%%%%%%%%%%%%%%%%%%%%%%%%%%%%%%%%%%%%%%%%%%%%%%%%%%%%%%%%%%%
\section{Linear profile decomposition}\label{S:LPD}
%%%%%%%%%%%%%%%%%%%%%%%%%%%%%%%%%%%%%%%%%%%%%%%%%%%%%%%%%%%%%%%%%%%%%%%%%%%%%%%%%%%%%%%%%%%%%%%%%%%%%%%%%%%%%%%%%%%%%%%%%%%%%%%%%%%%%%

The purpose of this section is to prove a linear profile decomposition for the propagator $e^{it\Delta_\Omega}$
for data in $\dot H^1_D(\Omega)$; see Theorem~\ref{T:LPD}.  As we will see below, the profiles can live
in different limiting geometries; this is one of the principal differences relative to previous analyses.

Throughout this section, $\Theta:\R^3\to[0,1]$ denotes a smooth function such that
\begin{align*}
\Theta(x)=\begin{cases} 0, & |x|\le \frac 14,\\1, & |x| \ge \frac
12.\end{cases}
\end{align*}
We also write $\Theta^c(x):=1-\Theta(x)$ and $d(x):=\dist(x,\Omega^c)$.

\begin{lem}[Refined Strichartz estimate]\label{lm:refs}
Let $f\in \hd$. Then we have
\begin{align*}
\|\prd f\|_{\lt}\lsm \|f\|_{\hd}^{\frac 15} \sup_{N\in \tz}\|\prd f_N\|_{\lt}^{\frac45}.
\end{align*}
\end{lem}
\begin{proof}
From the square function estimate Lemma~\ref{sq}, Bernstein, and Strichartz inequalities,
\begin{align*}
\|\prd & f\|_{L^{10}_{t,x}}^{10} \lsm \iint_{\R\times\Omega} \Bigl(\sum_{N\in \tz}|\prd f_N|^2 \Bigr)^5 \,dx \,dt\\
&\lsm \sum_{N_1\le \cdots\le N_5}\iint_{\R\times\Omega} |\prd f_{N_1}|^2 \cdots |\prd f_{N_5}|^2 \,dx\,dt\\
&\lsm \sum_{N_1\le\cdots\le N_5} \|\prd f_{N_1}\|_{L^\infty_{t,x}}\|\prd f_{N_1}\|_{L^{10}_{t,x}}
	\prod_{j=2}^4\|\prd f_{N_j}\|_{L^{10}_{t,x}}^2 \\
&\qquad\qquad \qquad\cdot\|\prd f_{N_5}\|_{L^{10}_{t,x}}\|\prd f_{N_5}\|_{L^5_{t,x}}\\
&\lsm \sup_{N\in \tz}\|\prd f_N\|_{L^{10}_{t,x}}^8 \sum_{N_1\le N_5} \bigr[1+\log\bigl(\tfrac {N_5}{N_1}\bigr)\bigr]^3 N_1^{\frac32}\| \prd f_{N_1}\|_{L^\infty_t L^2_x}\\
&\qquad\qquad\qquad\cdot N_5^{\frac 12}\|\prd f_{N_5}\|_{L^5_t L^{\frac{30}{11}}_x} \\
&\lsm \sup_{N\in \tz}\|\prd f_N\|_{L^{10}_{t,x}}^8 \sum_{N_1\le N_5} \bigr[1+\log\bigl(\tfrac {N_5}{N_1}\bigr)\bigr]^3 \bigl(\tfrac{N_1}{N_5}\bigr)^{\frac 12}
	\|f_{N_1}\|_{\hd} \|f_{N_5}\|_{\hd}\\
&\lsm \sup_{N\in \tz}\|\prd f_N\|_{L^{10}_{t,x}}^8 \|f\|_{\hd}^2,
\end{align*}
where all spacetime norms are over $\R\times\Omega$.  Raising this to the power $\frac 1{10}$ yields the lemma.
\end{proof}

The refined Strichartz inequality shows that linear solutions with non-trivial spacetime norm must concentrate on at least one frequency annulus.
The next proposition goes one step further and shows that they contain a bubble of concentration around some point in spacetime.  A novelty in our setting is that the bubbles of concentration may live in one of the limiting geometries identified earlier.

\begin{prop}[Inverse Strichartz inequality]\label{P:inverse Strichartz}
Let $\{f_n\}\subset \hd$. Suppose that
\begin{align*}
\lim_{n\to \infty}\|f_n\|_{\hd}=A < \infty \qtq{and}  \lim_{n\to\infty}\|\prd f_n\|_{\lt}=\eps >0.
\end{align*}
Then there exist a subsequence in $n$, $\{\phi_n\}\subset \hd$, $\{N_n\}\subset \tz$, $\{(t_n, x_n)\}\subset \R\times\Omega$
conforming to one of the four cases listed below such that
\begin{gather}
\liminf_{n\to\infty}\|\phi_n\|_{\hd}\gtrsim \eps(\tfrac{\eps}A)^{\frac 78}, \label{nontri}\\
\liminf_{n\to\infty}\Bigl\{ \|f_n\|_{\hd}^2-\|f_n-\phi_n\|_{\hd}^2\Bigr\} \gtrsim A^2 (\tfrac\eps A)^{\frac{15}4},\label{dech}\\
\liminf_{n\to\infty}\Bigl\{ \|\prd f_n\|_{\lt}^{10}-\|\prd (f_n-\phi_n)\|_{\lt}^{10}\Bigr\} \gtrsim \eps^{10}(\tfrac\eps A)^{\frac{35}4}.\label{dect}
\end{gather}
The four cases are:
\begin{CI}
\item Case 1: $N_n\equiv N_\infty \in 2^\Z$ and $x_n\to x_{\infty}\in \Omega$.  In this case, we choose $\phi\in \hd$ and the subsequence
so that $e^{it_n\ld}f_n\rightharpoonup \phi$ weakly in $\hd$ and we set $\phi_n:=e^{-it_n\ld}\phi$.

\item Case 2: $N_n\to 0$ and $-N_n x_n\to x_\infty\in \R^3$.  In this case, we choose $\tildphi\in \hr$ and the subsequence so that
$$
g_n(x) :=N_n^{-\frac 12}(e^{it_n\ld}f_n)(N_n^{-1}x+x_n) \rightharpoonup \tildphi(x) \quad\text{weakly in} \quad \dot H^1(\R^3)
$$
and we set
$$
\phi_n(x):=N_n^{\frac 12} e^{-it_n\ld}[(\chi_n\tilde\phi)(N_n(x-x_n))],
$$
where $\chi_n(x)=\chi(N_n^{-1}x+x_n)$ and  $\chi(x)=\Theta(\frac{d(x)}{\diam (\Omega^c)})$.

\item Case 3: $N_n d(x_n)\to\infty$.  In this case, we choose $\tilde\phi\in \hr$ and the subsequence so that
$$
g_n(x) :=N_n^{-\frac 12}(e^{it_n\ld}f_n)(N_n^{-1}x+x_n) \rightharpoonup \tildphi(x) \quad\text{weakly in} \quad \dot H^1(\R^3)
$$
and we set
$$
\phi_n(x) :=N_n^{\frac12}e^{-it_n\ld}[(\chi_n\tilde\phi)(N_n(x-x_n))],
$$
where $\chi_n(x)=1-\Theta(\frac{|x|}{N_n d(x_n)})$.

\item Case 4: $N_n\to \infty$ and  $N_n d(x_n)\to d_{\infty}>0$.  In this case, we choose $ \tilde\phi \in \dot H^1_D(\HH)$ and the subsequence so that
$$
g_n(x) := N_n^{-\frac12}(e^{it_n\ld}f_n)(N_n^{-1}R_nx+x^*_n)\rightharpoonup \tildphi(x) \quad\text{weakly in} \quad \hr
$$
and we set
$$
\phi_n(x) :=N_n^{\frac 12} e^{-it_n\ld}[\tilde\phi(N_nR_n^{-1}(\cdot-x^*_n))],
$$
where $R_n\in SO(3)$ satisfies $R_n e_3 = \frac{x_n-x^*_n}{|x_n-x^*_n|}$ and $x^*_n\in \partial \Omega$ is chosen so that $d(x_n)=|x_n-x^*_n|$.
\end{CI}
\end{prop}

\begin{rem}
The analogue of $\tilde \phi$ in Case 1 is related to $\phi$ via $\phi(x)= N_\infty^{\frac12} \tilde \phi(N_\infty (x-x_\infty))$; see \eqref{1converg}.
\end{rem}

\begin{proof}
From Lemma \ref{lm:refs} and the conditions on $f_n$, we know that for each $n$ there exists $N_n\in \tz$ such that
\begin{align*}
\|\prd \pnno f_n\|_{\lt}\gtrsim \eps^{\frac 54}A^{-\frac 14}.
\end{align*}
On the other hand, from the Strichartz and Bernstein inequalities we get
\begin{align*}
\|\prd \pnno f_n\|_{L_{t,x}^{\frac {10}3}(\R\times\Omega)}\lsm  \| \pnno f_n\|_{L_x^2(\Omega)} \lesssim N_n^{-1} A.
\end{align*}
By H\"older's inequality, these imply
\begin{align*}
A^{-\frac 14}\eps^{\frac 54}&\lsm \|\prd\pnno f_n\|_{\lt}\\
&\lsm \|\prd\pnno f_n\|_{L_{t,x}^{\frac{10}3}(\R\times\Omega)}^{\frac 13}\|\prd\pnno f_n\|_{L_{t,x}^{\infty}(\R\times\Omega)}^{\frac 23} \\
&\lsm N_n^{-\frac 13}A^{\frac 13}\|\prd \pnno f_n\|_{L_{t,x}^{\infty}(\R\times\Omega)}^{\frac 23},
\end{align*}
and so
\begin{align*}
\|\prd \pnno f_n\|_{L_{t,x}^{\infty}(\R\times\Omega)} \gtrsim N_n^{\frac 12}\eps (\tfrac \eps A)^{\frac78}.
\end{align*}
Thus there exist $(t_n,x_n)\in \R\times \Omega$ such that
\begin{align}\label{cncen}
\Bigl|(e^{it_n\ld}\pnno f_n)(x_n)\Bigr|\gtrsim N_n^{\frac12}\eps (\tfrac{\eps}A)^{\frac 78}.
\end{align}

The cases in the statement of the proposition are determined solely by the behaviour of $x_n$ and $N_n$.  We will now show
\begin{align}\label{lb}
N_n d(x_n)\gtrsim (\tfrac\eps A)^{\frac{15}8} \qtq{whenever} N_n \gtrsim 1,
\end{align}
which explains the absence of the scenario $N_n \gtrsim 1$ with $N_nd(x_n)\to0$.     The proof of \eqref{lb} is based on Theorem~\ref{T:heat}, which implies
\begin{align*}
\int_\Omega \bigl| e^{\Delta_\Omega / N_n^2}(x_n,y) \bigr|^2\,dy &\lesssim N_n^{6} \int_\Omega \Bigl| \bigl[N_n d(x_n)\bigr]\bigl[N_n d(x_n)+N_n|x_n-y| \bigr] e^{-c N_n^2|x_n-y|^2} \Bigr|^2 \,dy \\
&\lesssim  [N_n d(x_n)]^2[N_nd(x_n) + 1]^2 N_n^3,
\end{align*}
whenever $N_n\gtrsim 1$.  Writing
$$
(e^{it_n\ld} \pnno f_n)(x_n) = \int_\Omega e^{\Delta_\Omega / N_n^2}(x_n,y) \, \bigl[ P^\Omega_{\leq 2 N_n} e^{ - \Delta_\Omega / N_n^2} e^{it_n\ld} \pnno f_n \bigr](y) \,dy
$$
and using \eqref{cncen} and Cauchy--Schwarz gives
\begin{align*}
 N_n^{\frac12}\eps (\tfrac{\eps}A)^{\frac 78} &\lesssim \bigl[N_n d(x_n)\bigr] \bigl[N_nd(x_n) + 1\bigr] N_n^{\frac32}
	\bigl\| P^\Omega_{\leq 2 N_n} e^{ - \Delta_\Omega / N_n^2} e^{it_n\ld} \pnno f_n \bigr\|_{L^2_x} \\
& \lesssim \bigl[N_n d(x_n)\bigr] \bigl[N_nd(x_n) + 1\bigr] N_n^{\frac12}  \|  f_n \|_{\hd}.
\end{align*}	
The inequality \eqref{lb} now follows.

Thanks to the lower bound \eqref{lb}, after passing to a subsequence, we only need to consider the four cases below, which
correspond to the cases in the statement of the proposition.

\textbf{Case 1:} $N_n\sim 1$ and $N_n d(x_n)\sim 1$.

\textbf{Case 2:} $N_n\to 0$ and $N_n d(x_n) \lsm 1$.

\textbf{Case 3:} $N_n d(x_n)\to \infty$ as $n\to \infty$.

\textbf{Case 4:} $N_n\to \infty$ and $N_n d(x_n)\sim 1$.

We will address these cases in order.  The geometry in Case~1 is simplest and it allows us to introduce the basic framework for the argument.  The main new difficulty in the remaining cases is the variable geometry, which is where Proposition~\ref{P:converg} and Corollary~\ref{C:LF} play a crucial role.  Indeed, as we will see below, the four cases above reduce to the ones discussed in Sections~\ref{S:Domain Convergence} and~\ref{S:Linear flow convergence} after passing to a further subsequence.

With Proposition~\ref{P:converg} and Corollary~\ref{C:LF} already in place, the arguments in the four cases parallel each other rather closely.  There are four basic steps.  The most important is to embed the limit object $\tildphi$ back inside $\Omega$ in the form of $\phi_n$ and to show that $f_n-\phi_n$ converges to zero in suitable senses.  The remaining three steps use this information to prove the three estimates \eqref{nontri}, \eqref{dech}, and \eqref{dect}.

\textbf{Case 1:} Passing to a subsequence, we may assume
\begin{align*}
N_n\equiv N_\infty\in \tz \quad\text{and}\quad x_n\to x_\infty\in \Omega.
\end{align*}

To prefigure the treatment of the later cases we set
$$
g_n(x) :=N_n^{-\frac 12}(e^{it_n\ld}f_n)(N_n^{-1}x+x_n),
$$
even though the formulation of Case~1 does not explicitly include this sequence.  As $f_n$ is supported in $\Omega$, so $g_n$ is supported in
$\Omega_n :=N_n(\Omega-\{x_n\})$.  Moreover,
$$
\|g_n\|_{\dot H^1_D(\Omega_n)}=\|f_n\|_{\dot H^1_D(\Omega)}\lesssim A.
$$
Passing to a subsequence, we can choose $\tilde \phi$ so that $g_n\rightharpoonup \tilde\phi$ weakly in $\hr$.  Rescaling the relation $g_n\rightharpoonup \tilde\phi$ yields
\begin{align}\label{1converg}
(e^{it_n\ld}f_n)(x)\rightharpoonup \phi(x) :=N_\infty^{\frac 12}\tilde \phi(N_\infty(x-x_\infty)) \quad \text{weakly in} \quad \hd.
\end{align}
To see that $\phi\in\hd$ when defined in this way, we note that $\hd$ is a weakly closed subset of $\hr$; indeed, a convex set is weakly closed if and only if it is norm closed.

The next step is to prove \eqref{nontri} by showing that $\phi$ is non-trivial.  Toward this end, let $h:=P^{\Omega}_{N_\infty}\delta(x_\infty)$. Then from the Bernstein inequality we have
\begin{align}\label{h bd}
\|(-\ld)^{-\frac 12}h\|_{L^2(\Omega)}=\|(-\ld)^{-\frac 12}P_{N_\infty}^\Omega\delta(x_\infty)\|_{L^2(\Omega)}\lsm N_\infty^{\frac 12}.
\end{align}
In particular, $h\in \dot H^{-1}_D(\Omega)$. On the other hand, we have
\begin{align}\label{h meets phi}
\langle \phi,h\rangle
&=\lim_{n\to\infty}\langle e^{it_n\ld}f_n,h\rangle=\lim_{n\to\infty}\langle e^{it_n\ld} f_n,P_{N_\infty}^\Omega\delta(x_\infty)\rangle  \notag \\
&=\lim_{n\to\infty}(e^{it_n\ld}\pnno f_n)(x_n)+\lim_{n\to\infty}\langle e^{it_n\ld}f_n, P_{N_\infty}^{\Omega}[\delta({x_\infty})-\delta({x_n})]\rangle.
\end{align}

The second limit in \eqref{h meets phi} vanishes.  Indeed, basic elliptic theory shows that
\begin{align}\label{elliptic est}
\| \nabla v \|_{L^\infty(\{|x|\leq R\})} \lesssim  R^{-1} \| v \|_{L^\infty(\{|x|\leq 2R\})}  + R \| \Delta v \|_{L^\infty(\{|x|\leq 2R\})},
\end{align}
which we apply to $v(x) = (P_{N_\infty}^{\Omega} e^{it_n\ld}f_n )(x+x_n)$ with $R=\frac12 d(x_n)$.  By hypothesis, $d(x_n) \sim 1$,
while by the Bernstein inequalities,
$$
\| P_{N_\infty}^{\Omega} e^{it_n\ld}f_n \|_{L^\infty_x} \lesssim N_\infty^{\frac12} A  \qtq{and}
	\| \Delta P_{N_\infty}^{\Omega} e^{it_n\ld}f_n \|_{L^\infty_x} \lesssim N_\infty^{\frac52} A.
$$
Thus by the fundamental theorem of calculus and \eqref{elliptic est}, for $n$ sufficiently large,
\begin{align}\label{6:37}
|\langle e^{it_n\ld} f_n,P_{N_\infty}^{\Omega}[\delta(x_\infty)-\delta(x_n)]\rangle| &\lsm  |x_\infty-x_n| \, \| \nabla P_{N_\infty}^{\Omega} e^{it_n\ld} f_n\|_{L^\infty(\{|x|\leq R\})}\notag\\
&\lsm A \bigl[\tfrac{N_\infty^{\frac12}}{d(x_n)} + N_\infty^{\frac52} d(x_n)\bigr] |x_\infty-x_n|,
\end{align}
which converges to zero as $n\to \infty$.

Therefore, using \eqref{cncen}, \eqref{h bd}, and \eqref{h meets phi}, we have
\begin{align}
N_\infty^{\frac12} \eps \bigl(\tfrac{\eps}A\bigr)^{\frac78} \lesssim |\langle \phi, h\rangle| \lesssim  \|\phi\|_{\hd}\|h\|_{\dot H^{-1}_D(\Omega)}
\lesssim N_\infty^{\frac12}\|\phi\|_{\hd}.\label{lbf}
\end{align}
As $\prdn$ is unitary on $\hd$ we have $\|\phi_n\|_{\hd}=\|\phi\|_{\hd}$, and so \eqref{lbf} yields \eqref{nontri}.

Claim \eqref{dech} follows immediately from \eqref{nontri} and \eqref{1converg} since $\hd$ is a Hilbert space.

The only remaining objective is to prove decoupling for the $L_{t,x}^{10}$ norm. Note
\begin{align*}
(i\partial_t)^{\frac 12} \prd =(-\ld)^{\frac 12} \prd.
\end{align*}
Thus, by H\"older, on any compact domain $K$ in $\R\times\R^3$ we have
\begin{align*}
\|\prd \prdn f_n\|_{H^{\frac 12}_{t,x}(K)}\lsm \| \langle-\ld\rangle ^{\frac12} e^{i(t+t_n)\Delta_\Omega}f_n \|_{L^2_{t,x}(K)}\lsm_K A.
\end{align*}
From Rellich's Lemma, passing to a subsequence, we get
\begin{align*}
\prd \prdn f_n \to \prd \phi \qtq{strongly in} L_{t,x}^{2}(K)
\end{align*}
and so, passing to a further subsequence, $\prd\prdn f_n(x)\to \prd \phi(x)$ a.e. on $K$.   Using a diagonal argument
and passing again to a subsequence, we obtain
\begin{align*}
\prd\prdn f_n(x)\to \prd \phi(x) \quad\text{a.e. in $\R\times \R^3$}.
\end{align*}
Using the Fatou Lemma of Br\'ezis and Lieb (cf. Lemma~\ref{lm:rf}) and a change of variables, we get
\begin{align*}
\lim_{n\to \infty}\Bigl\{\|\prd f_n\|_{\lt}^{10}-\|\prd (f_n-\phi_n)\|_{\lt}^{10}\Bigr\} = \|\prd \phi\|_{\lt}^{10},
\end{align*}
from which \eqref{dect} will follow once we prove
\begin{align}\label{want}
\|\prd \phi\|_{\lt}\gtrsim \eps (\tfrac{\eps}A)^{\frac{7}{8}}.
\end{align}
To see this, we use \eqref{lbf}, the Mikhlin multiplier theorem (for $e^{it\Delta_\Omega} P^\Omega_{\leq 2N_\infty}$), and Bernstein to estimate
\begin{align*}
N_\infty^{\frac12} \eps \bigl(\tfrac{\eps}A\bigr)^{\frac78}\lesssim |\langle \phi, h\rangle|
&=|\langle e^{it\Delta_\Omega} \phi, e^{it\Delta_\Omega} h\rangle|
\lesssim  \|e^{it\Delta_\Omega}\phi\|_{L_x^{10}}\|e^{it\Delta_\Omega}h\|_{L_x^{\frac{10}9}}\\
&\lesssim  \|e^{it\Delta_\Omega}\phi\|_{L_x^{10}} \|h\|_{L_x^{\frac{10}9}}
\lesssim N_\infty^{\frac3{10}} \|e^{it\Delta_\Omega}\phi\|_{L_x^{10}},
\end{align*}
for each $|t|\le N_{\infty}^{-2}$.
Thus
$$
\|e^{it\Delta_\Omega}\phi\|_{L_x^{10}} \gtrsim N_\infty^{\frac15} \eps \bigl(\tfrac{\eps}A\bigr)^{\frac78},
$$
uniformly in $|t|\le N_{\infty}^{-2}$.  Integrating in $t$ leads to \eqref{want}.

\textbf{Case 2:}  As $N_n\to 0$, the condition $N_nd(x_n)\lsm 1$ guarantees that $\{N_nx_n\}_{n\geq 1}$ is a bounded sequence; thus, passing to a subsequence, we may assume $-N_n x_n\to x_\infty\in \R^3$.  As in Case 1, we define $\Omega_n :=N_n(\Omega-\{x_n\})$.  Note that the rescaled obstacles $\Omega_n^c$ shrink to $x_\infty$ as $n\to \infty$; this is the defining characteristic of Case~2.

As $f_n$ is bounded in $\hd$, so the sequence $g_n$ is bounded in $\dot H^1_D(\Omega_n)\subseteq\hr$.  Thus, passing to a subsequence, we
can choose $\tildphi$ so that $g_n \rightharpoonup \tildphi$ in $\hr$.

We cannot expect $\tildphi$ to belong to $\dot H^1_D(\Omega_n)$, since it has no reason to vanish on $\Omega_n^c$.
This is the role of $\chi_n$ in the definition of $\phi_n$.  Next we show that this does not deform $\tildphi$ too gravely; more precisely,
\begin{align}\label{E:no deform}
\chi_n\tildphi \to \tildphi, \qtq{or equivalently,} \bigl[ 1 - \chi(N_n^{-1}x+x_n)\bigr]\tildphi(x) \to 0 \quad \text{in $\hr$.}
\end{align}
Later, we will also need to show that the linear evolution of $\chi_n\tildphi$ in $\Omega_n$ closely approximates the whole-space linear evolution of $\tildphi$.

To prove \eqref{E:no deform} we first set $B_n:=\{ x\in \R^3 : \dist(x,\Omega_n^c) \leq \diam(\Omega_n^c)\}$, which contains $\supp(1-\chi_n)$ and $\supp(\nabla\chi_n)$.  Note that because $N_n\to0$, the measure of $B_n$ shrinks to zero as $n\to\infty$.  By H\"older's inequality,
\begin{align*}
\bigl\| [ 1 &- \chi(N_n^{-1}x+x_n)]\tildphi(x) \bigr\|_{\hr} \\
&\lesssim \bigl\| [ 1 - \chi(N_n^{-1}x+x_n)]\nabla \tildphi(x) \bigr\|_{L^2(\R^3)}
    + \bigl\| N_n^{-1} \bigl(\nabla\chi\bigr)(N_n^{-1}x+x_n) \tildphi(x) \bigr\|_{L^2(\R^3)} \\
&\lesssim \| \nabla \tildphi \|_{L^2(B_n)} + \| \tildphi \|_{L^6(B_n)},
\end{align*}
which converges to zero by the dominated convergence theorem.

With \eqref{E:no deform} in place, the proofs of \eqref{nontri} and \eqref{dech} now follow their Case~1 counterparts very closely; this will rely on key inputs from Section~\ref{S:Domain Convergence}.  We begin with the former.

Let $h:=P_1^{\R^3} \delta(0)$; then
\begin{align*}
\lng \tilde \phi, h\rng=\lim_{n\to\infty} \lng g_n, h\rng=\lim_{n\to\infty}\lng g_n, P_1^{\on} \delta(0)\rng+\lim_{n\to\infty}\lng g_n, (P_1^{\R^3}- P_1^{\on})\delta(0)\rng.
\end{align*}
The second term vanishes due to Proposition~\ref{P:converg} and the uniform boundedness of $\|g_n\|_{\hr}$. Therefore,
\begin{align}
|\lng\tilde\phi, h\rng|&=\Bigl|\lim_{n\to \infty}\lng g_n,P_1^{\on}\delta(0)\rng\Bigr|\notag\\
&=\Bigl|\lim_{n\to\infty} \lng\prdn f_n, N_n^{\frac52} (P_1^{\on}\delta(0))(N_n(x-x_n))\rng\Bigr|\notag\\
&=\Bigl|\lim_{n\to\infty}\lng e^{it_n\ld}f_n, N_n^{-\frac12}P_{N_n}^{\Omega}\delta(x_n)\rng\Bigr|\gtrsim\eps (\tfrac{\eps}A)^{\frac 78},\label{256}
\end{align}
where the last inequality follows from \eqref{cncen}.  Thus,
\begin{align*}
\|\tildphi\|_{\hr}\gtrsim \eps (\tfrac{\eps}A)^{\frac 78}
\end{align*}
as in \eqref{lbf}.  Combining this with \eqref{E:no deform}, for $n$ sufficiently large we obtain
\begin{align*}
\|\phi_n\|_{\hd}= \| \chi_n \tilde \phi\|_{\dot H^1_D(\Omega_n)}\gtrsim \eps (\tfrac{\eps}A)^{\frac 78}.
\end{align*}
This proves \eqref{nontri}.

To prove the decoupling in $\hd$, we write 
\begin{align*}
&\|f_n\|_{\hd}^2 -\|f_n-\phi_n\|_{\hd}^2 = 2\lng f_n, \phi_n\rng_{\hd}-\|\phi_n\|_{\hd}^2\\
&\quad=2\Bigl\lng N_n^{-\frac 12} (\prdn f_n)(N_n^{-1} x+x_n),\, \tildphi(x)\chi_n(x)\Bigr\rng_{\dot H^1_D(\on)}-\| \chi_n \tilde \phi\|_{\dot H^1_D(\Omega_n)}^2\\
&\quad=2\lng g_n, \tilde \phi\rng_{\hr}-2\bigl\lng g_n, \tildphi (1-\chi_n) \bigr\rng_{\hr} -\| \chi_n \tilde \phi\|_{\dot H^1_D(\Omega_n)}^2.
\end{align*}
From the weak convergence of $g_n$ to $\tildphi$, \eqref{E:no deform}, and \eqref{nontri}, we deduce
\begin{align*}
\lim_{n\to\infty}\Bigl\{\|f_n\|_{\hd}^2-\|f_n-\phi_n\|_{\hd}^2\Bigr\}=\|\tilde\phi\|_{\hr}^2 \gtrsim \eps^2 (\tfrac{\eps}A)^{\frac 74}.
\end{align*}
This completes the verification of \eqref{dech}.

We now turn to proving decoupling of the $\lt$ norm, which we will achieve by showing that
\begin{align}\label{305}
\liminf_{n\to\infty}\biggl\{\|e^{it\Delta_\Omega} f_n\|_{\ltr}^{10}-\|e^{it\Delta_\Omega}(f_n-\phi_n)&\|_{\lt}^{10}\biggr\} = \|e^{it\Delta_{\R^3}} \tilde\phi\|_{\ltr}^{10}.
\end{align}
Notice that \eqref{dect} then follows from the lower bound
\begin{align}\label{328}
\|e^{it\Delta_{\R^3}} \tilde \phi\|_{\ltr}^{10}\gtrsim \eps^{10} (\tfrac \eps A)^{\frac{35}4},
\end{align}
which we prove in much the same way as in Case~1: From \eqref{256} and the Mikhlin multiplier theorem, we have
\begin{align*}
\eps (\tfrac{\eps}A)^{\frac 78}&\lsm |\lng \tilde\phi, h\rng|\lsm
\|e^{it\Delta_{\R^3}} \tilde\phi\|_{L^{10}(\R^3)}\|e^{it\Delta_{\R^3}} P_1^{\R^3}\delta(0)\|_{L^{\frac
{10}9}(\R^3)}\lsm \|e^{it\Delta_{\R^3}} \tilde\phi\|_{L^{10}(\R^3)}
\end{align*}
uniformly for $|t|\leq 1$. Integrating in time yields \eqref{328} and then plugging this into \eqref{305} leads to \eqref{dect} in Case 2.

To establish \eqref{305} we need two ingredients: The first ingredient is
\begin{align}\label{c2i1}
e^{it\lon}[g_n-\chi_n\tilde\phi]\to 0 \quad \text{a.e. in } \R\times\R^3,
\end{align}
while the second ingredient is
\begin{align}\label{c2i2}
\|e^{it\lon}[\chi_n\tilde \phi]-e^{it\Delta_{\R^3}}\tilde\phi\|_{\ltr}\to 0.
\end{align}
Combining these and passing to a subsequence if necessary we obtain
$$
e^{it\lon}g_n-e^{it\Delta_{\R^3}}\tilde\phi\to 0 \quad \text{a.e. in } \R\times\R^3,
$$
which by the Fatou Lemma of Br\'ezis and Lieb (cf. Lemma~\ref{lm:rf}) yields
\begin{align*}
\liminf_{n\to\infty}\Bigl\{\|e^{it\lon}g_n\|_{\ltr}^{10}-\|e^{it\lon}g_n-e^{it\Delta_{\R^3}} &\tilde\phi\|_{\ltr}^{10}\Bigr\}\\
&= \|e^{it\Delta_{\R^3}} \tilde\phi\|_{\ltr}^{10}.
\end{align*}
Combining this with \eqref{c2i2} and rescaling yields \eqref{305}.

We start with the first ingredient \eqref{c2i1}.  Using the definition of $\tilde \phi$ together with \eqref{E:no deform}, we deduce
\begin{align*}
g_n-\chi_n\tilde \phi \rightharpoonup 0 \quad \text{weakly in} \quad \dot H^1(\R^3).
\end{align*}
Thus, by Proposition~\ref{P:converg},
\begin{align*}
e^{it\lon}[g_n-\chi_n\tilde \phi]\rightharpoonup 0 \quad \text{weakly in} \quad \dot H^1(\R^3)
\end{align*}
for each $t\in \R$.  By the same argument as in Case 1, using the fact that $( i\partial_t)^{1/2} e^{it\lon}=(-\lon)^{1/2} e^{it\lon}$ and passing to a subsequence, we obtain \eqref{c2i1}.

To establish \eqref{c2i2} we will make use of Corollary~\ref{C:LF}.  Note that $\tlim \Omega_n=\R^3\setminus\{x_\infty\}$ and by Lemma~\ref{L:dense}, $\tilde \phi$ can be well approximated in $\dot H^1(\R^3)$ by $\psi\in C^\infty_c(\tlim \Omega_n)$.  By \eqref{E:no deform}, for $n$ sufficiently large,
$\chi_n\tilde \phi$ are also well approximated in $\dot H^1(\R^3)$ by the same $\psi\in C^\infty_c(\tlim \Omega_n)$.  Thus, \eqref{c2i2} follows by combining Corollary~\ref{C:LF} with the Strichartz inequality.

\textbf{Case 3:} The defining characteristic of this case is that the rescaled obstacles $\Omega_n^c$ march off to infinity; specifically,
$\dist(0,\Omega_n^c)=N_nd(x_n)\to \infty$, where $\Omega_n:=N_n(\Omega-\{x_n\})$.

The treatment of this case parallels that of Case~2.  The differing geometry of the two cases enters only in the use of Proposition~\ref{P:converg}, Corollary~\ref{C:LF}, and the analogue of the estimate \eqref{E:no deform}.  As these first two inputs have already been proven in all cases, our only obligation is to prove
\begin{align}\label{E:no deform 3}
\chi_n\tildphi \to \tildphi, \qtq{or equivalently,} \Theta\bigl(\tfrac{|x|}{\dist(0,\Omega_n^c)}\bigr) \tildphi(x) \to 0 \quad \text{in $\hr$}.
\end{align}
To this end, let $B_n:= \{x\in \R^3: \, |x|\geq \frac14\dist(0, \Omega_n^c)\}$.  Then by H\"older,
\begin{align*}
\bigl\| \Theta\bigl(\tfrac{|x|}{\dist(0,\Omega_n^c)}\bigr)  \tildphi(x)\bigr\|_{\hr} \lesssim \|\nabla \tildphi(x) \|_{L^2(B_n)}+ \| \tildphi \|_{L^6(B_n)}.
\end{align*}
As $1_{B_n} \to 0$ almost everywhere, \eqref{E:no deform 3} follows from the dominated convergence theorem.

\textbf{Case 4:}  Passing to a subsequence, we may assume $N_nd(x_n)\to d_\infty>0$.  By weak sequential compactness of balls in $\dot H^1(\R^3)$, we are guaranteed that we can find a subsequence and a $\tildphi\in \dot H^1(\R^3)$ so that
$g_n \rightharpoonup \tildphi$ weakly in this space.  However, the proposition claims that $\tildphi\in \dot H^1_D(\HH)$.  This is a closed subspace isometrically
embedded in $\dot H^1(\R^3)$; indeed,
$$
\dot H^1_D(\HH) = \bigl\{ g\in\dot H^1(\R^3) : {\textstyle\int_{\R^3}} g(x)\psi(x) \,dx = 0 \text{ for all } \psi\in C^\infty_c(-\HH) \bigr\}.
$$
Using this characterization, it is not difficult to see that $\tildphi\in \dot H^1_D(\HH)$ since for any compact set $K$ in the left halfplane,
$K\subset \Omega_n^c$ for $n$ sufficiently large.  Here $\Omega_n:=N_n R_n^{-1}(\Omega-\{x_n^*\})$, which is where $g_n$ is supported.

As $\tildphi\in \dot H^1_D(\HH)$ we have $\phi_n \in \dot H^1_D(\Omega)$, as is easily seen from
$$
x\in\HH \iff N_n^{-1} R_n^{} x + x^*_n \in \HH_n := \{ y : (x_n - x_n^*)\cdot (y-x_n^*) >0 \} \subseteq \Omega;
$$
indeed, $\partial \HH_n$ is the tangent plane to $\partial\Omega$ at $x_n^*$.  This inclusion further shows that
\begin{align}\label{6:20}
\bigl\| \tildphi \bigr\|_{\dot H^1_D(\HH)} = \bigl\| \phi_n \bigr\|_{\dot H^1_D(\HH_n)} = \bigl\| \phi_n \bigr\|_{\dot H^1_D(\Omega)}.
\end{align}

To prove claim \eqref{nontri} it thus suffices to show a lower bound on $\| \tildphi \|_{\dot H^1_D(\HH)}$.  To this end, let
$h:=P_1^{\HH}\delta_{d_\infty e_3}$.  From the Bernstein inequality we have
\begin{align}\label{6:40}
\|(-\Delta_{\HH})^{-\frac 12}h\|_{L^2(\Omega)}\lsm 1.
\end{align}
In particular, $h\in \dot H^{-1}_D(\HH)$.  Now let $\tilde x_n:= N_nR_n^{-1}(x_n-x_n^*)$; by hypothesis, $\tilde x_n \to d_\infty e_3$.  Using Proposition~\ref{P:converg} we obtain
\begin{align*}
\langle \tilde \phi,h\rangle
&=\lim_{n\to\infty}\Bigl\{\langle g_n,P_1^{\Omega_n}\delta_{\tilde x_n}\rangle + \langle g_n,[P_{1}^{\HH}-P_1^{\Omega_n}]\delta_{d_\infty e_3}\rangle
+ \langle g_n,P_1^{\Omega_n}[\delta_{d_\infty e_3} - \delta_{\tilde x_n}]\rangle\Bigr\}\\
&=\lim_{n\to\infty}\Bigl\{N_n^{-\frac12}(e^{it_n\ld}\pnno f_n)(x_n)+\langle g_n,P_1^{\Omega_n}[\delta_{d_\infty e_3} - \delta_{\tilde x_n}]\rangle\Bigr\}.
\end{align*}
Arguing as in the treatment of \eqref{6:37} and applying \eqref{elliptic est} to $v(x)=(P_1^{\Omega_n}g_n)(x+\tilde x_n)$ with $R=\frac12 N_nd(x_n)$, for $n$ sufficiently large we obtain
\begin{align*}
|\langle g_n,P_1^{\Omega_n}[\delta_{d_\infty e_3} - \delta_{\tilde x_n}]\rangle|&\lsm A \bigl(d_\infty^{-1}+d_\infty\bigr) |d_\infty e_3- \tilde x_n|\to 0\qtq{as} n\to \infty.
\end{align*}
Therefore, we have
\begin{align*}
|\langle \tilde \phi,h\rangle|\gtrsim \eps (\tfrac{\eps}A)^{\frac78},
\end{align*}
which together with \eqref{6:20} and \eqref{6:40} yields \eqref{nontri}.

Claim \eqref{dech} is elementary; indeed,
\begin{align*}
\|f_n\|_{\hd}^2 -\|f_n-\phi_n\|_{\hd}^2 &= 2\lng f_n, \phi_n\rng_{\hd}-\|\phi_n\|_{\hd}^2\\
&= 2\lng g_n,\, \tildphi\rng_{\dot H^1_D(\Omega_n)}-\|\tilde\phi\|_{\dot H^{1}_D(\HH)}^2 \to \|\tilde\phi\|_{\dot H^{1}_D(\HH)}^2.
\end{align*}

The proof of \eqref{dect} differs little from the cases treated previously: One uses the Rellich Lemma and Corollary~\ref{C:LF} to
show $e^{it\Delta_{\Omega_n}} g_n\to e^{it\Delta_\HH}\tildphi$ almost everywhere and then the Fatou Lemma of Br\'ezis and Lieb to
see that
$$
\text{LHS\eqref{dect}} = \| e^{it\Delta_\HH} \tildphi \|_{L^{10}(\R\times\HH)}^{10}.
$$
The lower bound on this quantity comes from pairing with $h$; see Cases~1 and~2.
\end{proof}

To prove a linear profile decomposition for the propagator $e^{it\Delta_\Omega}$ we will also need the following weak convergence results.

\begin{lem}[Weak convergence]\label{L:converg}
Assume $\Omega_n\equiv\Omega$ or $\{\Omega_n\}$ conforms to one of the three scenarios considered in Proposition~\ref{P:convdomain}.  Let $f\in C_c^\infty(\tlim \Omega_n)$ and let $\{(t_n,x_n)\}_{n\geq 1}\subset\R\times\R^3$.  Then
\begin{align}\label{lc}
e^{it_n\Delta_{\Omega_n}}f(x+x_n) \rightharpoonup 0 \quad \text{weakly in } \dot H^1(\R^3) \quad \text{as } n\to \infty
\end{align}
whenever $|t_n|\to \infty$ or $|x_n|\to \infty$.
\end{lem}

\begin{proof}
By the definition of $\tlim\Omega_n$, we have $f\in C^\infty_c(\Omega_n)$ for $n$ sufficiently large.  Let $\Omega_\infty$ denote the limit of $\Omega_n$ in the sense of Definition~\ref{D:converg}.

We first prove \eqref{lc} when $t_n\to \infty$; the proof when $t_n\to-\infty$ follows symmetrically.  Let $\psi\in C_c^\infty(\R^3)$ and let
$$
F_n(t):=\langle e^{it\Delta_{\Omega_n}}f(x+x_n), \psi\rangle_{\dot H^1(\R^3)}.
$$
To establish \eqref{lc}, we need to show
\begin{align}\label{lc1}
F_n(t_n)\to 0 \qtq{as} n\to \infty.
\end{align}

We compute
\begin{align*}
|\partial_t F_n(t)|&= \bigl|\langle i\Delta_{\Omega_n} e^{it\Delta_{\Omega_n}}f(x+x_n), \psi\rangle_{\dot H^1(\R^3)} \bigr|\\
&= \bigl|\langle \Delta_{\Omega_n} e^{it\Delta_{\Omega_n}}f(x+x_n), \Delta \psi\rangle_{L^2(\R^3)} \bigr|
	\lesssim \|f\|_{\dot H^2} \|\psi\|_{\dot H^2}\lesssim_{f,\psi}1.
\end{align*}
On the other hand,
\begin{align*}
\|F\|_{L_t^{\frac{10} 3}([t_n,\infty))}
&\lsm \|e^{it\Delta_{\Omega_n}}f\|_{L_{t,x}^{\frac{10}3}([t_n,\infty)\times\R^3)}\|\Delta \psi\|_{L_x^{\frac{10}7}(\R^3)}\\
&\lsm_\psi \|[e^{it\Delta_{\Omega_n}}-e^{it\Delta_{\Omega_\infty}}]f\|_{L_{t,x}^{\frac{10}3}([0,\infty)\times\R^3)}
+\|e^{it\Delta_{\Omega_\infty}}f\|_{L_{t,x}^{\frac{10}3}([t_n,\infty)\times\R^3)}.
\end{align*}
The first term converges to zero as $n\to \infty$ by Theorem~\ref{T:LF}, while convergence to zero of the second term follows from the Strichartz inequality combined with the dominated convergence theorem.  Putting everything together, we derive \eqref{lc1} and so \eqref{lc} when $t_n\to \infty$.

Now assume $\{t_n\}_{n\geq 1}$ is bounded, but $|x_n|\to \infty$ as $n\to \infty$.  Without loss of generality, we may assume $t_n\to t_\infty\in \R$
as $n\to \infty$.  Let $\psi\in C_c^\infty(\R^3)$ and $R>0$ such that $\supp\psi\subseteq B(0,R)$.  We write
\begin{align*}
\langle e^{it_n\Delta_{\Omega_n}}f(x+x_n), \psi\rangle_{\dot H^1(\R^3)}
&=\langle e^{it_\infty\Delta_{\Omega_\infty}}f(x+x_n), \psi\rangle_{\dot H^1(\R^3)} \\
&\quad +\langle [e^{it_\infty\Delta_{\Omega_n}}-e^{it_\infty\Delta_{\Omega_\infty}}]f(x+x_n), \psi\rangle_{\dot H^1(\R^3)}\\
&\quad +\langle [e^{it_n\Delta_{\Omega_n}}-e^{it_\infty\Delta_{\Omega_n}}]f(x+x_n), \psi\rangle_{\dot H^1(\R^3)}.
\end{align*}
By the Cauchy--Schwarz inequality,
\begin{align*}
\bigl|\langle e^{it_\infty\Delta_{\Omega_\infty}}f(x+x_n), \psi\rangle_{\dot H^1(\R^3)}\bigr|
\lsm \|e^{it_\infty\Delta_{\Omega_{\infty}}}f\|_{L^2(|x|\ge |x_n|-R)} \|\Delta \psi\|_{L^2(\R^3)},
\end{align*}
which converges to zero as $n\to \infty$, by the monotone convergence theorem.  By duality and Proposition~\ref{P:converg},
\begin{align*}
\bigl\langle [e^{it_\infty\Delta_{\Omega_n}}-&e^{it_\infty\Delta_{\Omega_\infty}}]f(x+x_n), \psi\rangle_{\dot H^1(\R^3)}\bigr|\\
&\lsm \|[e^{it_\infty\Delta_{\Omega_n}}-e^{it_\infty\Delta_{\Omega_\infty}}]f\|_{\dot H^{-1}(\R^3)} \|\Delta\psi\|_{\dot H^1(\R^3)} \to 0 \qtq{as} n\to \infty.
\end{align*}
Finally, by the fundamental theorem of calculus,
\begin{align*}
\bigl|\langle [e^{it_n\Delta_{\Omega_n}}-e^{it_\infty\Delta_{\Omega_n}}]f(x+x_n), \psi\rangle_{\dot H^1(\R^3)}\bigr|
\lsm |t_n-t_\infty| \|\Delta_{\Omega_n} f\|_{L^2} \|\Delta \psi\|_{L^2},
\end{align*}
which converges to zero as $n\to \infty$.  Putting everything together we deduce
$$
\langle e^{it_n\Delta_{\Omega_n}}f(x+x_n), \psi\rangle_{\dot H^1(\R^3)} \to 0 \qtq{as} n\to \infty.
$$
This completes the proof of the lemma.
\end{proof}

\begin{lem}[Weak convergence]\label{L:compact} Assume $\Omega_n\equiv\Omega$ or $\{\Omega_n\}$ conforms to one of the three scenarios considered in  Proposition~\ref{P:convdomain}. Let $f_n\in \dot H_D^1(\Omega_n)$ be such that $f_n\rightharpoonup 0$ weakly in $\dot H^1(\R^3)$ and let $t_n\to t_\infty\in \R$. Then
\begin{align*}
e^{it_n\Delta_{\Omega_n}} f_n\rightharpoonup 0 \qtq{weakly in} \dot H^1(\R^3).
\end{align*}
\end{lem}

\begin{proof} For any $\psi\in C_c^{\infty}(\R^3)$,
\begin{align*}
\bigl|\langle [e^{it_n\Delta_{\Omega_n}}-e^{it_\infty\Delta_{\Omega_n}}]f_n, \psi\rangle_{\dot H^1(\R^3)}\bigr|
&\lsm \|[e^{it_n\Delta_{\Omega_n}}-e^{it_\infty\Delta_{\Omega_n}}]f_n\|_{L^2} \|\Delta\psi\|_{L^2}\\
&\lsm |t_n-t_\infty|^{\frac12} \|(-\Delta_{\Omega_n})^{\frac12}f_n\|_{L^2} \|\Delta\psi\|_{L^2},
\end{align*}
which converges to zero as $n\to \infty$.  To obtain the last inequality above, we have used the spectral theorem together with the elementary inequality $|e^{it_n\lambda}-e^{it_\infty\lambda}|\lsm |t_n-t_\infty|^{1/2}\lambda^{1/2}$ for $\lambda\geq 0$.  Thus, we are left to prove
\begin{align*}
\int_{\R^3} \nabla \bigl[e^{it_\infty\Delta_{\Omega_n}} f_n\bigr](x) \nabla \bar\psi(x)\,dx
= \int_{\R^3} e^{it_\infty\Delta_{\Omega_n}} f_n (x) (-\Delta\bar\psi)(x)\, dx \to 0 \qtq{as} n\to \infty
\end{align*}
for all $\psi\in C_c^\infty(\R^3)$.  By Sobolev embedding,
$$
\|e^{it_\infty\Delta_{\Omega_n}} f_n\|_{L^6}\lsm \|f_n\|_{\dot H^1(\R^3)}\lsm 1 \qtq{uniformly in} n\geq 1,
$$
and so using a density argument and the dominated convergence theorem (using the fact that the measure of $\Omega_n\triangle(\tlim \Omega_n)$ converges to zero), it suffices to show
\begin{align}\label{9:38am}
\int_{\R^3} e^{it_\infty\Delta_{\Omega_n}} f_n (x) \bar\psi(x)\, dx \to 0 \qtq{as} n\to \infty
\end{align}
for all $\psi\in C_c^\infty(\tlim \Omega_n)$.  To see that \eqref{9:38am} is true, we write
\begin{align*}
\langle e^{it_\infty\Delta_{\Omega_n}} f_n, \psi \rangle =\langle f_n, [e^{-it_\infty\Delta_{\Omega_n}} -e^{-it_\infty\Delta_{\Omega_\infty}}]\psi \rangle
+ \langle f_n,e^{-it_\infty\Delta_{\Omega_\infty}}\psi \rangle,
\end{align*}
where $\Omega_\infty$ denotes the limit of $\Omega_n$ in the sense of Definition~\ref{D:converg}.  The first term converges to zero by Proposition~\ref{P:converg}.  As $f_n\rightharpoonup 0$ in $\dot H^1(\R^3)$, to see that the second term converges to zero, we merely need to prove that $e^{-it_\infty\Delta_{\Omega_\infty}}\psi\in \dot H^{-1}(\R^3)$ for all $\psi\in C_c^\infty(\tlim \Omega_n)$.  Toward this end, we use the Mikhlin multiplier theorem and Bernstein's inequality to estimate
\begin{align*}
\|e^{-it_\infty\Delta_{\Omega_\infty}}\psi\|_{\dot H^{-1}(\R^3)}
&\lsm\|e^{-it_\infty\Delta_{\Omega_\infty}}P_{\leq 1}^{\Omega_\infty} \psi\|_{L^{\frac65}}+\sum_{N\geq 1}\|e^{-it_\infty\Delta_{\Omega_\infty}}P_N^{\Omega_\infty}\psi\|_{L^{\frac65}}\\
&\lsm\|\psi\|_{L^{\frac65}}+\sum_{N\geq 1} \langle N^2t_\infty\rangle^2\|P_N^{\Omega_\infty}\psi\|_{L^{\frac65}}\\
&\lsm \|\psi\|_{L^{\frac65}} + \|(-\Delta_{\Omega_\infty})^3\psi\|_{L^{\frac65}}\lsm_\psi 1.
\end{align*}
This completes the proof of the lemma.
\end{proof}

Finally, we turn to the linear profile decomposition for the propagator $e^{it\Delta_\Omega}$ in $\dot H^1_D(\Omega)$.  This is proved by the inductive application of Proposition~\ref{P:inverse Strichartz}. To handle the variety of cases in as systematic a way as possible, we introduce operators $G_n^j$ that act unitarily in $\dot H^1(\R^3)$.

\begin{thm}[$\dot H^1_D(\Omega)$ linear profile decomposition]\label{T:LPD}
Let $\{f_n\}$ be a bounded sequence in $\dot H^1_D(\Omega)$.  After passing to a subsequence, there exist $J^*\in \{0, 1, 2, \ldots,\infty\}$,
$\{\phi_n^j\}_{j=1}^{J^*}\subset \dot H_D^1(\Omega)$, $\{\lambda_n^j\}_{j=1}^{J^*}\subset(0,\infty)$, and $\{(t_n^j,x_n^j)\}_{j=1}^{J^*}\subset \R\times \Omega$
conforming to one of the following four cases for each $j$:
\begin{CI}
\item Case 1: $\lambda_n^j\equiv \lambda_\infty^j$, $x_n^j\to x_\infty^j$, and there is a $\phi^j\in \dot H^1_D(\Omega)$ so that
\begin{align*}
\phi_n^j = e^{it_n^j (\lambda_n^j)^2\Delta_\Omega}\phi^j.
\end{align*}
We define $[G_n^j f] (x) :=  (\lambda_n^j)^{-\frac 12} f\bigl(\tfrac{x-x_n^j}{\lambda_n^j} \bigr)$ and $\Omega_n^j:=(\lambda_n^j)^{-1}(\Omega - \{x_n^j\})$.

\item Case 2: $\lambda_n^j\to \infty$, $-(\lambda_n^j)^{-1}x_n^j \to x^j_\infty\in\R^3$, and there is a $\phi^j\in \dot H^1(\R^3)$ so that
\begin{align*}
\phi_n^j(x)= G_n^j \bigl[e^{it_n^j \Delta_{\Omega_n^j}} (\chi_n^j \phi^j)\bigr] (x) \qtq{with} [G_n^j f] (x) :=  (\lambda_n^j)^{-\frac 12} f\bigl(\tfrac{x-x_n^j}{\lambda_n^j} \bigr),
\end{align*}
$\Omega_n^j := (\lambda_n^j)^{-1}(\Omega - \{x_n^j\})$, $\chi_n^j(x)=\chi(\lambda_n^jx+x_n^j)$, and $\chi(x)=\Theta(\tfrac{d(x)}{\diam(\Omega^c)})$.

\item Case 3: $\frac{d(x_n^j)}{\lambda_n^j}\to \infty$ and there is a $\phi^j\in \dot H^1(\R^3)$ so that
\begin{align*}
\phi_n^j(x)= G_n^j \bigl[e^{it_n^j \Delta_{\Omega_n^j}} (\chi_n^j \phi^j)\bigr] (x) \qtq{with} [G_n^j f] (x) :=  (\lambda_n^j)^{-\frac 12} f\bigl(\tfrac{x-x_n^j}{\lambda_n^j} \bigr),
\end{align*}
$\Omega_n^j := (\lambda_n^j)^{-1}(\Omega - \{x_n^j\})$, and $\chi_n^j(x)=1-\Theta(\tfrac{\lambda_n^j|x|}{d(x_n^j)})$.

\item Case 4: $\lambda_n^j\to 0$, $\frac{d(x_n^j)}{\lambda_n^j}\to d_\infty^j>0$, and there is a $\phi^j\in \dot H^1_D(\HH)$ so that
\begin{align*}
\phi_n^j(x)= G_n^j \bigl[ e^{it_n^j \Delta_{\Omega_n^j}} \phi^j\bigr] (x) \qtq{with} [G_n^j f](x) := (\lambda_n^j)^{-\frac12} f\bigl(\tfrac{(R_n^j)^{-1}(x-(x_n^j)^*)}{\lambda_n^j}\bigr),
\end{align*}
$\Omega_n^j := (\lambda_n^j)^{-1}(R_n^j)^{-1}(\Omega - \{(x_n^j)^*\})$, $(x_n^j)^*\in \partial\Omega$ is defined by $d(x_n^j)=|x_n^j-(x_n^j)^*|$,
and $R_n^j\in SO(3)$ satisfies $R_n^j e_3=\tfrac{x_n^j-(x_n^j)^*}{|x_n^j-(x_n^j)^*|}$.
\end{CI}
Further, for any finite $0\le J\le J^*$, we have the decomposition
\begin{align*}
f_n=\sum_{j=1}^ J \phi_n^j +w_n^J,
\end{align*}
with $w_n^J\in \dot H^1_D(\Omega)$ satisfying
\begin{gather}
\lim_{J\to J^*} \limsup_{n\to\infty} \|e^{it\Delta_{\Omega}}w_n^J\|_{L_{t,x}^{10}(\R\times\Omega)}=0,  \label{E:LP1}\\
\lim_{n\to\infty}\Bigl\{\|f_n\|_{\dot H^1_D(\Omega)}^2-\sum_{j=1}^J\|\phi_n^j\|_{\dot H_D^1(\Omega)}^2-\|w_n^J\|_{\dot H^1_D(\Omega)}^2\Bigr\}=0, \label{E:LP2}\\
\lim_{n\to\infty}\Bigl\{\|f_n\|_{L^6(\Omega)}^6-\sum_{j=1}^J \|\phi_n^j\|_{L^6(\Omega)}^6-\|w_n^J\|_{L^6(\Omega)}^6\Bigr\}=0, \label{E:LP3}\\
e^{it_n^J\Delta_{\Omega_n^J}}(G_n^J)^{-1}w_n^J\rightharpoonup 0 \qtq{weakly in} \dot H^1(\R^3), \label{E:LP4}
\end{gather}
and for all $j\neq k$ we have the asymptotic orthogonality property
\begin{align}\label{E:LP5}
\lim_{n\to\infty} \ \frac{\lambda_n^j}{\lambda_n^k}+\frac{\lambda_n^k}{\lambda_n^j}+
 	\frac{|x_n^j-x_n^k|^2}{\lambda_n^j\lambda_n^k}+\frac{|t_n^j(\lambda_n^j)^2-t_n^k(\lambda_n^k)^2|}{\lambda_n^j\lambda_n^k}=\infty.
\end{align}
Lastly, we may additionally assume that for each $j$ either $t_n^j\equiv 0$ or $t_n^j\to \pm \infty$.
\end{thm}

\begin{proof} We will proceed inductively and extract one bubble at a time. To start with, we set $w_n^0:=f_n$.  Suppose
we have a decomposition up to level $J\geq 0$ obeying \eqref{E:LP2} through \eqref{E:LP4}.  (Note that conditions \eqref{E:LP1}
and \eqref{E:LP5} will be verified at the end.)  Passing to a subsequence if necessary, we set
\begin{align*}
A_J:=\lim_{n\to\infty} \|w_n^J\|_{\dot H^1_D(\Omega)} \qtq{and} \eps_J:=\lim_{n\to \infty} \|e^{it\Delta_{\Omega}}w_n^J\|_{L_{t,x}^{10}(\R\times\Omega)}.
\end{align*}
If $\eps_J=0$, we stop and set $J^*=J$. If not, we apply the inverse Strichartz inequality Proposition \ref{P:inverse Strichartz} to
$w_n^J$. Passing to a subsequence in $n$ we find $\{\phi_n^{J+1}\}\subset \dot H^1_D(\Omega)$, $\{\lambda_n^{J+1}\}\subset 2^{\mathbb Z}$, and
$\{(t_n^{J+1}, x_n^{J+1})\}\subset\R\times\Omega$, which conform to one of the four cases listed in the theorem.  Note that we rename
the parameters given by Proposition~\ref{P:inverse Strichartz} as follows: $\lambda_n^{J+1} := N_n^{-1}$ and $t_n^{J+1} := - N_n^{2} t_n$.

The profiles are defined as weak limits in the following way:
\begin{align*}
\tilde \phi^{J+1}=\wlim_{n\to\infty}(G_n^{J+1})^{-1} \bigl[ e^{-it_n^{J+1}(\lambda_n^{J+1})^2\Delta_\Omega}w_n^J\bigr]
=\wlim_{n\to\infty}e^{-it_n^{J+1}\Delta_{\Omega_n^{J+1}}}[(G_n^{J+1})^{-1}w_n^J],
\end{align*}
where $G_n^{J+1}$ is defined in the statement of the theorem.  In Cases 2, 3, 4, we define $\phi^{J+1}:=\tilde \phi^{J+1}$, while in Case 1,
$$
\phi^{J+1}(x):= G_\infty^{J+1}\tilde \phi^{J+1}(x):=(\lambda_\infty^{J+1})^{-\frac12} \tilde \phi^{J+1}\bigl(\tfrac{x-x_\infty^{J+1}}{\lambda_\infty^{J+1}} \bigr).
$$
Finally, $\phi_n^{J+1}$ is defined as in the statement of the theorem.  Note that in Case 1, we can rewrite this definition as
$$
\phi_n^{J+1}=e^{it_n^{J+1}(\lambda_n^{J+1})^2\Delta_\Omega}\phi^{J+1}=G_\infty^{J+1}e^{it_n^{J+1}\Delta_{\Omega_\infty^{J+1}}}\tilde \phi^{J+1},
$$
where $\Omega_\infty^{J+1}:=(\lambda_\infty^{J+1})^{-1}(\Omega - \{x_\infty^{J+1}\})$.  Note that in all four cases,
\begin{align}\label{strong}
\lim_{n\to \infty}\|e^{-it_n^{J+1}\Delta_{\Omega_n^{J+1}}}(G_n^{J+1})^{-1}\phi_n^{J+1}-\tilde \phi^{J+1}\|_{\dot H^1(\R^3)}=0;
\end{align}
see also \eqref{E:no deform} and \eqref{E:no deform 3} for Cases 2 and 3.

Now define $w_n^{J+1}:=w_n^J-\phi_n^{J+1}$.  By \eqref{strong} and the construction of $\tilde \phi^{J+1}$ in each case,
\begin{align*}
e^{-it_n^{J+1}\Delta_{\Omega_n^{J+1}}}(G_n^{J+1})^{-1}w_n^{J+1} \rightharpoonup 0 \quad \text{weakly in }\dot H^1(\R^3).
\end{align*}
This proves \eqref{E:LP4} at the level $J+1$.  Moreover, from Proposition \ref{P:inverse Strichartz} we also have
\begin{align*}
\lim_{n\to\infty}\Bigl\{\|w_n^J\|_{\dot H^1_D(\Omega)}^2-\|\phi_n^{J+1}\|_{\dot H^1_D(\Omega)}^2-\|w_n^{J+1}\|_{\dot H^1_D(\Omega)}^2\Bigr\}=0.
\end{align*}
This together with the inductive hypothesis give \eqref{E:LP2} at the level $J+1$.  A similar argument establishes \eqref{E:LP3} at the same level.

From Proposition~\ref{P:inverse Strichartz}, passing to a further subsequence we have
\begin{equation}\label{new a,eps}
\begin{aligned}
&A_{J+1}^2=\lim_{n\to\infty}\|w_n^{J+1}\|_{\dot H^1_D(\Omega)}^2\le A_J^2\Bigl[1 -C \bigl(\tfrac{\eps_J}{A_J}\bigr)^{\frac{15}4}\Bigr]\leq A_J^2\\
&\eps_{J+1}^{10}=\lim_{n\to\infty}\|e^{it\Delta_{\Omega}}w_n^{J+1}\|_{L_{t,x}^{10}(\R\times\Omega)}^{10}\le\eps_J^{10}\Bigl[1-C\bigl(\tfrac{\eps_J}{A_J}\bigr)^{\frac{35}4}\Bigr].
\end{aligned}
\end{equation}
If $\eps_{J+1}=0$ we stop and set $J^*=J+1$; moreover, \eqref{E:LP1} is automatic.  If $\eps_{J+1}>0$ we continue the induction.  If the algorithm does not terminate in finitely many steps, we set  $J^*=\infty$; in this case, \eqref{new a,eps} implies $\eps_J\to 0$ as $J\to \infty$ and so \eqref{E:LP1} follows.

Next we verify the asymptotic orthogonality condition \eqref{E:LP5}. We argue by contradiction.  Assume \eqref{E:LP5} fails to be true for some pair $(j,k)$.  Without loss of generality, we may assume $j<k$ and \eqref{E:LP5} holds for all pairs $(j,l)$ with $j<l<k$.  Passing to a subsequence, we may assume
\begin{align}\label{cg}
\frac{\lambda_n^j}{\lambda_n^k}\to \lambda_0\in (0,\infty), \quad \frac{x_n^j-x_n^k}{\sqrt{\lambda_n^j\lambda_n^k}}\to x_0, \qtq{and}
\frac{t_n^j(\lambda_n^j)^2-t_n^k(\lambda_n^k)^2}{\lambda_n^j\lambda_n^k}\to t_0.
\end{align}

From the inductive relation
\begin{align*}
w_n^{k-1}=w_n^j-\sum_{l=j+1}^{k-1}\phi_n^l
\end{align*}
and the definition for $\tilde \phi^k$, we obtain
\begin{align}
\tilde \phi^k&=\wlim_{n\to\infty}e^{-it_n^k\Delta_{\Omega_n^k}}[(G_n^k)^{-1}w_n^{k-1}]\notag\\
&=\wlim_{n\to\infty}e^{-it_n^k\Delta_{\Omega_n^k}}[(G_n^k)^{-1}w_n^j]\label{tp1}\\
&\quad-\sum_{l=j+1}^{k-1} \wlim_{n\to \infty}e^{-it_n^k\Delta_{\Omega_n^k}}[(G_n^k)^{-1}\phi_n^l]\label{tp2}.
\end{align}
We will prove that these weak limits are zero and so obtain a contradiction to the nontriviality of $\tilde \phi^k$.

We rewrite \eqref{tp1} as follows
\begin{align*}
e^{-it_n^k\Delta_{\Omega_n^k}}[(G_n^k)^{-1}w_n^j]
&=e^{-it_n^k\Delta_{\Omega_n^k}}(G_n^k)^{-1}G_n^je^{it_n^j\Delta_{\Omega_n^j}}[e^{-it_n^j\Delta_{\Omega_n^j}}(G_n^j)^{-1}w_n^j]\\
&=(G_n^k)^{-1}G_n^je^{i\bigl(t_n^j-t_n^k\tfrac{(\lambda_n^k)^2}{(\lambda_n^j)^2}\bigr)\Delta_{{\Omega_n^j}}}[e^{-it_n^j\Delta_{\Omega_n^j}}(G_n^j)^{-1}w_n^j].
\end{align*}
Note that by \eqref{cg},
\begin{align*}
t_n^j-t_n^k\frac{(\lambda_n^k)^2}{(\lambda_n^j)^2}=\frac{t_n^j(\lambda_n^j)^2-t_n^k(\lambda_n^k)^2}{\lambda_n^j\lambda_n^k}\cdot\frac{\lambda_n^k}
{\lambda_n^j}\to \frac{t_0}{\lambda_0}.
\end{align*}
Using this together with \eqref{E:LP4}, Lemma~\ref{L:compact}, and the fact that the adjoints of the unitary operators $(G_n^k)^{-1}G_n^j$ converge strongly, we obtain $\eqref{tp1}=0$.

To complete the proof of \eqref{E:LP5}, it remains to show $\eqref{tp2}=0$.  For all $j<l<k$ we write
\begin{align*}
e^{-it_n^k{\Delta_{\Omega_n^k}}}(G_n^k)^{-1}\phi_n^l
=(G_n^k)^{-1}G_n^je^{i\bigl(t_n^j-t_n^k\tfrac{(\lambda_n^k)^2}{(\lambda_n^j)^2}\bigr)\Delta_{{\Omega_n^j}}}[e^{-it_n^j\Delta_{\Omega_n^j}}(G_n^j)^{-1}\phi_n^l].
\end{align*}
Arguing as for \eqref{tp1}, it thus suffices to show
\begin{align*}
e^{-it_n^j\Delta_{\Omega_n^j}}(G_n^j)^{-1}\phi_n^l\rightharpoonup 0 \qtq{weakly in} \dot H^1(\R^3).
\end{align*}
Using a density argument, this reduces to
\begin{align}\label{need11}
I_n:=e^{-it_n^j\Delta_{\Omega_n^j}}(G_n^j)^{-1}G_n^le^{it_n^l\Delta_{\Omega_n^l}}\phi\rightharpoonup 0 \qtq{weakly in} \dot H^1(\R^3),
\end{align}
for all $\phi\in C_c^\infty(\tlim \Omega_n^l)$.  In Case 1, we also used the fact that $(G_n^l)^{-1} G_\infty^l$ converges strongly to the identity.

Depending on which cases $j$ and $l$ fall into, we can rewrite $I_n$ as follows:
\begin{CI}
\item Case a): If both $j$ and $l$ conform to Case 1, 2, or 3, then
\begin{align*}
I_n=\biggl(\frac{\lambda_n^j}{\lambda_n^l}\biggr)^{\frac12}\biggl[e^{i\bigl(t_n^l-t_n^j\bigl(\frac{\lambda_n^j}
{\lambda_n^l}\bigr)^2\bigr)\Delta_{\Omega_n^l}}\phi\biggr]\biggl(\frac{\lambda_n^j x+x_n^j- x_n^l}{\lambda_n^l}\biggr).
\end{align*}

\item Case b): If $j$ conforms to Case 1, 2, or 3 and $l$ to Case 4, then
\begin{align*}
I_n=\biggl(\frac{\lambda_n^j}{\lambda_n^l}\biggr)^{\frac12}\biggl[e^{i\bigl(t_n^l-t_n^j\bigl(\frac{\lambda_n^j}
{\lambda_n^l}\bigr)^2\bigr) \Delta_{\Omega_n^l}}\phi\biggr]\biggl(\frac{(R_n^l)^{-1}(\lambda_n^j x+x_n^j-(x_n^l)^*)}{\lambda_n^l}\biggr).
\end{align*}

\item Case c): If $j$ conforms to Case 4 and $l$ to Case 1, 2, or 3, then
\begin{align*}
I_n=\biggl(\frac{\lambda_n^j}{\lambda_n^l}\biggr)^{\frac12}\biggl[e^{i\bigl(t_n^l-t_n^j\bigl(\frac{\lambda_n^j}
{\lambda_n^l}\bigr)^2\bigr) \Delta_{\Omega_n^l}}\phi\biggr]\biggl(\frac{R_n^j\lambda_n^j x+(x_n^j)^*-x_n^l}{\lambda_n^l}\biggr).
\end{align*}

\item Case d): If both $j$ and $l$ conform to Case 4, then
\begin{align*}
I_n=\biggl(\frac{\lambda_n^j}{\lambda_n^l}\biggr)^{\frac12}\biggl[e^{i\bigl(t_n^l-t_n^j\bigl(\frac{\lambda_n^j}
{\lambda_n^l}\bigr)^2\bigr) \Delta_{\Omega_n^l}}\phi\biggr]\biggl(\frac{(R_n^l)^{-1}(R_n^j\lambda_n^j x+(x_n^j)^*-(x_n^l)^*)}{\lambda_n^l}\biggr).
\end{align*}
\end{CI}

We first prove \eqref{need11} when the scaling parameters are not comparable, that is,
\begin{align*}
\lim_{n\to\infty}\frac{\lambda_n^j}{\lambda_n^l}+\frac{\lambda_n^l}{\lambda_n^j}=\infty.
\end{align*}
We treat all cases simultaneously.  By Cauchy--Schwarz,
\begin{align*}
\bigl|\langle I_n, \psi\rangle_{\dot H^1(\R^3)}\bigr|
&\lsm \min\Bigl\{\|\Delta I_n\|_{L^2(\R^3)}\|\psi\|_{L^2(\R^3)}, \|I_n\|_{L^2(\R^3)}\|\Delta\psi\|_{L^2(\R^3)}\Bigr\}\\
&\lsm \min\biggl\{\frac{\lambda_n^j}{\lambda_n^l}\|\Delta\phi\|_{L^2(\R^3)}\|\psi\|_{L^2(\R^3)}, \frac{\lambda_n^l}{\lambda_n^j}\|\phi\|_{L^2(\R^3)}\|\Delta\psi\|_{L^2(\R^3)}\biggr\},
\end{align*}
which converges to zero as $n\to \infty$, for all $\psi\in C_c^\infty(\R^3)$.  Thus, in this case $\eqref{tp2}=0$ and we get the desired contradiction.

Henceforth we may assume
\begin{align*}
\lim_{n\to \infty}\frac{\lambda_n^j}{\lambda_n^l}=\lambda_0\in (0,\infty).
\end{align*}
We now suppose the time parameters diverge, that is,
\begin{align*}
\lim_{n\to \infty}\frac{|t_n^j(\lambda_n^j)^2-t_n^l(\lambda_n^l)^2|}{\lambda_n^j\lambda_n^l}=\infty;
\end{align*}
then we also have
\begin{align*}
\biggl|t_n^l-t_n^j\biggl(\frac{\lambda_n^j}{\lambda_n^l}\biggr)^2\biggr|
=\frac{|t_n^l(\lambda_n^l)^2-t_n^j(\lambda_n^j)^2|}{\lambda_n^l\lambda_n^j}\cdot\frac{\lambda_n^j}{\lambda_n^l}\to \infty \qtq{as} n\to\infty.
\end{align*}
We first discuss Case a). Under the above condition, \eqref{need11} follows from
\begin{align*}
\lambda_0^{\frac 12}\biggl(e^{i\bigl(t_n^l-t_n^j\bigl(\frac{\lambda_n^j}{\lambda_n^l}\bigr)^2\bigr)\Delta_{\Omega_n^l}}\phi\biggr)\bigl(\lambda_0 x+(\lambda_n^l)^{-1}(x_n^j-x_n^l)\bigr)\rightharpoonup 0 \qtq{weakly in} \dot H^1(\R^3),
\end{align*}
which is an immediate consequence of Lemma \ref{L:converg}.  In Cases b), c), and d), the proof proceeds similarly since $SO(3)$ is a compact group; indeed, passing to a subsequence we may assume that $R_n^j\to R_0$ and $R_n^l\to R_1$, which places us in the same situation as in Case a).

Finally, we deal with the situation when
\begin{align}\label{cdition}
\frac{\lambda_n^j}{\lambda_n^l}\to \lambda_0, \quad \frac{t_n^l(\lambda_n^l)^2-t_n^j(\lambda_n^j)^2}{\lambda_n^j\lambda_n^l}\to t_0,
\qtq{but} \frac{|x_n^j-x_n^l|^2}{\lambda_n^j\lambda_n^l}\to \infty.
\end{align}
Then we also have $t_n^l-t_n^j(\lambda_n^j)^2/(\lambda_n^l)^2\to \lambda_0t_0$. Thus, in Case a) it suffices to show
\begin{align}\label{524}
\lambda_0^{\frac 12}
e^{it_0\lambda_0\Delta_{\Omega_n^l}}\phi(\lambda_0x+y_n)\rightharpoonup0 \qtq{weakly in} \dot H^1(\R^3),
\end{align}
where
\begin{align*}
y_n:=\frac{x_n^j-x_n^l}{\lambda_n^l}=\frac{x_n^j-x_n^l}{\sqrt{\lambda_n^l\lambda_n^j}}\sqrt{\frac{\lambda_n^j}{\lambda_n^l}}\to \infty \qtq{as} n\to \infty.
\end{align*}
The desired weak convergence \eqref{524} follows from Lemma \ref{L:converg}.

As $SO(3)$ is a compact group, in Case b) we can proceed similarly if we can show
\begin{align*}
\frac{|x_n^j-( x_n^l)^*|}{\lambda_n^l}\to \infty \qtq{as} n\to \infty.
\end{align*}
But this is immediate from an application of the triangle inequality: for $n$ sufficiently large,
\begin{align*}
\frac{|x_n^j-(x_n^l)^*|}{\lambda_n^l}\ge\frac{|x_n^j-x_n^l|}{\lambda_n^l}-\frac{|x_n^l-(x_n^l)^*|}{\lambda_n^l}
\ge\frac{|x_n^j-x_n^l|}{\lambda_n^l}-2d_\infty^l\to \infty.
\end{align*}
Case c) can be treated symmetrically. Finally, in Case d) we note that for $n$ sufficiently large,
\begin{align*}
\frac{|(x_n^j)^*-(x_n^l)^*|}{\lambda_n^l}&\ge\frac{|x_n^j-x_n^l|}{\lambda_n^l}-\frac{|x_n^j-(x_n^j)^*|}{\lambda_n^l}-\frac{|x_n^l-(x_n^l)^*|}{\lambda_n^l}\\
&\ge\frac{|x_n^j-x_n^l|}{\sqrt{\lambda_n^j\lambda_n^l}}\sqrt{\frac{\lambda_n^j}{\lambda_n^l}}-\frac{d(x_n^j)}{\lambda_n^j}\frac{\lambda_n^j}{\lambda_n^l}-\frac{d(x_n^l)}{\lambda_n^l}\\
&\ge \frac12\sqrt{\lambda_0}\frac{|x_n^j-x_n^l|}{\sqrt{\lambda_n^j\lambda_n^l}}-2\lambda_0d_\infty^j-2d_\infty^l\to \infty \qtq{as} n\to \infty.
\end{align*}
The desired weak convergence follows again from Lemma \ref{L:converg}.

Finally, we prove the last assertion in the theorem regarding the behaviour of $t_n^j$.  For each $j$, by passing to a subsequence we may assume $t_n^j\to t^j\in [-\infty, \infty]$. Using a standard diagonal argument, we may assume that the limit exists for all $j\ge 1$.

Given $j$, if $t^j=\pm\infty$, there is nothing more to be proved; thus, let us suppose that $t^j\in (-\infty, \infty)$.
We claim that we may redefine $t_n^j\equiv 0$, provided we replace the original profile $\phi^j$ by $\exp\{it^j\Delta_{\Omega^j_\infty}\} \phi^j$,
where $\Omega^j_\infty$ denotes the limiting geometry dictated by the case to which $j$ conforms.  Underlying this claim is the
assertion that the errors introduced by these changes can be incorporated into $w_n^J$.  The exact details of proving this depend
on the case to which $j$ conforms; however, the principal ideas are always the same.  Let us give the details in Case~2 alone (for which
$\Omega_\infty^j=\R^3$).  Here, the claim boils down to the assertion that
\begin{align}\label{s1}
\lim_{n\to\infty} \bigl\|e^{it_n^j(\lambda_n^j)^2\Delta_\Omega}[G_n^j(\chi_n^j\phi^j)]
	-  G_n^j(\chi_n^j e^{it^j\Delta_{\R^3}} \phi^j) \bigr\|_{\dot H^1_D(\Omega)} = 0.
\end{align}

To prove \eqref{s1} we first invoke Lemma~\ref{L:dense} to replace $\phi^j$ by a function $\psi\in C^\infty_c(\tlim \Omega^j_n)$.
Moreover, for such functions $\psi$ we have $\chi_n\psi=\psi$ for $n$ sufficiently large.  Doing this and also changing variables, we reduce \eqref{s1} to
\begin{align}\label{s1'}
\lim_{n\to\infty} \bigl\|e^{it_n^j\Delta_{\Omega_n^j}} \psi
	-  \chi_n^j e^{it^j\Delta_{\R^3}} \psi \bigr\|_{\dot H^1_D(\Omega_n^j)} = 0.
\end{align}
We prove this by breaking it into three pieces.

First, by taking the time derivative, we have
$$
 \bigl\|e^{it_n^j\Delta_{\Omega_n^j}} \psi - e^{it^j\Delta_{\Omega_n^j}}\psi \bigr\|_{\dot H^1(\R^3)} \leq | t_n^j - t^j | \| \Delta \psi \|_{\dot H^1(\R^3)},
$$
which converges to zero since $t_n^j \to t^j$.  Secondly, we claim that
$$
e^{it^j\Delta_{\Omega_n^j}}\psi \to e^{it^j\Delta_{\R^3}}\psi \qtq{strongly in} \dot H^1(\R^3) \qtq{as} n\to\infty.
$$
Indeed, the $\dot H^1(\R^3)$ norms of both the proposed limit and all terms in the sequence are the same, namely, $\|\psi\|_{\dot H^1(\R^3)}$.
Thus, strong convergence can be deduced from weak convergence, which follows from Proposition~\ref{P:converg}.  The third
and final part of \eqref{s1'}, namely,
\begin{align*}
\bigl\| (1 - \chi_n^j) e^{it_j\Delta_{\R^3}} \psi \bigr\|_{\dot H^1(\R^3)} \to 0 \qtq{as} n\to\infty,
\end{align*}
can be shown by direct computation using that $\lambda_n^j\to\infty$; see the proof of \eqref{E:no deform}.

This completes the proof of the Theorem~\ref{T:LPD}.
\end{proof}

%%%%%%%%%%%%%%%%%%%%%%%%%%%%%%%%%%%%%%%%%%%%%%%%%%%%%%%%%%%%%%%%%%%%%%%%%%%%%%%%%%%%%%%%%%%%%%%%%%%%%%%%%%%%%%%%%%%%%%%%%%%%%%%%%%%%%%
\section{Embedding of nonlinear profiles}\label{S:Nonlinear Embedding}
%%%%%%%%%%%%%%%%%%%%%%%%%%%%%%%%%%%%%%%%%%%%%%%%%%%%%%%%%%%%%%%%%%%%%%%%%%%%%%%%%%%%%%%%%%%%%%%%%%%%%%%%%%%%%%%%%%%%%%%%%%%%%%%%%%%%%%

The next step in the proof of Theorem~\ref{T:main} is to use the linear profile decomposition obtained in the previous section to derive a Palais--Smale condition for minimizing sequences of blowup solutions to \eqref{nls}.  This essentially amounts to proving a nonlinear profile decomposition for solutions to $\text{NLS}_\Omega$; in the next section, we will prove this decomposition and then combine it with the stability result Theorem~\ref{T:stability} to derive the desired compactness for minimizing sequences of solutions.  This leads directly to the Palais--Smale condition.

In order to prove a nonlinear profile decomposition for solutions to \eqref{nls}, we have to address the possibility that the nonlinear profiles we will extract are solutions to the energy-critical equation in \emph{different} limiting geometries.  In this section, we will see how to embed these nonlinear profiles corresponding to different limiting geometries back inside $\Omega$.  Specifically, we need to approximate these profiles \emph{globally in time} by actual solutions to \eqref{nls} that satisfy \emph{uniform} spacetime bounds.  This section contains three theorems, one for each of the Cases 2, 3, and 4 discussed in the previous sections.

As in Section~\ref{S:LPD}, throughout this section $\Theta:\R^3\to [0,1]$ denotes a smooth function such that
\begin{align*}
\Theta(x)=\begin{cases}0, \ & |x|\le \frac 14, \\1, \ & |x|\geq \frac 12.
\end{cases}
\end{align*}
We will also use the following notation:
$$
\dot X^1(I\times\Omega):=L_{t,x}^{10}(I\times\Omega)\cap L_t^5\dot H^{1,\frac{30}{11}}_D(I\times\Omega).
$$

Our first result in this section concerns the scenario when the rescaled obstacles $\Omega_n^c$ are shrinking to a point (cf. Case 2 in Theorem~\ref{T:LPD}).

\begin{thm}[Embedding nonlinear profiles for shrinking obstacles]\label{T:embed2} Let $\{\lambda_n\}\subset 2^{\mathbb Z}$ be such that $\lambda_n\to \infty$. Let $\{t_n\}\subset\R$ be such that $t_n\equiv0$ or $t_n\to \pm\infty$.  Let $\{x_n\}\subset \Omega$ be such that
$-\lambda_n^{-1}x_n\to x_\infty\in \R^3$.  Let
$\phi\in \dot H^1(\R^3)$ and
\begin{align*}
\phi_n(x)=\lambda_n^{-\frac12}e^{it_n\lambda_n^2\Delta_\Omega}\bigl[(\chi_n\phi)\bigl(\tfrac{x-x_n}{\lambda_n}\bigr)\bigr],
\end{align*}
where $\chi_n(x)=\chi(\lambda_n x+x_n)$ with $\chi(x)=\Theta(\tfrac{d(x)}{\diam(\Omega^c)})$.  Then for $n$ sufficiently large there exists a global solution $v_n$ to $\text{NLS}_{\Omega}$ with initial data $v_n(0)=\phi_n$ which satisfies
\begin{align*}
\|v_n\|_{L_{t,x}^{10}(\R\times\Omega)}\lsm 1,
\end{align*}
with the implicit constant depending only on $\|\phi\|_{\dot H^1}$.  Furthermore, for every $\eps>0$ there exists $N_\eps\in \N$ and
$\psi_\eps\in C_c^{\infty}(\R\times\R^3)$ such that for all $n\geq N_\eps$ we have
\begin{align}\label{dense2}
\|v_n(t-\lambda_n^2 t_n,x+x_n)-\lambda_n^{-\frac12}\psi_\eps(\lambda_n^{-2}t,\lambda_n^{-1} x)\|_{\dot X^1(\R\times\R^3)}<\eps.
\end{align}
\end{thm}

\begin{proof}
The proof contains five steps.  In the first step we construct global solutions to the energy-critical NLS in the limiting geometry $\R^3$ and we record some of their properties.  In the second step we construct a candidate for the sought-after solution to $\text{NLS}_\Omega$.  In the third step we prove that our candidate asymptotically matches the initial data $\phi_n$, while in the fourth step we prove that it is an approximate solution to \eqref{nls}.  In the last step we invoke the stability result Theorem~\ref{T:stability} to find $v_n$ and then prove the approximation result \eqref{dense2}.

To ease notation, throughout the proof we will write $-\Delta=-\Delta_{\R^3}$.

\textbf{Step 1:} Constructing global solutions to $\text{NLS}_{\R^3}$.

Let $\theta:=\frac 1{100}$.  The construction of the solutions to $\text{NLS}_{\R^3}$ depends on the behaviour of $t_n$.  If $t_n\equiv0$, let $w_n$ and
$w_\infty$ be solutions to $\text{NLS}_{\R^3}$ with initial data $w_n(0)=\phi_{\le \lambda_n^{\theta}}$ and $w_\infty(0)=\phi$.

If instead $t_n\to \pm\infty$, let $w_n$ be the solution to $\text{NLS}_{\R^3}$ such that
\begin{align*}
\|w_n(t)-e^{it\Delta}\phi_{\le \lambda_n^{\theta}}\|_{\dot H^1(\R^3)}\to 0 \qtq{as} t\to \pm\infty.
\end{align*}
Similarly, we define $w_\infty$ as the solution to $\text{NLS}_{\R^3}$ such that
\begin{align}\label{n24}
\|w_\infty(t)-e^{it\Delta}\phi\|_{\dot H^1(\R^3)}\to 0 \qtq{as} t\to \pm\infty.
\end{align}

By \cite{CKSTT:gwp}, in all cases $w_n$ and $w_\infty$ are global solutions and satisfy
\begin{align}\label{258}
\|w_n\|_{\dot S^1(\R\times\R^3)}+\|w_\infty\|_{\dot S^1(\R\times\R^3)}\lsm 1,
\end{align}
with the implicit constant depending only on $\|\phi\|_{\dot H^1}$.  Moreover, by the perturbation theory described in that paper,
\begin{align}\label{258'}
\lim_{n\to \infty}\|w_n-w_\infty\|_{\dot S^1(\R\times\R^3)}=0.
\end{align}

By Bernstein's inequality,
\begin{align*}
\|\phi_{\le \lambda_n^{\theta}}\|_{\dot H^s(\R^3)}\lsm \lambda_n^{\theta(s-1)} \qtq{for any} s\ge 1,
\end{align*}
and so the persistence of regularity result Lemma~\ref{lm:persistencer3} gives
\begin{align}\label{persist2}
\||\nabla|^s w_n\|_{\dot S^1(\R\times\R^3)}\lsm \lambda_n^{\theta s}  \qtq{for any} s\ge 0,
\end{align}
with the implicit constant depending solely on $\|\phi\|_{\dot H^1}$.  Combining this with the Gagliardo--Nirenberg inequality
\begin{align*}
\|f\|_{L^\infty_x}\lsm \|\nabla f\|_{L^2_x}^{\frac 12}\|\Delta f\|_{L^2_x}^{\frac12},
\end{align*}
we obtain
\begin{align}\label{259}
\||\nabla|^s w_n\|_{L_{t,x}^{\infty}(\R\times\R^3)}\lsm \lambda_n^{\theta(s+\frac12)},
\end{align}
for all $s\geq 0$.  Finally, using the equation we get
\begin{align}\label{260}
\|\partial_t w_n\|_{L_{t,x}^{\infty}(\R\times\R^3)}\le \|\Delta w_n\|_{L_{t,x}^{\infty}(\R\times\R^3)}+\|w_n\|_{L_{t,x}^{\infty}(\R\times\R^3)}^5
\lsm\lambda_n^{\frac 52\theta}.
\end{align}

\textbf{Step 2:} Constructing the approximate solution to $\text{NLS}_\Omega$.

As previously in this scenario, let $\Omega_n:=\lambda_n^{-1}(\Omega-\{x_n\})$.  The most naive way to embed $w_n(t)$ into $\Omega_n$
is to choose $\tilde v_n(t) = \chi_n w_n(t)$; however, closer investigation reveals that this is \emph{not} an approximate solution to $\text{NLS}_\Omega$,
unless one incorporates  some high-frequency reflections off the obstacle, namely,
\begin{align*}
z_n(t):=i\int_0^t e^{i(t-s)\Delta_{\Omega_n}}(\Delta_{\Omega_n}\chi_n)w_n(s,-\lambda_n^{-1}x_n)\,ds.
\end{align*}
The source of these waves is nonresonant in spacetime due to the slow variation in time when compared to the small spatial scale involved.
This allows us to estimate these reflected waves; indeed, we have the following lemma:

\begin{lem} For any $T>0$, we have
\begin{align}
\limsup_{n\to\infty}\|z_n\|_{\dot X^1([-T,T]\times\Omega_n)}&=0\label{209}\\
\|(-\Delta_{\Omega_n})^{\frac s2}z_n\|_{L_t^\infty L_x^2([-T,T]\times\Omega_n)}&\lsm \lambda_n^{s-\frac 32+\frac52\theta}(T+\lambda_n^{-2\theta}) \qtq{for all} 0\le s<\tfrac 32.\label{209'}
\end{align}
\end{lem}

\begin{proof}
Throughout the proof, all spacetime norms will be over $[-T,T]\times\Omega_n$.  We write
\begin{align*}
z_n(t)&=-\int_0^t [e^{it\Delta_{\Omega_n}}\partial_se^{-is\Delta_{\Omega_n}}\chi_n]w_n(s,-\lambda_n^{-1}x_n) \,ds\\
&=-\chi_nw_n(t,-\lambda_n^{-1}x_n)+e^{it\Delta_{\Omega_n}}[\chi_nw_n(0,-\lambda_n^{-1}x_n)]\\
&\quad +\int_0^t [e^{i(t-s)\Delta_{\Omega_n}}\chi_n]\partial_sw_n(s,-\lambda_n^{-1}x_n)\,ds.
\end{align*}

We first estimate the $L_t^5\dot H^{1,\frac{30}{11}}_D$ norm of $z_n$.  Using the Strichartz inequality, the equivalence of Sobolev spaces
Theorem~\ref{T:Sob equiv}, \eqref{259}, and \eqref{260}, we get
\begin{align*}
\|z_n\|_{L_t^5\dot H_D^{1,\frac{30}{11}}}
&\lsm \|\nabla\chi_n(x) w_n(t, -\lambda_n^{-1}x_n)\|_{L_t^5L_x^\frac{30}{11}} +\|\nabla \chi_n(x)w_n(0,-\lambda_n^{-1}x_n)\|_{L^2_x}\\
&\quad+\|\nabla \chi_n(x) \partial_t w_n(t,-\lambda_n^{-1}x_n)\|_{L_t^1L_x^2}\\
&\lsm T^{\frac 15}\|\nabla \chi_n\|_{L^{\frac{30}{11}}_x}\|w_n\|_{L_{t,x}^\infty}+\|\nabla\chi_n\|_{L^2_x}\|w_n\|_{L_{t,x}^\infty}
	+T\|\nabla\chi_n\|_{L^2_x}\|\partial_t w_n\|_{L_{t,x}^\infty}\\
&\lsm T^{\frac15}\lambda_n^{-\frac{1}{10}+\frac{\theta}2}+\lambda_n^{-\frac12+\frac{\theta}2}+T\lambda_n^{-\frac 12+\frac 52\theta}\to 0\qtq{as} n\to \infty.
\end{align*}
Similarly, using also Sobolev embedding we obtain
\begin{align}\label{zn in L10}
\|z_n\|_{L_{t,x}^{10}}&\lsm\|(-\Delta_{\Omega_n})^{\frac 12}z_n\|_{L_t^{10}L_x^{\frac{30}{13}}}\notag\\
&\lsm \|\nabla \chi_n(x) w_n(t,-\lambda_n^{-1}x_n)\|_{L_t^{10}L_x^{\frac{30}{13}}}+\|\nabla\chi_n(x)w_n(0,-\lambda_n^{-1}x_n)\|_{L^2_x}\notag\\
&\quad+\|\nabla\chi_n(x)\partial_tw_n(t,-\lambda_n^{-1}x_n)\|_{L_t^1L_x^2}\notag\\
&\lsm T^{\frac 1{10}}\|\nabla\chi_n\|_{L_x^{\frac{30}{13}}}\|w_n\|_{L_{t,x}^\infty}+\|\nabla\chi_n\|_{L_x^2}\|w_n\|_{L_{t,x}^\infty}
+T\|\nabla\chi_n\|_{L_x^2}\|\partial_t w_n\|_{L_{t,x}^\infty}\notag\\
&\lsm T^{\frac1{10}}\lambda_n^{-\frac{3}{10}+\frac{\theta}2}+\lambda_n^{-\frac12+\frac{\theta}2}+T\lambda_n^{-\frac 12+\frac 52\theta} \to 0\qtq{as} n\to \infty.
\end{align}
This proves \eqref{209}.

To establish \eqref{209'}, we argue as before and estimate
\begin{align*}
\|(-\Delta_{\Omega_n})^{\frac s2}z_n\|_{L_t^\infty L_x^2}
&\lsm \|(-\Delta)^{\frac s2}\chi_n w_n(t,-\lambda_n^{-1}x_n)\|_{L_t^{\infty}L_x^2}+\|(-\Delta)^{\frac s2}\chi_nw_n(0,-\lambda_n^{-1}x_n)\|_{L_x^2}\\
&\quad+\|(-\Delta)^{\frac s2}\chi_n\partial_tw_n(t,-\lambda_n^{-1}x_n)\|_{L_t^1L_x^2}\\
&\lsm \|(-\Delta)^{\frac s2}\chi_n\|_{L_x^2}\|w_n\|_{L_{t,x}^\infty}+T\|(-\Delta)^{\frac s2}\chi_n\|_{L_x^2}\|\partial_t w_n\|_{L_{t,x}^\infty}\\
&\lsm \lambda_n^{s-\frac 32+\frac{\theta}2}+T\lambda_n^{s-\frac 32+\frac 52\theta}\\
&\lsm \lambda_n^{s-\frac 32+\frac 52\theta}(T+\lambda_n^{-2\theta}).
\end{align*}
This completes the proof of the lemma.
\end{proof}

We are now in a position to introduce the approximate solution
\begin{align*}
\tilde v_n(t,x):=\begin{cases}
\lambda_n^{-\frac12}(\chi_nw_n+z_n)(\lambda_n^{-2} t, \lambda_n^{-1}(x-x_n)), &|t|\le\lambda_n^2 T,\\
e^{i(t-\lambda_n^2 T)\Delta_{\Omega}}\tilde v_n(\lambda_n^2 T,x), & t>\lambda_n^2 T, \\
e^{i(t+\lambda_n^2 T)\Delta_\Omega}\tilde v_n(-\lambda_n^2 T,x), & t<-\lambda_n^2 T,
\end{cases}
\end{align*}
where $T>0$ is a parameter to be chosen later.  Note that $\tilde v_n$ has finite scattering size.  Indeed, using a change of variables, the Strichartz inequality, \eqref{258}, \eqref{209'}, and \eqref{zn in L10}, we get
\begin{align}\label{tildevn2}
\|\tilde v_n\|_{L_{t,x}^{10}(\R\times\Omega)}
&\lsm \|\chi_nw_n +z_n\|_{L_{t,x}^{10}([-T,T]\times\Omega_n)}+\|(\chi_nw_n+z_n)(\pm T)\|_{\dot H^1_D(\Omega_n)}\notag\\
&\lsm \|w_n\|_{L_{t,x}^{10}(\R\times\R^3)}+\|z_n\|_{L_{t,x}^{10}([-T,T]\times\Omega_n)}+\|\chi_n\|_{L_x^\infty}\|\nabla w_n\|_{L_t^\infty L_x^2(\R\times\R^3)}\notag\\
&\quad+\|\nabla \chi_n\|_{L_x^3}\|w_n\|_{L_t^\infty L_x^6(\R\times\R^3)}+\|(-\Delta_{\Omega_n})^{\frac12}z_n\|_{L_t^\infty L_x^2([-T,T]\times\Omega_n)}\notag\\
&\lsm 1+ T^{\frac1{10}}\lambda_n^{-\frac{3}{10}+\frac{\theta}2}+\lambda_n^{-\frac12+\frac{\theta}2}+T\lambda_n^{-\frac 12+\frac 52\theta} .
\end{align}

\textbf{Step 3:} Asymptotic agreement of the initial data.

In this step, we show (cf. the smallness hypothesis in Theorem~\ref{T:stability})
\begin{align}\label{match2}
\lim_{T\to \infty}\limsup_{n\to\infty}\|(-\Delta_\Omega)^{\frac 12}e^{it\Delta_\Omega}[\tilde v_n(\lambda_n^2t_n)-\phi_n]\|_{L_t^{10}L_x^{\frac{30}{13}}(\R\times\Omega)}=0.
\end{align}

We first prove \eqref{match2} in the case when $t_n\equiv0$.  Using the Strichartz inequality, a change of variables, and H\"older, we estimate
\begin{align*}
\|(-\Delta_\Omega)^{\frac 12} e^{it\Delta_\Omega}&[\tilde v_n(0)-\phi_n]\|_{L_t^{10}L_x^{\frac{30}{13}}(\R\times\Omega)}\\
&\lsm \|(-\Delta_\Omega)^{\frac 12}[\tilde v_n(0)-\phi_n]\|_{L_x^2}\\
%&\lsm \|(-\Delta_{\Omega_n})^{\frac 12}[\chi_n\phi_{\le \lambda_n^{\theta}}-\chi_n\phi]\|_{L_x^2}\\
&\lsm \|\nabla[\chi_n\phi_{\le\lambda_n^{\theta}}-\chi_n\phi]\|_{L_x^2}\\
&\lsm \|\nabla\chi_n\|_{L_x^3}\|\phi_{\le\lambda_n^{\theta}}-\phi\|_{L_x^6}+\|\chi_n\|_{L_x^\infty}\|\nabla[\phi_{\le\lambda_n^\theta}-\phi]\|_ {L_x^2},
\end{align*}
which converges to zero as $n\to \infty$.

It remains to prove \eqref{match2} in the case $t_n\to \infty$; the case $t_n\to -\infty$ can be treated similarly.  As $T>0$ is fixed, for sufficiently large $n$ we have $t_n>T$ and so
\begin{align*}
\tilde v_n(\lambda_n^2t_n,x)&=e^{i(t_n-T)\lambda_n^2\Delta_\Omega}\tilde v_n(\lambda_n^2 T,x)
=e^{i(t_n-T)\lambda_n^2\Delta_\Omega}\bigl[\lambda_n^{-\frac12}\bigl(\chi_nw_n+z_n\bigr)\bigl(T,\tfrac{x-x_n}{\lambda_n}\!\bigl)\bigr] .
\end{align*}
Thus by a change of variables and the Strichartz inequality,
\begin{align*}
\|(-\Delta_\Omega)^{\frac 12}& e^{it\Delta_\Omega}[\tilde v_n(\lambda_n^2 t_n)-\phi_n]\|_{L_t^{10}L_x^{\frac{30}{13}}(\R\times\Omega)}\\
&=\|(-\Delta_{\Omega_n})^{\frac 12}\{e^{i(t-T)\Delta_{\Omega_n}}(\chi_nw_n+z_n)(T)-e^{it\Delta_{\Omega_n}}(\chi_n\phi)\}\|_{L_t^{10}L_x^{\frac{30}{13}}(\R\times\Omega_n)}\\
&\lsm \|(-\Delta_{\Omega_n})^{\frac 12}z_n(T)\|_{L_x^2}+\|(-\Delta_{\Omega_n})^{\frac 12}[\chi_n(w_n-w_\infty)(T)]\|_{L_x^2}\\
&\quad+\|(-\Delta_{\Omega_n})^{\frac 12}e^{it\Delta_{\Omega_n}}[e^{-iT\Delta_{\Omega_n}}(\chi_nw_\infty(T))-\chi_n\phi]\|
_{L_t^{10}L_x^{\frac{30}{13}}(\R\times\Omega_n)}.
\end{align*}
Using \eqref{258'} and \eqref{209'}, we see that
\begin{align*}
&\|(-\Delta_{\Omega_n})^{\frac 12}z_n(T)\|_{L_x^2}+\|(-\Delta_{\Omega_n})^{\frac 12}[\chi_n(w_n-w_\infty)(T)]\|_{L_x^2}\\
&\lsm \lambda_n^{-\frac 12+\frac 52\theta}(T+\lambda_n^{-2\theta})+\|\nabla\chi_n\|_{L_x^3}\|w_n-w_\infty\|_{L_t^\infty L_x^6}+\|\chi_n\|_{L_x^\infty}\|\nabla(w_n- w_\infty)\|_{L_t^\infty L_x^2},
\end{align*}
which converges to zero as $n\to \infty$.  Thus, to establish \eqref{match2} we are left to prove
\begin{align}\label{n23}
\lim_{T\to \infty}\limsup_{n\to\infty}\|(-\Delta_{\Omega_n})^{\frac12}e^{it\Delta_{\Omega_n}}[e^{-iT\Delta_{\Omega_n}}(\chi_nw_\infty(T))-\chi_n\phi]\|
_{L_t^{10}L_x^{\frac{30}{13}}(\R\times\Omega_n)}=0.
\end{align}
Using the triangle and Strichartz inequalities, we obtain
\begin{align*}
\|(-\Delta_{\Omega_n})^{\frac12}&e^{it\Delta_{\Omega_n}}[e^{-iT\Delta_{\Omega_n}}(\chi_nw_\infty(T))-\chi_n\phi]\|_{L_t^{10}L_x^{\frac{30}{13}}(\R\times\Omega_n)}\\
&\lsm \|(-\Delta_{\Omega_n})^{\frac 12}(\chi_n w_\infty(T))-\chi_n(-\Delta)^{\frac 12}w_\infty(T)\|_{L_x^2}\\
&\quad+\|[e^{i(t-T)\Delta_{\Omega_n}}-e^{i(t-T)\Delta}][\chi_n(-\Delta)^{\frac12}w_{\infty}(T)]\|_{L_t^{10}L_x^{\frac{30}{13}}(\R\times\Omega_n)}\\
&\quad+\|e^{-iT\Delta}[\chi_n(-\Delta)^{\frac12}w_\infty(T)]-\chi_n(-\Delta)^{\frac12}\phi\|_{L_x^2}\\
&\quad+\|[e^{it\Delta_{\Omega_n}}-e^{it\Delta}][\chi_n(-\Delta)^{\frac12}\phi]\|_{L_t^{10}L_x^{\frac{30}{13}}(\R\times\Omega_n)}\\
&\quad+ \|(-\Delta_{\Omega_n})^{\frac12}(\chi_n\phi)-\chi_n(-\Delta)^{\frac12}\phi\|_{L_x^2}.
\end{align*}
The fact that the second and fourth terms above converge to zero as $n\to \infty$ follows from Theorem~\ref{T:LF} and the density in $L_x^2$ of $C_c^{\infty}$ functions supported in $\R^3$ minus a point.  To see that the first and fifth terms above converge to zero, we note that for any $f\in \dot H^1(\R^3)$,
\begin{align*}
\|(-\Delta_{\Omega_n})^{\frac 12}(\chi_n f)-\chi_n(-\Delta)^{\frac12}f\|_{L_x^2}
&\le \|(1-\chi_n)(-\Delta)^{\frac 12}f\|_{L_x^2}+\|(-\Delta)^{\frac12}[(1-\chi_n)f]\|_{L_x^2}\\
&\quad+\|(-\Delta_{\Omega_n})^{\frac 12}(\chi_n f)-(-\Delta)^{\frac12}(\chi_n f)\|_{L_x^2}\to 0
\end{align*}
as $n\to \infty$ by Lemma~\ref{L:n3} and the monotone convergence theorem.  Finally, for the third term we use \eqref{n24} and the monotone convergence theorem to obtain
\begin{align*}
\|e^{-iT\Delta}[\chi_n(-\Delta)^{\frac12}&w_\infty(T)]-\chi_n(-\Delta)^{\frac 12}\phi\|_{L_x^2}\\
&\lsm \|(1-\chi_n)(-\Delta)^{\frac12}w_\infty(T)\|_{L^2_x}+\|(1-\chi_n)(-\Delta)^{\frac 12}\phi\|_{L_x^2}\\
&\quad+\|e^{-iT\Delta}(-\Delta)^{\frac12}w_\infty(T)-(-\Delta)^{\frac 12}\phi\|_{L_x^2} \to 0,
\end{align*}
by first taking $n\to \infty$ and then $T\to \infty$.  This completes the proof of \eqref{n23} and so the proof of \eqref{match2}.

\textbf{Step 4:} Proving that $\tilde v_n$ is an approximate solution to $\text{NLS}_{\Omega}$ in the sense that
\begin{align*}
i\partial_t \tilde v_n+\Delta_\Omega\tilde v_n=|\tilde v_n|^4\tilde v_n+e_n
\end{align*}
with
\begin{align}\label{error2}
\lim_{T\to\infty}\limsup_{n\to\infty}\|e_n\|_{\dot N^1(\R\times\Omega)}=0.
\end{align}

We start by verifying \eqref{error2} on the large time interval $t>\lambda_n^2 T$; symmetric arguments can be used to treat $t<-\lambda_n^2 T$.  By the definition of $\tilde v_n$, in this regime we have
$e_n=-|\tilde v_n|^4\tilde v_n$.  Using the equivalence of Sobolev spaces, Strichartz, and \eqref{209'}, we estimate
\begin{align*}
\|e_n\|_{\dot N^1(\{t>\lambda_n^2 T\}\times\Omega)}
&\lsm \|(-\Delta_\Omega)^{\frac12}(|\tilde v_n|^4\tilde v_n)\|_{L_t^{\frac53} L_x^{\frac{30}{23}}(\{t>\lambda_n^2 T\}\times\Omega)}\\
&\lsm \|(-\Delta_\Omega)^{\frac12}\tilde v_n\|_{L_t^5 L_x^{\frac{30}{11}}(\{t>\lambda_n^2 T\}\times\Omega)} \|\tilde v_n\|^4_{L_{t,x}^{10}(\{t>\lambda_n^2 T\}\times\Omega)}\\
&\lsm \|(-\Delta_{\Omega_n})^{\frac12}[\chi_n w_n (T)+z_n(T)] \|_{L_x^2}\|\tilde v_n\|^4_{L_{t,x}^{10}(\{t>\lambda_n^2 T\}\times\Omega)}\\
&\lsm \bigl[ 1+ \lambda_n^{-\frac12+\frac52\theta}(T+\lambda_n^{-2\theta}) \bigr] \|\tilde v_n\|^4_{L_{t,x}^{10}(\{t>\lambda_n^2 T\}\times\Omega)}.
\end{align*}
Thus, to establish \eqref{error2} it suffices to show
\begin{align}\label{largetime2}
\lim_{T\to\infty}\limsup_{n\to\infty}\|e^{i(t-\lambda_n^2 T)\Delta_{\Omega}}\tilde v_n(\lambda_n^2T)\|_{L_{t,x}^{10}(\{t>\lambda_n^2 T\}\times\Omega)}=0,
\end{align}
to which we now turn.

As a consequence of the spacetime bounds \eqref{258}, the global solution $w_\infty$ scatters.  Let $w_+$ denote the forward asymptotic state, that is,
\begin{align}\label{as2}
\|e^{-it\Delta}w_\infty(t)- w_+\|_{\dot H^1(\R^3)}\to 0 \qtq{as} t\to \infty.
\end{align}
(Note that in the case when $t_n\to \infty$, from the definition of $w_\infty$ we have $w_+=\phi$.)  Using a change of variables, \eqref{209'}, 	the Strichartz and H\"older inequalities, and Sobolev embedding, we obtain
\begin{align*}
\|&e^{i(t-\lambda_n^2 T)\Delta_\Omega}\tilde v_n(\lambda_n^2T)\|_{L_{t,x}^{10}((\lambda_n^2T,\infty)\times\Omega)}\\
&=\|e^{it\Delta_{\Omega_n}}(\chi_nw_n(T)+z_n(T))\|_{L_{t,x}^{10}([0,\infty)\times\Omega_n)}\\
&\lsm \|(-\Delta_{\Omega_n})^{\frac12}z_n(T)\|_{L_x^2}+\|(-\Delta_{\Omega_n})^{\frac12}[\chi_n(w_n(T)-w_\infty(T))]\|_{L_x^2}\\
&\quad+\|(-\Delta_{\Omega_n})^{\frac12}[\chi_n(w_{\infty}(T)-e^{iT\Delta}w_+)]\|_{L_x^2}+\|e^{it\Delta_{\Omega_n}}[\chi_ne^{iT\Delta}w_+]\|_{L_{t,x}^{10}([0,\infty)\times\Omega_n)}\\
&\lsm \lambda_n^{-\frac12+\frac 52\theta}(T+\lambda_n^{-2\theta})+\|w_n(T)-w_\infty(T)\|_{\dot H^1_x}+\|w_\infty(T)-e^{iT\Delta}w_+\|_{\dot H^1_x}\\
&\quad+\|[e^{it\Delta_{\Omega_n}}-e^{it\Delta}][\chi_ne^{iT\Delta}w_+]\|_{L_{t,x}^{10}([0,\infty)\times\R^3)} +\|\nabla [(1-\chi_n)e^{iT\Delta}w_+]\|_{L_x^2}\\
&\quad+\|e^{it\Delta}w_+\|_{L_{t,x}^{10}((T,\infty)\times\R^3)},
\end{align*}
which converges to zero by first letting $n\to \infty$ and then $T\to \infty$ by \eqref{258'}, \eqref{as2}, Corollary~\ref{C:LF}, and the monotone convergence theorem.

We are left to prove \eqref{error2} on the middle time interval $|t|\leq \lambda_n^2T$.   For these values of time, we compute
\begin{align*}
e_n(t,x)&=[(i\partial_t+\Delta_\Omega )\tilde v_n- |\tilde v_n|^4\tilde v_n](t,x)\\
&=-\lambda_n^{-\frac52}[\Delta\chi_n](\lambda_n^{-1}(x-x_n))w_n(\lambda_n^{-2}t,-\lambda_n^{-1}x_n)\\
&\quad+\lambda_n^{-\frac 52}[\Delta\chi_n w_n](\lambda_n^{-2}t,\lambda_n^{-1}(x-x_n))\\
&\quad+2\lambda_n^{-\frac 52}(\nabla\chi_n\cdot\nabla w_n)(\lambda_n^{-2}t, \lambda_n^{-1}(x-x_n))\\
&\quad+\lambda_n^{-\frac52}[\chi_n|w_n|^4w_n-|\chi_nw_n+z_n|^4(\chi_nw_n+z_n)](\lambda_n^{-2}t,\lambda_n^{-1}(x-x_n)).
\end{align*}
Thus, using a change of variables and the equivalence of Sobolev norms Theorem~\ref{T:Sob equiv}, we estimate
\begin{align}
\|e_n\|_{\dot N^1(\{|t|\leq\lambda_n^2 T\}\times\Omega)}
&\lsm \|(-\Delta_\Omega)^{\frac12} e_n\|_{L_{t,x}^{\frac{10}7}(\{|t|\le\lambda_n^2T\}\times\Omega)}\notag\\
&\lsm\|\nabla[\Delta\chi_n(w_n(t,x)-w_n(t,-\lambda_n^{-1}x_n))]\|_{L_{t,x}^{\frac{10}7}([-T,T]\times\Omega_n)}\label{51}\\
&\quad+\|\nabla[\nabla\chi_n\cdot\nabla w_n]\|_{L_{t,x}^{\frac{10}7}([-T,T]\times\Omega_n)}\label{52}\\
&\quad+\|\nabla[\chi_n|w_n|^4 w_n-|\chi_nw_n+z_n|^4(\chi_nw_n+z_n)]\|_{L_{t,x}^{\frac{10}7}([-T,T]\times\Omega_n)}.\label{53}
\end{align}

Using H\"older, the fundamental theorem of calculus, and \eqref{259}, we estimate
\begin{align*}
\eqref{51}&\lsm T^{\frac 7{10}}\|\Delta\chi_n\|_{L_x^{\frac{10}7}}\|\nabla w_n\|_{L_{t,x}^\infty}\\
&\quad+T^{\frac7{10}}\|\nabla\Delta\chi_n\|_{L_x^\frac{10}7}\|w_n(t,x)-w_n(t,-\lambda_n^{-1}x_n)\|_{L_{t,x}^\infty(\R\times\supp\Delta\chi_n)}\\
&\lsm T^{\frac 7{10}}\lambda_n^{-\frac 1{10}+ \frac32\theta}+T^{\frac7{10}}\lambda_n^{\frac 9{10}}\lambda_n^{-1}\|\nabla w_n\|_{L_{t,x}^\infty}\\
&\lsm T^{\frac 7{10}}\lambda_n^{-\frac 1{10}+\frac 32\theta} \to 0 \qtq{as} n\to \infty.
\end{align*}
Notice that the cancellations induced by the introduction of $z_n$ were essential in order to control this term.  Next,
\begin{align*}
\eqref{52}
&\le T^{\frac7{10}}\bigl[\|\Delta\chi_n\|_{L_x^{\frac{10}7}}\|\nabla w_n\|_{L_{t,x}^\infty}+\|\nabla \chi_n\|_{L_x^{\frac{10}7}}\|\Delta w_n\|_{L_{t,x}^\infty}\bigr]\\
&\le T^{\frac 7{10}}[\lambda_n^{-\frac 1{10}+\frac 32\theta}+\lambda_n^{-\frac{11}{10}+\frac 52\theta}]\to 0 \qtq{as} n\to \infty.
\end{align*}

Finally, we turn our attention to \eqref{53}.  A simple algebraic computation yields
\begin{align*}
\eqref{53}&\lsm T^{\frac7{10}}\Bigl\{ \|\nabla[(\chi_n-\chi_n^5)w_n^5] \|_{L_t^\infty L_x^{\frac{10}7}} + \|z_n^4\nabla z_n\|_{L_t^\infty L_x^{\frac{10}7}}\\
&\quad+\sum_{k=1}^4\Bigl[ \|w_n^{k-1}z_n^{5-k}\nabla (\chi_n w_n) \|_{L_t^\infty L_x^{\frac{10}7}}+
\|w_n^k z_n^{4-k}\nabla z_n\|_{L_t^\infty L_x^{\frac{10}7}}\Bigr]\Bigr\},
\end{align*}
where all spacetime norms are over $[-T,T]\times\Omega_n$.  Using H\"older and \eqref{259}, we estimate
\begin{align*}
\|\nabla[(\chi_n-\chi_n^5)w_n^5] \|_{L_t^\infty L_x^{\frac{10}7}}
&\lesssim  \|\nabla \chi_n\|_{L_x^{\frac{10}7}}\|w_n\|_{L_{t,x}^\infty}^5 + \|\chi_n-\chi_n^5\|_{L_x^{\frac{10}7}}\|w_n\|_{L_{t,x}^\infty}^4\|\nabla w_n\|_{L_{t,x}^\infty} \\
&\lesssim  \lambda_n^{-\frac{11}{10}+\frac 52\theta} + \lambda_n^{-\frac{21}{10}+\frac 72\theta}.
\end{align*}
Using also \eqref{209'}, Sobolev embedding, and Theorem~\ref{T:Sob equiv}, we obtain
\begin{align*}
\|z_n^4\nabla z_n\|_{L_t^{\infty}L_x^{\frac{10}7}}
\lsm\|\nabla z_n\|_{L_t^\infty L_x^2}\|z_n\|_{L_t^\infty L_x^{20}}^4
&\lsm \|\nabla z_n\|_{L_t^\infty L_x^2}\||\nabla |^{\frac{27}{20}} z_n\|_{L_t^\infty L_x^2}^4\\
&\lsm \lambda_n^{-\frac{11}{10}+\frac{25}2\theta}(T+\lambda_n^{-2\theta})^5.
\end{align*}
Similarly,
\begin{align*}
\| & w_n^{k-1}z_n^{5-k}\nabla (\chi_n w_n) \|_{L_t^\infty L_x^{\frac{10}7}} \\
&\lesssim \|\nabla\chi_n\|_{L_x^3}\|w_n\|_{L_t^\infty L_x^{\frac{150}{11}}}^k \|z_n\|_{L_t^\infty L_x^{\frac{150}{11}}}^{5-k}
	+ \|\nabla w_n\|_{L_t^\infty L_x^2}\|w_n\|_{L_t^\infty L_x^{20}}^{k-1}\|z_n\|_{L_t^\infty L_x^{20}}^{5-k} \\
&\lesssim \||\nabla|^{\frac{32}{25}}w_n\|_{L_t^\infty L_x^2}^k\||\nabla|^{\frac{32}{25}}z_n\|_{L_t^\infty L_x^2}^{5-k}
	+ \||\nabla|^{\frac{27}{20}}w_n\|_{L_t^\infty L_x^2}^{k-1}\||\nabla|^{\frac{27}{20}}z_n\|_{L_t^\infty L_x^2}^{5-k} \\
&\lesssim \lambda_n^{\frac7{25}\theta k+(-\frac{11}{50}+\frac 52\theta)(5-k)}(T+\lambda_n^{-2\theta})^{5-k}
	+  \lambda_n^{\frac 7{20}\theta(k-1)}\lambda_n^{(-\frac 3{20}+\frac 52\theta)(5-k)}(T+\lambda_n^{-2\theta})^{5-k}
\end{align*}
and
\begin{align*}
\|w_n^k z_n^{4-k}\nabla z_n\|_{L_t^\infty L_x^{\frac{10}7}}
&\lsm \|\nabla z_n\|_{L_t^\infty L_x^2}\|w_n\|_{L_t^\infty L_x^{20}}^k\|z_n\|_{L_t^\infty L_x^{20}}^{4-k}\\
&\lsm \lambda_n^{-\frac 12+\frac 52\theta}\lambda_n^{\frac7{20}\theta k}\lambda_n^{(-\frac3{20}+\frac 52\theta)(4-k)}(T+\lambda_n^{-2\theta})^{5-k}.
\end{align*}
Putting everything together and recalling $\theta=\frac1{100}$, we derive
\begin{align*}
\eqref{53}\to 0 \qtq{as} n\to \infty.
\end{align*}

Therefore,
\begin{align*}
\lim_{T\to\infty}\limsup_{n\to\infty}\|e_n\|_{\dot N^1(\{|t|\leq \lambda_n^2T\}\times\Omega)}=0,
\end{align*}
which together with \eqref{largetime2} gives \eqref{error2}.

\textbf{Step 5} Constructing $v_n$ and approximation by $C_c^\infty$ functions.

Using \eqref{tildevn2}, \eqref{match2}, and \eqref{error2}, and invoking the stability result Theorem~\ref{T:stability},  for $n$ (and $T$) sufficiently large we obtain a global solution $v_n$ to $\text{NLS}_\Omega$ with initial data $v_n(0)=\phi_n$ and
\begin{align*}
\|v_n\|_{L_{t,x}^{10}(\R\times\Omega)}\lsm 1.
\end{align*}
Moreover,
\begin{align}\label{vncase2}
\lim_{T\to\infty}\limsup_{n\to\infty}\|v_n(t-\lambda_n^2 t_n)-\tilde v_n(t)\|_{\dot S^1(\R\times\Omega)}=0.
\end{align}

To complete the proof of the theorem, it remains to prove the approximation result \eqref{dense2}, to which we now turn.  From the density of
$C_c^\infty(\R\times\R^3)$ in $\dot X^1(\R\times\R^3)$, for any $\eps>0$ there exists $\psi_\eps\in C_c^\infty(\R\times\R^3)$ such that
\begin{align}\label{approxwinfty2}
\|w_\infty-\psi_\eps\|_{\dot X^1(\R\times\R^3)}<\tfrac \eps 3.
\end{align}
Using a change of variables, we estimate
\begin{align*}
\|v_n(t-\lambda_n^2 t_n, x+x_n)&-\lambda_n^{-\frac12}\psi_\eps(\lambda_n^{-2}t, \lambda_n^{-1}x)\|_{\dot X^1(\R\times\R^3)}\\
&\le \|v_n(t-\lambda_n^2 t_n)-\tilde v_n(t)\|_{\dot X^1(\R\times\R^3)}+\|w_\infty-\psi_\eps\|_{\dot X^1(\R\times\R^3)}\\
&\quad+\|\tilde v_n(t,x)-\lambda_n^{-\frac12}w_\infty(\lambda_n^{-2},\lambda_n^{-1}(x-x_n))\|_{\dot X^1(\R\times\R^3)}.
\end{align*}
In view of \eqref{vncase2} and \eqref{approxwinfty2}, proving \eqref{dense2} reduces to showing
\begin{align}\label{remaincase2}
\|\tilde v_n(t,x)-\lambda_n^{-\frac 12}w_\infty(\lambda_n^{-2}t,\lambda_n^{-1}(x-x_n))\|_{\dot X^1(\R\times\R^3)}< \tfrac \eps 3
\end{align}
for sufficiently large $n$ and $T$.

To prove \eqref{remaincase2} we discuss two different time regimes.  On the middle time interval $|t|\leq \lambda_n^2 T$, we have
\begin{align*}
&\|\tilde v_n(t,x)-\lambda_n^{-\frac 12}w_\infty(\lambda_n^{-2}t,\lambda_n^{-1}(x-x_n))\|_{\dot X^1(\{|t|\leq \lambda_n^2T\}\times\R^3)}\\
&\lsm\|\chi_n w_n+z_n-w_\infty\|_{\dot X^1([-T,T]\times\R^3)}\\
&\lsm \|(1-\chi_n)w_\infty\|_{\dot X^1([-T,T]\times\R^3)}+\|\chi_n(w_n-w_\infty)\|_{\dot X^1([-T,T]\times\R^3)}+\|z_n\|_{\dot X^1([-T,T]\times\R^3)},
\end{align*}
which converges to zero by \eqref{258}, \eqref{258'}, and \eqref{209}.

We now consider $|t|> \lambda_n^2 T$; by symmetry, it suffices to control the contribution of positive times.  Using the Strichartz inequality, we estimate
\begin{align*}
\|\tilde v_n(t,x)&-\lambda_n^{-\frac 12}w_\infty(\lambda_n^{-2}t,\lambda_n^{-1}(x-x_n))\|_{\dot X^1((\lambda_n^2T, \infty)\times\R^3)}\\
&=\|e^{i(t-T)\Delta_{\Omega_n}}[\chi_nw_n(T)+z_n(T)]-w_\infty\|_{\dot X^1((T,\infty)\times\R^3)}\\
&\lsm \|z_n(T)\|_{\dot H^1_D(\Omega_n)}+\|\nabla[\chi_n(w_\infty-w_n)]\|_{L_x^2} +\|w_\infty\|_{\dot X^1((T,\infty)\times\R^3)}\\
&\quad +\|e^{i(t-T)\Delta_{\Omega_n}}[\chi_nw_\infty(T)]\|_{\dot X^1((T,\infty)\times\R^3)}\\
&=o(1) +\|e^{i(t-T)\Delta_{\Omega_n}}[\chi_nw_\infty(T)]\|_{\dot X^1((T,\infty)\times\R^3)} \qtq{as} n, T\to \infty
\end{align*}
by \eqref{209'}, \eqref{258'}, and the monotone convergence theorem.  Using the triangle and Strichartz inequalities, we estimate the last term as follows:
\begin{align*}
\|&e^{i(t-T)\Delta_{\Omega_n}}[\chi_n w_\infty(T)]\|_{\dot X^1((T,\infty)\times\R^3)}\\
&\lsm\|[e^{i(t-T)\Delta_{\Omega_n}}-e^{i(t-T)\Delta}][\chi_nw_\infty(T)]\|_{\dot X^1((T,\infty)\times\R^3)}+\|\nabla[(1-\chi_n)w_\infty]\|_{L_x^2}\\
&\quad+\|\nabla[e^{-iT\Delta}w_\infty(T)-w_+]\|_{L_x^2} + \|e^{it\Delta}w^+\|_{\dot X^1((T,\infty)\times\R^3)},
\end{align*}
which converges to zero by letting $n\to \infty$ and then $T\to \infty$ by Theorem~\ref{T:LF}, \eqref{as2}, and the monotone convergence theorem.

Putting everything together we obtain \eqref{remaincase2} and so \eqref{dense2}.  This completes the proof of Theorem~\ref{T:embed2}.
\end{proof}

Our next result concerns the scenario when the rescaled obstacles $\Omega_n^c$ are retreating to infinity (cf. Case 3 in Theorem~\ref{T:LPD}).

\begin{thm}[Embedding nonlinear profiles for retreating obstacles]\label{T:embed3}
Let $\{t_n\}\subset \R$ be such that $t_n\equiv0$ or $t_n\to \pm\infty$. Let $\{x_n\}\subset \Omega$ and $\{\lambda_n\}\subset 2^{\Z}$ be such that $\frac{d(x_n)}{\lambda_n}\to \infty$. Let $\phi\in\dot H^1(\R^3)$ and define
\begin{align*}
\phi_n(x)=\lambda_n^{-\frac 12}e^{i\lambda_n^2 t_n\Delta_\Omega}\bigl[(\chi_n\phi)\bigl(\tfrac{x-x_n}{\lambda_n}\bigr)\bigr],
\end{align*}
where $\chi_n(x)=1-\Theta(\lambda_n|x|/d(x_n))$.  Then for $n$ sufficiently large there exists a global solution $v_n$ to $\text{NLS}_\Omega$ with initial data $v_n(0)=\phi_n$ which satisfies
\begin{align*}
\|v_n\|_{L_{t,x}^{10}(\R\times\Omega)}\lsm 1,
\end{align*}
with the implicit constant depending only on $\|\phi\|_{\dot H^1}$.  Furthermore, for every $\eps>0$ there exist $N_\eps\in \N$ and
$\psi_\eps\in C_c^{\infty}(\R\times\R^3)$ such that for all $n\ge N_\eps$ we have
\begin{align}\label{apcase3}
\|v_n(t-\lambda_n^2 t_n, x+x_n)-\lambda_n^{-\frac12}\psi_\eps(\lambda_n^{-2}t, \lambda_n^{-1} x)\|_{\dot X^1(\R\times\R^3)}<\eps.
\end{align}
\end{thm}

\begin{proof} The proof of this theorem follows the general outline of the proof of Theorem~\ref{T:embed2}.  It consists of the same five steps.  Throughout the proof we will write $-\Delta=-\Delta_{\R^3}$.

\textbf{Step 1:} Constructing global solutions to $\text{NLS}_{\R^3}$.

Let $\theta:=\frac 1{100}$.  As in the proof of Theorem~\ref{T:embed2}, the construction of the solutions to $\text{NLS}_{\R^3}$ depends on the behaviour of $t_n$.  If $t_n\equiv0$, we let $w_n$ and $w_\infty$ be solutions to $\text{NLS}_{\R^3}$ with initial data $w_n(0)=\phi_{\le(d(x_n)/\lambda_n)^{\theta}}$
and $w_\infty(0)=\phi$.  If $t_n\to \pm\infty$, we let $w_n$ and $w_\infty$ be solutions to $\text{NLS}_{\R^3}$ satisfying
\begin{align*}
\|w_n(t)-e^{it\Delta}\phi_{\le (d(x_n)/\lambda_n)^{\theta}}\|_{\dot H^1(\R^3)}\to0 \qtq{and} \|w_\infty(t)-e^{it\Delta}\phi\|_{\dot H^1(\R^3)}\to 0
\end{align*}
as $t\to \pm \infty$.

In all cases, \cite{CKSTT:gwp} implies that $w_n$ and $w_\infty$ are global solutions obeying global spacetime norms.  Moreover, arguing as in the proof of Theorem~\ref{T:embed2} and invoking perturbation theory and the persistence of regularity result Lemma \ref{lm:persistencer3}, we see that $w_n$ and $w_\infty$ satisfy the following:
\begin{equation}\label{cond3}
\left\{ \quad \begin{aligned}
&\|w_n\|_{\dot S^1(\R\times\R^3)}+\|w_\infty\|_{\dot S^1(\R\times\R^3)}\lsm 1,\\
&\lim_{n\to \infty}\|w_n-w_\infty\|_{\dot S^1(\R\times\R^3)}= 0,\\
&\||\nabla|^s w_n\|_{\dot S^1(\R\times\R^3)}\lsm \bigl(\tfrac{d(x_n)}{\lambda_n}\bigr)^{s\theta} \qtq{for all} s\ge 0.
\end{aligned} \right.
\end{equation}

\textbf{Step 2:} Constructing the approximate solution to $\text{NLS}_{\Omega}$.

Fix $T>0$ to be chosen later.  We define
\begin{align*}
\tilde v_n(t,x):=\begin{cases} \lambda_n^{-\frac12}[\chi_nw_n](\lambda_n^{-2}t, \lambda_n^{-1}(x-x_n)), & |t|\le \lambda_n^2 T, \\
e^{i(t-\lambda_n^2 T)\Delta_\Omega}\tilde v_n(\lambda_n^2 T,x), & t>\lambda_n^2 T, \\
e^{i(t+\lambda_n^2 T)\Delta_\Omega}\tilde v_n(-\lambda_n^2 T,x), & t<-\lambda_n^2 T.
\end{cases}
\end{align*}
Note that $\tilde v_n$ has finite scattering size; indeed, using a change of variables, the Strichartz inequality, H\"older, Sobolev embedding, and \eqref{cond3}, we get
\begin{align}\label{tildevn3}
\|\tilde v_n\|_{L_{t,x}^{10}(\R\times\Omega)}
&\lsm \|\chi_n w_n\|_{L_{t,x}^{10}([-T,T]\times\Omega_n)}+\|\chi_n w_n(\pm T)\|_{\dot H^1_D(\Omega_n)} \notag\\
&\lsm \|w_n\|_{\dot S^1(\R\times\R^3)} + \|\nabla\chi_n\|_{L_x^3} \|w_n(\pm T)\|_{L_x^6} + \|\chi_n\|_{L_x^\infty} \|\nabla w_n(\pm T)\|_{L_x^2}\notag\\
&\lsm 1,
\end{align}
where $\Omega_n:=\lambda_n^{-1}(\Omega-\{x_n\})$.

\textbf{Step 3:}  Asymptotic agreement of the initial data:
\begin{align}\label{n0}
\lim_{T\to\infty}\limsup_{n\to\infty}\|(-\Delta_\Omega)^{\frac12}e^{it\Delta_\Omega}[\tilde v_n(\lambda_n^2t_n)-\phi_n]\|_{L_t^{10}L_x^{\frac{30}{13}}(\R\times\Omega)}=0.
\end{align}

We first consider the case when $t_n\equiv0$.  By Strichartz and a change of variables,
\begin{align*}
\|&(-\Delta_{\Omega})^{\frac 12}e^{it\Delta_\Omega}[\tilde v_n(0)-\phi_n]\|_{L_t^{10}L_x^{\frac{30}{13}}(\R\times\Omega)}\\
&\lsm \|(-\Delta_{\Omega_n})^{\frac 12}[\chi_n\phi_{>(d(x_n)/\lambda_n)^{\theta}}]\|_{L^2_x(\Omega_n)}\\
&\lsm \|\nabla \chi_n\|_{L_x^3}\|\phi_{>(d(x_n)/\lambda_n)^{\theta}}\|_{L_x^6}+\|\chi_n\|_{L_x^\infty}\|\nabla \phi_{>(d(x_n)/\lambda_n)^{\theta}}\|_{L_x^2}\to0\qtq{as} n\to \infty.
\end{align*}

It remains to prove \eqref{n0} when $t_n\to \infty$; the case $t_n\to-\infty$ can be treated similarly.  As $T$ is fixed, for sufficiently large $n$ we have $t_n>T$ and so
\begin{align*}
\tilde v_n(\lambda_n^2t_n,x)=e^{i(t_n-T)\lambda_n^2\Delta_{\Omega}}\bigl[\lambda_n^{-\frac 12}(\chi_nw_n(T))\bigl(\tfrac{x-x_n}{\lambda_n}\bigr)\bigr].
\end{align*}
Thus, by a change of variables and the Strichartz inequality,
\begin{align}
\|(-\Delta_{\Omega})^{\frac 12}& e^{it\Delta_{\Omega}}[\tilde v_n(\lambda_n^2 t_n)-\phi_n]\|_{L_t^{10}L_x^{\frac{30}{13}}(\R\times\Omega)}\notag\\
&=\|(-\Delta_{\Omega_n})^{\frac 12}e^{it\Delta_{\Omega_n}}[e^{-iT\Delta_{\Omega_n}}(\chi_nw_n(T))-\chi_n\phi)]\|_{L_t^{10}L_x^{\frac{30}{13}}(\R\times\Omega_n)}\notag\\
&\lsm \|(-\Delta_{\Omega_n})^{\frac 12}[\chi_n(w_n(T)-w_\infty(T))]\|_{L^2(\Omega_n)}\label{n1}\\
&\quad +\|(-\Delta_{\Omega_n})^{\frac12}e^{it\Delta_{\Omega_n}}[e^{-iT\Delta_{\Omega_n}}(\chi_nw_\infty(T))-\chi_n\phi]\|_{L_t^{10}L_x^{\frac{30}{13}}(\R\times\Omega_n)}\label{n2}.
\end{align}
Using \eqref{cond3} and Sobolev embedding, we see that
\begin{align*}
\eqref{n1}\lsm \|\nabla\chi_n\|_{L_x^3}\|w_n(T)-w_\infty(T)\|_{L_x^6}+\|\chi_n\|_{L_x^\infty}\|\nabla[w_n(T)-w_\infty(T)]\|_{L_x^2} \to 0
\end{align*}
as $n\to \infty$.  The proof of
$$
\lim_{T\to\infty}\limsup_{n\to\infty}\eqref{n2}=0
$$
is identical to the proof of \eqref{n23} in Theorem~\ref{T:embed2} and we omit it.  This completes the proof of \eqref{n0}.

\textbf{Step 4:}  Proving that $\tilde v_n$ is an approximate solution to $\text{NLS}_{\Omega}$ in the sense that
\begin{align}\label{n6}
\lim_{T\to\infty}\limsup_{n\to\infty}\|(-\Delta_\Omega)^{\frac12}[(i\partial_t+\Delta_\Omega)\tilde v_n-|\tilde v_n|^4\tilde v_n]\|_{\dot N^0(\R\times\Omega)}=0.
\end{align}

We first verify \eqref{n6} for $|t|>\lambda_n^2 T$.  By symmetry, it suffices to consider positive times.  Arguing as in the proof of Theorem~\ref{T:embed2}, we see that in this case \eqref{n6} reduces to
\begin{align}\label{n7}
\lim_{T\to \infty}\limsup_{n\to\infty}\|e^{i(t-\lambda_n^2T)\Delta_{\Omega}}\tilde v_n(\lambda_n^2T)\|_{L_{t,x}^{10}((\lambda_n^2T,\infty)\times\Omega)}=0.
\end{align}
Let $w_+$ denote the forward asymptotic state of $w_\infty$.  Using a change of variables and the Strichartz inequality, we get
\begin{align*}
&\|e^{i(t-\lambda_n^2 T)\Delta_\Omega}\tilde v_n(\lambda_n^2T)\|_{L_{t,x}^{10}((\lambda_n^2T,\infty)\times\Omega)}\\
&=\|e^{it\Delta_{\Omega_n}}[\chi_nw_n(T)]\|_{L_{t,x}^{10}((0,\infty)\times\Omega_n)}\\
&\lsm \|e^{it\Delta_{\Omega_n}}[\chi_ne^{iT\Delta}w_+]\|_{L_{t,x}^{10}((0,\infty)\times\Omega_n)}+\|\chi_n[w_\infty(T)-e^{iT\Delta}w_+]\|_{\dot H^1(\R^3)}\\
&\quad+\|\chi_n[w_\infty(T)-w_n(T)]\|_{\dot H^1(\R^3)}\\
&\lsm \|[e^{it\Delta_{\Omega_n}}-e^{it\Delta}][\chi_n e^{iT\Delta}w_+]\|_{L_{t,x}^{10}((0,\infty)\times\R^3)}+\|(1-\chi_n)e^{iT\Delta}w_+\|_{\dot H^1(\R^3)}\\
&\quad +\|e^{it\Delta}w_+\|_{L_{t,x}^{10}((T,\infty)\times\R^3)}+\|w_\infty(T) -e^{iT\Delta}w_+\|_{\dot H^1(\R^3)}+\|w_\infty(T)-w_n(T)\|_{\dot H^1(\R^3)},
\end{align*}
which converges to zero by letting $n\to \infty$ and then $T\to \infty$ in view of Corollary~\ref{C:LF} (and the density of $C_c^\infty(\R^3)$ functions in
$\dot H^1(\R^3)$), \eqref{cond3}, the definition of $w_+$, and the monotone convergence theorem.

Next we show \eqref{n6} on the middle time interval $|t|\le \lambda_n^2 T$.  We compute
\begin{align*}
[(i\partial_t+\Delta_{\Omega})\tilde v_n-|\tilde v_n|^4\tilde v_n](t,x)
&=\lambda_n^{-\frac 52}[(\chi_n-\chi_n^5)|w_n|^4w_n](\lambda_n^{-2}t,\lambda_n^{-1}(x-x_n))\\
&\quad+2\lambda_n^{-\frac 52}[\nabla\chi_n \cdot\nabla w_n](\lambda_n^{-2} t, \lambda_n^{-1}(x-x_n))\\
&\quad+\lambda_n^{-\frac 52}[\Delta\chi_nw_n](\lambda_n^{-2} t,\lambda_n^{-1}(x-x_n)).
\end{align*}
Thus, using a change of variables and the equivalence of Sobolev spaces we obtain
\begin{align}
\|(-\Delta_\Omega)^{\frac 12}&[(i\partial_t+\Delta)\tilde v_n-|\tilde v_n|^4\tilde v_n]\|_{\dot N^0((|t|\le \lambda_n^2 T)\times\Omega)}\notag\\
&\lsm \|\nabla[(\chi_n-\chi_n^5)|w_n|^4 w_n]\|_{\dot N^0([-T,T]\times\Omega_n)}\label{n9}\\
&\quad+\|\nabla(\nabla\chi_n\cdot \nabla w_n)\|_{\dot N^0([-T,T]\times\Omega_n)}+\|\nabla (\Delta\chi_nw_n)\|_{\dot N^0([-T,T]\times\Omega_n)}.\label{n10}
\end{align}
Using H\"older, we estimate the contribution of \eqref{n9} as follows:
\begin{align*}
\eqref{n9}&\lsm \|(\chi_n-\chi_n^5)|w_n|^4 \nabla w_n\|_{L_{t,x}^{\frac{10}7}}+\|\nabla \chi_n(1-5\chi_n^4)w_n^5\|_{L_t^{\frac 53} L_x^{\frac{30}{23}}}\\
&\lsm\|\nabla w_n\|_{L_{t,x}^{\frac{10}3}}\Bigl[\|w_n-w_\infty\|_{L_{t,x}^{10}}^4+\|1_{|x|\sim \frac{d(x_n)}{\lambda_n}} w_\infty\|_{L_{t,x}^{10}}^4\Bigr]\\
&\quad+\|w_n\|_{L_t^5 L_x^{30}}\|\nabla \chi_n\|_{L_x^3}\Bigl[\|w_n-w_\infty\|_{L_{t,x}^{10}}^4+\|1_{|x|\sim\frac{d(x_n)}{\lambda_n}}w_\infty\|_{L_{t,x}^{10}}^4\Bigr]\to0,
\end{align*}
by the dominated convergence theorem and \eqref{cond3}.  Similarly,
\begin{align*}
\eqref{n10}&\lsm T \Bigl[\|\Delta\chi_n\|_{L_x^\infty} \|\nabla w_n\|_{L_t^\infty L_x^2} +\|\nabla \chi_n\|_{L_x^\infty}\|\Delta w_n\|_{L_t^\infty L_x^2}
+\|\nabla\Delta\chi_n\|_{L_x^3}\|w_n\|_{L_t^\infty L_x^6}\Bigr]\\
&\lsm T\Bigl[\bigl(\tfrac{d(x_n)}{\lambda_n}\bigr)^{-2}+\bigl(\tfrac{d(x_n)}{\lambda_n}\bigr)^{\theta-1}\Bigr] \to 0 \qtq{as} n\to \infty.
\end{align*}
This completes the proof of \eqref{n6}.

\textbf{Step 5:} Constructing $v_n$ and approximation by $C_c^\infty$ functions.

Using \eqref{tildevn3}, \eqref{n0}, and \eqref{n6}, and invoking the stability result Theorem~\ref{T:stability}, for $n$ sufficiently large we obtain a global solution $v_n$ to $\text{NLS}_\Omega$ with initial data $v_n(0)=\phi_n$ which satisfies
\begin{align*}
\|v_n\|_{L_{t,x}^{10}(\R\times\Omega)}\lsm 1
\qtq{and}
\lim_{T\to\infty}\limsup_{n\to \infty}\|v_n(t-\lambda_n^2t_n)-\tilde v_n (t)\|_{\dot S^1(\R\times\Omega)}=0.
\end{align*}
It remains to prove the approximation result \eqref{apcase3}.

From the density of $C_c^\infty(\R\times\R^3)$ in $\dot X^1(\R\times\R^3)$, for any $\eps>0$ we can find $\psi_\eps\in C_c^\infty(\R\times\R^3)$ such that
\begin{align*}
\|w_\infty-\psi_\eps\|_{\dot X^1(\R\times\R^3)}< \tfrac \eps 3.
\end{align*}
Thus, to prove \eqref{apcase3} it suffices to show
\begin{align}\label{n11}
\|\tilde v_n(t,x)-\lambda_n^{-\frac 12}w_\infty(\lambda_n^{-2}t, \lambda_n^{-1}(x-x_n))\|_{\dot X^1(\R\times\R^3)}<\tfrac \eps 3
\end{align}
for $n, T$ sufficiently large.  A change of variables gives
\begin{align*}
\text{LHS}\eqref{n11}
&\le \|\chi_n w_n-w_\infty\|_{\dot X^1([-T,T]\times\R^3)}+\|e^{i(t-T)\Delta_{\Omega_n}}[\chi_nw_n(T)]-w_\infty\|_{\dot X^1((T,\infty)\times\R^3)}\\
&\quad+\|e^{i(t+T)\Delta_{\Omega_n}}[\chi_nw_n(-T)]-w_\infty\|_{\dot X^1((-\infty,-T)\times\R^3)}.
\end{align*}
We estimate the contribution from each term separately. For the first term we use the monotone convergence theorem and \eqref{cond3} to see that
\begin{align*}
\|\chi_n w_n-w_\infty\|_{\dot X^1([-T,T]\times\R^3)}&\lsm \|(1-\chi_n)w_\infty\|_{\dot X^1(\R\times\R^3)}+\|w_n-w_\infty\|_{\dot X^1(\R\times\R^3)}\to 0,
\end{align*}
as $n\to \infty$. For the second term we use Strichartz to get
\begin{align*}
&\|e^{i(t-T)\Delta_{\Omega_n}}[\chi_n w_n(T)]-w_\infty\|_{\dot X^1((T,\infty)\times\R^3)}\\
&\lsm \|w_\infty\|_{\dot X^1((T,\infty)\times\R^3)}+\|\chi_n[w_\infty(T)-w_n(T)]\|_{\dot H^1(\R^3)}\\
&\quad+\|e^{i(t-T)\Delta_{\Omega_n}}[\chi_n w_\infty(T)]\|_{\dot X^1((T,\infty))}\\
&\lsm \|w_\infty\|_{\dot X^1((T,\infty)\times\R^3)}+\|w_\infty(T)-w_n(T)\|_{\dot H^1(\R^3)}+\|(1-\chi_n)w_\infty(T)\|_{\dot H^1(\R^3)}\\
&\quad+\|[e^{i(t-T)\Delta_{\Omega_n}}-e^{i(t-T)\Delta}][\chi_nw_\infty(T)]\|_{\dot X^1(\R\times\R^3)}+\|e^{it\Delta}w_+\|_{\dot X^1((T,\infty)\times\R^3)}\\
&\quad+\|w_+-e^{-iT\Delta}w_\infty(T)\|_{\dot H^1(\R^3)}\to 0 \qtq{as} n\to \infty \qtq{and then} T\to \infty
\end{align*}
by Theorem~\ref{T:LF}, \eqref{cond3}, the definition of the asymptotic state $w_+$, and the monotone convergence theorem.
The third term can be treated analogously to the second term.

This completes the proof of \eqref{n11} and with it, the proof of Theorem~\ref{T:embed3}.
\end{proof}

Our final result in this section treats the case when the obstacle expands to fill a halfspace (cf. Case~4 in Theorem~\ref{T:LPD}).

\begin{thm}[Embedding $\text{NLS}_{\HH}$ into $\text{NLS}_{\Omega}$]\label{T:embed4}
Let $\{t_n\}\subset \R$ be such that $t_n\equiv0$ or $t_n\to\pm\infty$. Let $\{\lambda_n\}\subset 2^{\Z}$ and $\{x_n\}\subset \Omega$ be such that
\begin{align*}
\lambda_n\to 0 \qtq{and} \tfrac{d(x_n)}{\lambda_n}\to d_\infty>0.
\end{align*}
Let $x_n^*\in \partial\Omega$ be such that $|x_n-x_n^*|=d(x_n)$ and let $R_n\in SO(3)$ be such that $R_n e_3=\frac{x_n-x_n^*}{|x_n-x_n^*|}$.
Finally, let $\phi\in \dot H^1_D(\HH)$ and define
\begin{align*}
\phi_n(x)=\lambda_n^{-\frac 12}e^{i\lambda_n^2t_n\Delta_\Omega}\bigl[\phi\bigl(\tfrac{R_n^{-1}(x-x_n^*)}{\lambda_n}\bigr)\bigr].
\end{align*}
Then for $n$ sufficiently large there exists a global solution $v_n$ to $\text{NLS}_\Omega$ with initial data $v_n(0)=\phi_n$ which satisfies
\begin{align*}
\|v_n\|_{L_{t,x}^{10}(\R\times\Omega)}\lsm 1,
\end{align*}
with the implicit constant depending only on $\|\phi\|_{\dot H^1}$.  Furthermore, for every $\eps>0$ there exist $N_\eps\in \N$ and
$\psi_\eps\in C_c^\infty(\R\times\HH)$ such that for all $n\geq N_\eps$ we have
\begin{align}\label{ap4}
\|v_n(t-\lambda_n^2 t_n, R_nx+x_n^*)-\lambda_n^{-\frac12}\psi_\eps(\lambda_n^{-2}t, \lambda_n^{-1}x)\|_{\dot X^1(\R\times\R^3)}<\eps.
\end{align}
\end{thm}

\begin{proof}
Again, the proof follows the outline of the proofs of Theorems~\ref{T:embed2} and \ref{T:embed3}.

\textbf{Step 1:} Constructing global solutions to $\text{NLS}_{\HH}$.

Let $\theta:=\frac 1{100}$.  If $t_n\equiv0$, let $w_n$ and $w_\infty$ be solutions to $\text{NLS}_{\HH}$ with initial data
$w_n(0)=\phi_{\le \lambda_n^{-\theta}}$ and $w_\infty(0)=\phi$.  If $t_n\to \pm\infty$, let $w_n$ and $w_\infty$ be solutions to $\text{NLS}_{\HH}$ that satisfy
\begin{align}\label{m12}
\|w_n(t)-e^{it\Delta_{\HH}}\phi_{\le \lambda_n^{-\theta}}\|_{\dot H^1_D(\HH)}\to 0 \qtq{and} \|w_\infty(t)-e^{it\Delta_{\HH}}\phi\|_{\dot H^1_D(\HH)}\to 0,
\end{align}
as $t\to \pm \infty$.

In all cases, \cite{CKSTT:gwp} implies that $w_n$ and $w_\infty$ are global solutions and obey
\begin{align*}
\|w_n\|_{\dot S^1(\R\times\HH)}+\|w_\infty\|_{\dot S^1(\R\times\HH)}\lsm 1,
\end{align*}
with the implicit constant depending only on $\|\phi\|_{\dot H^1_D(\HH)}$.  Indeed, we may interpret such solutions as solutions to $\text{NLS}_{\R^3}$ that are odd under reflection in $\partial\HH$.  Moreover, arguing as in the proof of Theorems~\ref{T:embed2} and \ref{T:embed3} and using the stability result from \cite{CKSTT:gwp} and the persistence of regularity result Lemma~\ref{lm:persistenceh}, we have
\begin{align}\label{cond4}
\begin{cases}
\lim_{n\to \infty}\|w_n-w_\infty\|_{\dot S^1(\R\times\HH)}=0,\\
\|(-\Delta_\HH)^{\frac k2}w_n\|_{L_t^\infty L_x^2(\R\times\HH)}\lsm\lambda_n^{-\theta(k-1)} \qtq{for} k=1,2,3.
\end{cases}
\end{align}

\textbf{Step 2:} Constructing approximate solutions to $\text{NLS}_\Omega$.

Let $\Omega_n:=\lambda_n^{-1}R_n^{-1}(\Omega-\{x_n^*\})$ and let $T>0$ to be chosen later.  On the middle time interval $|t|<\lambda_n^2 T$, we embed
$w_n$ by using a boundary straightening diffeomorphism $\Psi_n$ of a neighborhood of zero in $\Omega_n$ of size $L_n:=\lambda_n^{-2\theta}$ into a corresponding neighborhood in $\HH$.

To this end, we define a smooth function $\psi_n$ on the set $|x^\perp|\le L_n$ so that $x^\perp\mapsto (x^\perp, -\psi_n(x^\perp))$ traces out $\partial\Omega_n$.  Here and below we write $x\in \R^3$ as $x=(x^\perp, x_3)$.  By our choice of $R_n$, $\partial \Omega_n$ has unit normal $e_3$ at zero.  Moreover, $\partial\Omega_n$ has curvatures that are $O(\lambda_n)$.  Thus,
$\psi_n$ satisfies the following:
\begin{align}\label{psin}
\begin{cases}
&\psi_n(0)=0, \quad \nabla\psi_n(0)=0, \quad |\nabla\psi_n(x^\perp)|\lsm \lambda_n^{1-2\theta},\\
&|\partial^{\alpha}\psi_n(x^\perp)|\lsm\lambda_n^{|\alpha|-1} \qtq{for all}  |\alpha|\ge 2.
\end{cases}
\end{align}

We now define the map $\Psi_n: \Omega_n\cap\{|x^\perp|\le L_n\}\to \HH$ and a cutoff $\chi_n:\R^3\to[0,1]$ via
\begin{align*}
\Psi_n(x):=(\xp, x_3+\psi_n(x^\perp)) \qtq{and} \chi_n(x):=1-\Theta\bigl(\tfrac{x}{L_n}\bigr).
\end{align*}
Note that on the domain of $\Psi_n$, which contains $\supp\chi_n$, we have
\begin{align}\label{detpsin}
|\det(\partial \Psi_n)|\sim 1 \qtq{and} |\partial\Psi_n|\lsm 1.
\end{align}

We are now ready to define the approximate solution. Let $\tilde w_n:=\chi_nw_n$ and define
\begin{align*}
\tilde v_n(t,x):=\begin{cases} \lambda_n^{-\frac12}[\tilde
w_n(\lambda_n^{-2}t)\circ\Psi_n](\lambda_n^{-1}R_n^{-1}(x-x_n^*)), &|t|\le \lambda_n^2 T, \\
e^{i(t-\lambda_n^2 T)\Delta_\Omega}\tilde v_n(\lambda_n^2 T,x), &t>\lambda_n^2 T,\\
e^{i(t+\lambda_n^2 T)\Delta_\Omega}\tilde v_n(-\lambda_n^2T,x), &t<-\lambda_n^2 T .
\end{cases}
\end{align*}
We first prove that $\tilde v_n$ has finite scattering size.  Indeed, by the Strichartz inequality, a change of variables, and \eqref{detpsin},
\begin{align}\label{tildevn4}
\|\tilde v_n\|_{L_{t,x}^{10}(\R\times\Omega)}
&\lsm \|\tilde w_n\circ\Psi_n\|_{L_{t,x}^{10}(\R\times\Omega_n)}+\|\tilde w_n(\pm T)\circ\Psi_n\|_{\dot H^1_D(\Omega_n)}\notag\\
&\lsm \|\tilde w_n\|_{L_{t,x}^{10}(\R\times\HH)} + \|\tilde w_n(\pm T)\|_{\dot H^1_D(\HH)}\lsm 1.
\end{align}

\textbf{Step 3:}  In this step we prove asymptotic agreement of the initial data, namely,
\begin{align}\label{match4}
\lim_{T\to\infty}\limsup_{n\to \infty}\|(-\Delta_\Omega)^{\frac12}e^{it\Delta_\Omega}[\tilde v_n(\lambda_n^2 t_n)-\phi_n]\|_{L_t^{10} L_x^{\frac{30}{13}}(\R\times\Omega)}=0.
\end{align}

We discuss two cases. If $t_n\equiv0$, then by Strichartz and a change of variables,
\begin{align*}
\| & (-\Delta_{\Omega})^{\frac 12} e^{it\Delta_\Omega}[\tilde v_n(0)-\phi_n]\|_{L_t^{10} L_x^{\frac{30}{13}}(\R\times\Omega)}\\
&\lsm \|(\chi_n\phi_{\le \lambda_n^{-\theta}})\circ\Psi_n-\phi\|_{\dot H^1_D(\Omega_n)}\\
&\lsm \|\nabla[(\chi_n\phi_{>\lambda_n^{-\theta}})\circ\Psi_n]\|_{L^2_x}+\|\nabla[(\chi_n\phi)\circ\Psi_n-\chi_n\phi]\|_{L^2_x}+\|\nabla[(1-\chi_n)\phi]\|_{L^2_x}.
\end{align*}
As $\lambda_n\to 0$ we have $\|\nabla \phi_{>\lambda_n^{-\theta}}\|_{L^2_x}\to 0$ as $n\to \infty$; thus, using \eqref{detpsin} we see that the first
term converges to $0$.  For the second term, we note that $\Psi_n(x)\to x$ in $C^1$; thus, approximating $\phi$ by $C_c^\infty(\HH)$ functions we see that the second term converges to $0$. Finally, by the dominated convergence theorem and $L_n=\lambda_n^{-2\theta}\to \infty$, the last term converges to $0$.

It remains to prove \eqref{match4} when $t_n\to +\infty$; the case when $t_n\to -\infty$ can be treated similarly.
Note that as $T>0$ is fixed, for $n$ sufficiently large we have $t_n>T$ and so
\begin{align*}
\tilde v_n(\lambda_n^2t_n,x)&=e^{i(t_n-T)\lambda_n^2\Delta_\Omega}[\lambda_n^{-\frac12}(\tilde w_n(T)\circ\Psi_n)(\lambda_n^{-1}R_n^{-1}(x-x_n^*))].
\end{align*}
Thus, a change of variables gives
\begin{align}
\|(-\Delta_\Omega)^{\frac12} &e^{it\Delta_\Omega}[\tilde v_n(\lambda_n^2 t_n)-\phi_n]\|_{L_t^{10} L_x^{\frac{30}{13}}(\R\times\Omega)}\notag\\
&\lsm \|(-\Delta_{\Omega_n})^{\frac 12}[\tilde w_n(T)\circ\Psi_n-w_\infty(T)]\|_{L^2_x}\label{n13}\\
&\quad+\|(-\Delta_{\Omega_n})^{\frac 12}[e^{i(t-T)\Delta_{\Omega_n}}w_\infty(T)-e^{it\Delta_{\Omega_n}}\phi]\|_{L_t^{10} L_x^{\frac{30}{13}}(\R\times\Omega_n)}.\label{n12}
\end{align}
Using the triangle inequality,
\begin{align*}
\eqref{n13}
&\lsm\|(-\Delta_{\Omega_n})^{\frac12}[(\chi_nw_\infty(T))\circ\Psi_n-w_\infty(T)]\|_{L^2_x}\\
&\quad+\|(-\Delta_{\Omega_n})^{\frac 12}[(\chi_n(w_n(T)-w_\infty(T)))\circ\Psi_n]\|_{L^2_x},
\end{align*}
which converges to zero as $n\to \infty$ by \eqref{cond4} and the the fact that $\Psi_n(x)\to x$ in $C^1$.  Using Strichartz, Lemma~\ref{L:n3},
Theorem~\ref{T:LF}, and \eqref{m12}, we see that
\begin{align*}
\eqref{n12}
&\lsm \|e^{i(t-T)\Delta_{\Omega_n}}(-\Delta_{\HH})^{\frac12}w_\infty(T)-e^{it\Delta_{\Omega_n}}(-\Delta_{\HH})^{\frac12}\phi\|_{L_t^{10} L_x^{\frac{30}{13}}(\R\times\Omega_n)}\\
&\quad +\|[(-\Delta_{\Omega_n})^{\frac 12}-(-\Delta_{\HH})^{\frac12}]w_\infty(T)\|_{L^2_x}+\|[(-\Delta_{\Omega_n})^{\frac 12}-(-\Delta_{\HH})^{\frac 12}]\phi\|_{L^2_x}\\
&\lsm\|[e^{i(t-T)\Delta_{\Omega_n}}-e^{i(t-T)\Delta_{\HH}}](-\Delta_{\HH})^{\frac 12}w_\infty(T)\|_{L_t^{10} L_x^{\frac{30}{13}}(\R\times\Omega_n)}\\
&\quad+\|[e^{it\Delta_{\Omega_n}}-e^{it\Delta_{\HH}}](-\Delta_{\HH})^{\frac12}\phi\|_{L_t^{10} L_x^{\frac{30}{13}}(\R\times\Omega_n)}\\
&\quad+\|e^{-iT\Delta_{\HH}}w_\infty(T)-\phi\|_{\dot H^1_D(\HH)}+o(1),
\end{align*}
and that this converges to zero by first taking $n\to \infty$ and then $T\to \infty$.

\textbf{Step 4:} In this step we prove that $\tilde v_n$ is an approximate solution to $\text{NLS}_\Omega$ in the sense that
\begin{align}\label{n14}
\lim_{T\to\infty}\limsup_{n\to\infty}\|(-\Delta_\Omega)^{\frac12}[(i\partial_t+\Delta_\Omega)\tilde v_n-|\tilde v_n|^4\tilde v_n]\|_{\dot N^0(\R\times\Omega)}=0.
\end{align}
We first control the contribution of $|t|\ge \lambda_n^2T$.  As seen previously, this reduces to proving
\begin{align}\label{n15}
\lim_{T\to\infty}\limsup_{n\to\infty}\|e^{i(t-\lambda_n^2T)\Delta_{\Omega}}\tilde v_n(\lambda_n^2 T)\|_{L_{t,x}^{10}((\lambda_n^2 T,\infty)\times\Omega)}=0.
\end{align}
and the analogous estimate in the opposite time direction, which follows similarly.

Let $w_+$ denote the forward asymptotic state of $w_\infty$.  Using Strichartz, our earlier estimate on \eqref{n13}, and the monotone convergence theorem, we see that
\begin{align*}
\|&e^{i(t-\lambda_n^2 T)\Delta_{\Omega}}\tilde v_n(\lambda_n^2T)\|_{L_{t,x}^{10}((\lambda_n^2 T,\infty)\times\Omega)}\\
&=\|e^{i(t-T)\Delta_{\Omega_n}}[\tilde w_n(T)\circ \Psi_n]\|_{L_{t,x}^{10}((T,\infty)\times\Omega_n)}\\
&\lsm \|e^{i(t-T)\Delta_{\Omega_n}}[e^{iT\Delta_{\HH}}w_+]\|_{L_{t,x}^{10}((T,\infty)\times\Omega_n)}+\|w_\infty(T)-e^{iT\Delta_{\HH}}w_+\|_{\dot H^1_D(\HH)}\\
&\quad+\|\tilde w_n(T)\circ\Psi_n-w_\infty(T)\|_{\dot H^1_D(\Omega_n)}\\
&\lsm \|[e^{i(t-T)\Delta_{\Omega_n}}-e^{i(t-T)\Delta_{\HH}}][e^{iT\Delta_{\HH}}w_+]\|_{L_{t,x}^{10}((0,\infty)\times\Omega_n)}\\
&\quad+\|e^{it\Delta_\HH}w_+\|_{L_{t,x}^{10} ((T,\infty)\times\HH)}+o(1)
\end{align*}
and that this converges to zero by Theorem~\ref{T:LF} and the monotone convergence theorem by first taking $n\to \infty$ and then $T\to \infty$.  Thus \eqref{n15} is proved.

Next we control the contribution of the middle time interval $\{|t|\le \lambda_n^2 T\}$ to \eqref{n14}. We compute
\begin{align*}
\Delta(\tilde w_n\circ \Psi_n)&=(\partial_k\tilde w_n\circ\Psi_n)\Delta\Psi_n^k+(\partial_{kl}\tilde w_n\circ\Psi_n)\partial_j\Psi_n^l\partial_j\Psi_n^k,
\end{align*}
where $\Psi_n^k$ denotes the $k$th component of $\Psi_n$ and repeated indices are summed.  As $\Psi_n(x)=x+(0,\psi_n(\xp))$, we have
\begin{align*}
&\Delta\Psi_n^k=O(\partial^2\psi_n), \quad \partial_j\Psi_n^l=\delta_{jl}+O(\partial\psi_n), \\
&\partial_j\Psi_n^l\partial_j\Psi_n^k=\delta_{jl}\delta_{jk}+O(\partial\psi_n)+O((\partial\psi_n)^2),
\end{align*}
where we use $O$ to denote a collection of similar terms.  For example, $O(\partial\psi_n)$ contains terms of the form $c_j\partial_{x_j}\psi_n$ for some constants $c_j\in \R$, which may depend on the indices $k$ and $l$ appearing on the left-hand side.  Therefore,
\begin{align*}
(\partial_k\tilde w_n\circ\Psi_n)\Delta\Psi_n^k&=O\bigl((\partial\tilde w_n\circ\Psi_n)(\partial^2\psi_n)\bigr),\\
(\partial_{kl}\tilde w_n\circ\Psi_n)\partial_j\Psi_n^l\partial_j\Psi_n^k
&=\Delta\tilde w_n\circ\Psi_n+O\bigl(\bigl(\partial^2\tilde w_n\circ\Psi_n\bigr)\bigl(\partial\psi_n+(\partial\psi_n)^2\bigr)\bigr)
\end{align*}
and so
\begin{align*}
(i\partial_t+\Delta_{\Omega_n})&(\tilde w_n\circ \Psi_n)-(|\tilde w_n|^4\tilde w_n)\circ\Psi_n\\
&=[(i\partial_t+\Delta_{\HH})\tilde w_n-|\tilde w_n|^4\tilde w_n]\circ \Psi_n \\
&\quad+O\bigl((\partial\tilde w_n\circ\Psi_n)(\partial^2\psi_n)\bigr)+O\bigl(\bigl(\partial^2\tilde w_n\circ\Psi_n\bigr)\bigl(\partial\psi_n+(\partial\psi_n)^2\bigr)\bigr).
\end{align*}
By a change of variables and \eqref{detpsin}, we get
\begin{align}
\|(-\Delta_\Omega)^{\frac 12}&[(i\partial_t+\Delta_\Omega)\tilde v_n-|\tilde v_n|^4\tilde v_n]\|_{L_t^1L_x^2((|t|\le \lambda_n^2T)\times\Omega)}\notag\\
&=\|(-\Delta_{\Omega_n})^{\frac12}[(i\partial_t+\Delta_{\Omega_n})(\tilde w_n\circ\Psi_n)-(|\tilde w_n|^4\tilde w_n)\circ \Psi_n]\|_{L_t^1L_x^2((|t|\le T)\times\Omega_n)}\notag\\
&\lsm \|(-\Delta_{\Omega_n})^{\frac12}[((i\partial_t+\Delta_{\HH})\tilde w_n-|\tilde w_n|^4\tilde w_n)\circ\Psi_n]\|_{L_t^1L_x^2([-T,T]\times\Omega_n)}\notag\\
&\quad+\|(-\Delta_{\Omega_n})^{\frac 12}[(\partial\tilde w_n\circ \Psi_n)\partial^2\psi_n]\|_{L_t^1L_x^2([-T,T]\times\Omega_n)}\notag\\
&\quad+\bigl\|(-\Delta_{\Omega_n})^{\frac 12}\bigl[(\partial^2\tilde w_n\circ\Psi_n)\bigl(\partial\psi_n+(\partial\psi_n)^2\bigr)\bigr]\bigr\|_{L_t^1L_x^2([-T,T]\times\Omega_n)}\notag\\
&\lsm \|\nabla[(i\partial_t+\Delta_{\HH})\tilde w_n -|\tilde w_n|^4\tilde w_n]\|_{L_t^1L_x^2([-T,T]\times\HH)}\label{n18}\\
&\quad+\|\nabla[(\partial \tilde w_n\circ\Psi_n)\partial^2\psi_n]\|_{L_t^1L_x^2([-T,T]\times\Omega_n)}\label{n16}\\
&\quad+\bigl\|\nabla\bigl[(\partial^2 \tilde w_n\circ \Psi_n)\bigl(\partial\psi_n+(\partial\psi_n)^2\bigr)\bigr]\bigr\|_{L_t^1L_x^2([-T,T]\times\Omega_n)}\label{n17}.
\end{align}
Using \eqref{cond4}, \eqref{psin}, and \eqref{detpsin}, we can control the last two terms as follows:
\begin{align*}
\eqref{n16}
&\lsm\|(\partial\tilde w_n\circ\Psi_n)\partial^3\psi_n\|_{L_t^1L_x^2([-T,T]\times\Omega_n)}+\|(\partial^2\tilde w_n\circ\Psi_n)\partial\Psi_n\partial^2\psi_n\|_{L_t^1L_x^2([-T,T]\times\Omega_n)}\\
&\lsm T\lambda_n^2\|\nabla \tilde w_n\|_{L_t^\infty L_x^2}+T\lambda_n\|\partial^2\tilde w_n\|_{L_t^\infty L_x^2}\\
&\lsm T\lambda_n^2\bigl[\|\nabla \chi_n\|_{L^3_x}\|w_n\|_{L_t^\infty L^6_x}+\|\nabla w_n\|_{L_t^\infty L_x^2}\bigr]\\
&\quad+T\lambda_n\bigl[\|\partial^2 \chi_n\|_{L^3_x}\|w_n\|_{L_t^\infty L^6_x}+\|\nabla \chi_n\|_{L_x^\infty}\|\nabla w_n\|_{L_t^\infty L_x^2}
+\|\partial^2w_n\|_{L_t^\infty L_x^2}\bigr]\\
&\lsm T\lambda_n^2+T\lambda_n[L_n^{-1}+\lambda_n^{-\theta}]\to 0\qtq{as} n\to \infty
\end{align*}
and similarly,
\begin{align*}
\eqref{n17}
&\lsm \|(\partial^2 \tilde w_n\circ\Psi_n)(\partial^2\psi_n+\partial\psi_n\partial^2\psi_n)\|_{L_t^1L_x^2([-T,T]\times\Omega_n)}\\
&\quad+ \|(\partial^3\tilde w_n\circ\Psi_n)[\partial\Psi_n(\partial\psi_n+(\partial\psi_n)^2)]\|_{L_t^1L_x^2([-T,T]\times\Omega_n)}\\
&\lsm T[\lambda_n+\lambda_n^{2-2\theta}] \|\partial^2\tilde w_n\|_{L_t^\infty L_x^2}+T[\lambda_n^{1-2\theta}+\lambda_n^{2-4\theta}]\|\partial^3\tilde w_n\|_{L_t^\infty L_x^2}\\
&\lsm T\lambda_n[L_n^{-1}+\lambda_n^{-\theta}]+T\lambda_n^{1-2\theta}\bigl[\|\partial^3\chi_n\|_{L^3_x}\|w_n\|_{L_t^\infty L^6_x}+\|\partial^2\chi_n\|_{L_x^\infty}\|\nabla w_n\|_{L^2_x}\\
&\quad+\|\nabla\chi\|_{L_x^\infty}\|\partial^2w_n\|_{L_t^\infty L^2_x}+\|\partial^3 w_n\|_{L_t^\infty L^2_x}\bigr]\\
&\lsm T\lambda_n[L_n^{-1}+\lambda_n^{-\theta}]+T\lambda_n^{1-2\theta}\bigl[L_n^{-2}+L_n^{-1}\lambda_n^{-\theta}+\lambda_n^{-2\theta}\bigr]\to 0\qtq{as} n\to \infty.
\end{align*}

Finally, we consider \eqref{n18}.  A direct computation gives
\begin{align*}
(i\partial_t+\Delta_{\HH})\tilde w_n-|\tilde w_n|^4\tilde w_n=(\chi_n-\chi_n^5)|w_n|^4w_n+2\nabla\chi_n\cdot\nabla w_n+\Delta\chi_n w_n.
\end{align*}
We then bound each term as follows:
\begin{align*}
\|\nabla(\Delta \chi_n w_n)\|_{L_t^1L_x^2([-T,T]\times\HH)}
&\lsm T\bigl[ \|\partial^3\chi_n\|_{L^3_x}\|w_n\|_{L^\infty_t L^6_x}+\|\partial^2 \chi_n\|_{L^\infty_x} \|\nabla w_n\|_{L^\infty_t L^2_x} \bigr] \\
&\lsm TL_n^{-2} \to 0 \qtq{as} n\to \infty\\
%%%%%%%
\|\nabla(\nabla\chi_n\cdot \nabla w_n)\|_{L_t^1L_x^2([-T,T]\times\HH)}
&\lsm T\bigl[ \|\partial^2 \chi_n\|_{L^\infty_x} \|\nabla w_n\|_{L^\infty_t L^2_x} + \|\nabla\chi_n\|_{L^\infty_x} \|\partial^2 w_n\|_{L^\infty_t L^2_x}\bigr]\\
&\lsm T[L_n^{-2}+L_n^{-1}\lambda_n^{-\theta}] \to 0 \qtq{as} n\to \infty.
\end{align*}
Finally, for the first term, we have
\begin{align*}
\| & \nabla[(\chi_n-\chi_n^5)|w_n|^4w_n]\|_{\dot N^0([-T,T]\times\HH)}\\
&\lsm \|(\chi_n-\chi_n^5) |w_n|^4\nabla w_n\|_{L_{t,x}^{\frac{10}7}([-T,T]\times\HH)}+\| |w_n|^5\nabla \chi_n\|_{L_t^{\frac53}L_x^{\frac{30}{23}}([-T,T]\times\HH)}\\
&\lsm \|w_n 1_{|x|\sim L_n}\|_{L^{10}_{t,x}}^4 \|\nabla w_n\|_{L^{\frac{10}3}_{t,x}}
	+ \|\nabla \chi_n\|_{L^3_x} \|w_n 1_{|x|\sim L_n}\|_{L^{10}_{t,x}}^4 \|\nabla w_n\|_{L^5_t L^\frac{30}{11}_x}\\
&\lsm \|1_{|x|\sim L_n}w_\infty \|_{L^{10}_{t,x}}^4+\|w_\infty-w_n\|_{L^{10}_{t,x}}^4 \to 0 \qtq{as} n\to \infty.
\end{align*}
This completes the proof of \eqref{n14}.

\textbf{Step 5:}  Constructing $v_n$ and approximating by $C_c^{\infty}$ functions.

Using \eqref{tildevn4}, \eqref{match4}, and \eqref{n14}, and invoking the stability result Theorem~\ref{T:stability}, for $n$ large enough we obtain a global solution $v_n$ to $\text{NLS}_\Omega$ with initial data $v_n(0)=\phi_n$ and
\begin{align*}
\|v_n\|_{L_{t,x}^{10}(\R\times\Omega)}\lsm 1.
\end{align*}
Moreover,
\begin{align}\label{n19}
\lim_{T\to\infty}\limsup_{n\to\infty}\|v_n(t-\lambda_n^2t_n)-\tilde v_n(t)\|_{\dot S^1(\R\times\Omega)}=0.
\end{align}

It remains to prove the approximation result \eqref{ap4}.  By the density of $C_c^{\infty}(\R\times\HH)$ in $\dot X^1(\R\times\HH)$, for
every $\eps>0$ there exists $\psi_\eps\in C_c^{\infty}(\R\times\HH)$ such that
\begin{align*}
\|w_\infty-\psi_\eps\|_{\dot X^1(\R\times\HH)}<\tfrac \eps 3.
\end{align*}
This together with \eqref{n19} reduce matters to showing
\begin{align}\label{c4e3}
\|\tilde v_n(t,x)-\lambda_n^{-\frac 12}w_\infty(\lambda_n^{-2}t, \lambda_n^{-1}R_n^{-1}(x-x_n^*))\|_{\dot X^1(\R\times\R^3)}<\tfrac\eps3
\end{align}
for $n,\ T$ sufficiently large. A change of variables shows that
\begin{align*}
\text{LHS\eqref{c4e3}}&\lsm \|\tilde w_n\circ \Psi_n-w_\infty\|_{\dot X^1([-T,T]\times\R^3)} \\
& \quad +\|e^{i(t-T)\Delta_{\Omega_n}}(\tilde w_n(T)\circ\Psi_n)-w_\infty\|_{\dot X^1((T,\infty)\times\R^3)}\\
&\quad+\|e^{i(t+T)\Delta_{\Omega_n}}(\tilde w_n(-T)\circ\Psi_n)-w_\infty\|_{\dot X^1((-\infty,-T)\times\R^3)}.
\end{align*}
The first term can be controlled as follows:
\begin{align*}
\|\tilde w_n\circ\Psi_n-w_\infty\|_{\dot X^1([-T,T]\times\R^3)}
& \lesssim \| (\chi_n w_\infty)\circ\Psi_n-w_\infty\|_{\dot X^1([-T,T]\times\R^3)}\\
&\quad +\|[\chi_n(w_n-w_\infty)]\circ\Psi_n\|_{\dot X^1([-T,T]\times\R^3)},
\end{align*}
which converges to zero as $n\to\infty$ by \eqref{cond4} and the fact that $\Psi_n(x)\to x$ in $C^1$.
Similarly, we can use the Strichartz inequality to replace $\tilde w_n(T)\circ \Psi_n$ by $w_\infty(T)$ in the second term by making a $o(1)$ error as $n\to \infty$.  Then we can use the convergence of propagators result Theorem~\ref{T:LF} to replace $e^{i(t-T)\Delta_{\Omega_n}}$ by $e^{i(t-T)\Delta_{\HH}}$ with an additional $o(1)$ error.  It then suffices to show
\begin{align*}
\|e^{i(t-T)\Delta_{\HH}}w_\infty(T)-w_\infty\|_{\dot X^1((T,\infty)\times\R^3)}\to 0 \qtq{as}T\to \infty,
\end{align*}
which follows from the fact that $w_\infty$ scatters forward in time, just as in the proofs of Theorem~\ref{T:embed2} and~\ref{T:embed3}.  The treatment of the third term is similar.  This completes the proof of \eqref{ap4} and so the proof of Theorem~\ref{T:embed4}.
\end{proof}

%%%%%%%%%%%%%%%%%%%%%%%%%%%%%%%%%%%%%%%%%%%%%%%%%%%%%%%%%%%%%%%%%%%%%%%%%%%%%%%%%%%%%%%%%%%%%%%%%%%%%%%%%%%%%%%%%%%%%%%%%%%%%%%%%%%%%%
\section{Palais--Smale and the proof of Theorem~ \ref{T:main}}\label{S:Proof}
%%%%%%%%%%%%%%%%%%%%%%%%%%%%%%%%%%%%%%%%%%%%%%%%%%%%%%%%%%%%%%%%%%%%%%%%%%%%%%%%%%%%%%%%%%%%%%%%%%%%%%%%%%%%%%%%%%%%%%%%%%%%%%%%%%%%%%

In this section we prove a Palais--Smale condition for minimizing sequences of blowup solutions to \eqref{nls}.  This will allow us to conclude that failure of Theorem~\ref{T:main} would imply the existence of special counterexamples that are almost periodic.  At the end of this section, we rule out these almost periodic solutions by employing a spatially truncated (one-particle) Morawetz inequality in the style of \cite{borg:scatter}.  This will complete the proof of Theorem~\ref{T:main}.

We first define operators $T_n^j$ on general functions of spacetime.  These act on linear solutions in a manner corresponding to the action of $G_n^j \exp\{it_n^j\Delta_{\Omega_n^j}\}$ on initial data in Theorem~\ref{T:LPD}.  As in that theorem, the exact definition depends on the case to which
the index $j$ conforms.  In Cases~1, 2,~and~3, we define
\begin{align*}
(T_n^j f)(t,x) :=(\lambda_n^j)^{-\frac 12}f\bigl((\lambda_n^j)^{-2} t+t_n^j, (\lambda_n^j)^{-1}(x-x_n^j)\bigr).
\end{align*}
In Case 4, we define
\begin{align*}
(T_n^j f)(t,x):=(\lambda_n^j)^{-\frac12}f\bigl((\lambda_n^j)^{-2}t+t_n^j, (\lambda_n^j)^{-1}(R_n^j)^{-1}(x-(x_n^j)^*)\bigr).
\end{align*}
Here, the parameters $\lambda_n^j, t_n^j, x_n^j, (x_n^j)^*$, and $R_n^j$ are as defined in Theorem~\ref{T:LPD}.  Using the asymptotic orthogonality
condition \eqref{E:LP5}, it is not hard to prove the following

\begin{lem}[Asymptotic decoupling]\label{L:ortho} Suppose that the parameters associated to $j,k$ are orthogonal in the sense
of \eqref{E:LP5}.  Then for any $\psi^j, \psi^k\in C_c^{\infty}(\R\times\R^3)$,
\begin{align*}
\|T_n^j\psi^j T_n^k\psi^k\|_{L_{t,x}^5(\R\times\R^3)}+\|T_n^j\psi^j \nabla(T_n^k\psi^k)\|_{L_{t,x}^{\frac 52}(\R\times\R^3)}
	+\|\nabla(T_n^j\psi^j)\nabla (T_n^k\psi^k)\|_{L_{t,x}^{\frac53}(\R\times\R^3)}
\end{align*}
converges to zero as $n\to\infty$.
\end{lem}

\begin{proof}
From a change of variables, we get
\begin{align*}
\|T_n^j&\psi^jT_n^k \psi^k\|_{L_{t,x}^5} + \|T_n^j\psi^j\nabla(T_n^k\psi^k)\|_{L_{t,x}^{\frac 52}}
+\|\nabla(T_n^j\psi^j)\nabla (T_n^k\psi^k)\|_{L_{t,x}^{\frac53}}\\
&= \|\psi^j (T_n^j)^{-1}T_n^k\psi^k\|_{L_{t,x}^5}+\|\psi^j\nabla(T_n^j)^{-1}T_n^k\psi^k\|_{L_{t,x}^{\frac 52}}
+\|\nabla \psi^j\nabla (T_n^j)^{-1}T_n^k\psi^k\|_{L_{t,x}^{\frac53}},
\end{align*}
where all spacetime norms are over $\R\times\R^3$.  Depending on the cases to which $j$ and $k$ conform, $(T_n^j)^{-1}T_n^k$ takes one of the following forms:
\begin{CI}
\item Case a): $j$ and $k$ each conform to one of Cases 1, 2, or 3.
\begin{align*}
[(T_n^j)^{-1}T_n^k\psi^k](t,x)=\bigl(\tfrac{\lambda_n^j}{\lambda_n^k}\bigr)^{\frac12}
	\psi^k\Bigl(\bigl(\tfrac{\lambda_n^j}{\lambda_n^k}\bigr)^{2} \bigl(t-\tfrac{t_n^j(\lambda_n^j)^2-t_n^k(\lambda_n^k)^2}{(\lambda_n^j)^2}\bigr),
	\tfrac{\lambda_n^j}{\lambda_n^k}\bigl(  x - \tfrac{x_n^k-x_n^j}{\lambda_n^j}\bigr) \Bigr).
\end{align*}

\item Case b): $j$ conforms to Case 1, 2, or 3 and $k$ conforms to Case 4.
\begin{align*}
[(T_n^j)^{-1}&T_n^k\psi^k](t,x) \\
&=\bigl(\tfrac{\lambda_n^j}{\lambda_n^k}\bigr)^{\frac12}
	\psi^k\Bigl(\bigl(\tfrac{\lambda_n^j}{\lambda_n^k}\bigr)^{2} \bigl(t-\tfrac{t_n^j(\lambda_n^j)^2-t_n^k(\lambda_n^k)^2}{(\lambda_n^j)^2}\bigr),
	\tfrac{\lambda_n^j}{\lambda_n^k} (R_n^k)^{-1}\bigl(  x - \tfrac{(x_n^k)^*-x_n^j}{\lambda_n^j}\bigr) \Bigr).
\end{align*}

\item Case c): $j$ conforms to Case 4 and $k$ to Case 1, 2, or 3.
\begin{align*}
[(T_n^j)^{-1}T_n^k\psi^k](t,x)=\bigl(\tfrac{\lambda_n^j}{\lambda_n^k}\bigr)^{\frac12}
	\psi^k\Bigl(\bigl(\tfrac{\lambda_n^j}{\lambda_n^k}\bigr)^{2} \bigl(t-\tfrac{t_n^j(\lambda_n^j)^2-t_n^k(\lambda_n^k)^2}{(\lambda_n^j)^2}\bigr),
	\tfrac{\lambda_n^j}{\lambda_n^k} \bigl( R_n^j x - \tfrac{x_n^k-(x_n^j)^*}{\lambda_n^j}\bigr) \Bigr).
\end{align*}

\item Case d): Both $j$ and $k$ conform to Case 4.
\begin{align*}
[(T_n^j&)^{-1}T_n^k\psi^k](t,x) \\
&=\bigl(\tfrac{\lambda_n^j}{\lambda_n^k}\bigr)^{\frac12}
	\psi^k\Bigl(\bigl(\tfrac{\lambda_n^j}{\lambda_n^k}\bigr)^{2} \bigl(t-\tfrac{t_n^j(\lambda_n^j)^2-t_n^k(\lambda_n^k)^2}{(\lambda_n^j)^2}\bigr),
	\tfrac{\lambda_n^j}{\lambda_n^k}(R_n^k)^{-1}\bigl( R_n^j x - \tfrac{(x_n^k)^*-(x_n^j)^*}{\lambda_n^j}\bigr) \Bigr).
\end{align*}
\end{CI}

We only present the details for decoupling in the $L_{t,x}^5$ norm; the argument for decoupling in the other norms is very similar.

We first assume $\frac{\lambda_n^j}{\lambda_n^k}+\frac{\lambda_n^k}{\lambda_n^j}\to\infty$. Using H\"older and a change of
variables, we estimate
\begin{align*}
\|\psi^j(T_n^j)^{-1}T_n^k\psi^k\|_{L^5_{t,x}}
&\le\min\bigl\{\|\psi^j\|_{L^\infty_{t,x}}\|(T_n^j)^{-1}T_n^k\psi^k\|_{L^5_{t,x}}+\|\psi^j\|_{L^5_{t,x}}\|(T_n^j)^{-1}T_n^k\psi^k\|_{L^\infty_{t,x}}\bigr\}\\
&\lsm \min\Bigl\{ \bigl(\tfrac{\lambda_n^j}{\lambda_n^k}\bigr)^{-\frac12}, \bigl(\tfrac{\lambda_n^j}{\lambda_n^k}\bigr)^{\frac12}\Bigr\}
 \to 0 \qtq{as} n\to \infty.
\end{align*}
Henceforth, we may assume $\frac{\lambda_n^j}{\lambda_n^k}\to \lambda_0\in (0,\infty)$.

If $\frac{|t_n^j(\lambda_n^j)^2-t_n^k(\lambda_n^k)^2|}{\lambda_n^k\lambda_n^j}\to \infty$, it is easy to see
that the temporal supports of $\psi^j$ and $(T_n^j)^{-1}T_n^k\psi^k$ are disjoint for $n$ sufficiently large.  Hence
\begin{align*}
\lim_{n\to \infty}\|\psi^j(T_n^j)^{-1}T_n^k\psi^k\|_{L^5_{t,x}} = 0.
\end{align*}

The only case left is when
\begin{align}\label{s13}
\tfrac{\lambda_n^j}{\lambda_n^k}\to\lambda_0, \quad \tfrac{t_n^j(\lambda_n^j)^2-t_n^k(\lambda_n^k)^2}{\lambda_n^k\lambda_n^j} \to t_0,
\qtq{and} \tfrac{|x_n^j-x_n^k|}{\sqrt{\lambda_n^j\lambda_n^k}}\to\infty.
\end{align}
In this case we will verify that the spatial supports of $\psi^j$ and $(T_n^j)^{-1}T_n^k\psi^k$ are disjoint for $n$ sufficiently large.
Indeed, in Case a),
\begin{align*}
\tfrac{|x_n^j-x_n^k|}{\lambda_n^j}=\tfrac{|x_n^j-x_n^k|}{\sqrt{\lambda_n^j\lambda_n^k}} \sqrt{\tfrac{\lambda_n^k}{\lambda_n^j}}\to \infty
\qtq{as} n\to\infty.
\end{align*}
In Case b), for $n$ sufficiently large we have
\begin{align*}
\tfrac{|x_n^j-(x_n^k)^*|}{\lambda_n^j}
&\ge \tfrac{|x_n^j-x_n^k|}{\lambda_n^j}-\tfrac{|x_n^k-(x_n^k)^*|}{\lambda_n^j}
\ge\tfrac{|x_n^j-x_n^k|}{\lambda_n^j}-2\tfrac{d^k_\infty}{\lambda_0},
\end{align*}
which converges to infinity as $n\to \infty$.  In Case c), for $n$ sufficiently large we have
\begin{align*}
\tfrac{|(x_n^j)^*-x_n^k|}{\lambda_n^j}
&\ge \tfrac{|x_n^j-x_n^k|}{\lambda_n^j}-\tfrac{|x_n^j-(x_n^j)^*|}{\lambda_n^j}
\ge \tfrac{|x_n^j-x_n^k|}{\lambda_n^j}-2d^j_{\infty},
\end{align*}
which converges to infinity as $n\to \infty$.  Finally, in Case d) for $n$ sufficiently large,
\begin{align*}
\tfrac{|(x_n^j)^*-(x_n^k)^*|}{\lambda_n^j}
&\ge \tfrac{|x_n^j-x_n^k|}{\lambda_n^j}-\tfrac{|x_n^j-(x_n^j)^*|}{\lambda_n^j} -\tfrac{|x_n^k-(x_n^k)^*|}{\lambda_n^j}
\ge \tfrac{|x_n^j-x_n^k|}{\lambda_n^j}-2d_\infty^j-2\tfrac{d^k_\infty}{\lambda_0},
\end{align*}
which converges to infinity as $n\to\infty$. Thus, in all cases,
\begin{align*}
\lim_{n\to \infty}\|\psi^j(T_n^j)^{-1}T_n^k\psi^k\|_{L^5_{t,x}} = 0.
\end{align*}
This completes the proof of the lemma.
\end{proof}

Theorem~\ref{T:main} claims that for any initial data $u_0\in \dot H^1_D(\Omega)$ there is a global solution $u:\R\times\Omega\to \C$ to
\eqref{nls} with $S_\R(u)\leq C(\|u_0\|_{\dot H^1_D(\Omega)})$.  Recall that for a time interval~$I$, the scattering size of $u$ on $I$ is given by
\begin{align*}
S_I(u)=\iint_{\R\times\Omega}|u(t,x)|^{10} \,dx\,dt.
\end{align*}

Supposing that Theorem~\ref{T:main} failed, there would be a critical energy $0<E_c<\infty$ so that
\begin{align}\label{LofE}
L(E)<\infty \qtq{for} E<E_c \qtq{and} L(E)=\infty \qtq{for} E\geq E_c.
\end{align}
Recall from the introduction that $L(E)$ is the supremum of $S_I(u)$ over all solutions $u:I\times\Omega\to \C$ with $E(u)\leq E$ and defined on any interval $I\subseteq \R$.

The positivity of $E_c$ follows from small data global well-posedness.  Indeed, the argument proves the stronger statement
\begin{align}\label{SbyE}
\|u\|_{\dot X^1(\R\times\Omega)}\lesssim E(u_0)^{\frac12} \quad \text{for all data with } E(u_0)\leq \eta_0,
\end{align}
where $\eta_0$ denotes the small data threshold. Recall $\dot X^1= L_{t,x}^{10} \cap L_t^5\dot H^{1,\frac{30}{11}}$.

Using the induction on energy argument together with \eqref{LofE} and the stability result Theorem~\ref{T:stability}, we now prove a compactness result for optimizing sequences of blowup solutions.

\begin{prop}[Palais--Smale condition]\label{P:PS}
Let $u_n: I_n\times\Omega\to \C$ be a sequence of solutions with $E(u_n)\to E_c$, for which there is a sequence of times $t_n\in I_n$
so that
\begin{align*}
\lim_{n\to\infty} S_{\ge t_n}(u_n)=\lim_{n\to\infty}S_{\le t_n}(u_n)=\infty.
\end{align*}
Then the sequence $u_n(t_n)$ has a subsequence that converges strongly in $\dot H^1_D(\Omega)$.
\end{prop}

\begin{proof}
Using the time translation symmetry of \eqref{nls}, we may take $t_n\equiv0$ for all $n$; thus,
\begin{align}\label{scat diverge}
\lim_{n\to\infty} S_{\ge 0}(u_n)=\lim_{n\to \infty} S_{\le 0} (u_n)=\infty.
\end{align}
Applying Theorem~\ref{T:LPD} to the bounded sequence $u_n(0)$ in $\dot H^1_D(\Omega)$ and passing to a subsequence if necessary, we
obtain the linear profile decomposition
\begin{align}\label{s0}
u_n(0)=\sum_{j=1}^J \phi_n^j+w_n^J
\end{align}
with the properties stated in that theorem. In particular, for any finite $0\leq J \leq J^*$ we have the energy decoupling condition
\begin{align}\label{s01}
\lim_{n\to \infty}\Bigl\{E(u_n)-\sum_{j=1}^J E(\phi_n^j)-E(w_n^J)\Bigr\}=0.
\end{align}

To prove the proposition, we need to show that $J^*=1$, that $w_n^1\to 0$ in $\dot H^1_D(\Omega)$, that the only profile $\phi^1_n$ conforms to Case~1, and that $t_n^1\equiv 0$.  All other possibilities will be shown to contradict \eqref{scat diverge}.  We discuss two scenarios:

\textbf{Scenario I:} $\sup_j \limsup_{n\to \infty} E(\phi_n^j) =E_c$.

From the non-triviality of the profiles, we have $\liminf_{n\to \infty} E(\phi_n^j)>0$ for every finite $1\leq j\leq J^*$; indeed,
$\|\phi_n^j\|_{\dot H^1_D(\Omega)}$ converges to $\|\phi^j\|_{\dot H^1}$.  Thus, passing to a subsequence, \eqref{s01} implies that there is a single profile in the decomposition \eqref{s0} (that is, $J^*=1$) and we can write
\begin{equation}\label{s11}
u_n(0)=\phi_n +w_n \qtq{with} \lim_{n\to \infty} \|w_n\|_{\dot H_D^1(\Omega)}=0.
\end{equation}
If $\phi_n$ conforms to Cases 2, 3, or 4, then by the Theorems~\ref{T:embed2}, \ref{T:embed3}, or \ref{T:embed4}, there are global solutions $v_n$ to $\text{NLS}_\Omega$ with data $v_n(0)=\phi_n$ that admit a uniform spacetime bound.  By Theorem~\ref{T:stability}, this spacetime bound extends to the solutions $u_n$ for $n$ large enough.  However, this contradicts \eqref{scat diverge}.  Therefore, $\phi_n$ must conform to Case~1 and \eqref{s11} becomes
\begin{equation}\label{s11'}
u_n(0)=e^{it_n\lambda_n^2\Delta_\Omega}\phi +w_n \qtq{with} \lim_{n\to \infty} \|w_n\|_{\dot H_D^1(\Omega)}=0
\end{equation}
and $t_n\equiv 0$ or $t_n\to \pm \infty$.  If $t_n\equiv 0$, then we obtain the desired compactness.  Thus, we only need to preclude that $t_n\to\pm\infty$.

Let us suppose $t_n\to \infty$; the case $t_n\to -\infty$ can be treated symmetrically.  In this case, the Strichartz inequality and the monotone convergence theorem yield
\begin{align*}
S_{\ge 0}(e^{it\Delta_\Omega}u_n(0))=S_{\ge 0}(e^{i(t+t_n\lambda_n^2)\Delta_{\Omega}}\phi+e^{it\Delta_{\Omega}}w_n) \to 0 \qtq{as} n\to \infty.
\end{align*}
By the small data theory, this implies that $S_{\geq 0}(u_n)\to 0$, which contradicts \eqref{scat diverge}.

\textbf{Scenario 2:} $\sup_j \limsup_{n\to \infty} E(\phi_n^j) \leq E_c-2\delta$ for some $\delta>0$.

We first observe that for each finite $J\leq J^*$ we have $E(\phi_n^j) \leq E_c-\delta$ for all $1\leq j\leq J$ and $n$ sufficiently large.  This is important for constructing global nonlinear profiles for $j$ conforming to Case~1, via the induction on energy hypothesis~\eqref{LofE}.

If $j$ conforms to Case~1 and $t_n^j\equiv 0$, we define $v^j:I^j\times\Omega\to \C$ to be the maximal-lifespan solution to
\eqref{nls} with initial data $v^j(0)=\phi^j$.  If instead $t_n^j\to \pm \infty$, we define $v^j:I^j\times\Omega\to \C$ to be the
maximal-lifespan solution to \eqref{nls} which scatters to $e^{it\Delta_\Omega}\phi^j$ as $t\to \pm\infty$.
Now define $v_n^j(t,x):=v^j(t+t_n^j(\lambda_n^j)^2,x)$. Then $v_n^j$ is also a solution to \eqref{nls} on the time interval
$I_n^j:=I^j-\{t_n^j(\lambda_n^j)^2\}$.  In particular, for $n$ sufficiently large we have $0\in I_n^j$ and
\begin{align}\label{bb1}
\lim_{n\to\infty}\|v_n^j(0)-\phi_n^j\|_{\dot H^1_D(\Omega)}=0.
\end{align}
Combining this with $E(\phi_n^j) \leq E_c-\delta$ and \eqref{LofE}, we deduce that for $n$ sufficiently large, $v_n^j$ (and also $v^j$) are global solutions that obey
\begin{align*}
S_\R(v^j)=S_\R(v_n^j)\le L(E_c-\delta)<\infty.
\end{align*}
Combining this with the Strichartz inequality shows that all Strichartz norms of $v_n^j$ are finite and, in particular, the $\dot X^1$ norm.
This allows us to approximate $v_n^j$ in $\dot X^1(\R\times\Omega)$ by $C_c^\infty(\R\times\R^3)$ functions.   More precisely, for any $\eps>0$ there exist $N_\eps^j\in \N$ and $\psi^j_\eps\in C_c^\infty(\R\times\R^3)$ so that for $n\geq N_\eps^j$ we have
\begin{align}\label{ap case1}
\|v_n^j - T_n^j\psi^j_\eps\|_{\dot X^1(\R\times\R^3)}<\eps.
\end{align}
Speciffically, choosing $\tilde\psi_\eps^j\in C_c^\infty(\R\times\R^3)$ such that
\begin{align*}
\|v^j-\tilde \psi_\eps^j\|_{\dot X^1(\R\times\R^3)}<\tfrac \eps2, \qtq{we set} \psi_\eps^j(t,x):=(\lambda_\infty^j)^{\frac 12}\tilde\psi_\eps^j
\bigl((\lambda_\infty^j)^2 t, \lambda_\infty^j x+x_\infty^j\bigr).
\end{align*}

When $j$ conforms to Cases 2, 3, or 4, we apply the nonlinear embedding theorems of the previous section to construct the nonlinear profiles.  More precisely, let $v_n^j$ be the global solutions to $\text{NLS}_\Omega$ constructed in Theorems~\ref{T:embed2}, \ref{T:embed3}, or \ref{T:embed4}, as appropriate. In particular, these $v_n^j$ also obey \eqref{ap case1} and $\sup_{n,j} S_\R(v_n^j)<\infty$.

In all cases, we may use \eqref{SbyE} together with our bounds on the spacetime norms of $v_n^j$ and the finiteness of $E_c$ to deduce
\begin{align}\label{s2}
\|v_n^j\|_{\dot X^1(\R\times\Omega)}\lsm_{E_c, \delta} E(\phi_n^j)^{\frac12} \lsm_{E_c, \delta}1.
\end{align}
Combining this with \eqref{s01} we deduce
\begin{align}\label{s2lim}
\limsup_{n\to \infty} \sum_{j=1}^J \|v_n^j\|_{\dot X^1(\R\times\Omega)}^2
\lsm_{E_c, \delta} \limsup_{n\to \infty} \sum_{j=1}^J E(\phi_n^j) \lsm_{E_c,\delta} 1,
\end{align}
uniformly for finite $J\leq J^*$.

The asymptotic orthogonality condition \eqref{E:LP5} gives rise to asymptotic decoupling of the nonlinear profiles.

\begin{lem}[Decoupling of nonlinear profiles] \label{L:npd} For $j\neq k$ we have 
\begin{align*}
\lim_{n\to \infty} \|v_n^j v_n^k\|_{L_{t,x}^5(\R\times\Omega)} +\|v_n^j \nabla v_n^k\|_{L_{t,x}^{\frac52}(\R\times\Omega)}
+\|\nabla v_n^j \nabla v_n^k\|_{L_t^{\frac52} L_x^{\frac{15}{11}}(\R\times\Omega)}=0.
\end{align*}
\end{lem}

\begin{proof}
Recall that for any $\eps>0$ there exist $N_\eps\in \N$ and $\psi_\eps^j,\psi_\eps^k\in C_c^\infty(\R\times\R^3)$ so that
\begin{align*}
\|v_n^j - T_n^j\psi^j_\eps\|_{\dot X^1(\R\times\R^3)} + \|v_n^k - T_n^k\psi^k_\eps\|_{\dot X^1(\R\times\R^3)}<\eps.
\end{align*}
Thus, using \eqref{s2} and Lemma~\ref{L:ortho} we get
\begin{align*}
\|v_n^j v_n^k\|_{L_{t,x}^5} &\leq \|v_n^j(v_n^k-T_n^k\psi_{\eps}^k)\|_{L_{t,x}^5}+\|(v_n^j-T_n^j\psi_\eps^j)T_n^k\psi_{\eps}^k\|_{L_{t,x}^5}
+\|T_n^j\psi_\eps^j\, T_n^k\psi_\eps^k\|_{L_{t,x}^5}\\
&\lsm \|v^j_n\|_{\dot X^1} \|v_n^k - T_n^k\psi^k_\eps\|_{\dot X^1} + \|v_n^j - T_n^j\psi^j_\eps\|_{\dot X^1} \|\psi_\eps^k\|_{\dot X^1} + \|T_n^j\psi_\eps^j\, T_n^k\psi_\eps^k\|_{L_{t,x}^5}\\
&\lsm_{E_c,\delta} \eps + o(1) \qtq{as}n\to \infty.
\end{align*}
As $\eps>0$ was arbitrary, this proves the first asymptotic decoupling statement.

The second decoupling statement follows analogously.  For the third assertion, a little care has to be used to estimate the error terms, due to the asymmetry of the spacetime norm and to the restrictions placed by Theorem~\ref{T:Sob equiv}.  Using the same argument as above and interpolation, we estimate
\begin{align*}
\|\nabla v_n^j \nabla v_n^k\|_{L_t^{\frac52} L_x^{\frac{15}{11}}}
&\leq \|\nabla v_n^j(\nabla v_n^k-\nabla T_n^k\psi_{\eps}^k)\|_{L_t^{\frac52} L_x^{\frac{15}{11}}}+\|(\nabla v_n^j-\nabla T_n^j\psi_\eps^j)\nabla T_n^k\psi_{\eps}^k\|_{L_t^{\frac52} L_x^{\frac{15}{11}}}\\
&\quad +\|\nabla T_n^j\psi_\eps^j\, \nabla T_n^k\psi_\eps^k\|_{L_t^{\frac52} L_x^{\frac{15}{11}}}\\
&\lsm_{E_c,\delta}\eps + \|\nabla T_n^j\psi_\eps^j\, \nabla T_n^k\psi_\eps^k\|_{L_{t,x}^{\frac53}}^{\frac23}\|\nabla T_n^j\psi_\eps^j\, \nabla T_n^k\psi_\eps^k\|_{L_t^\infty L_x^1}^{\frac13}\\
&\lsm_{E_c,\delta}\eps + \|\nabla T_n^j\psi_\eps^j\, \nabla T_n^k\psi_\eps^k\|_{L_{t,x}^{\frac53}}^{\frac23}\|\nabla\psi_\eps^j\|_{L_x^2}^{\frac13}\|\nabla\psi_\eps^k\|_{L_x^2}^{\frac13}\\
&\lsm_{E_c,\delta}\eps + o(1) \qtq{as}n\to \infty,
\end{align*}
where we used Lemma~\ref{L:ortho} in the last step. As $\eps>0$ was arbitrary, this proves the last decoupling statement.
\end{proof}

As a consequence of this decoupling we can bound the sum of the nonlinear profiles in $\dot X^1$, as follows:
\begin{align}\label{sum vnj}
\limsup_{n\to \infty} \Bigl\|\sum_{j=1}^J v_n^j\Bigr\|_{\dot X^1(\R\times\Omega)}\lsm_{E_c,\delta}1 \quad\text{uniformly for finite $J\leq J^*$}.
\end{align}
Indeed, by Young's inequality, \eqref{s2}, \eqref{s2lim}, and Lemma~\ref{L:npd},
\begin{align*}
S_\R\Bigl(\sum_{j=1}^J v_n^j\Bigr)
&\lsm \sum_{j=1}^J S_\R(v_n^j)+J^8\sum_{j\neq k}\iint_{\R\times\Omega}|v_n^j||v_n^k|^9 \,dx\,dt\\
&\lsm_{E_c,\delta}1 + J^8 \sum_{j\neq k}\|v_n^jv_n^k\|_{L_{t,x}^5}\|v_n^k\|_{L_{t,x}^{10}}^8\\
&\lsm_{E_c,\delta}1 + J^8 o(1) \qtq{as} n\to \infty.
\end{align*}
Similarly,
\begin{align*}
\Bigl\|\sum_{j=1}^J \nabla v_n^j\Bigr\|_{L_t^5L_x^{\frac{30}{11}}}^2 &= \Bigl\|\Bigl(\sum_{j=1}^J \nabla v_n^j\Bigr)^2\Bigr\|_{L_t^{\frac52}L_x^{\frac{15}{11}}}\lsm \sum_{j=1}^J \|\nabla v_n^j\|_{L_t^5L_x^{\frac{30}{11}}}^2 + \sum_{j\neq k} \|\nabla v_n^j \nabla v_n^k\|_{L_t^{\frac52} L_x^{\frac{15}{11}}}\\
&\lsm_{E_c,\delta}1 + o(1)\qtq{as} n\to \infty.
\end{align*}
This completes the proof of \eqref{sum vnj}.  The same argument combined with \eqref{s01} shows that given $\eta>0$, there exists $J'=J'(\eta)$ such that
\begin{align}\label{sum vnj tail}
\limsup_{n\to \infty} \Bigl\|\sum_{j=J'}^J v_n^j\Bigr\|_{\dot X^1(\R\times\Omega)}\leq \eta \quad \text{uniformly in $J\geq J'$}.
\end{align}

Now we are ready to construct an approximate solution to $\text{NLS}_\Omega$.  For each $n$ and $J$, we define
\begin{align*}
u_n^J:=\sum_{j=1}^J v_n^j+e^{it\Delta_{\Omega}}w_n^J.
\end{align*}
Obviously $u_n^J$ is defined globally in time.  In order to apply Theorem~\ref{T:stability}, it suffices to verify the following
three claims for $u_n^J$:

Claim 1: $\|u_n^J(0)-u_n(0)\|_{\dot H^1_D(\Omega)}\to 0$ as $n\to \infty$ for any $J$.

Claim 2: $\limsup_{n\to \infty} \|u_n^J\|_{\dot X^1(\R\times\Omega)}\lsm_{E_c, \delta} 1$ uniformly in $J$.

Claim 3: $\lim_{J\to\infty}\limsup_{n\to\infty}\|(i\partial_t+\Delta_{\Omega})u_n^J-|u_n^J|^4u_n^J\|_{\dot N^1(\R\times\Omega)}=0$.

The three claims imply that for sufficiently large $n$ and $J$, $u_n^J$ is an approximate solution to \eqref{nls} with finite scattering size, which asymptotically matches $u_n(0)$ at time $t=0$.  Using the stability result Theorem~\ref{T:stability} we see that for $n, J$ sufficiently large, the solution $u_n$ inherits\footnote{In fact, we obtain a nonlinear profile decomposition for the sequence of solutions $u_n$ with an error that goes to zero in $L^{10}_{t,x}$.}  the spacetime bounds of $u_n^J$, thus contradicting \eqref{scat diverge}.  Therefore, to complete the treatment of the second scenario, it suffices to verify the three claims above.

The first claim follows trivially from \eqref{s0} and \eqref{bb1}.  To derive the second claim, we use \eqref{sum vnj} and the Strichartz inequality, as follows:
\begin{align*}
\limsup_{n\to \infty}\|u_n^J\|_{\dot X^1(\R\times\Omega)}&\lsm \limsup_{n\to \infty}\Bigl\|\sum_{j=1}^J v_n^j\Bigr\|_{\dot X^1(\R\times\Omega)}+\limsup_{n\to \infty}\|w_n^J\|_{\dot H^1_D(\Omega)}\lsm_{E_c,\delta}1.
\end{align*}

Next we verify the third claim.  Adopting the notation $F(z)=|z|^4 z$, a direct computation gives
\begin{align}
(i\partial_t+\Delta_\Omega)u_n^J-F(u_n^J)
&=\sum_{j=1}^JF(v_n^j)-F(u_n^J)\notag\\
&=\sum_{j=1}^J F(v_n^j)-F\biggl(\sum_{j=1}^J v_n^j\biggr)+F\bigl(u_n^J-e^{it\Delta_{\Omega}}w_n^J\bigr)-F(u_n^J)\label{s6}.
\end{align}
Taking the derivative, we estimate
\begin{align*}
\biggl|\nabla\biggl\{ \sum_{j=1}^JF(v_n^j)-F\biggl(\sum_{j=1}^J v_n^j\biggr)\biggr\} \biggr|
	\lesssim_{J} \sum_{j\neq k}|\nabla v_n^j||v_n^k|^4+|\nabla v_n^j||v_n^j|^3|v_n^k|
\end{align*}
and hence, using \eqref{s2} and Lemma~\ref{L:npd},
\begin{align*}
\biggl\|\nabla\biggl\{ \sum_{j=1}^JF(v_n^j)-F\biggl(\sum_{j=1}^J v_n^j\biggr)\bigg\}\biggr \|_{\dot N^0(\R\times \Omega)}
&\lesssim_{J} \sum_{j\neq k} \bigl\| |\nabla v_n^j| |v_n^k|^4 + |\nabla v_n^j||v_n^j|^3|v_n^k| \bigr\|_{L^{\frac {10}7}_{t,x}} \\
&\lesssim_{J} \sum_{j\neq k} \bigl\|\nabla v_n^j v_n^k\bigr\|_{L^{\frac 52}_{t,x}}\bigl[\|v_n^k\|_{L^{10}_{t,x}}^3 +\|v_n^j\|_{L^{10}_{t,x}}^3\bigr]\\
&\lesssim_{J,E_c,\delta} o(1) \qtq{as} n\to \infty.
\end{align*}
Thus, using the equivalence of Sobolev spaces Theorem~\ref{T:Sob equiv}, we obtain
\begin{equation}\label{s7}
\lim_{J\to \infty}\limsup_{n\to \infty} \biggl\| \sum_{j=1}^J F(v_n^j)-F\biggl(\sum_{j=1}^J v_n^j\biggr) \biggr\|_{\dot N^1(\R\times \Omega)}=0.
\end{equation}

We now turn to estimating the second difference in \eqref{s6}.  We will show
\begin{equation}\label{s8}
\lim_{J\to \infty} {\limsup_{n\to \infty}} \bigl\| F(u_n^J-e^{it\Delta_{\Omega}}w_n^J)-F(u_n^J) \bigr\|_{\dot N^1(\R\times \Omega)}=0.
\end{equation}
By the equivalence of Sobolev spaces, it suffices to estimate the usual gradient of the difference in dual Strichartz spaces.  Taking the derivative, we get
\begin{align*}
\bigl|\nabla \bigl\{F\bigl(u_n^J-e^{it\Delta_{\Omega}}w_n^J\bigr)-F(u_n^J)\bigr\}\bigr|
&\lsm \sum_{k=0}^3 |\nabla u_n^J| |u_n^J|^k |e^{it\Delta_{\Omega}}w_n^J|^{4-k} \\
&\quad + \sum_{k=0}^4 |\nabla e^{it\Delta_{\Omega}}w_n^J| |u_n^J|^k |e^{it\Delta_{\Omega}}w_n^J|^{4-k}.
\end{align*}
Using H\"older and the second claim, we obtain
\begin{align*}
\sum_{k=0}^3\bigl\||\nabla u_n^J| |u_n^J|^k |e^{it\Delta_{\Omega}}w_n^J|^{4-k}\bigr\|_{L^{\frac53}_t L^{\frac{30}{23}}_x}
&\lsm \sum_{k=0}^3\|\nabla u_n^J\|_{L^5_tL^{\frac {30}{11}}_x} \|u_n^J\|_{L_{t,x}^{10}}^k \|e^{it\Delta_{\Omega}}w_n^J\|_{L_{t,x}^{10}}^{4-k}\\
&\lsm_{E_c, \delta} \sum_{k=0}^3\|e^{it\Delta_{\Omega}}w_n^J\|_{L_{t,x}^{10}}^{4-k},
\end{align*}
which converges to zero as $n,J\to \infty$ by \eqref{E:LP1}.  The same argument gives
\begin{align*}
\lim_{J\to \infty}\limsup_{n\to \infty} \sum_{k=0}^3\bigl\||\nabla e^{it\Delta_{\Omega}}w_n^J| |u_n^J|^k |e^{it\Delta_{\Omega}}w_n^J|^{4-k}\bigr\|_{L^{\frac53}_t L^{\frac{30}{23}}_x}=0.
\end{align*}

This leaves us to prove
\begin{align}\label{708}
\lim_{J\to \infty}\limsup_{n\to \infty}\||\nabla e^{it\Delta_{\Omega}}w_n^J| |u_n^J |^4\|_{L_{t,x}^{\frac{10}7}}=0.
\end{align}
Using H\"older, the second claim, Theorem~\ref{T:Sob equiv}, and the Strichartz inequality, we get
\begin{align*}
\||\nabla e^{it\Delta_{\Omega}}w_n^J| |u_n^J|^4\|_{L_{t,x}^{\frac{10}7}}
&\lsm \|u_n^J \nabla e^{it\Delta_{\Omega}}w_n^J \|_{L_{t,x}^{\frac52}} \|u_n^J\|_{L_{t,x}^{10}}^3\\
&\lsm_{E_c,\delta}\|e^{it\Delta_{\Omega}}w_n^J\nabla e^{it\Delta_{\Omega}}w_n^J \|_{L_{t,x}^{\frac52}} +\Bigl\|\sum_{j=1}^J v_n^j \nabla e^{it\Delta_{\Omega}}w_n^J \Bigr\|_{L_{t,x}^{\frac52}}\\
&\lsm_{E_c,\delta}\|e^{it\Delta_{\Omega}}w_n^J\|_{L_{t,x}^{10}}^{\frac2{11}}\|e^{it\Delta_{\Omega}}w_n^J\|_{L_t^{\frac92}L_x^{54}}^{\frac9{11}}\|\nabla e^{it\Delta_{\Omega}}w_n^J \|_{L_t^5 L_x^{\frac{30}{11}}} \\
&\quad +\Bigl\|\sum_{j=1}^J v_n^j \nabla e^{it\Delta_{\Omega}}w_n^J \Bigr\|_{L_{t,x}^{\frac52}}\\
&\lsm_{E_c,\delta}\|e^{it\Delta_{\Omega}}w_n^J\|_{L_{t,x}^{10}}^{\frac2{11}}+\Bigl\|\sum_{j=1}^J v_n^j \nabla e^{it\Delta_{\Omega}}w_n^J \Bigr\|_{L_{t,x}^{\frac52}}.
\end{align*}
By \eqref{E:LP1}, the contribution of the first term to \eqref{708} is acceptable.  We now turn to the second term.

By \eqref{sum vnj tail},
\begin{align*}
\limsup_{n\to \infty} \Bigl\| \Bigl(\sum_{j=J'}^J v_n^j\Bigr) \nabla e^{it\Delta_\Omega} w_n^J  \Bigr\|_{L_{t,x}^{\frac52}}
&\lsm\limsup_{n\to \infty}\Bigl\|\sum_{j=J'}^J v_n^j\Bigr\|_{\dot X^1}\|\nabla e^{it\Delta_\Omega} w_n^J\|_{L_t^5L_x^{\frac{30}{11}}}\lsm_{E_c,\delta} \eta,
\end{align*}
where $\eta>0$ is arbitrary and $J'=J'(\eta)$ is as in \eqref{sum vnj tail}.  Thus, proving \eqref{708} reduces to showing
\begin{align}\label{834}
\lim_{J\to\infty} \limsup_{n \to \infty}\| v_n^j \nabla e^{it\Delta_\Omega} w_n^J  \|_{L_{t,x}^{\frac52}}=0  \qtq{for each} 1\leq j\leq J'.
\end{align}

To this end, fix $1\leq j\leq J'$.  Let $\eps>0$, $\psi_\eps^j\in C^\infty_c(\R\times\R^3)$ be as in \eqref{ap case1}, and let $R,T>0$ be such that
$\psi^j_\eps$ is supported in the cylinder $[-T,T]\times\{|x|\leq R\}$.  Then
$$
\supp(T_n^j \psi^j_\eps ) \subseteq [(\lambda_n^j)^2 (-T-t_n^j), (\lambda_n^j)^2 (T-t_n^j)]\times\{|x -x_n^j|\leq \lambda_n^j R\}
$$
and $\|T_n^j \psi^j_\eps\|_{L^\infty_{t,x}} \lesssim (\lambda_n^j)^{-\frac12} \| \psi^j_\eps\|_{L^\infty_{t,x}}$.  If $j$ conforms to Case~4,
then $x_n^j$ above should be replaced by $(x_n^j)^*$.  Thus, using Corollary~\ref{C:Keraani3.7} we deduce that
\begin{align*}
\| (T_n^j \psi^j_\eps) \nabla e^{it\Delta_\Omega} w_n^J  \|_{L_{t,x}^{\frac52}}
&\lesssim T^{\frac{31}{180}} R^{\frac7{45}}  \| \psi^j_\eps\|_{L^\infty_{t,x}} \| e^{it\Delta_\Omega} w_n^J \|_{L^{10}_{t,x}}^{\frac1{18}}
     	\| w_n^J \|_{\dot H^1_D(\Omega)}^{\frac{17}{18}} \\
&\lesssim_{\psi^j_\eps,E_c}  \| e^{it\Delta_\Omega} w_n^J \|_{L^{10}_{t,x}}^{\frac1{18}}.
\end{align*}
Combining this with \eqref{ap case1} and using Theorem~\ref{T:Sob equiv} and Strichartz, we deduce that
\begin{align*}
\| v_n^j \nabla e^{it\Delta_\Omega} w_n^J  \|_{L_{t,x}^{\frac52}}
& \lesssim  \| v^j_n - T_n^j \psi^j_\eps \|_{\dot X^1} \| \nabla e^{it\Delta_\Omega} w_n^J \|_{L_t^5 L_x^{\frac{30}{11}}} +
	C(\psi^j_\eps,E_c) \| e^{it\Delta_\Omega} w_n^J \|_{L^{10}_{t,x}}^{\frac1{18}} \\
&\lesssim \eps E_c +  C(\psi^j_\eps,E_c) \| e^{it\Delta_\Omega} w_n^J \|_{L^{10}_{t,x}}^{\frac1{18}}.
\end{align*}
Using \eqref{E:LP1} we get $\text{LHS\eqref{834}} \lesssim_{E_c} \eps$. As $\eps>0$ was arbitrary, this proves \eqref{834}.

This completes the proof of \eqref{708} and with it, the proof of \eqref{s8}.  Combining \eqref{s7} and \eqref{s8} yields the third claim.
This completes the treatment of the second scenario and so the proof of the proposition.
\end{proof}

As an immediate consequence of the Palais--Smale condition, we obtain that the failure of Theorem~\ref{T:main} implies the existence
of almost periodic counterexamples:

\begin{thm}[Existence of almost periodic solutions]\label{T:mmbs}
Suppose Theorem \ref{T:main} fails to be true. Then there exist a critical energy \/$0<E_c<\infty$ and a global solution $u$ to \eqref{nls} with
$E(u)=E_c$, which blows up in both time directions in the sense that
$$
S_{\ge 0}(u)=S_{\le 0}(u)=\infty,
$$
and whose orbit $\{ u(t):\, t\in \R\}$ is precompact in $\dot H_D^1(\Omega)$.  Moreover, there exists $R>0$ so that
\begin{align}\label{E:unif L6}
\int_{\Omega\cap\{|x|\leq R\}} |u(t,x)|^6\, dx\gtrsim 1 \quad \text{uniformly for } t\in \R.
\end{align}
\end{thm}

\begin{proof}
If Theorem~\ref{T:main} fails to be true, there must exist a critical energy $0<E_c<\infty$ and a sequence of solutions $u_n:I_n\times\Omega\to \C$ such that $E(u_n)\to E_c$ and $S_{I_n}(u_n)\to \infty$.  Let $t_n\in I_n$ be such that $S_{\ge t_n}(u_n)=S_{\le t_n}(u_n)=\frac 12 S_{I_n}(u_n)$; then
\begin{align}\label{s12}
\lim_{n\to\infty} S_{\ge t_n}(u_n)=\lim_{n\to\infty}S_{\le t_n}(u_n)=\infty.
\end{align}
Applying Proposition \ref{P:PS} and passing to a subsequence, we find $\phi\in \dot H^1_D(\Omega)$ such that $u_n(t_n)\to \phi$
in $\dot H^1_D(\Omega)$.  In particular, $E(\phi)=E_c$.  We take $u:I\times\Omega\to \C$ to be the maximal-lifespan solution to \eqref{nls} with initial data $u(0)=\phi$.  From the stability result Theorem~\ref{T:stability} and \eqref{s12}, we get
\begin{align}\label{s14}
S_{\ge 0}(u)=S_{\le 0}(u)=\infty.
\end{align}

Next we prove that the orbit of $u$ is precompact in $\dot H_D^1(\Omega)$.  For any sequence $\{t'_n\}\subset I$, \eqref{s14} implies
$S_{\ge t_n'}(u)=S_{\le t_n'}(u)=\infty$.  Thus by Proposition~\ref{P:PS}, we see that $u(t_n')$ admits a subsequence that converges strongly in
$\dot H^1_D(\Omega)$. Therefore, $\{u(t): t\in \R\}$ is precompact in $\dot H^1_D(\Omega)$.

We now show that the solution $u$ is global in time. We argue by contradiction; suppose, for example, that $\sup I<\infty$.  Let $t_n\to \sup I$.  Invoking Proposition~\ref{P:PS} and passing to a subsequence, we find $\phi\in \dot H^1_D(\Omega)$ such that $u(t_n)\to \phi$ in
$\dot H^1_D(\Omega)$.  From the local theory, there exist $T=T(\phi)>0$ and a unique solution $v:[-T,T]\times\Omega\to \C$ to \eqref{nls} with initial data $v(0)=\phi$ such that $S_{[-T,T]}v<\infty$.  By the stability result Theorem~\ref{T:stability}, for $n$ sufficiently large we find a unique solution
$\tilde u_n:[t_n-T,t_n+T]\times\Omega\to \C$ to \eqref{nls} with data $\tilde u_n(t_n)=u(t_n)$ and $S_{[t_n-T,t_n+T]}(\tilde u_n)<\infty$.  From uniqueness of solutions to \eqref{nls}, we must have $\tilde u_n=u$.  Thus taking $n$ sufficiently large, we see that $u$ can be extended beyond $\sup I$, which contradicts the fact that $I$ is the maximal lifespan of $u$.

Finally, we prove the uniform lower bound \eqref{E:unif L6}.  We again argue by contradiction. Suppose there exist sequences $R_n\to \infty$ and
$\{t_n\}\subset \R$ along which
$$
\int_{\Omega\cap\{|x|\leq R_n\}} |u(t_n,x)|^6\, dx \to 0.
$$
Passing to a subsequence, we find $u(t_n)\to \phi$ in $\dot H^1_D(\Omega)$ for some non-zero $\phi\in \dot H^1_D(\Omega)$.
Note that if $\phi$ were zero, then the solution $u$ would have energy less than the small data threshold, which would contradict \eqref{s14}.
By Sobolev embedding, $u(t_n)\to\phi$ in $L^6$, and since $R_n\to\infty$,
$$
\int_\Omega |\phi(x)|^6\,dx = \lim_{n\to\infty}  \int_{\Omega\cap\{|x|\leq R_n\}} |\phi(x)|^6\, dx = \lim_{n\to\infty} \int_{\Omega\cap\{|x|\leq R_n\}} |u(t_n,x)|^6\, dx =0.
$$
This contradicts the fact that $\phi \neq 0$ and  completes the proof of the theorem.
\end{proof}

Finally, we are able to prove the main theorem.

\begin{proof}[Proof of Theorem~\ref{T:main}]
We argue by contradiction.  Suppose Theorem~\ref{T:main} fails.  By Theorem~\ref{T:mmbs}, there exists a minimal energy blowup solution
$u$ that is global in time, whose orbit is precompact in $\dot H_D^1(\Omega)$, and that satisfies
$$
\int_{\Omega\cap\{|x|\leq R\}} |u(t,x)|^6\, dx \gtrsim 1 \quad \text{uniformly for } t\in \R
$$
for some large $R>1$.  Integrating over a time interval of length $|I|\geq 1$, we obtain
\begin{align*}
|I|\lsm R\int_{I}\int_{\Omega\cap\{|x|\leq R\}} \frac{|u(t,x)|^6}{|x|}\,dx\,dt \lsm R\int_{I}\int_{\Omega\cap\{|x|\leq R|I|^{\frac12}\}}\frac{|u(t,x)|^6}{|x|}\,dx\,dt.
\end{align*}

On the other hand, for $R|I|^{\frac12}\geq 1$, the Morawetz inequality Lemma~\ref{L:morawetz} gives
\begin{align*}
\int_{I}\int_{\Omega\cap\{|x|\leq R|I|^{\frac12}\}}\frac{|u(t,x)|^6}{|x|}\,dx\,dt \lsm R|I|^{\frac12},
\end{align*}
with the implicit constant depending only on $E(u)=E_c$.

Taking $I$ sufficiently large depending on $R$ and $E_c$ (which is possible since $u$ is global in time), we derive a contradiction.
This completes the proof of Theorem~\ref{T:main}.
\end{proof}

%%%%%%%%%%%%%%%%%%%%%%%%%%%%%%%%%%%%%%%%%%%%%%%%%%%%%%%%%%%%%%%%%%%%%%%%%%%%%%%%%%%%%%%%%%%%%%%%%%%%%%%%%%%%%%%%%%%%%%%%%%%%%%
%%%%%%%%%%%%%%%%%%%%%%%%%%%%%%%%%%%%%%%%%%%%%%%%%%%%%%%%%%%%%%%%%%%%%%%%%%%%%%%%%%%%%%%%%%%%%%%%%%%%%%%%%%%%%%%%%%%%%%%%%%%%%%

\end{document}